\documentclass[oneside]{scrbook}

\usepackage[T1]{fontenc}
\usepackage[latin9]{inputenc}
\usepackage{verbatim}
\usepackage{prettyref}
\usepackage{mathtools}
\usepackage{amsmath}
\usepackage{amsthm}
\usepackage{amssymb}

\usepackage{mathdots}
\usepackage{csquotes}
\usepackage{pgffor}
\usepackage{makeidx}
\makeindex
\usepackage{esint}
\usepackage[backend=biber,refsection=chapter,style=alphabetic,citestyle=alphabetic,giveninits=true,natbib=true,doi=false,isbn=false,url=false]{biblatex}
\DeclareFontFamily{U}{mathx}{\hyphenchar\font45}
\DeclareFontShape{U}{mathx}{m}{n}{
      <5> <6> <7> <8> <9> <10>
      <10.95> <12> <14.4> <17.28> <20.74> <24.88>
      mathx10
      }{}
\DeclareSymbolFont{mathx}{U}{mathx}{m}{n}
\DeclareFontSubstitution{U}{mathx}{m}{n}
\DeclareMathAccent{\widecheck}{0}{mathx}{"71}
\DeclareMathAccent{\wideparen}{0}{mathx}{"75}
 
\makeatletter
\theoremstyle{definition}
\newtheorem*{defn*}{\protect\definitionname}
\theoremstyle{plain}
\newtheorem{thm}{\protect\theoremname}[section]
\theoremstyle{remark}
\newtheorem{rem}[thm]{\protect\remarkname}
\theoremstyle{plain}
\newtheorem{lem}[thm]{\protect\lemmaname}
\theoremstyle{plain}
\newtheorem{cor}[thm]{\protect\corollaryname}
\theoremstyle{plain}
\newtheorem{prop}[thm]{\protect\propositionname}
\theoremstyle{plain}

\theoremstyle{definition}
\newtheorem{example}[thm]{\protect\examplename}
\theoremstyle{definition}
\newtheorem{xca}[thm]{\protect\exercisename}

\@ifundefined{date}{}{\date{}}

\makeatother
\usepackage{faktor}
\usepackage{prettyref}
\usepackage{enumerate}
\usepackage{dsfont}
\usepackage{color}

\usepackage{tikz}
\usepackage[]{circuitikz} 
\usepackage{wrapfig}

\allowdisplaybreaks

\renewcommand{\labelenumi}{{\upshape (\alph{enumi})}}
\renewcommand{\theenumi}{{\upshape (\alph{enumi})}}

\providecommand{\set}[2]{\left\{ #1 \,;\, #2 \right\}}
\providecommand{\scp}[2]{\left\langle #1 ,#2\right\rangle}

\providecommand{\sbb}[1]{\mathrm{s}_{\mathrm{b}}\left( #1 \right)}
\providecommand{\Lm}[1]{L_{2,#1}}
\DeclareMathOperator{\grad}{grad}
\DeclareMathOperator{\Grad}{Grad}
\DeclareMathOperator{\Dive}{Div}
\DeclareMathOperator{\dive}{div}
\DeclareMathOperator{\curl}{curl}

\newcommand{\rmc}{\mathrm{c}}
\newcommand{\rmpe}{\mathrm{pe}}
\newcommand{\rmhom}{\mathrm{hom}}
\newcommand{\rmsym}{\mathrm{sym}}
\newcommand{\bo}{L}
\newcommand{\cci}{C_{\mathrm{c}}^\infty}
\newcommand{\cco}{C_{\mathrm{c}}^1}
\newcommand{\cb}{C_{\mathrm{b}}}
\newcommand{\cc}{C_{\mathrm{c}}}
\renewcommand{\L}{L_2}
\newcommand{\Lnu}{L_{2,\nu}}
\newcommand{\cHt}{\mathcal{H}_2}

\providecommand{\ball}[2]{B\left(#1,#2\right)}
\providecommand{\cball}[2]{B\left[#1,#2\right]}

\renewcommand{\Re}{\operatorname{Re}}
\renewcommand{\Im}{\operatorname{Im}}
\providecommand{\td}[1]{\partial_{t,#1}}
\newcommand{\dom}{\operatorname{dom}}
\renewcommand{\ker}{\operatorname{ker}}
\newcommand{\ran}{\operatorname{ran}}
\newcommand{\lin}{\operatorname{lin}}
\newcommand{\id}{\operatorname{id}}
\newcommand{\idop}{1}
\providecommand{\card}[1]{\operatorname{card}(#1)}
\newcommand{\IV}{\operatorname{IV}}
\providecommand{\calDH}{\mathcal{D}_H}
\providecommand{\Lip}{\mathrm{Lip}}
\newcommand{\BD}{\mathrm{BD}}

\newcommand{\cran}{\operatorname{\overline{ran}}}

\newcommand{\loc}{\mathrm{loc}}
\renewcommand{\d}{\,\mathrm{d}}
\newcommand{\e}{\mathrm{e}}
\renewcommand{\i}{\mathrm{i}}

\renewcommand{\leq}{\leqslant}
\renewcommand{\geq}{\geqslant}

\newcommand{\Drazin}[1]{#1^{\mathrm{D}}}

\DeclareMathOperator{\spt}{spt}

\renewcommand{\Re}{\operatorname{Re}}
\renewcommand{\Im}{\operatorname{Im}}
\DeclareMathOperator{\trace}{trace}
\DeclareMathOperator{\diag}{diag}
\DeclareMathOperator{\dist}{dist}

\makeatletter
\newcommand\MyPairedDelimiter{%
  \@ifstar{\My@Paired@Delimiter{{}}}
          {\My@Paired@Delimiter{}}%
}
\newcommand\My@Paired@Delimiter[4]{%
  \newcommand#2{%
    \@ifstar{\start@PD{#1}{\delimitershortfall=-1sp}{#3}{#4}}
            {\start@PD{#1}{}{#3}{#4}}%
  }%
}
\newcommand\start@PD[5]{%
  #1\mathopen{\mathpalette\put@delim@helper{\put@delim{#2}{#3}{.}{#5}}}%
  #5%
  \mathclose{\mathpalette\put@delim@helper{\put@delim{#2}{.}{#4}{#5}}}%
}
\newcommand\put@delim@helper[2]{%
  \hbox{$\m@th\nulldelimiterspace=0pt #2#1$}%
}
\newcommand\put@delim[5]{%
  \setbox\z@\hbox{$\m@th#5{#4}$}%
  \setbox\tw@\null
  \ht\tw@\ht\z@ \dp\tw@\dp\z@
  #1#5%
  \left#2\box\tw@\right#3%
}
\makeatother

\MyPairedDelimiter*{\abs}{\lvert}{\rvert}
\MyPairedDelimiter*{\norm}{\lVert}{\rVert}

\newcommand\llim{
\mathchoice{\vcenter{\hbox{${\scriptstyle{-}}$}}}
{\vcenter{\hbox{$\scriptstyle{-}$}}}
{\vcenter{\hbox{$\scriptscriptstyle{-}$}}}
{\vcenter{\hbox{$\scriptscriptstyle{-}$}}}}
\newcommand\rlim{
\mathchoice{\vcenter{\hbox{${\scriptstyle{+}}$}}}
{\vcenter{\hbox{$\scriptstyle{+}$}}}
{\vcenter{\hbox{$\scriptscriptstyle{+}$}}}
{\vcenter{\hbox{$\scriptscriptstyle{+}$}}}}

\newcommand{\N}{\mathbb{N}}

\newcommand{\Z}{\mathbb{Z}}

\newcommand{\R}{\mathbb{R}}
\providecommand{\Rl}[1]{\mathbb{R}_{<{#1}}}
\providecommand{\Rle}[1]{\mathbb{R}_{\le{#1}}}
\providecommand{\Rg}[1]{\mathbb{R}_{>{#1}}}
\providecommand{\Rge}[1]{\mathbb{R}_{\ge{#1}}}
\newcommand{\C}{\mathbb{C}}
\newcommand{\K}{\mathbb{K}}
\newcommand{\1}{\mathds{1}}
\newcommand{\m}{\mathrm{m}}
\renewcommand{\tilde}{\widetilde}
\renewcommand{\hat}{\widehat}
\renewcommand{\check}{\widecheck}
\newcommand{\Fun}{\operatorname{Fun}}
\newcommand{\ind}{\operatorname{ind}}
\renewcommand{\emptyset}{\varnothing}
\newcommand{\from}{\colon}

\providecommand{\loi}[2]{\left(#1,#2\right]}
\providecommand{\roi}[2]{\left[#1,#2\right)}
\providecommand{\ci}[2]{\left[#1,#2\right]}
\providecommand{\oi}[2]{\left(#1,#2\right)}

\DeclareMathOperator{\esssup}{ess-sup}
\DeclareMathOperator{\essran}{ess-ran}

\newrefformat{prop}{Proposition \ref{#1}}
\newrefformat{lem}{Lemma \ref{#1}}
\newrefformat{thm}{Theorem \ref{#1}}
\newrefformat{cor}{Corollary \ref{#1}}
\newrefformat{rem}{Remark \ref{#1}}
\newrefformat{eq}{(\ref{#1})}
\newrefformat{item}{(\ref{#1})}
\newrefformat{exa}{Example \ref{#1}}
\newrefformat{exer}{Exercise \ref{#1}}
\newrefformat{chap}{Chapter \ref{#1}}
\newrefformat{problem}{Problem \ref{#1}}

\makeatletter
\def\renewtheorem#1{%
  \expandafter\let\csname#1\endcsname\relax
  \expandafter\let\csname c@#1\endcsname\relax
  \gdef\renewtheorem@envname{#1}
  \renewtheorem@secpar
}
\def\renewtheorem@secpar{\@ifnextchar[{\renewtheorem@numberedlike}{\renewtheorem@nonumberedlike}}
\def\renewtheorem@numberedlike[#1]#2{\newtheorem{\renewtheorem@envname}[#1]{#2}}
\def\renewtheorem@nonumberedlike#1{  
\def\renewtheorem@caption{#1}
\edef\renewtheorem@nowithin{\noexpand\newtheorem{\renewtheorem@envname}{\renewtheorem@caption}}
\renewtheorem@thirdpar
}
\def\renewtheorem@thirdpar{\@ifnextchar[{\renewtheorem@within}{\renewtheorem@nowithin}}
\def\renewtheorem@within[#1]{\renewtheorem@nowithin[#1]}
\makeatother

\renewtheorem{xca}{Exercise}[chapter]

\makeatother

\providecommand{\corollaryname}{Corollary}
\providecommand{\definitionname}{Definition}
\providecommand{\examplename}{Example}
\providecommand{\exercisename}{Exercise}
\providecommand{\lemmaname}{Lemma}
\providecommand{\propositionname}{Proposition}
\providecommand{\remarkname}{Remark}
\providecommand{\theoremname}{Theorem}
\providecommand{\problemname}{Problem}

\begin{document}

\frontmatter

\title{23rd Internet Seminar\\
``Evolutionary Equations''}
\author{Christian Seifert, Sascha Trostorff \& Marcus Waurick}
\date{Final Version, \today}

\maketitle

\tableofcontents{}

\mainmatter



\chapter{Introduction}

This chapter is intended to give a brief introduction as well as a
summary of the course to be presented throughout the semester. We
shall highlight some of the main ideas and methods behind the theory
and will also aim to provide some background on the main concept, which
will be the central object of study in the forthcoming weeks: the
notion of so-called
\[
\textbf{Evolutionary Equations}
\]
dating back to Picard in the seminal paper \cite{PicPhy}; see also
\cite[Chapter 6]{A11}. 

Another expression used to describe the same thing (and in order to
distinguish the concept from \emph{\index{evolution equation}evolution equations}) is that of
\emph{evo-systems}.\index{evo-system} Before going into detail on what we think of when
using the term evolutionary equations, we shall look into a seemingly
similar class of equations first.

\section{Evolution Equations}

The term evolution equation\index{evolution equation} is commonly referred to as a (partial)
differential equation involving time. This is a well developed concept
that can be found, for example, in the standard references \cite{Engel2000,Hille1957,Pazy1983}.
Before addressing a solution strategy for these kinds of problems
we mention some examples. We shall revisit these examples again in
the course later. One of the main examples of evolution equations\index{evolution equation},
in the sense to be discussed in this section, is the heat equation\index{heat equation}
in its second order form. More precisely, 
\[
\begin{cases}
\partial_{t}\theta(t,x)=\Delta\theta(t,x), & (t,x)\in\oi{0}{\infty}\times\Omega,\\
\theta(0,x)=\theta_{0}(x), & x\in\Omega,
\end{cases}
\]
where $\Omega\subseteq\R^{d}$ is some open set, and $\Delta=\sum_{j=1}^{d}\partial_{j}^{2}$
is the usual Laplacian\index{Laplacian} carried out with respect to the `$x$-variables' or `spatial variables',
and $\theta_{0}$ is a given initial heat distribution and $\theta$
is the unknown (scalar-valued) heat distribution. The above heat equation
is also accompanied with some boundary conditions for $\theta(t,x)$ which are required 
to be valid for all $t>0$ and $x\in\partial\Omega$.

We shall explain one way of solving this problem. To this end we make
a detour to the theory of ordinary differential equations. Let us consider
an $n\times n$-matrix $A$ with entries from the field $\K$ of complex
or real numbers, $\C$ or $\R$, and address the system of ordinary
differential equations 
\[
\begin{cases}
u'(t)=Au(t), & t>0,\\
u(0)=u_{0}
\end{cases}
\]
for some given initial datum, $u_{0}\in\K^{n}$.
In this case, we know that there exists a unique solution. This solution can
be computed with the help of the so-called matrix exponential\index{matrix exponential}
\[
\e^{tA}=\sum_{n=0}^{\infty}\frac{(tA)^{k}}{k!}\in\K^{n\times n}
\]
in the form 
\[
u(t)=\e^{tA}u_{0}.
\]
As it turns out, this $u$ is continuously differentiable and
 $u$ satisfies the above equation. We note in particular that
 $\e^{tA}u_{0}\to\e^{0 A}u_{0}=u_{0}$ as $t\to0\rlim$
and that $\e^{\left(t+s\right)A}=\e^{tA}\e^{sA}$. In a way, to obtain
the solution for the system of ordinary differential equations we
need to construct $(\e^{tA})_{t\geq 0}$. This is the same idea behind the process for obtaining a solution
for the aforementioned heat equation.

Indeed, given a suitable Banach space $X$ one aims to construct a
so-called \emph{$C_{0}$-semigroup\index{$C_0$-semigroup}, $(T(t))_{t\geq0}$}, that
is, for all $t\geq 0$, $T(t)$ is a bounded linear operator acting
in $X$, $T(t)\in\bo(X)$, and the following conditions are satisfied
\begin{enumerate}
\item semigroup law: $T(0)= I$ and $T(t+s)=T(t)T(s)$ for all $t,s\geq 0$,
\item strong continuity: for all $x\in X$, $\lim_{t\to0\rlim}T(t)x=x$.
\end{enumerate}
For instance in the case of $X=\L(\Omega)$, it is possible to construct
such a family $(T(t))_{t\geq0}$, written as $(\e^{t\Delta})_{t\geq0}$,
satisfying the just mentioned criteria. For every $\theta_{0}\in\L(\Omega)$ this $C_0$-semigroup provides a function $\theta\colon t\mapsto\e^{t\Delta}\theta_{0}\in\L(\Omega)$
which satisfies the above heat equation in a certain \emph{generalised}
sense. It is then an a posteriori question as to which additional conditions,
for example on $\theta_{0}$, need to be imposed in order to assure that $\theta$
solves the above heat equation as it stands. 

In the abstract setting of $C_{0}$-semigroups, the orbit
$u\colon t\mapsto T(t)x$ for some $x\in X$ then satisfies the equation
\begin{equation}
\begin{cases}
u'(t)=Bu(t), & t>0,\\
u(0)=x
\end{cases}\label{eq:aCP}
\end{equation}
provided that $x\in\set{y\in X}{By\coloneqq\lim_{t\to0\rlim}\frac{1}{t}(T(t)y-y)\in X\text{ exists}}$.
It can be shown that $B$ is uniquely
determined. $B$ is called the \emph{generator}\index{generator} of the $C_0$-semigroup
$(T(t))_{t\geq 0}$. The $C_0$-semigroup and the generator are in direct correspondence to
each other. In applications, given some operator $B$ the task is
to find a $C_{0}$-semigroup such that $Bx=\lim_{t\to0\rlim}\frac{1}{t}(T(t)x-x)$
for all $x\in X$ where either the left-hand side or the right-hand
side is well-defined. In other words, the question is whether a given
operator $B$ is actually the generator of a $C_{0}$-semigroup.

Note that in the case of the heat equation, the $C_{0}$-semigroup
is also known as the \emph{fundamental solution or Green's function}\index{fundamental solution or Green's function} of the problem considered;
in the abstract setting, $(T(t))_{t\geq0}$ is the fundamental
solution of \prettyref{eq:aCP}.

As we have seen, $C_{0}$-semigroups focus on initial value problems. Moreover,
the heat equation (as a partial differential equation) above is viewed as an \emph{ordinary} differential
equation with values in an \emph{infinite-dimensional} state space
$X$. While the left-hand side of the equation is
always of the same form, the complexity of the problem class is stored
in the choice of $X$ and the operator $B$ (and, naturally, its domain of definition). 

In the literature, explicit initial value problems like that of the ODE system
or the heat equation are gathered under the umbrella term \emph{evolution
equation}. It has become customary to refer to any problem where $C_{0}$-semigroups
play a fundamental role as an evolution equation\index{evolution equation}. This is particularly
the case when $X$ is infinite-dimensional. Then, arguably, the study
of $C_{0}$-semigroups is the study of fundamental solutions\index{fundamental solution or Green's function} (or abstract
Green's functions) associated to a certain class of initial
value problems for (partial) differential equations. A solution theory,\index{solution theory, general notion}
that is, the proof for existence, uniqueness and continuous dependence
on the data, is then contained in the construction of the fundamental
solution in terms of the ingredients of the equation. More precisely,
in the case of the ODE above, the fundamental solution\index{fundamental solution or Green's function} is constructed
in terms of $A$ and in case of the heat equation, the fundamental
solution is constructed in terms of (a particular realisation of)
$\Delta$ as an (unbounded) operator in $X$. We emphasise that there
is some bias towards the temporal direction in the theory of $C_{0}$-semigroups in the sense that $C_0$-semigroups impose continuity in time while this is in general not assumed for the spatial directions.

\section{Time-independent Problems}

The construction of fundamental solutions\index{fundamental solution or Green's function} is also a valuable method
for obtaining a solution for time-independent problems, see, e.g., \cite{Evans1998}.
To see this, let us consider Poisson's equation\index{Poisson's equation} in $\R^{3}$: Given
$f\in\cci(\R^{3})$ we want to find a function $u\colon\R^{3}\to\R$
with the property that
\[
-\Delta u(x)=f(x)\quad(x\in\R^{3}).
\]
It can be shown that $u$ given by
\[
u(x)=\frac{1}{4\pi}\int_{\R^{3}}\frac{1}{\abs{x-y}}f(y)\d y
\]
is well-defined, twice continuously differentiable and satisfies Poisson's
equation; cf.\ \prettyref{exer:Poisson_solution}. Note that $x\mapsto\frac{1}{4\pi\abs{x}}$ is also referred
to as the \emph{fundamental solution} or \emph{Green's} \emph{function}\index{fundamental solution or Green's function}
for Poisson's equation. The formula presented for $u$ is the \emph{convolution}
with the fundamental solution. The formula used
to define $u$ also works for $f$ being merely bounded and measurable
with compact support. In this case, however, the pointwise formula
of Poisson's equation cannot be expected to hold anymore, simply because
$f$ is well-defined only up to a set of measure 0. Thus, only a posteriori
estimates under additional conditions on $f$ render $u$ to be twice
continuously differentiable (say) with Poisson's equation holding
for all $x\in\R^{3}$. However, similar to the semigroup setting,
it is possible to \emph{generalise} the meaning of $-\Delta u=f$.
Then, again, the fundamental solution\index{fundamental solution or Green's function} can be used to construct a solution
for Poisson's equation for more general $f$. 

The situation becomes different when we consider a boundary value
problem instead of the problem above. More precisely, let $\Omega\subseteq\R^{3}$
be an open set and let $f\in\L(\Omega).$ We then need to ask whether
there exists $u\in\L(\Omega)$ such that 
\[
\begin{cases}
-\Delta u =f, & \text{ on }\Omega,\\
\quad\;\; u = 0, & \text{ on }\partial\Omega.
\end{cases}
\]
Notice that the task of just (mathematically) formulating this equation, let alone establishing a solution theory, is something that needs to be addressed. Indeed, we emphasise
that it is unclear as to what $\Delta u$ is supposed to mean if $u\in\L(\Omega)$,
only. It turns out that the problem described is not well-posed in general.
In particular -- depending on the shape of $\Omega$ and the norms
involved -- it might, for instance, lack continuous dependence on the
data, $f$.

In any case, the solution formula that we have used for the case when $\Omega=\R^{3}$
does not work anymore. Indeed, only particular shapes of $\Omega$
permit to construct a fundamental solution\index{fundamental solution or Green's function}; see \cite[Section 2.2]{Evans1998}.
Despite this, when $\Omega$ is merely bounded, it is still possible
to construct a solution, $u$, for the above problem. There are two key ingredients required for this approach. One is a clever application of Riesz's representation
theorem for functionals in Hilbert spaces and the other one involves inventing `suitable'
interpretations of $\Delta u$ in $\Omega$ and $u=0$ on $\partial\Omega$. Thus,
the method of `solving' Poisson's equation amounts to posing the
correct question, which then can be addressed \emph{without} invoking
the fundamental solution\index{fundamental solution or Green's function}. With this in mind, one could argue that the \emph{setting}
makes the problem solvable. 

\section{Evolution\emph{ary} Equations}

The central aim for \index{evolutionary equation}evolutionary equations is to combine the rationales
from both the $C_{0}$-semigroup theory and that from the time-independent
case. That is to say, we wish to establish a setting that treats time-independent
problems as well as time-dependent problems. At the
same time we need to \emph{generalise} solution concepts. We
shall not aim to construct the fundamental solution in either
the spatial or the temporal directions. The problem class will comprise of problems that can be written in the form
\[
\left(\partial_{t}M(\partial_{t})+A\right)U=F
\]
where $U$ is the unknown and $F$ the known right-hand side.
Furthermore, $A$ is an (unbounded, skew-selfadjoint) operator acting
in some Hilbert space that is thought of as modelling spatial coordinates;
$\partial_{t}$ is a realisation of the (time-)derivative operator and $M(\partial_{t})$
is an analytic, bounded operator-valued function $M$, which is evaluated at the time derivative.
In the course of the next chapters, we shall specfiy the definitions
and how standard problems fit into this problem class. 

Before going into greater depth on this approach, we would like to emphasise
the key differences and similarities which arise when compared to the derivation of more traditional 
solution theories that we outlined above.

Since the solution theory for evolutionary equations\index{evolutionary equation} will also encapsulate
time-in\-de\-pen\-dent problems, we cannot just focus on initial value problems
but rather on inhomogeneous problems. As we do not want to
require the existence of any fundamental solution we will also need to introduce
a \emph{generalisation} of the concept of a solution. Indeed, continuity
in the form of the existence of a $C_{0}$-semigroup, or variants thereof,
will neither be shown nor be expected. Moreoever, we shall see that both
$\partial_{t}$ and $A$ are \emph{unbounded} operators with $M(\partial_{t})$
being a bounded operator. Thus, we need to make sense of the operator
sum of the two unbounded operators $\partial_{t}M(\partial_{t})$
and $A$, which, in general, cannot be realised as being onto but rather as
having dense range, only.

A post-processing procedure will then ensure that for more regular
right-hand sides, $F$, the solution $U$ will also be more regular.
In some cases this will, for instance, amount to $U$ being continuous
in the time variable. In this way, phrased in similar settings, $C_{0}$-semigroup
theory may be viewed as a \emph{regularity} theory for a subclass of
evolutionary equations\index{evolutionary equation}. We shall entirely confine ourselves within the
Hilbert space case though. In this sense, the solution theory to
be presented will be, in essence,
an application of the projection theorem (similar to time-independent problems). In our case, however, there
will not be as much of a regularity bias as there is in $C_{0}$-semigroup
theory or in abstract ODEs with an (infinite-)dimensional state space. In
fact, the projection theorem is applied in a Hilbert space, which
combines both spatial and temporal variables.

The operator $M(\partial_{t})$ is thought of as carrying all the `complexity'
of the model. This is different to $C_{0}$-semigroups, where
this complexity is put on to the (domain of the) generator. What we mean by complexity  will become more
apparent when we discuss some examples. 

Finally, let us stress that $A$ being `skew-selfadjoint' is a way
of implementing first order systems in our abstract setting. In fact,
deviating from classical approaches, we shall focus on first order
equations in \emph{both} time \emph{and} space. This is also another change
in perspective when compared to classical approaches. As classical treatments
might emphasise the importance of the Laplacian (and hence  Poisson's
equation) and variants thereof, evolutionary equations\index{evolutionary equation} rather emphasise
\emph{Maxwell's equations} as the prototypical PDE. This change of 
point of view will be illustrated in the following section, where we
address some classical examples.

\section{Particular Examples and the Change of Perspective}

Here we will focus on three examples. These examples will also be the first to be readdressed when we
discuss the solution theory of evolutionary equations\index{evolutionary equation} in a later chapter. In order to
simplify the current presentation we will not consider boundary value
problems but solely concentrate on problems posed on $\Omega=\R^{3}$.
Furthermore, we shall dispose of any initial conditions. 

\subsection*{Maxwell's Equations}\index{Maxwell's equations}

The prototypical evolutionary equation is the system provided by Maxwell's
equations. Maxwell's equations consist of two equations describing
an electro-magnetic field, $(E,H)$, subject to a given certain external
current, $J$, 
\begin{align*}
\partial_{t}\varepsilon E+\sigma E-\curl H & =-J,\\
\partial_{t}\mu H+\curl E & =0.
\end{align*}
We shall detail the properties of the material parameters $\varepsilon,\mu$, and $\sigma$
later on. For the time being it is safe to assume that they are non-negative
real numbers and that they additionally satisfy that $\mu(\varepsilon+\sigma)>0$.
Now, in the setting of evolutionary equations, we gather the electro-magnetic
field into one column vector and obtain
\[
\left(\partial_{t}\begin{pmatrix}
\varepsilon & 0\\
0 & \mu
\end{pmatrix}+\begin{pmatrix}
\sigma & 0\\
0 & 0
\end{pmatrix}+\begin{pmatrix}
0 & -\curl\\
\curl & 0
\end{pmatrix}\right)\begin{pmatrix}
E\\
H
\end{pmatrix}=\begin{pmatrix}
-J\\
0
\end{pmatrix}.
\]
We shall see later that we obtain an evolutionary equation by setting
\begin{align*}
M(\partial_{t}) & \coloneqq\begin{pmatrix}
\varepsilon & 0\\
0 & \mu
\end{pmatrix}+\partial_{t}^{-1}\begin{pmatrix}
\sigma & 0\\
0 & 0
\end{pmatrix}\text{ and }A\coloneqq\begin{pmatrix}
0 & -\curl\\
\curl & 0
\end{pmatrix}.
\end{align*}

A formulation that fits well into the $C_{0}$-semigroup setting would
be, for example, 
\[
\partial_{t}\begin{pmatrix}
E\\
H
\end{pmatrix}=\begin{pmatrix}
\varepsilon & 0\\
0 & \mu
\end{pmatrix}^{-1}\begin{pmatrix}
-\sigma & \curl\\
-\curl & 0
\end{pmatrix}\begin{pmatrix}
E\\
H
\end{pmatrix}+\begin{pmatrix}
\varepsilon & 0\\
0 & \mu
\end{pmatrix}^{-1}\begin{pmatrix}
-J\\
0
\end{pmatrix},
\]
provided that $\varepsilon>0$. The inhomogeneous right-hand side $(-\frac{1}{\varepsilon} J, 0)$
can then be dealt with by means of the variation of constants formula,
which is the incarnation of the convolution of $(-\frac{1}{\varepsilon} J, 0)$ with the fundamental solution
in this time-dependent situation. Thus, in order to apply semigroup
theory, the main task lies in showing that 
\[
\begin{pmatrix}
-\frac{1}{\varepsilon}\sigma & \frac{1}{\varepsilon}\curl\\
-\frac{1}{\mu}\curl & 0
\end{pmatrix}
\]
is the generator of a $C_{0}$-semigroup. 

A different formulation needs to be put in place if $\varepsilon=0$.
The situation becomes even more complicated if $\varepsilon$ and $\sigma$
are bounded, non-negative, measurable functions of the spatial variable such that $\varepsilon+\sigma\geq c$
for some $c>0$. In the setting of evolutionary equations, this problem,
however, \emph{can} be dealt with. Note that then one cannot expect $E$ to be continuous with respect to the temporal variable unless $J$ is smooth enough.

\subsection*{Wave Equation}\index{wave equation}

We shall discuss the scalar wave equation in a medium where the wave
propagation speed is inhomogeneous in different directions of space.
This is modelled by finding $u\colon\R\times\R^{3}\to\R$ such that,
given a suitable forcing term $f\colon\R\times\R^{3}\to\R$ (again
we skip initial values here), we have
\[
\partial_{t}^{2}u-\dive a\grad u=f,
\]
where $a=a^{\top}\in\R^{3\times3}$ is positive definite\index{positive definite}; that is, $\scp{\xi}{a\xi}_{\R^{3}}>0$
for all $\xi\in\R^{3}\setminus\{0\}$. In the context of evolutionary
equations, we rewrite this as a first order problem in time \emph{and}
space. For this, we introduce $v\coloneqq\partial_{t}u$ and $q\coloneqq-a\grad u$
and obtain that
\[
\left(\partial_{t}\begin{pmatrix}
1 & 0\\
0 & a^{-1}
\end{pmatrix}+\begin{pmatrix}
0 & \dive\\
\grad & 0
\end{pmatrix}\right)\begin{pmatrix}
v\\
q
\end{pmatrix}=\begin{pmatrix}
f\\
0
\end{pmatrix}.
\]
Thus, 
\[
M(\partial_{t})\coloneqq\begin{pmatrix}
1 & 0\\
0 & a^{-1}
\end{pmatrix}\text{ and }A\coloneqq\begin{pmatrix}
0 & \dive\\
\grad & 0
\end{pmatrix}
\]
yield a corresponding formulation.
A semigroup formulation would work again by multiplying through
by $\begin{pmatrix}
1 & 0\\
0 & a
\end{pmatrix}$ . 

Let us mention briefly that it is also possible to rewrite the wave
equation as a first order system in time only. For this, a standard
ODE trick is used: one simply sticks with the additional variable $v=\partial_{t}u$
and obtains that
\[
\partial_{t}\begin{pmatrix}
u\\
v
\end{pmatrix}=\begin{pmatrix}
0 & 1\\
\dive a\grad & 0
\end{pmatrix}\begin{pmatrix}
u\\
v
\end{pmatrix}+\begin{pmatrix}
f\\
0
\end{pmatrix}.
\]
 In this formulation the `complexity' of the model is contained in the
operator 
\[
\begin{pmatrix}
0 & 1\\
\dive a\grad & 0
\end{pmatrix}.
\]
One would then have to show that this operator is the generator of a $C_{0}$-semigroup.

\subsection*{Heat Equation}\index{heat equation}

We have already formulated the semigroup perspective of the heat equation
\[
\partial_{t}\theta-\dive a\grad\theta=Q,
\]
in which we have added a heat source $Q$ and a conductivity $a=a^{\top}\in\R^{3\times3}$
being positive definite. Here, again, we reformulate the heat equation as
a first order system in time and space to end up (again setting $q\coloneqq-a\grad\theta$) with 
\[
\left(\partial_{t}\begin{pmatrix}
1 & 0\\
0 & 0
\end{pmatrix}+\begin{pmatrix}
0 & 0\\
0 & a^{-1}
\end{pmatrix}+\begin{pmatrix}
0 & \dive\\
\grad & 0
\end{pmatrix}\right)\begin{pmatrix}
\theta\\
q
\end{pmatrix}=\begin{pmatrix}
Q\\
0
\end{pmatrix}.
\]
In the context of evolutionary equations we then have that
\[
M(\partial_{t})\coloneqq\begin{pmatrix}
1 & 0\\
0 & 0
\end{pmatrix}+\partial_{t}^{-1}\begin{pmatrix}
0 & 0\\
0 & a^{-1}
\end{pmatrix}\text{ and }A\coloneqq\begin{pmatrix}
0 & \dive\\
\grad & 0
\end{pmatrix}.
\]
The advantage of this reformulation is that it becomes easily comparable
to the first order formulation of the wave equation outlined above.
For instance it is now possible to easily consider mixed type problems
of the form
\[
\left(\partial_{t}\begin{pmatrix}
1 & 0\\
0 & (1-s)a^{-1}
\end{pmatrix}+\begin{pmatrix}
0 & 0\\
0 & sa^{-1}
\end{pmatrix}+\begin{pmatrix}
0 & \dive\\
\grad & 0
\end{pmatrix}\right)\begin{pmatrix}
\theta\\
q
\end{pmatrix}=\begin{pmatrix}
Q\\
0
\end{pmatrix},
\]
with $s\colon\R^{3}\to[0,1]$ being an arbitrary measurable function. In fact,
in the solution theory for evolutionary equations, this does not amount to any additional complication of the problem. Models of this type are particularly
interesting in the context of so-called solid-fluid interaction\index{solid-fluid interaction}, where
the relations of a solid body and a flow of fluid surrounding it
are addressed.

\section{A Brief Outline of the Course}

We now present an overview of the contents of the following chapters.

\subsection*{Basics}

In order to properly set the stage, we shall begin with
some background operator theory in Banach and Hilbert spaces. We
assume the readers to be acquainted with some knowledge on bounded
linear operators, such as the uniform boundedness principle, and basic
concepts in the topology of metric spaces, such as density and closure.
The most important new material will be the adjoint of an operator,
which need not be bounded anymore. In order to deal with this notion, we will
consider relations rather than operators as they provide the natural setting for \emph{unbounded} operators. Having finished this brief detour
on operator theory, we will turn to a generalisation of Lebesgue spaces.
More precisely, we will survey ideas from Lebesgue's integration theory for functions
attaining values in an infinite-dimensional Banach space. 

\subsection*{The Time Derivative}

Banach space-valued (or rather Hilbert space-valued) integration theory
will play a fundamental role in defining the time derivative as
an unbounded, continuously invertible operator in a suitable Hilbert
space. In order to obtain continuous invertibility, we have to introduce
an exponential weighting function, which is akin to the exponential
weight introduced in the space of continuous functions for a proof of the Picard--Lindelöf
theorem. It is therefore natural to discuss the application of
this operator to ordinary differential equations. In particular,
we will present a Hilbert space solution theory for ordinary differential
equations. Here, we will also have the opportunity to discuss ordinary differential
equations with delay and memory. After this short detour, we will
turn back to the time derivative operator and describe its spectrum.
For this we introduce the so-called Fourier--Laplace transformation
which transforms the time derivative into a multiplication operator.
This unitary transformation will additionally serve to define (analytic and
bounded) functions of the time derivative. This is absolutely essential for the
formulation of evolutionary equations.

\subsection*{Evolutionary Equations}

Having finished the necessary preliminary work, we will then be in a
position to provide the proper justification of the formulation and solution theory for evolutionary
equations. We will accompany this solution theory not only with the
three leading examples from above, but also with some more sophisticated
equations. Amazingly, the considered space-time setting will allow us to
discuss (time-)fractional differential equations, partial differential
equations with delay terms and even a class of integro-differential equations.
Withdrawing the focus on regularity with respect to the temporal variable,
we are en passant able to generalise well-posedness conditions from
the classical literature. However, we shall stick with the treatment of
analytic operator-valued functions $M$ only. Therefore, we will
also include some arguments as to why this assumption seems to be \emph{physically}
meaningful. It will turn out that analyticity and causality are intimately
related via both the so-called Paley--Wiener theorem and a representation
theorem for time translation invariant causal operators.

\subsection*{Initial Value Problems for Evolutionary Equations}

As it has been outlined above, the focus of evolutionary equations
is on inhomogeneous right-hand sides rather than on initial value
problems. However, there is also the possibility to treat initial value
problems with the approach discussed here. For this, we need to introduce
extrapolation spaces. This then enables us to formulate initial value
problems as inhomogeneous equations. We have to make a concession
on the structure of the problem, however. In fact, we will focus on the case when $M(\partial_{t})=M_{0}+\partial_{t}^{-1}M_{1}$ for some bounded
linear operators $M_{0},M_{1}$ acting in the spatial variables alone.
The initial condition will then read as $\left(M_{0}U\right)(0\rlim)=M_{0}U_{0}$.
Hence, one might argue that the initial condition $U(0\rlim)=U_{0}$ is
only assumed in a rather generalised sense. This is due to the fact
that $M_{0}$ might be zero. However, for the case $A=0$ we will
also discuss the initial condition $U(0\rlim)=U_{0}$, which amounts to
a treatment of so-called differential-algebraic equations in both
finite- and inifinite-dimensional state spaces. 

\subsection*{Properties of Solutions and Inhomogeneous Boundary Value Problems}

Turning back to the case when $A\neq0$ we will discuss qualitative
properties of solutions of evolutionary equations. One of which will
be exponential decay. We will identify a subclass of evolutionary
equations where it is comparatively easy to show that if the right-hand side decays exponentially then so too must the solution. This is the
proper replacement in our setting for the notion of exponential stability
 from $C_{0}$-semigroups. If the right-hand side is smooth enough
we obtain that $U(t)$, the solution of the evolutionary equation at time $t$,
decays exponentially if $t\to\infty$. Furthermore, we will frame
inhomogeneous boundary value problems in the setting of evolutionary
equations. The method will require a bit more on the regularity theory
for evolutionary equations and a definition of suitable boundary values.
In particular, we shall present a way of formulating classical inhomogeneous
boundary value problems for domains without any boundary regularity.

\subsection*{Properties of the Solution Operator}

In the final part, we shall have another
look at the advantages of the problem formulation. In fact, we will
have a look at the notion of homogenisation of differential equations.
In the problem formulation presented here, we shall analyse the continuity
properties of the solution operator with respect to weak operator topology convergence
of the operator $M(\partial_{t})$. We will address an example for
ordinary differential equations (when $A=0$) and one for partial differential
equations (when $A\neq0$). It will turn out that the respective continuity
properties are profoundly different from one another. 

\section{Comments}

The focus presented here on the main notions behind evolutionary equations is
mostly in order to properly motivate the theory and highlight the
most striking differences in the philosophy. There are other solution
concepts (and corresponding general settings) developed for partial
differential equations; either time-dependent or without involving time. 

There is an abundance of examples and additional concepts for $C_{0}$-semigroups
for which we refer to the aforementioned standard treatments again.
There is also a generalisation to problems that are second order in
time, e.g., $u''=Au$, where $u(0)$ and $u'(0)$
are given. This gives rise to cosine families of bounded linear operators which is
another way of generalising the fundamental solution\index{fundamental solution or Green's function} concept, see,
for example, \cite{Sova1966}.

The main focus of all of these equations is to address \emph{initial
value problems, }where the (first/second) time derivative of the unknown
is explicit.

With a focus on static, that is, time-independent partial differential
equations, the notion of Friedrichs systems is also concerned with
a way of writing many PDEs from mathematical physics into a common
form, see \cite{Friedrichs1954,Friedrichs1958}. A time-dependent
variant of constant coefficient Friedrichs systems are so-called symmetric-hyperbolic
systems, see e.g.~\cite{Benzoni-Gavage2007}. In these cases, whether
the authors treat constant coefficients or not, the framework of evolutionary
equations adds a profound
amount of additional complexity by including the operator $M(\partial_{t})$.

The treatment of time-dependent problems in space-time settings and
addressing corresponding well-posedness properties of a sum of two
unbounded operators has also been considered in \cite{DaPrato1975}
with elaborate conditions on the operators involved. In their studies,
the flexibility introduced by the operator $M(\partial_{t})$ in our
setting is missing, thus the time derivative operator is not
thought of having any variable coefficients attached to it.

\section*{Exercises}
\addcontentsline{toc}{section}{Exercises}

\begin{xca}
Let $\phi\in C(\R,\R)$. Assume that $\phi(t+s)=\phi(t)\phi(s)$ for
all $t,s\in\R$, $\phi(0)=1$. Show that $\phi(t)=\e^{\alpha t}$ ($t\in\R$)
for some $\alpha\in\R$.
\end{xca}

\begin{xca}
Let $n\in\N$, $T\colon\R\to\R^{n\times n}$ continuously differentiable such that
$T(t+s)=T(t)T(s)$ for all $t,s\in\R$, $T(0)=I$. Show that there exists $A\in\R^{n\times n}$
with the property that $T(t)=\e^{tA}$ ($t\in\R$).
\end{xca}

\begin{xca}
\label{exer:Poisson_solution}
Show that $x\mapsto u(x)=\frac{1}{4\pi}\int_{\R^{3}}\frac{1}{\abs{x-y}}f(y)\d y$
satisfies Poisson's equation, given $f\in\cci(\R^{3})$.
\end{xca}

\begin{xca}
Let $f\in\cci(\R)$. Define $u(t,x)\coloneqq f(x+t)$ for $x,t\in\R$.
Show that $u$ satisfies the differential equation $\partial_{t}u=\partial_{x}u$
and $u(0,x)=f(x)$ for all $x\in\R$. 
\end{xca}

\begin{xca}
Let $X,Y$ be Banach spaces, $(T_{n})_{n\in\N}$ be a sequence in
$\bo(X,Y)$, the set of bounded linear operators. If $\sup\set{\norm{T_{n}}}{n\in\N}=\infty,$
show that there is $x\in X$ and a strictly increasing sequence $(n_{k})_{k}$
in $\N$ such that $\norm{T_{n_{k}}x}\to\infty.$
\end{xca}

\begin{xca}
Let $n\in\N$. Denote by $\mathrm{GL}(n;\K)$ the set of continuously
invertible $n\times n$ matrices. Show that $\mathrm{GL}(n;\K)\subseteq \K^{n\times n}$
is open.
\end{xca}

\begin{xca}
Let $n\in\N$. Show that $\Phi\colon\mathrm{GL}(n;\K)\ni A\mapsto A^{-1}\in \K^{n\times n}$
is continuously differentiable. Compute $\Phi'$.
\end{xca}

\printbibliography[heading=subbibliography]

\chapter{Unbounded Operators}

We will gather some information on operators
in Banach and Hilbert spaces. Throughout this chapter let $X_{0}$, $X_{1},$ and
$X_{2}$ be Banach spaces and $H_{0}$, $H_{1}$, and $H_{2}$ be
Hilbert spaces over the field $\K\in\{\R,\C\}$.

\section{Operators in Banach Spaces}

The main difference of \index{continuous linear operator}continuous linear operators,
that is, 
\[
\bo(X_{0},X_{1})\coloneqq\set{B\colon X_{0}\to X_{1}}{B\text{ linear, }\norm{B}\coloneqq\sup_{x\in X_{0}\setminus\{0\}}\frac{\norm{Bx}}{\norm{x}}<\infty}
\]
(with the usual abbreviation $\bo(X_0):=\bo(X_0,X_0)$)
and the set of uncontinuous or unbounded linear operators is that the latter only need to be defined 
on a subset of $X_{0}$. In order to define unbounded linear operators, we will first take a more general point of view and introduce (linear) relations.
This perspective will turn out to be the natural setting later on.
\begin{defn*}
A subset $A\subseteq X_{0}\times X_{1}$ is called a \emph{relation
in $X_{0}$ and $X_{1}$}. We define the \emph{domain}\index{domain},
\emph{range\index{range}} and \emph{kernel\index{kernel} of $A$}
as follows
\begin{align*}
\dom(A) & \coloneqq\set{x\in X_{0}}{\exists y\in X_{1}\colon(x,y)\in A},\\
\ran(A) & \coloneqq\set{y\in X_{1}}{\exists x\in X_{0}\colon(x,y)\in A},\\
\ker(A) & \coloneqq\set{x\in X_{0}}{(x,0)\in A}.
\end{align*}
The \emph{image\index{image}, $A[M]$, of a set $M\subseteq X_{0}$
under $A$} is given by
\[
A[M]\coloneqq\set{y\in X_{1}}{\exists x\in M\colon(x,y)\in A}.
\]
A relation $A$ is called \emph{bounded}, if for all bounded $M\subseteq X_{0}$
the set $A[M]\subseteq X_{1}$ is bounded. For a given relation $A$
we define the \emph{inverse relation\index{inverse relation}}
\[
A^{-1}\coloneqq\set{(y,x)\in X_{1}\times X_{0}}{(x,y)\in A}.
\]
A relation $A$ is called \emph{linear}, if $ A\subseteq X_0 \times X_1$ is a linear subspace. A linear
relation $A$ is called \emph{linear operator} or just \emph{operator
from $X_{0}$ to $X_{1}$}, if 
\[
A[\{0\}] = \set{y\in X_{1}}{(0,y)\in A}=\{0\}.
\]
In this case, we also write 
\[
A\colon\dom(A)\subseteq X_{0}\to X_{1}
\]
to denote a linear operator from $X_{0}$ to $X_{1}$. Moreover, we
shall write $Ax=y$ instead of $(x,y)\in A$ in this case.
A linear operator $A$, which is not bounded, is called \emph{unbounded}.
\end{defn*}
For completeness, we also define the sum, scalar multiples, and composition
of relations.
\begin{defn*}
Let $A\subseteq X_{0}\times X_{1}$, $B\subseteq X_{0}\times X_{1}$
and $C\subseteq X_{1}\times X_{2}$ be relations, $\lambda\in\mathbb{K}$. Then
we define
\begin{align*}
A+B & \coloneqq\set{(x,y+w)\in X_{0}\times X_{1}}{(x,y)\in A,(x,w)\in B},\\
\lambda A & \coloneqq\set{(x,\lambda y)\in X_{0}\times X_{1}}{(x,y)\in A},\\
CA & \coloneqq\set{(x,z)\in X_{0}\times X_{2}}{\exists y\in X_{1}\colon(x,y)\in A,(y,z)\in C}.
\end{align*}
\end{defn*}

For a relation $A\subseteq X_{0}\times X_{1}$ we will use the abbreviation $-A:=-1A$ (so that the minus sign only acts on the second component).
We now proceed with topological notions for relations.

\begin{defn*}
Let $A\subseteq X_{0}\times X_{1}$ be a relation. $A$ is called
\emph{densely defined}\index{densely defined}, if $\dom(A)$ is dense
in $X_{0}$. We call $A$ \emph{closed}\index{closed}, if $A$ is
a closed subset of the direct sum of the Banach spaces $X_{0}$ and
$X_{1}$. If $A$ is a linear operator then we will call $A$ \emph{closable}\index{closable},
whenever $\overline{{A}}\subseteq X_{0}\times X_{1}$ is a linear operator. 
\end{defn*}

\begin{prop}
\label{prop:inv-sum-comp-closed}Let $A\subseteq X_{0}\times X_{1}$ be a relation,
$C\in\bo(X_{2},X_{0})$ and $B\in\bo(X_{0},X_{1})$. Then the following
statements hold.
\begin{enumerate}
\item $A$ is closed if and only if $A^{-1}$ is closed. Moreover, we have $(\overline{A})^{-1}=\overline{A^{-1}}$.
\item $A$ is closed if and only if $A+B$ is closed;
\item if $A$ is closed, then $AC$ is closed.
\end{enumerate}
\end{prop}

\begin{proof}
Statement (a) follows upon realising that $X_{0}\times X_{1}\ni(x,y)\mapsto(y,x)\in X_{1}\times X_{0}$
is an isomorphism. 

For statement (b), it suffices to show that
the closedness of $A$ implies the same for $A+B$. Let $\left((x_{n},y_{n})\right)_{n}$
be a sequence in $A+B$ convergent in $X_{0}\times X_{1}$ to some $(x,y)$. Since $B\in\bo(X_{0},X_{1})$, it follows that
$\left((x_{n},y_{n}-Bx_{n})\right)_{n}$ in $A$ is convergent to
$(x,y-Bx)$ in $X_{0}\times X_{1}$. Since $A$ is closed, $(x,y-Bx)\in A$.
Thus, $(x,y)\in A+B$.

For statement (c), let $\left((w_{n},y_{n})\right)_{n}$
be a sequence in $AC$ convergent in $X_{2}\times X_{1}$ to some $(w,y)$. Since $C$ is continuous, $\left(Cw_{n}\right)_{n}$
converges to $Cw$. Hence, $(Cw_{n},y_{n})\to(Cw,y)$ in $X_{0}\times X_{1}$
and since $(Cw_{n},y_{n})\in A$ and $A$ is closed, it follows that
$(Cw,y)\in A$. Equivalently, $(w,y)\in AC$, which yields closedness of $AC$.
\end{proof}
We shall gather some other elementary facts about closed operators
in the following. We will make use of the following notion.

\begin{defn*}
Let $A\colon\dom(A)\subseteq X_{0}\to X_{1}$ be
a linear operator. Then the \emph{graph norm\index{graph norm}} of $A$ is defined by $\dom(A)\ni x\mapsto\norm{x}_A\coloneqq \sqrt{\norm{x}^{2}+\norm{Ax}^{2}}$.
\end{defn*}

\begin{lem}
\label{lem:grscp}Let $A\colon\dom(A)\subseteq X_{0}\to X_{1}$ be
a linear operator. Then the following statements are equivalent:
\begin{enumerate}
\renewcommand{\labelenumi}{{\upshape (\roman{enumi})}}
\item $A$ is closed.
\item $\dom(A)$ equipped with the graph norm
is a Banach space.
\item For all $(x_{n})_{n}$ in $\dom(A)$
convergent in $X_{0}$ such that $\left(Ax_{n}\right)_{n}$ is convergent in
$X_{1}$ we have $\lim_{n\to\infty}x_{n}\in\dom(A)$ and $A\lim_{n\to\infty}x_{n}=\lim_{n\to\infty}Ax_{n}$.
\end{enumerate}
\end{lem}

\begin{proof}
For the equivalence (i)$\Leftrightarrow$(ii), it suffices to observe that $\dom(A)\ni x\mapsto(x,Ax)\in A$,
where $\dom(A)$ is endowed with the graph norm, is an isomorphism.
The equivalence (ii)$\Leftrightarrow$(iii) is an easy reformulation of the definition of closedness of $A\subseteq X_{0}\times X_{1}$.
\end{proof}

Unless explicitly stated otherwise (e.g.~in the form $\dom(A)\subseteq X_0$, where we regard $\dom(A)$ as a subspace of $X_0$), for closed operators $A$ we always consider $\dom(A)$ as a Banach
space in its own right; that is, we shall regard it as being endowed 
with the graph norm.

\begin{lem}
\label{lem:bbc}Let $A\colon\dom(A)\subseteq X_{0}\to X_{1}$ be a
closed linear operator. Then $A$ is bounded if and only if $\dom(A)\subseteq X_{0}$
is closed.
\end{lem}

\begin{proof}
First of all note that boundedness of $A$ is equivalent to the fact
that the graph norm and the $X_{0}$-norm on $\dom(A)$ are equivalent.
Hence, the closedness and boundedness of $A$ implies that $\dom(A)\subseteq X_0$
is closed. On the other hand, the embedding 
\[
\iota\colon(\dom(A),\norm{\cdot}_{A})\hookrightarrow(\dom(A),\norm{\cdot}_{X_{0}})
\]
 is continuous and bijective. Since the range is closed, the open
mapping theorem implies that $\iota^{-1}$ is continuous. This yields
the equivalence of the graph norm and the $X_{0}$-norm and, thus,
the boundedness of $A$.
\end{proof}

For unbounded operators, obtaining a precise description of the domain may be difficult. However, there may be a subset of the domain which essentially (or approximately) describes the operator.
This gives rise to the following notion of a core.

\begin{defn*}
Let $A\subseteq X_{0}\times X_{1}$. A set $D\subseteq\dom(A)$ is
called a \emph{core\index{core} for $A$} provided $\overline{A\cap(D\times X_{1})}=\overline{A}$.
\end{defn*}

\begin{prop}
\label{prop:coreBounded}Let $A\in\bo(X_{0},X_{1})$, and $D\subseteq X_{0}$ a
dense linear subspace. Then $D$ is a core for $A$.
\end{prop}

\begin{cor}
\label{cor:uniquecont}Let $A\colon\dom(A)\subseteq X_{0}\to X_{1}$
be a densely defined, bounded linear operator. Then there exists a
unique $B\in L(X_{0},X_{1})$ with $B\supseteq A$. In particular,
we have $B=\overline{A}$ and 
\[
\norm{B}=\sup_{x\in\dom(A), x\neq0}\frac{{\normalcolor \norm{Ax}}}{\norm{x}}.
\]
\end{cor}
The proofs of \prettyref{prop:coreBounded} and \prettyref{cor:uniquecont} are asked for in \prettyref{exer:propcor}.

\section{Operators in Hilbert Spaces}

Let us now focus on operators on Hilbert spaces. In this setting, we can additionally make use of scalar products $\scp{\cdot}{\cdot}$, which in this course are considered to be linear in the second argument (and anti-linear in the first, in the case when $\K=\C$).

For a linear operator $A\colon \dom(A)\subseteq H_0\to H_1$ the graph norm of $A$ is induced
by the scalar product
\[
(x,y)\mapsto\scp{x}{y}+\scp{Ax}{Ay},
\]
known as the \emph{graph scalar product\index{graph scalar product} of $A$}. If $A$ is closed then $\dom(A)$ (equipped with the graph norm) is a Hilbert space.

Of course, no presentation of operators in Hilbert spaces would be complete without the central notion of the adjoint operator. We wish to pose the adjoint within the relational framework just established. The definition is as follows.
\begin{defn*}
For a relation $A\subseteq H_{0}\times H_{1}$ we define the \emph{adjoint
relation\index{adjoint relation} $A^{*}$} by
\[
A^{*}\coloneqq-\left(\left(A^{-1}\right)^{\bot}\right)\subseteq H_{1}\times H_{0},
\]
where the orthogonal complement is computed in the direct sum of the
Hilbert spaces $H_{1}$ and $H_{0}$; that is, the set $H_{1}\times H_{0}$
endowed with the scalar product $\bigl((x,y),(u,v)\bigr)\mapsto\scp{x}{u}_{H_1}+\scp{y}{v}_{H_0}$. 
\end{defn*}
\begin{rem}
Let $A\subseteq H_{0}\times H_{1}$. Then we have
\[
A^{*}=\set{(u,v)\in H_{1}\times H_{0}}{\forall(x,y)\in A:\scp{u}{y}_{H_1}=\scp{v}{x}_{H_0}}.
\]
In particular, if $A$ is a linear operator, we have
\[
A^{*}=\set{(u,v)\in H_{1}\times H_{0}}{\forall x\in\dom(A):\scp{u}{Ax}_{H_1}=\scp{v}{x}_{H_0}}.
\]
\end{rem}

\begin{lem}
\label{lem:adj}Let $A\subseteq H_{0}\times H_{1}$ be a relation. Then $A^{*}$
is a linear relation. Moreover, we have
\[
A^{*}=-\left(\left(A^{\bot}\right)^{-1}\right)=\left(\left(-A\right)^{-1}\right)^{\bot}=\left(-\left(A^{-1}\right)\right)^{\bot}=\left(\left(-A\right)^{\bot}\right)^{-1}=\left(-\left(A^{\bot}\right)\right)^{-1}.
\]
\end{lem}

The proof of this lemma is left as \prettyref{exer:lem_adj}.

\begin{rem}
Let $A\subseteq H_{0}\times H_{1}$. Since $A^{*}$ is the orthogonal
complement of $-A^{-1}$, it follows immediately that $A^{*}$ is closed. Moreover, $A^* = \bigl(\overline{A}\bigr)^*$ since $A^\bot = \bigl(\overline{A}\bigr)^\bot$.
\end{rem}

\begin{lem}
\label{lem:ass}Let $A\subseteq H_{0}\times H_{1}$ be a linear relation.
Then
\[
A^{**}\coloneqq\left(A^{*}\right)^{*}=\overline{{A}}.
\]
\end{lem}

\begin{proof}
We compute using \prettyref{lem:adj}
\[
A^{**}=\left(\left(-\left(A^{*}\right)\right)^{-1}\right)^{\bot}=\left(\left(-\left(-\left(\left(A^{\bot}\right)^{-1}\right)\right)\right)^{-1}\right)^{\bot}=\left(A^{\bot}\right)^{\bot}=\overline{A}.\tag*{{\qedhere}}
\]
\end{proof}
\begin{thm}
\label{thm:ran-kernerl}Let $A\subseteq H_{0}\times H_{1}$ be a linear relation.
Then
\[
\ran(A)^{\bot}=\ker(A^{*})\quad\text{and}\quad \cran(A^{*})=\ker(\overline{A})^{\bot}.
\]
\end{thm}

\begin{proof}
Let $u\in\ker(A^{*})$ and let $y\in\ran(A)$. Then we find $x\in\dom(A)$
such that $(x,y)\in A$. Moreover, note that $(u,0)\in A^{*}$. Then, we
compute
\[
\scp{u}{y}_{H_1}=\scp{0}{x}_{H_0}=0.
\]
This equality shows that $\ran(A)^{\bot}\supseteq\ker(A^{*})$. If
on the other hand, $u\in\ran(A)^{\bot}$ then for all $(x,y)\in A$
we have that
\[
0=\scp{u}{y}_{H_1},
\]
which implies $(u,0)\in A^{*}$ and hence $u\in\ker(A^{*})$. The
remaining equation follows from \prettyref{lem:ass} together with
the first equation applied to $A^{*}$.
\end{proof}
The following decomposition result is immediate from the latter theorem
and will be used frequently throughout the text.
\begin{cor}
\label{cor:abstrHelmH} Let $A\subseteq H_{0}\times H_{1}$ be a closed linear relation.
Then
\begin{align*}
H_{1} & =\cran(A)\oplus\ker(A^{*})\quad\text{and}\quad
H_{0} =\ker(A)\oplus\cran(A^{*}).
\end{align*}
\end{cor}

We will now turn to the case where the adjoint relation is actually a linear operator.

\begin{lem}
\label{lem:ddc}Let $A\subseteq H_{0}\times H_{1}$ be a linear relation.
Then $A^{*}$ is a linear operator if and only if $A$ is densely
defined. If, in addition, $A$ is a linear operator, then $A$ is
closable if and only if $A^{*}$ is densely defined.
\end{lem}

\begin{proof}
For the first equivalence, it suffices to observe that 
\begin{equation}
 A^{*}[\{0\}]=\dom(A)^{\bot}.\label{eq:rel}
\end{equation}
Indeed, $A$ being densely defined is equivalent to having $\dom(A)^{\bot}=\{0\}$. Moreover, $A^{*}$ is an operator if and only if
$A^{*}[\{0\}]=\{0\}$. Next, we show \prettyref{eq:rel}. For this, apply \prettyref{thm:ran-kernerl} to the linear relation $A^{-1}$. One obtains $(\ran A^{-1})^\bot=\ker(A^{-1})^*$. Hence, $(\dom (A))^\bot = \ker(A^*)^{-1}=A^*[\{0\}]$, which is \prettyref{eq:rel}.
For the remaining equivalence,
we need to characterise $\overline{A}$ being an operator. Using \prettyref{lem:ass}
and the first equivalence, we deduce that $\overline{A}=\left(A^{*}\right)^{*}$
is a linear operator if and only if $A^{*}$ is densely defined.
\end{proof}

\begin{rem} Note that the statement ``$A^{*}$ is an operator if  $A$ is densely
defined''  asserted in \prettyref{lem:ddc} is also true for \emph{any} relation. For this, it suffices to observe that \prettyref{eq:rel} is true for any relation $A\subseteq H_0 \times H_1$. Indeed, let $A\subseteq H_0\times H_1$ be a relation; define $B\coloneqq \lin A$. Then $\dom(B)= \lin \dom(A)$. Also, we have \[A^*=-(A^\bot)^{-1}=-(B^\bot)^{-1}=B^*.\] With these preparations, we can write
\begin{align*}
   \dom(A)^\bot & = (\lin  \dom (A))^\bot
   = \dom(B)^\bot
   = B^*[\{0\}]
  = A^{*}[\{0\}],
\end{align*}
where we used that \prettyref{eq:rel} holds for linear relations.
\end{rem}

\begin{lem}
\label{lem:bdds}Let $A\subseteq H_{0}\times H_{1}$ be a linear relation.
Then $\overline{A}\in\bo(H_{0},H_{1})$ if and only if $A^{*}\in\bo(H_{1},H_{0})$. In either case, $\norm{A^*} = \norm{\overline{A}}$.
\end{lem}

\begin{proof}
Note that $\overline{A}\in L(H_{0},H_{1})$ implies that $A$ is closable
and densely defined. Thus, by \prettyref{lem:ddc}, $A^{*}$ is a
densely defined, closed linear operator. For $u\in\dom(A^{*})$ we
compute using \prettyref{lem:ass}
\[
\norm{A^{*}u}=\sup_{x\in H_{0}\setminus\{0\}}\frac{\abs{\scp{A^{*}u}{x}}}{\norm{x}}=\sup_{x\in H_{0}\setminus\{0\}}\frac{\abs{\scp{u}{\overline{A}x}}}{\norm{x}}\leq\norm{\overline{A}}\norm{u},
\]
yielding $\norm{A^*}\leq \norm{\overline{A}}$.
On the one hand, this implies that $A^{*}$ is bounded, and on the other, since $A^{*}$ is densely
defined we deduce $A^{*}\in L(H_{1},H_{0})$ by \prettyref{lem:bbc}.
The other implication (and the other inequality) follows from the first one applied to $A^{*}$
instead of $A$ using $A^{**}=\overline{A}$.
\end{proof}

We end this section by defining some special classes of relations and operators.

\begin{defn*}
Let $H$ be a Hilbert space and $A\subseteq H\times H$ a linear relation.
We call $A$ \emph{(skew-)Hermitian}\index{(skew-)Hermitian} if
$A\subseteq A^{*}$ ($A\subseteq -A^{*}$). We say that $A$ is \emph{(skew-)symmetric}\index{(skew-)symmetric}
if $A$ is (skew-)Hermitian and densely defined (so that $A^{*}$ is a linear operator), and $A$
is called \emph{(skew-)self\-adjoint}\index{(skew-)selfadjoint} if $A=A^{*}$
($A=-A^{*}$). Additionally, if $A$ is densely defined, then we say that
$A$ is \emph{normal}\index{normal} if $AA^{*}=A^{*}A$. 
\end{defn*}

\section{Computing the Adjoint}

In general it is a very difficult task to compute the adjoint of a
given (unbounded) operator. There are, however, cases, where the adjoint
of a sum or the product can be computed more readily. We start with the most basic
case of bounded linear operators.
\begin{prop}
\label{prop:adj-cont-sum-comp}Let $A,B\in\bo(H_{0},H_{1}),C\in\bo(H_{2},H_{0})$.
Then $\left(A+B\right)^{*}=A^{*}+B^{*}$ and $\left(AC\right)^{*}=C^{*}A^{*}$.
\end{prop}

The latter results are special cases of more general statements to
follow.
\begin{thm}
\label{thm:adj-sum}Let $A\subseteq H_{0}\times H_{1}$ be a relation and $B\in L(H_{0},H_{1})$.
Then $(A+B)^{*}=A^{*}+B^{*}$.
\end{thm}

\begin{proof}
Let $(u,v)\in H_{1}\times H_{0}$. Then we compute
\begin{align*}
(u,v)\in A^{*}+B^{*} & \iff(u,v-B^{*}u)\in A^{*}\\
 & \iff\forall\, (x,y)\in A\colon\scp{y}{u}_{H_1}=\scp{x}{v-B^{*}u}_{H_0}\\
 & \iff\forall\, (x,y)\in A\colon\scp{y+Bx}{u}_{H_1}=\scp{x}{v}_{H_0}\\
 & \iff\forall\, (x,z)\in A+B\colon\scp{z}{u}_{H_1}=\scp{x}{v}_{H_0}\\
 & \iff(u,v)\in\left(A+B\right)^{*}.
\end{align*}
Note that for the first, third and fourth equivalence, we have used the fact that
$B\in\bo(H_{0},H_{1})$ together with \prettyref{lem:bdds}.
\end{proof}
\begin{cor}
\label{cor:adj-sum}Let $A\subseteq H_{0}\times H_{1}$, $B\in L(H_{0},H_{1}).$
If $A$ is densely defined, then $A^{*}+B^{*}$ is an operator
and $(A+B)^{*}=A^{*}+B^{*}$.
\end{cor}

\begin{thm}
\label{thm:adj-comp}Let $A\subseteq H_{0}\times H_{1}$ be a closed
linear relation and $C\in\bo(H_{2},H_{0})$. Then $\left(AC\right)^{*}=\overline{C^{*}A^{*}}.$
\end{thm}

\begin{proof}
We first show that $AC\subseteq\left(C^{*}A^{*}\right)^{*}$. For
this, let $(w,y)\in AC$. Then $(Cw,y)\in A$. Hence, for all $(u,z)\in C^{*}A^{*}$;
that is, for all $(u,v)\in A^{*}$ and $z=C^{*}v$, we compute
\[
\scp{u}{y}_{H_1}=\scp{v}{Cw}_{H_0}=\scp{C^{*}v}{w}_{H_2}=\scp{z}{w}_{H_2},
\]
which implies that $AC\subseteq\left(C^{*}A^{*}\right)^{*}$. Next, let
$(w,y)\in\left(C^{*}A^{*}\right)^{*}$. Then for all $(u,v)\in A^{*}$
and $z=C^{*}v$ we obtain
\[
\scp{u}{y}_{H_1}=\scp{z}{w}_{H_2}=\scp{C^{*}v}{w}_{H_2}=\scp{v}{Cw}_{H_0}.
\]
Thus, we obtain $(Cw,y)\in A^{**}=\overline{A}=A$. Thus, $(w,y)\in AC.$
Hence, 
\[
AC=\left(C^{*}A^{*}\right)^{*},
\]
which yields the assertion by adjoining this equation.
\end{proof}
\begin{cor}
\label{cor:adj-comp1}Let $A\subseteq H_{0}\times H_{1}$ be a linear
relation and $C\in\bo(H_{2},H_{0})$. Then $\left(\overline{A}C\right)^{*}=\overline{C^{*}A^{*}}.$
\end{cor}

\begin{proof}
The result follows upon realising that $A^{*}=A^{***}=\left(\overline{A}\right)^{*}$.
\end{proof}
\begin{cor}
\label{cor:adj-comp2}Let $A\subseteq H_{0}\times H_{1}$ be a linear
relation and $C\in\bo(H_{2},H_{0})$. If $\overline{A}C$ is densely
defined, then $C^{*}A^{*}$ is a closable linear operator with $\overline{C^{*}A^{*}}=\left(\overline{A}C\right)^{*}.$
\end{cor}

\begin{rem}
\label{rem:adj-comp}
\begin{enumerate}
 \item 
Note that if $B\in L(H_{1},H_{2})$ and $A\subseteq H_{0}\times H_{1}$
linear, then $\left(B\overline{A}\right)^{*}=A^{*}B^{*}$. Indeed,
this follows from \prettyref{thm:adj-comp} applied to $A^{*}$ and $B$ instead
of $A$ and $C^{*}$, respectively, since then we obtain $\left(A^{*}B^{*}\right)^{*}=\overline{B^{**}A^{**}}=\overline{B\overline{A}}.$
Computing adjoints on both sides again and using that $A^{*}B^{*}$
is closed by \prettyref{prop:inv-sum-comp-closed}, we get the assertion.

\item We note here that in \prettyref{cor:adj-comp1} and \prettyref{cor:adj-comp2} we $\overline{A}C$ cannot be replaced by $\overline{AC}$
 and encourage the reader to find a counterexample for $A$ being a
closable linear operator. We also refer to \cite{Picard2013a} for
a counterexample due to J. Epperlein.
\end{enumerate}
\end{rem}

We have already seen that $A^{*}={\overline{A}}^{*}$. We can even
restrict $A$ to a core and still obtain the same adjoint.

\begin{prop}
\label{prop:adjoint_core}Let $A\subseteq H_{0}\times H_{1}$ be a linear
relation, $D\subseteq\dom(A)$  a linear subspace. Then $D$
is a core for $A$ if and only if $\left(A\cap(D\times H_{1})\right)^{*}=A^{*}.$
\end{prop}

\begin{proof} We set $A|_D\coloneqq A\cap (D\times H_1)$. Then
\begin{align*}
D\text{ core} & \iff \overline{A|_D}=\overline{A} \iff \overline{A|_D}^\bot=\overline{A}^\bot 
\iff  {A|_D}^\bot={A}^\bot \iff A|_D^*=A^*.\qedhere
\end{align*}
\end{proof}

\section{The Spectrum and Resolvent Set}

In this section, we focus on operators acting on a single Banach space. As such, throughout this section let $X$ be a Banach space
over $\K\in\{\R,\C\}$ and let $A\colon\dom(A)\subseteq X\to X$ be
a closed linear operator.
\begin{defn*}
The set 
\[
\rho(A)\coloneqq\set{\lambda\in\K}{\left(\lambda-A\right)^{-1}\in\bo(X)}
\]
is called the \emph{resolvent set\index{resolvent set}} of $A$. We define
\[
\sigma(A)\coloneqq\K\setminus\rho(A)
\]
to be the \emph{spectrum\index{spectrum}} of $A$. 
\end{defn*}
We state and prove some elementary properties of the spectrum and
the resolvent set. We shall see natural examples for $A$ which satisfy that
$\sigma(A)=\K$ or $\sigma(A)=\emptyset$ later on.

For a metric space $(X,d)$, we will write $\ball{x}{r}=\set{y\in X}{d(x,y)<r}$ for the open ball around $x$ of radius $r$ and $\cball{x}{r}=\set{y\in X}{d(x,y)\leq r}$ for the closed ball.

\begin{prop}
\label{prop:resolvent}If $\lambda,\mu\in\rho(A)$, then the \emph{\index{resolvent identity}resolvent
identity} holds. That is
\[
\left(\lambda-A\right)^{-1}-\left(\mu-A\right)^{-1}=\left(\mu-\lambda\right)\left(\lambda-A\right)^{-1}\left(\mu-A\right)^{-1}.
\]
Moreover, the set $\rho(A)$ is open. More precisely, if $\lambda\in\rho(A)$
then $\ball{\lambda}{1\big/\norm{(\lambda-A)^{-1}}}\subseteq\rho(A)$
and for $\mu\in\ball{\lambda}{1\big/\norm{(\lambda-A)^{-1}}}$
we have
\[
\norm{(\mu-A)^{-1}}\leq\frac{\norm{(\lambda-A)^{-1}}}{1-\abs{\lambda-\mu}\norm{(\lambda-A)^{-1}}}.
\]

The mapping $\rho(A)\ni\lambda\mapsto(\lambda-A)^{-1}\in\bo(X)$
is analytic.
\end{prop}

\begin{proof}
For the first assertion, we let $\lambda,\mu\in\rho(A)$ and compute
\begin{align*}
(\lambda-A)^{-1}-(\mu-A)^{-1} & =(\lambda-A)^{-1}\bigl((\mu-A)-(\lambda-A)\bigr)(\mu-A)^{-1}\\
 & =(\lambda-A)^{-1}(\mu-\lambda)(\mu-A)^{-1}\\
 & =(\mu-\lambda)(\lambda-A)^{-1}(\mu-A)^{-1}.
\end{align*}
Next, let $\lambda\in\rho(A)$ and $\mu\in\ball{\lambda}{1/\norm{(\lambda-A)^{-1}}}$.
Then 
\[
\norm{(\lambda-\mu)(\lambda-A)^{-1}}<1.
\]
Hence, $1-(\lambda-\mu)(\lambda-A)^{-1}$ admits
an inverse in $\bo(X)$ satisfying
\begin{equation}
\left(1-(\lambda-\mu)(\lambda-A)^{-1}\right)^{-1}=\sum_{k=0}^{\infty}\left((\lambda-\mu)(\lambda-A)^{-1}\right)^{k}.\label{eq:Neu}
\end{equation}
We claim that $\mu\in\rho(A)$. For this, we compute
\begin{align*}
\mu-A & =\lambda-A-(\lambda-\mu)
 =(\lambda-A)\left(1-(\lambda-\mu)(\lambda-A)^{-1}\right).
\end{align*}
Since $\left(1-(\lambda-\mu)(\lambda-A)^{-1}\right)$ is
an isomorphism in $\bo(X)$, we deduce that the right-hand side admits
a continuous inverse if and only if the left-hand side does. As $\lambda\in\rho(A)$,
we thus infer $\mu\in\rho(A)$. The estimate follows from \prettyref{eq:Neu}.
Indeed, we have
\begin{align*}
\norm{(\mu-A)^{-1}} & \leq\norm{(\lambda-A)^{-1}}\norm{\sum_{k=0}^{\infty}\left((\lambda-\mu)(\lambda-A)^{-1}\right)^{k}}\\
 & \leq\norm{(\lambda-A)^{-1}}\sum_{k=0}^{\infty}\norm{(\lambda-\mu)(\lambda-A)^{-1}}^{k}
 =\frac{\norm{(\lambda-A)^{-1}}}{1-\norm{(\lambda-\mu)(\lambda-A)^{-1}}}.
\end{align*}

For the final claim of the present proposition, we observe that 
\begin{align*}
(\mu-A)^{-1} & =\left(1-(\lambda-\mu)(\lambda-A)^{-1}\right)^{-1}(\lambda-A)^{-1}
 =\sum_{k=0}^{\infty}(\lambda-\mu)^{k}\left((\lambda-A)^{-1}\right)^{k+1},
\end{align*}
which is an operator norm convergent power series expression for the
resolvent at $\mu$ about $\lambda$. Thus, analyticity follows.
\end{proof}
We shall consider multiplication operators in $\L(\mu)$ next. For
a measurable function $V\colon\Omega\to\R$ we will use the notation
$[V\leq c]\coloneqq V^{-1}\bigl[\loi{-\infty}{c}\bigr]$ for some
constant $c\in\mathbb{R}$ (and similarly for relational symbols other
than $\leq$).

\begin{thm}\label{thm:mult1}Let $\left(\Omega,\Sigma,\mu\right)$ be a measure
space and $V\colon\Omega\to\K$ a measurable function. Then the operator
\begin{align*}
V(\m)\colon\dom(V(\m))\subseteq\L(\mu) & \to\L(\mu)\\
f & \mapsto\bigl(\omega\mapsto V(\omega)f(\omega)\bigr),
\end{align*} 
with $\dom(V(\m))\coloneqq\set{f\in\L(\mu)}{\bigl(\omega \mapsto V(\omega)f(\omega)\bigr)\in\L(\mu)}$ satisfies the following prop\-er\-ties:
\begin{enumerate}
  \item $V(\m)$ is densely defined and closed.
  \item $\left(V(\m)\right)^{*}=V^{*}(\m)$ where $V^{*}(\omega)=V(\omega)^{*}$
for all $\omega\in\Omega$ (here $V(\omega)^{*}$ denotes the complex
conjugate of $V(\omega)$).
\item If $V$ is $\mu$-almost everywhere bounded,
then $V(\m)$ is continuous. Moreover, we have $\norm{V(\m)}_{\bo(\L(\mu))}\leq \norm{V}_{L_\infty(\mu)}$.
\item If $V\neq0$ $\mu$-a.e. then $V(\m)$ is
injective and $V(\m)^{-1}=\frac{1}{V}(\m)$, where 
\[\frac{1}{V}(\omega)\coloneqq\begin{cases}
\frac{1}{V(\omega)}, & V(\omega)\neq0,\\
0, & V(\omega)=0,
\end{cases}\]
for all $\omega\in\Omega$. 
\end{enumerate}  
\end{thm}
\begin{proof}
For the whole proof we let $\Omega_{n}\coloneqq[\abs{V}\leq n]$ and put $\1_{n}\coloneqq\1_{\Omega_{n}}$.  

(a) We first show that $V(\m)$ is densely defined.
Let $f\in\L(\mu)$. Then, we have for all $n\in\N$ that $\1_{n}f\in\dom(V(\m))$.
From $\Omega=\bigcup_{n}\Omega_{n}$ and $\Omega_{n}\subseteq\Omega_{n+1}$
it follows that $\1_{n}f\to f$ in $\L(\mu)$ as $n\to\infty$.

Next, we confirm that $V(\m)$ is closed. Let $(f_{k})_{k}$ in $\dom(V(\m))$
convergent in $\L(\mu)$ with $(V(\m)f_{k})_{k}$ be convergent in $\L(\mu)$.
Denote the respective limits by $f$ and $g$. It is clear that for
all $n\in\N$ we have $\1_{n}f_{k}\to\1_{n}f$ as $k\to\infty$. Also,
we have 
\[
\1_{n}g=\lim_{k\to\infty}\1_{n}V(\m)f_{k}=\lim_{k\to\infty}V(\m)(\1_{n}f_{k})=V(\m)(\1_{n}f) = \1_{n}Vf.
\]
Hence, $g=Vf$ $\mu$-almost everywhere and since $g\in\L(\mu)$,
we have that $f\in\dom(V(\m))$. 

(b) It is easy to see that $V^{*}(\m)\subseteq V(\m)^{*}$. For the other inclusion, we
let $u\in\dom(V(\m)^{*})$. Then, for all $f\in \L(\mu)$ and $n\in \N$ we have $\1_n f\in\dom(V(\m))$ and, hence,
\[
   \scp{f}{\1_nV^*u} = \int_{\Omega_n} f^* V^* u \d\mu = \scp{V(\m)(\1_nf)}{u}=\scp{\1_nf}{V(\m)^*u}=\scp{f}{\1_nV(\m)^*u}.
\]
It follows that $\1_nV^*u=\1_nV(\m)^*u$ for all $n\in\N$.  Thus, $\Omega=\bigcup_{n}\Omega_{n}$ implies $V^*u=V(\m)^*u$ and therefore $u\in \dom(V^*(\m))$ and $V^*(\m)u=V(\m)^*u$.

(c) If $\abs{V}\leq\kappa$ $\mu$-almost everywhere for some $\kappa\geq0$,
then for all $f\in\L(\mu)$ we have $\abs{V(\omega)f(\omega)}\leq\kappa\abs{f(\omega)}$
for $\mu$-almost every $\omega\in\Omega$. Squaring and integrating this
inequality yields boundedness of $V(\m)$ and the asserted inequality.

(d) Assume that $V\neq0$ $\mu$-a.e.~and $V(\m)f=0$. Then, $f(\omega)=0$
for $\mu$-a.e.~$\omega\in\Omega$, which implies $f=0$ in $\L(\mu)$.
Moreover, if $V(\m)f=g$ for $f,g\in\L(\mu)$, then for $\mu$-a.e.~$\omega\in\Omega$ we deduce that $f(\omega)=\frac{1}{V}(\omega)g(\omega)$,
which shows $\frac{1}{V}(\m)\supseteq V(\m)^{-1}$. If on the other
hand $g\in\dom\left(\frac{1}{V}(\m)\right)$, then a similar computation
reveals that $\frac{1}{V}(\m)g\in\dom(V(\m))$ and $V(\m)\frac{1}{V}(\m)g=g$. 
\end{proof}

The spectrum of $V(\m)$ from the latter example can be computed once
we consider a less general class of measure spaces. For later use, we provide a characterisation of these measure spaces first.

\begin{prop}\label{prop:mult1.5} Let $\left(\Omega,\Sigma,\mu\right)$ be a measure
space.
Then the following statements are equivalent:
\begin{enumerate}
\renewcommand{\labelenumi}{{\upshape (\roman{enumi})}}
  \item $\left(\Omega,\Sigma,\mu\right)$ is \emph{semi-finite}\index{semi-finite}, that is, for every $A\in\Sigma$ with $\mu(A)=\infty$,
there exists $B\in\Sigma$ with $0<\mu(B)<\infty$ such that $B\subseteq A$.
  \item For all measurable $V\colon\Omega\to\K$ with $V(\m)\in \bo(\L(\mu))$, we have $V\in L_\infty(\mu)$ and $\norm{V}_{L_\infty(\mu)}\leq \norm{V(\m)}_{\bo(\L(\mu))}$.
\end{enumerate}
\end{prop}
\begin{proof}
 (i)$\Rightarrow$(ii): Let $\varepsilon>0$ and $A_\varepsilon\coloneqq [\abs{V}\geq \norm{V(\m)}_{\bo(\L(\mu))}+\varepsilon]$. Assume that $\mu(A_\varepsilon)>0$. Since $(\Omega,\Sigma,\mu)$ is semi-finite we find $B_\varepsilon\subseteq A_\varepsilon$ such that $0<\mu(B_\varepsilon)<\infty$. Define $f\coloneqq \mu(B_\varepsilon)^{-1/2}\1_{B_\varepsilon} \in \L(\mu)$ with $\norm{f}_{\L(\mu)}=1$. Consequently, we obtain
 \[
     \norm{V(\m)}_{\bo(\L(\mu))}\geq      \norm{V(\m)f}_{\L(\mu)} \geq \norm{V(\m)}_{\bo(\L(\mu))}+\varepsilon,  
 \]which yields a contradiction, and hence (ii).
 
 (ii)$\Rightarrow$(i): Assume that $(\Omega,\Sigma,\mu)$ is not semi-finite. Then we find $A\in \Sigma$ with $\mu(A)=\infty$ such that for each $B\subseteq A$ measurable, we have $\mu(B)\in \{0,\infty\}$. Then $V\coloneqq \1_{A}$ is bounded and measurable with $\norm{V}_{L_\infty(\mu)}=1$. However, $V(\m)=0$. Indeed, if $f\in \L(\mu)$ then $[f\neq 0]=\bigcup_{n\in\N}[\abs{f}^2\geq n^{-1}]$. Thus, 
 \[
 [V(\m)f\neq 0] = [f\neq 0]\cap A = \bigcup_{n\in\N} [\abs{f}^2\geq n^{-1}]\cap A.
 \] Since $\mu([\abs{f}^2\geq n^{-1}])<\infty$ as $f\in \L(\mu)$, we infer $\mu([\abs{f}^2\geq n^{-1}]\cap A)=0$ by the property assumed for $A$. Thus, $\mu([V(\m)f\neq 0])=0$ implying $V(\m)=0$. Hence, $\norm{V(\m)}_{\bo(\L(\mu))}=0<1=\norm{V}_{L_\infty(\mu)}$.
 \end{proof}

A straightforward consequence of \prettyref{thm:mult1} (c) and \prettyref{prop:mult1.5} is the following.

\begin{prop}\label{prop:mult1.75} Let $\left(\Omega,\Sigma,\mu\right)$ be a semi-finite measure
space, $V\colon \Omega\to\K$ measurable and bounded. Then $\norm{V(\m)}_{\bo(\L(\mu))}=\norm{V}_{L_\infty(\mu)}$.
\end{prop}

\begin{thm}
\label{thm:mult2}Let $\left(\Omega,\Sigma,\mu\right)$ be a semi-finite measure
space and let $V\colon\Omega\to\K$ be measurable. Then
\[
\sigma\left(V(\m)\right)=\essran V\coloneqq\set{\lambda\in\K}{\forall\varepsilon>0\colon\mu\left(\left[\abs{\lambda-V}<\varepsilon\right]\right)>0}.
\]
\end{thm}
\begin{proof}
Let $\lambda\in\essran V$. For all $n\in\N$ we find $B_{n}\in\Sigma$
with non-zero, but finite measure such that $B_{n}\subseteq\left[\abs{\lambda-V}<\frac{1}{n}\right].$
We define $f_{n}\coloneqq\sqrt{\frac{1}{\mu(B_{n})}}\1_{B_{n}}\in\L(\mu)$.
Then $\norm{f_{n}}_{\L(\mu)}=1$ and 
\[
\abs{V(\omega)f_{n}(\omega)}\leq\abs{V(\omega)-\lambda}\abs{f_{n}(\omega)}+\abs{\lambda}\abs{f_{n}(\omega)}\leq\left(\frac{1}{n}+\abs{\lambda}\right)\abs{f_{n}(\omega)}
\]
for $\omega\in\Omega,$ which shows that $(f_{n})_{n}$ is in $\dom(V(\m))$. A similar estimate, on the other hand, shows that
\[
\norm{\left(V(\m)-\lambda\right)f_{n}}_{\L(\mu)}\to0\quad(n\to\infty).
\]
Thus, $\left(V(\m)-\lambda\right)^{-1}$ cannot be continuous as $\|f_{n}\|_{\L(\mu)}=1$
for all $n\in\mathbb{N}$. 

Let now $\lambda\in\K\setminus\essran V$. Then there exists $\varepsilon>0$
such that $N\coloneqq\left[\abs{\lambda-V}<\varepsilon\right]$ is
a $\mu$-null set. In particular, $\lambda-V\neq0$ $\mu$-a.e. Hence,
$\left(\lambda-V(\m)\right)^{-1}=\frac{1}{\lambda-V}(\m)$ is a linear
operator. Since, $\abs{\frac{1}{\lambda-V}}\leq1/\varepsilon$ $\mu$-almost
everywhere, we deduce that $\left(\lambda-V(\m)\right)^{-1}\in\bo(\L(\mu))$
and hence, $\lambda\in\rho(V(\m))$. 
\end{proof}

We conclude this chapter by stating that multiplication operators
as discussed in \prettyref{thm:mult1}, \prettyref{prop:mult1.5}, \prettyref{prop:mult1.75}, and \prettyref{thm:mult2}
are \emph{the} prototypical example for normal operators. It is also important to note that, as we have seen in \prettyref{thm:mult1},
 a multiplication operator in $\L(\mu)$ is self-adjoint if and
only if $V$ assumes values in the real numbers, only.

\section{Comments}

The material presented in this lecture is basic textbook knowledge.
We shall thus refer to the monographs \cite{Kato1995,Weidmann1980}.
Note that spectral theory for self-adjoint operators is a classical
topic in functional analysis. For a glimpse on further theory of linear
relations we exemplarily refer to \cite{Arens1961,Berger2016,Cross1998}. 
The restriction in \prettyref{prop:mult1.75} and \prettyref{thm:mult2} to semi-finite measure spaces is not very severe. In fact, if $(\Omega,\Sigma,\mu)$ was not semi-finite, it is possible to construct a semi-finite measure space $(\Omega_{\textnormal{loc}},\Sigma_{\textnormal{loc}},\mu_{\textnormal{loc}})$ such that $L_p(\mu)$ is isometrically isomorphic to $L_p(\mu_{\textnormal{loc}})$, see \cite[Section 2]{VV2017}.

\section*{Exercises}
\addcontentsline{toc}{section}{Exercises}

\begin{xca}
Let $A\subseteq X_{0}\times X_{1}$ be an unbounded linear operator. 
Show that for every linear operator $B\subseteq X_{0}\times X_{1}$
with $B\supseteq A$ and $\dom(B)=X_{0}$, we have that $B$ is not
closed.
\end{xca}

\begin{xca}\label{exer:propcor}
Prove \prettyref{prop:coreBounded} and \prettyref{cor:uniquecont}. Hint: One might use that bounded linear relations are always operators.
\end{xca}

\begin{xca}\label{exer:lem_adj}
Prove \prettyref{lem:adj}.
\end{xca}

\begin{xca}\label{exer:resolvent_adj}
Let $A\colon\dom(A)\subseteq H_{0}\to H_{0}$ be a closed and densely
defined linear operator. Show that for all $\lambda\in\K$ we have
\[
\lambda\in\rho(A)\iff\lambda^{*}\in\rho(A^{*}).
\]
\end{xca}

\begin{xca}
Let $U\subseteq H_{0}\times H_{1}$ satisfy $U^{-1}=U^{*}$. Show
that $U\in L(H_{0},H_{1})$ and that $U$ is \emph{unitary\index{unitary}},
that is, $U$ is onto and for all $x\in H_{0}$ we have $\norm{Ux}_{H_1}=\norm{x}_{H_0}$.
\end{xca}

\begin{xca}
Let $\delta\colon C\ci{0}{1}\subseteq\L(0,1)\to\K,f\mapsto f(0)$,
where $C\ci{0}{1}$ denotes the set of $\K$-valued continuous functions
on $[0,1]$. Show that $\delta$ is not closable. Compute $\overline{\delta}$.
\end{xca}

\begin{xca}
Let $C\subseteq\C$ be closed. Provide a Hilbert space $H$ and a densely
defined closed linear operator $A$ on $H$ such that $\sigma(A)=C$.
\end{xca}

  \printbibliography[heading=subbibliography]
  
  \chapter{The Time Derivative}

It is the aim of this chapter to define a derivative operator on a
suitable $\L$-space, which will be used as the derivative with respect
to the temporal variable in our applications. As we want to deal with
Hilbert space-valued functions, we start by introducing the concept
of Bochner--Lebesgue spaces, which generalises the classical scalar-valued
$L_{p}$-spaces to the Banach space-valued case. 

\section{Bochner--Lebesgue Spaces}

Throughout, let $(\Omega,\Sigma,\mu)$ be a $\sigma$-finite measure space and $X$
a Banach space over the field $\mathbb{K}\in\{\R,\C\}$. We are aiming
to define the spaces $L_{p}(\mu;X)$ for $1\leq p\leq\infty$. This is the
space of (equivalence classes of) measurable functions attaining values
in $X$, which are $p$-integrable (if $p<\infty)$, or essentially
bounded (if $p=\infty$) with respect to the measure $\mu.$ We begin
by defining the space of simple functions on $\Omega$ with values
in $X$ and the notion of Bochner-measurability.
\begin{defn*}
For a function $f\from \Omega\to X$ and $x\in X$ we set 
\[
A_{f,x}\coloneqq f^{-1}[\{x\}].
\]
A function $f\from\Omega\to X$ is called \emph{simple}\index{simple function}
if $f[\Omega]$ is finite and for each $x\in X\setminus\{0\}$ the
set $A_{f,x}$ belongs to $\Sigma$ and has finite measure. We denote
the set of simple functions by $S(\mu;X)$.
A function $f\from\Omega\to X$ is called \index{Bochner-measurable}\emph{Bochner-measurable}
if there exists a sequence $(f_{n})_{n\in\N}$ in $S(\mu;X)$ such
that 
\[
f_{n}(\omega)\to f(\omega)\quad(n\to\infty)
\]
for $\mu$-a.e.~$\omega\in\Omega$.
\end{defn*}
\begin{rem} 
\label{rem:Bochner}
\begin{enumerate}
\item For a simple function $f$ we have 
\[
 f=\sum_{x\in X} x\cdot\1_{A_{f,x}},
\]
where the sum is actually finite, since $\1_{A_{f,x}}=0$ for all $x\notin f[\Omega]$.
\item\label{rem:Bochner:item:2} If $X=\K$, then a function is Bochner-measurable if and only if it has a $\mu$-measurable representative.
\item It is easy to check that $S(\mu;X)$ is a vector space and an $S(\mu;\K)$-module;
that is, for $f\in S(\mu;X)$ and $g\in S(\mu;\K)$ we have 
$g\cdot f\in S(\mu;X)$.

\item\label{rem:Bochner:item:4} If $f\from\Omega\to X$ is Bochner-measurable, then $\norm{f(\cdot)}_{X}\from\Omega\to\R$
is Bochner-measurable. Indeed, since 
\[
\norm{f(\cdot)}_{X}=\lim_{n\to\infty}\norm{f_{n}(\cdot)}_{X}
\]
for $\mu$-a.e.~$\omega\in\Omega$ and a sequence $(f_{n})_{n\in\N}$
in $S(\mu;X)$, it suffices to show that $\norm{f_{n}(\cdot)}_{X}$
is simple for all $n\in\N.$ The latter follows since $A_{f,x}\cap A_{f,y}=\emptyset$
for $x\ne y$ and thus
\[
\norm{f_{n}(\cdot)}_{X}=\sum_{x\in f_{n}[\Omega]}\norm{x}_{X}\cdot\1_{A_{f,x}}
\]
is a real-valued simple function.

\item If one deals with arbitrary measure spaces, the definition of simple functions has to be weakened by allowing the sets $A_{f,x}$ to have infinite measure. However, since in the applications to follow we only work with weighted Lebesgue measures, we restrict ourselves to $\sigma$-finite measure spaces. 
\end{enumerate}
\end{rem}

Next, we state and prove the celebrated Theorem of Pettis, which characterises Bochner-measurabilty in terms of weak measurability. In what follows, let $X':=\bo(X,\K)$\index{dual space $X'$} denote the dual space of $X$.
\begin{thm}[Theorem of Pettis]
 \label{thm:Pettis}
 Let $f:\Omega\to X.$
Then $f$ is Bochner-measurable if and only if
\begin{enumerate}
\item $f$ is \emph{weakly Bochner-measurable\index{weakly Bochner-measurable}}; that is, $x'\circ f:\Omega\to\mathbb{K}$
is Bochner-measurable for each $x'\in X'$, and
\item $f$ is \emph{almost separably-valued\index{almost separably-valued}}; that is, $\overline{\lin f[\Omega\setminus N_{0}]}$
is separable for some $N_{0}\in\Sigma$ with $\mu(N_{0})=0$. 
\end{enumerate}
\end{thm}

\begin{proof}
If $f$ is Bochner-measurable, then clearly it is weakly Bochner-measurable.
Further, as $f$ is the almost everywhere limit of simple functions,
it is almost separably-valued, since each simple function attains
values in a finite dimensional subspace of $X$.\\
Assume now conversely that $f$ satisfies (a) and (b). We define $Y\coloneqq\overline{\lin f[\Omega\setminus N_{0}]}$,
which is a separable Banach space by (b). Thus, there exists a sequence
$(x_{n}')_{n\in\mathbb{N}}$ in $X'$ such that 
\[
\|y\|=\sup_{n\in\mathbb{N}}|x_{n}'(y)|\quad(y\in Y).
\]
Since for each $n\in\N$ the function $g_{n}\coloneqq|x_{n}'\circ f|$
is Bochner-measurable by (a) and \prettyref{rem:Bochner}\ref{rem:Bochner:item:4}, we find a $\mu$-nullset $N_{n}$ and
a measurable function $\tilde{g}_{n}:\Omega\to\R$ such that $g_{n}=\tilde{g}_{n}$
on $\Omega\setminus N_{n}$ by \prettyref{rem:Bochner}\ref{rem:Bochner:item:2}. Then $\sup_{n\in\N}\tilde{g}_{n}(\cdot)$
is measurable and 
\[
\|f(\omega)\|=\sup_{n\in\N}\tilde{g}_{n}(\omega)\quad(\omega\in\Omega\setminus N),
\]
where $N\coloneqq\bigcup_{n\in\N_{0}}N_{n}$, which shows that $\|f(\cdot)\|$
is Bochner-measurable. Let $\varepsilon>0$, $(y_{n})_{n\in\mathbb{N}}$
a dense sequence in $Y$. Applying the previous argument to the function
$f_{k}(\cdot)\coloneqq f(\cdot)-y_{k}$ for $k\in\N$ we infer that
$\|f_{k}(\cdot)\|$ is Bochner measurable and hence, there is a $\mu$-nullset
$N_{k}'$ and a measurable funtion $\tilde{f_{k}}:\Omega\to\R$ such
that $\|f_{k}\|=\tilde{f}_{k}$ on $\Omega\setminus N_{k}'$. Consequently,
the sets 
\[
E_{k}\coloneqq[\tilde{f}_{k}\leq\varepsilon]\coloneqq \set{\omega\in \Omega}{|\tilde{f}_k(\omega)|\leq \varepsilon}\quad(k\in\mathbb{N})
\]
are measurable. Moreover, by the density of $\{y_{n}\,;\,n\in\N\}$
in $Y$, we get that $\Omega\setminus N'\subseteq\bigcup_{k\in\N}E_{k}$
with $N'\coloneqq\bigcup_{k=1}^{\infty}N_{k}'\cup N_{0}$. Setting
$F_{1}\coloneqq E_{1}$ and $F_{n+1}=E_{n+1}\setminus\bigcup_{k=1}^{n}F_{k}$
for $n\in\mathbb{N}$, we obtain a sequence of pairwise disjoint measurable
sets $(F_{n})_{n\in\mathbb{N}}$ with $\Omega\setminus N'\subseteq\bigcup_{n\in\mathbb{N}}F_{n}.$
We set 
\[
g\coloneqq\sum_{k=1}^{\infty}y_{k}\1_{F_{k}}
\]
and obtain $\|f(\omega)-g(\omega)\|\leq\varepsilon$ for each $\omega\in\Omega\setminus N'.$
Hence, if $g$ is Bochner-measurable, then $f$ is Bochner-measurable as well. For showing the Bochner-measurability of $g$, let $(\Omega_{k})_{k\in\mathbb{N}}$
be a sequence of pairwise disjoint measurable sets such that $\bigcup_{k\in\mathbb{N}}\Omega_{k}=\Omega$
and $\mu(\Omega_{k})<\infty$ for each $k\in\mathbb{N}.$ For $n\in\mathbb{N}$
we set 
\[
g_{n}\coloneqq\sum_{k,j=1}^{n}y_{k}\1_{F_{k}\cap\Omega_{j}}.
\]
Then $\left(g_{n}\right)_{n\in\mathbb{N}}$ is a sequence of simple
functions with $g_{n}\to g$ pointwise as $n\to\infty$ and thus,
$g$ is Bochner-measurable.
\end{proof}

\begin{cor}\label{cor:limit-measurable}
 Let $f_n,f:\Omega \to X$ for $n\in \N$. Moreover, assume that $f_n$ is Bochner-measurbale for each $n\in \N$ and $f_n(\omega)\to f(\omega)$ as $n\to \infty$ for $\mu$-almost every $\omega \in \Omega$. Then $f$ is Bochner-measurable 
\end{cor}

\begin{proof}
 By \prettyref{thm:Pettis} we find $\mu$-nullsets $N_n\in \Sigma$ such that $X_n\coloneqq \overline{\lin f_n[\Omega\setminus N_n]}$ is separable. Moreover, we find a $\mu$-nullset $N'\in \Sigma$ such that 
 \[
  f_n(\omega)\to f(\omega)\quad (n\to \infty, \omega \in \Omega \setminus N').
 \]
We set $N\coloneqq \bigcup_{n\in \N} N_n\cup N'$. Then 
\[
 \overline{\lin f[\Omega \setminus N]} \subseteq \overline{\lin \bigcup_{n\in \N} X_n}
\]
and thus, $f$ is almost separably-valued. Moreover, for $x'\in X'$ we have that 
\[
 (x'\circ f)(\omega)=\lim_{n\to \infty} (x'\circ f_n)(\omega)\quad (\omega \in \Omega \setminus N')
\]
and since all functions $x'\circ f_n$ are $\mu$-measurable outside a $\mu$-nullset, so is $x'\circ f$ and thus, $x'\circ f$ is Bochner-measurable, see \prettyref{rem:Bochner}\ref{rem:Bochner:item:2}. Thus, \prettyref{thm:Pettis} yields the Bochner-measurability of $f$.
\end{proof}

\begin{defn*}[Bochner--Lebesgue spaces]
 \index{Bochner-Lebesgue spaces}For $p\in\ci{1}{\infty}$ we define
\[
\mathcal{L}_{p}(\mu;X)\coloneqq\set{f\from\Omega\to X}{f\mbox{ Bochner-measurable},\,\norm{f(\cdot)}_{X}\in \mathcal{L}_{p}(\mu)},
\]
as well as 
\[
L_{p}(\mu;X)\coloneqq\faktor{\mathcal{L}_{p}(\mu;X)}{\sim},
\]
where $\sim$ denotes the usual equivalence relation of equality $\mu$-almost
everywhere.
\end{defn*}
\begin{prop}
\label{prop:Lp Banach space}Let $p\in\ci{1}{\infty}.$ Then 
\[
\norm{f}_{p}\coloneqq\begin{cases}
\left(\intop_{\Omega}\norm{f(\omega)}_{X}^{p}\d\mu(\omega)\right)^{\frac{1}{p}}, & \text{ if }p<\infty,\\
\esssup_{\omega\in\Omega}\norm{f(\omega)}_{X}, & \text{ if }p=\infty,
\end{cases}
\]
defines a seminorm on $\mathcal{L}_{p}(\mu;X)$ where $\norm{f}_{p}=0$
if and only if $f=0$ $\mu$-a.e. Consequently, $\norm{\cdot}_{p}$
defines a norm on $L_{p}(\mu;X).$ Moreover, $(L_{p}(\mu;X),\norm{\cdot}_{p})$
is a Banach space and if $X=H$ is a Hilbert space, then so too is $L_{2}(\mu;H)$
with the scalar product given by 
\[
\scp{f}{g}_{2}\coloneqq\intop_{\Omega}\scp{f(\omega)}{g(\omega)}_{H}\d\mu(\omega)\quad(f,g\in L_{2}(\mu;H)).
\]
\end{prop}

\begin{proof}
We just show the completeness of $L_{p}(\mu;X).$ Let $(f_{n})_{n\in\N}$
be a sequence in $L_{p}(\mu;X)$ such that $\sum_{n=1}^{\infty}\norm{f_{n}}_{p}<\infty$.
We set 
\[
g_{n}(\omega)\coloneqq\norm{f_{n}(\omega)}_X\quad(n\in\N,\omega\in\Omega).
\]
Then $(g_{n})_{n\in\N}$ is a sequence in $L_{p}(\mu)$ such that
$\sum_{n=1}^{\infty}\norm{g_{n}}_{p}<\infty.$ By the completeness
of $L_{p}(\mu)$ we infer that 
\[
g\coloneqq\sum_{n=1}^{\infty}g_{n}
\]
exists and is an element in $L_{p}(\mu).$ In particular, $g(\omega)<\infty$
for $\mu$-a.e.~$\omega\in\Omega$ and thus, 
\[
\sum_{n=1}^{\infty}\norm{f_{n}(\omega)}_{X}=\sum_{n=1}^{\infty}g_{n}(\omega)<\infty
\]
for $\mu$-a.e.~$\omega\in\Omega.$ By the completeness of
$X$ we can define 
\[
f(\omega)\coloneqq\sum_{n=1}^{\infty}f_{n}(\omega)
\]
for $\mu$-a.e.~$\omega\in\Omega$. Note that $f$ is Bochner-measurable by \prettyref{cor:limit-measurable}. We need to prove that
$f\in L_{p}(\mu;X)$ and that $\sum_{n=1}^{k}f_{n}\to f$ in $L_{p}(\mu;X)$
as $k\to\infty.$  For this, it suffices to prove that
\begin{equation}
\sum_{n=k}^{\infty}f_{n}\in L_{p}(\mu;X)\mbox{ and }\sum_{n=k}^{\infty}f_{n}\to0\mbox{ in }L_{p}(\mu;X)\mbox{ as }k\to\infty.\label{eq:series_to_zero}
\end{equation}
Indeed, this would imply both that $f-\sum_{n=1}^{k}f_{n}\in L_{p}(\mu;X)$
and the desired convergence result. We prove \prettyref{eq:series_to_zero}
for $p<\infty$ and $p=\infty$ separately.

First, let $p=\infty$. For each $n\in\N$ we have $f_{n}\in L_{\infty}(\mu;X)$
and thus $\norm{f_{n}(\omega)}_{X}\leq\norm{f_{n}}_{\infty}$ for
all $\omega\in\Omega\setminus N_{n}$ and some null set $N_{n}\subseteq\Omega$. We set $N\coloneqq\bigcup_{n=1}^{\infty}N_{n},$
which is again a null set. For $k\in\N$ and $\omega\in\Omega\setminus N$
we then estimate 
\[
\norm{\sum_{n=k}^{\infty}f_{n}(\omega)}_{X}\leq\sum_{n=k}^{\infty}\norm{f_{n}(\omega)}_{X}\leq\sum_{n=k}^{\infty}\norm{f_{n}}_{\infty},
\]
which yields \prettyref{eq:series_to_zero}.

Now, let $p<\infty$. For $k\in\N$ we estimate 
\begin{align*}
\left(\intop_{\Omega}\left(\norm{\sum_{n=k}^{\infty}f_{n}(\omega)}_{X}\right)^{p}\d\mu(\omega)\right)^{\frac{1}{p}} & \leq\left(\intop_{\Omega}\left(\sum_{n=k}^{\infty}\norm{f_{n}(\omega)}_{X}\right)^{p}\d\mu(\omega)\right)^{\frac{1}{p}}\\
 & =\left(\intop_{\Omega}\lim_{m\to\infty}\left(\sum_{n=k}^{m}\norm{f_{n}(\omega)}_{X}\right)^{p}\d\mu(\omega)\right)^{\frac{1}{p}}\\
 & =\lim_{m\to\infty}\left(\intop_{\Omega}\left(\sum_{n=k}^{m}\norm{f_{n}(\omega)}_{X}\right)^{p}\d\mu(\omega)\right)^{\frac{1}{p}}\\
 & \leq\lim_{m\to\infty}\sum_{n=k}^{m}\norm{f_{n}}_{p}
 =\sum_{n=k}^{\infty}\norm{f_{n}}_{p},
\end{align*}
where we have used monotone convergence in the third line. This estimate
yields \prettyref{eq:series_to_zero}.
\end{proof}

We now want to define an $X$-valued integral for functions in $L_{1}(\mu;X)$;
the so-called Bochner-integral. To do this, we need the following
density result.
\begin{lem}
\label{lem:simple_fcts_dense}The space $S(\mu;X)$ is dense in $L_{p}(\mu;X)$
for $p\in\roi{1}{\infty}$.
\end{lem}

\begin{proof}
Let $f\in L_{p}(\mu;X).$ Then there exists a sequence $(f_{n})_{n\in\N}$
in $S(\mu;X)$ such that $f_{n}(\omega)\to f(\omega)$ for all $\omega\in\Omega\setminus N$
for some null set $N\subseteq\Omega$. W.l.o.g. we may assume that $\norm{f_n(\cdot)}_X$ and $\norm{f(\cdot)}_X$ are measurable on $\Omega\setminus N$ for each $n\in \N$. For
$n\in\N$ we define the set 
\[
I_{n}\coloneqq\set{\omega\in\Omega\setminus N}{\norm{f_{n}(\omega)}_{X}\leq2\norm{f(\omega)}_{X}}\in\Sigma,
\]
and set $\tilde{f}_{n}\coloneqq f_{n}\1_{I_{n}}$. Then $\tilde{f}_{n}\in S(\mu;X)$
and we claim that $\tilde{f}_{n}(\omega)\to f(\omega)$ for all $\omega\in\Omega\setminus N$.
Indeed, if $f(\omega)=0$ then $\tilde{f}_{n}(\omega)=0$ also and
the claim follows. If $f(\omega)\ne0$, then there is some $n_{0}\in\N$
such that $\norm{f_{n}(\omega)}_{X}\leq2\norm{f(\omega)}_{X}$ for
$n\ge n_{0}$, and hence $\omega\in\bigcap_{n\geq n_{0}}I_{n}.$ Consequently
$\tilde{f}_{n}(\omega)=f_{n}(\omega)\to f(\omega).$ By dominated
convergence, it now follows that 
\[
\intop_{\Omega}\norm{\tilde{f}_{n}(\omega)-f(\omega)}_{X}^{p}\d\mu(\omega)\to0\quad(n\to\infty),
\]
which proves the claim. 
\end{proof}
\begin{prop}
\label{prop:Bochner-integral}The mapping\footnote{Note that the sum is indeed finite and all summands are well-defined
if we set $0_{X}\cdot\infty=0_{X}.$} 
\begin{align*}
\int_{\Omega}\d\mu\colon S(\mu;X)\subseteq L_{1}(\mu;X) & \to X\\
f & \mapsto\sum_{x\in X} x\cdot \mu(A_{f,x})
\end{align*}
is linear and continuous, and thus has a unique continuous linear extension
to $L_{1}(\mu;X)$, called the \emph{Bochner-integral\index{Bochner-integral}}. Moreover,
\[
\norm{\int_{\Omega}f\d\mu}_{X}\leq\norm{f}_{1}\quad(f\in L_{1}(\mu;X)),
\]
and for $A\in\Sigma,f\in L_{1}(\mu;X)$ we set 
\[
\int_{A}f\d\mu\coloneqq\int_{\Omega}f\cdot\1_{A}\d\mu.
\]
 
\end{prop}

\begin{proof}
We first show linearity. Let $f,g\in S(\mu;X)$ and $\lambda\in\K$. We then compute
\begin{align*}
\int_{\Omega}(\lambda f+g)\d\mu & =\sum_{x\in X}x\cdot\mu(A_{\lambda f+g,x})
 =\sum_{x\in X}x\cdot\mu\left(\bigcup_{y\in X}\left(A_{f,y}\cap A_{g,x-\lambda y}\right)\right)\\
 & =\sum_{x\in X}\sum_{y\in X}x\cdot\mu(A_{f,y}\cap A_{g,x-\lambda y})\\
 & =\sum_{y\in X}\sum_{x\in X}\lambda y\cdot\mu(A_{f,y}\cap A_{g,x-\lambda y})+\sum_{x\in X}\sum_{y\in X}\left(x-\lambda y\right)\mu(A_{f,y}\cap A_{g,x-\lambda y})\\
 & =\lambda\sum_{y\in X}y\cdot\mu\left(A_{f,y}\cap\bigcup_{x\in X}A_{g,x-\lambda y}\right)+\sum_{z\in X}z\cdot\mu\left(\bigcup_{y\in X}A_{f,y}\cap A_{g,z}\right)\\
 & =\lambda\int_{\Omega}f\d\mu+\int_{\Omega}g\d\mu,
\end{align*}
where we have used the fact that all occuring unions and sums are
finite. In order to prove continuity, let $f\in S(\mu;X)$. We estimate 
\begin{align*}
\norm{\int_{\Omega}f\d\mu}_{X} & =\norm{\sum_{x\in f[\Omega]}x\cdot\mu(A_{f,x})}_X
 \leq\sum_{x\in f[\Omega]}\norm{x}_{X}\mu(A_{f,x})
 =\int_{\Omega}\sum_{x\in f[\Omega]}\norm{x}_{X}\1_{A_{f,x}}\d\mu\\
 & =\int_{\Omega}\norm{f(\cdot)}_{X}\d\mu
 =\norm{f}_{1}.
\end{align*}
The remaining assertions now follow from \prettyref{lem:simple_fcts_dense}
by continuous extension (see \prettyref{cor:uniquecont}).
\end{proof}

The next proposition tells us how the Bochner-integral of a function behaves if we compose the function with a bounded or closed linear operator first.

\begin{prop}
\label{prop:interchange_integral}Let $f\in L_{1}(\mu;X)$, $Y$ a Banach
space. 
\begin{enumerate}
\item Let $B\in\bo(X,Y)$. Then $B\circ f\in L_{1}(\mu;Y)$ and 
\[
\int_{\Omega}B\circ f\d\mu=B\int_{\Omega}f\d\mu.
\]
\item If $X_{0}\subseteq X$ is a closed subspace and $f(\omega)\in X_{0}$
for $\mu$-a.e.~$\omega\in\Omega$, then $\int_{\Omega}f\d\mu\in X_{0}$.
\item (Theorem of Hille)\index{Theorem of Hille} Let $A\colon\dom(A)\subseteq X\to Y$
be a closed linear operator and assume that $f(\omega)\in\dom(A)$ for
$\mu$-a.e.~$\omega\in\Omega$ and that $A\circ f\in L_{1}(\mu;Y)$.
Then $\int_{\Omega}f\d\mu\in\dom(A)$ and 
\[
A\int_{\Omega}f\d\mu=\int_{\Omega}A\circ f\d\mu.
\]
\end{enumerate}
\end{prop}

\begin{proof}
\begin{enumerate}

\item Let $(f_{n})_{n\in\N}$ be a sequence in $S(\mu;X)$ such that
$f_{n}\to f$ $\mu$-a.e. Then $B\circ f_{n}\in S(\mu;Y)$ since
\[
\left(B\circ f_{n}\right)[\Omega]=B\bigl[f_{n}[\Omega]\bigr]
\]
is finite and for $y\in(B\circ f_{n})[\Omega]\setminus\{0\}$ we have
that 
\[
A_{B\circ f_{n},y}=\bigcup_{x\in B^{-1}[\{y\}]\cap f_{n}[\Omega]}A_{f_{n},x}\in\Sigma,
\]
and thus, $\mu(A_{B\circ f_{n},y})<\infty.$ Due to the continuity
of $B$ we have that $B\circ f_{n}\to B\circ f$ $\mu$-a.e., hence
$B\circ f$ is Bochner-measurable. Moreover, $\norm{B\circ f(\cdot)}_{Y}\leq\norm{B}\norm{f(\cdot)}_{X}$,
which yields that $B\circ f\in L_{1}(\mu;Y).$ By continuity of both $B$
and $\int_{\Omega}\d\mu$, it suffices to check the interchanging property
for any $f\in S(\mu;X)$ alone. However, this is clear, since for a simple function $f$
\[
 B\circ f = B\left(\,\sum_{x\in X} x \cdot \1_{A_{f,x}}\right)= \sum_{x\in X} Bx \cdot \1_{A_{f,x}},
\]
where the sum is actually finite and hence, 
\begin{align*}
\int_{\Omega}B\circ f\d\mu & =\int_\Omega \sum_{x\in X} Bx\cdot\1_{A_{f,x}} \d\mu= \sum_{x\in X} \int_\Omega Bx\cdot \1_{A_{f,x}}\d\mu\\
 &=\sum_{x\in X}Bx\cdot\mu(A_{f,x})= B\left(\,\sum_{x\in X} x\cdot\mu(A_{f,x})\right)=B\,\int_{\Omega}f\d\mu,
\end{align*}
where in the third equality we have used that $Bx\cdot\1_{A_{f,x}}$ is a simple function.
\item Let $x'\in X'$ with $x'|_{X_{0}}=0.$ It follows from (a) that
\[
x'\left(\int_{\Omega}f\d\mu\right)=\int_{\Omega}x'\circ f\d\mu=0,
\]
and since $x'$ was arbitrary, it follows that $\int_{\Omega}f\d\mu\in X_{0}$
from the Theorem of Hahn--Banach.

\item Consider the space $L_{1}(\mu;X\times Y).$ By assumption,
it follows that 
\[
(f,A\circ f)\in L_{1}(\mu;X\times Y).
\]
However, $(f,A\circ f)(\omega)=(f(\omega),\left(A\circ f\right)(\omega))\in A\subseteq X\times Y$
for $\mu$-a.e.~$\omega\in\Omega$, and since $A$ is closed we can
use (b) to derive that 
\begin{equation}
\int_{\Omega}(f,A\circ f)\d\mu\in A.\label{eq:int in A}
\end{equation}
Let $\pi_{1},\pi_{2}$ be the projection from $X\times Y$ to $X$
and $Y$, respectively. It then follows from part (a) that 
\[
\pi_{1}\left(\int_{\Omega}(f,A\circ f)\d\mu\right)=\int_{\Omega}\pi_{1}(f,A\circ f)\d\mu=\int_{\Omega}f\d\mu,
\]
and analogously for $\pi_{2}.$ Using these equalities we derive from
\prettyref{eq:int in A} that $\int_{\Omega}f\d\mu\in\dom(A)$ and that
$A\int_{\Omega}f\d\mu=\int_{\Omega}A\circ f\d\mu$.\qedhere

\end{enumerate}
\end{proof}
As a consequence of the latter proposition, we derive the fundamental
theorem of calculus for Banach space-valued functions.
\begin{cor}[fundamental theorem of calculus]
\label{cor:Hauptsatz}\index{fundamental theorem of calculus}Let
$a,b\in\R,a<b$ and consider the measure space $(\ci{a}{b},\mathcal{B}(\ci{a}{b}),\lambda)$,
where $\mathcal{B}(\ci{a}{b})$ denotes the Borel-$\sigma$-algebra
of $\ci{a}{b}$ and $\lambda$ is the Lebesgue measure. Let $f\from\ci{a}{b}\to X$
be continuously differentiable.\footnote{By this we mean that $f$ is continuous on $\ci{a}{b}$, continuously
differentiable on $\oi{a}{b}$ and $f'$ has a continuous extension
to $\ci{a}{b}$.} Then 
\[
f(b)-f(a)=\int_{\ci{a}{b}}f'\d\lambda.
\]
\end{cor}

\begin{proof}
Note first of all that continuous functions are Bochner-measurable (which can be easily seen using \prettyref{thm:Pettis}). Thus, the integral
on the right-hand side is well-defined. Let $\varphi\in X'.$ Then $\varphi\circ f\from[a,b]\to\K$
is continuously differentiable, and $\left(\varphi\circ f\right)'(t)=\left(\varphi\circ f'\right)(t)$.
Using \prettyref{prop:interchange_integral}~(a) together with the fundamental
theorem of calculus for the scalar-valued case we get 
\begin{align*}
\varphi\left(\int_{\ci{a}{b}}f'\d\lambda\right) & =\int_{\ci{a}{b}}\left(\varphi\circ f'\right)\d\lambda
 =\varphi\left(f(b)\right)-\varphi\left(f(a)\right)
 =\varphi\left(f(b)-f(a)\right).
\end{align*}
Since this holds for all $\varphi\in X',$ the assertion follows from
the Theorem of Hahn--Banach.
\end{proof}
We conclude this section with a density result, which will be useful
throughout the course.
\begin{lem}
\label{lem:dense sets}Let $\mathcal{D}\subseteq \L(\mu)$ be total
in $\L(\mu)$ and $H$ a Hilbert space. Then the set $\set{\varphi(\cdot)x}{x\in H,\,\varphi\in\mathcal{D}}$
is total in $\L(\mu;H)$.
\end{lem}

\begin{proof}By \prettyref{lem:simple_fcts_dense}, we know that $S(\mu;X)$ is dense in $\L(\mu;H)$. Thus, it suffices to approximate $\1_A x$ for some $A\in \Sigma$ and $x\in H$. For this, however, take a sequence $(\phi_n)_n$ in the linear hull of $\mathcal{D}$ with $\phi_n\to \1_A$ in $\L(\mu)$ as $n\to \infty$. Then
\[
    \norm{\1_Ax-\phi_n x}_{\L(\mu;H)} = \norm{x}_H\norm{\1_A - \phi_n}_{\L(\mu)}\to 0 \quad(n\to\infty).
\]Thus, the claim follows.
\end{proof}
\begin{lem}
\label{lem:dense sets 2} Let $\mathcal{D}\subseteq \L(\mu)$ be total
in $\L(\mu)$, $H$ a Hilbert space, $D_{0}\subseteq H$ total in
$H$. Then $\set{\varphi(\cdot)x}{x\in D_{0},\varphi\in\mathcal{D}}$
is total in $\L(\mu;H)$.
\end{lem}

\begin{proof}
The proof follows upon realising that the set $\set{\varphi(\cdot)x}{x\in D_{0},\,\varphi\in\mathcal{D}}$
is total in the set $\set{\varphi(\cdot)x}{x\in H,\,\varphi\in\mathcal{D}}$. From here we just apply \prettyref{lem:dense sets}.
\end{proof}

\section{The Time Derivative as a Normal Operator}
\label{sec:time_derivative}

Now let $H$ be a Hilbert space over $\K\in\{\R,\C\}$. For $\nu\in\R$
and $p\in\roi{1}{\infty}$ we define the measure 
\[
\mu_{p,\nu}(A)\coloneqq\int_{A}\e^{-p\nu t}\d\lambda(t)
\]
for $A$ in the Borel-$\sigma$-algebra, $\mathcal{B}(\R)$, of $\R$.
As our underlying Hilbert space for the time derivative we set 
\[
\Lnu(\R;H)\coloneqq\L(\mu_{2,\nu};H).
\]
In the same way we define 
\[
L_{p,\nu}(\R;H)\coloneqq L_{p}(\mu_{p,\nu};H)
\]
for $p\in\roi{1}{\infty}$. If $H=\K$ we abbreviate $L_{p,\nu}(\R):=L_{p,\nu}(\R;\K)$.

Our aim is to define the time derivative on $\Lnu(\R;H)$.
For this, we define the integral as an operator, which for $\nu\neq0$ turns out to be one-to-one and bounded. Then we introduce the time derivative as the inverse of this integral. The reason for doing it that way is to easily get a formula for the adjoint for the time derivative using the boundedness of the integral.

We start our considerations with the definition
of convolution operators in $\Lnu(\R;H)$. 
\begin{lem}
\label{lem:convolution}Let $k\in L_{1,\nu}(\R)$. We define the convolution
operator
\[
k\ast\colon\Lnu(\R;H)\to\Lnu(\R;H)
\]
by 
\[
\left(k\ast f\right)(t)\coloneqq\int_{\R}k(s)f(t-s)\d s,
\]
which exists for a.e.~$t\in\R$. Then, $k\ast$ is linear and
bounded with $\norm{k\ast}\leq\norm{k}_{L_{1,\nu}(\R)}$.
\end{lem}

\begin{proof}
We first prove that $s\mapsto k(s)f(t-s)\in L_{1}(\R;H)$ for
a.e.~$t\in\R$. The Bochner-measurability is clear since $k$ and
$f$ are both Bochner-measurable. Moreover, 
\begin{align*}
 & \int_{\R}\left(\int_{\R}\norm{k(s)f(t-s)}_{H}\d s\right)^{2}\e^{-2\nu t}\d t\\
 & =\int_{\R}\left(\int_{\R}\abs{k(s)}^{\frac{1}{2}}\e^{-\frac{\nu}{2}s}\abs{k(s)}^{\frac{1}{2}}\e^{-\frac{\nu}{2}s}\norm{f(t-s)}_{H}\e^{-\nu(t-s)}\d s\right)^{2}\d t\\
 & \leq\int_{\R}\left(\int_{\R}\abs{k(s)}\e^{-\nu s}\d s\right)\left(\int_{\R}\abs{k(s)}\e^{-\nu s}\norm{f(t-s)}_{H}^{2}\e^{-2\nu(t-s)}\d s\right)\d t\\
 & =\norm{k}_{L_{1,\nu}(\R)}\int_{\R}\abs{k(s)}\int_{\R}\norm{f(t-s)}^2\e^{-2\nu(t-s)}\d t\,\e^{-\nu s}\d s\\
 & =\norm{k}_{L_{1,\nu}(\R)}^{2}\norm{f}_{\Lnu(\R;H)}^{2},
\end{align*}
which on the one hand proves that 
\[
\int_{\R}\norm{k(s)f(t-s)}_{H}\d s<\infty
\]
for a.e.~$t\in\R$ and on the other hand shows the norm estimate, once we have shown the Bochner-measurability of $k\ast f$. For proving the latter, we apply \prettyref{thm:Pettis}. Since $f$ is Bochner-measurable, we find a nullset $N$ such that $H_0 \coloneqq \overline{\lin f[\R\setminus N]}$ is separable. Hence, for almost every $t\in \R$ we have 
\[
 (k\ast f)(t)=\int_{\R} k(s)f(t-s)\d s=\int_{\R\setminus N} k(t-s) f(s)\d s\in H_0
\]
by \prettyref{prop:interchange_integral} (b). Thus, $k \ast f$ is almost separably-valued. Moreover, for $x' \in H'$ we have by \prettyref{prop:interchange_integral} (a) 
\[
 x'\circ (k\ast f)= k\ast (x'\circ f)
\]
almost everywhere and thus, the weak Bochner-measurability follows from the fact that the convolution of two measurable scalar-valued functions is measurable. Since the linearity of $k\ast$ is clear the proof is done.
\end{proof}
\begin{defn*}
For $\nu\ne0$ we define the operator 
\[
I_{\nu}\colon\Lnu(\R;H)\to\Lnu(\R;H)
\]
by 
\[
I_{\nu}\coloneqq\begin{cases}
\1_{\roi{0}{\infty}}\ast, & \text{ if }\nu>0,\\
-\1_{\loi{-\infty}{0}}\ast, & \text{ if }\nu<0.
\end{cases}
\]
\end{defn*}
Note that, by \prettyref{lem:convolution}, $I_{\nu}$ is bounded with $\norm{I_{\nu}}\leq\frac{1}{\abs{\nu}}$.

\begin{rem}
\label{rem:action_Inu}
  For $\nu>0$, $f\in \Lnu(\R,H)$ we have
  \[I_\nu f(t) = \1_{\roi{0}{\infty}}\ast f (t) = \int_0^\infty f(t-s)\d s = \int_{-\infty}^t f(s)\d s\quad (t\in \R \mbox{ a.e.}).\]
  Analogously, for $\nu<0$, $f\in \Lnu(\R,H)$ we have
  \[I_\nu f(t) = -\int_t^\infty f(s)\d s\quad (t\in \R \mbox{ a.e.}).\]
\end{rem}

\begin{prop}
\label{prop:int_nice}Let $\nu\ne0.$ Then $I_{\nu}$ is one-to-one
and $\cco(\R;H)$\index{C_{c}^{1}(R;H)@$\cco(\R;H)$},
the space of continuously differentiable, compactly supported functions
on $\R$ with values in $H$, is in the range of $I_{\nu}$.
\end{prop}

\begin{proof}
We just prove the assertion for the case when $\nu>0$. Let $f\in\Lnu(\R;H)$
satisfy $I_{\nu}f=0$. In particular, we obtain for all $t\in\R\setminus N$
that $0=I_{\nu}f(t)=\int_{-\infty}^{t}f(s)\d s$
for some Lebesgue null set, $N\subseteq\R$. Then for $a,b\in\R\setminus N$
with $a<b$ and $x\in H$ we have that 
\begin{align*}
\scp{f}{\e^{2\nu(\cdot)}\1_{\ci{a}{b}}\cdot x}_{\Lnu(\R;H)} & =\int_{\R}\scp{f(t)}{\e^{2\nu t}\1_{\ci{a}{b}}(t)\cdot x}_{H}\e^{-2\nu t}\d t\\
 & =\scp{\int_{a}^{b}f(t)\d t}{x}_{H}\\
 & =\scp{\left(I_{\nu}f\right)(b)-\left(I_{\nu}f\right)(a)}{x}_{H}=0.
\end{align*}
Thus $f=0$. Indeed,  since $\R\setminus N$ is dense in $\R$, $\set{\e^{2\nu(\cdot)}\1_{\ci{a}{b}}}{a,b\in\R\setminus N}$
is total in $\Lnu(\R)$. Hence, $\set{\e^{2\nu(\cdot)}\1_{\ci{a}{b}}\cdot x}{a,b\in\R\setminus N,\,x\in H}$
is total in $\Lnu(\R;H)$ by \prettyref{lem:dense sets}. This proves
the injectivity of $I_{\nu}$. Moreover, if $\varphi\in \cco(\R;H)$
then by \prettyref{cor:Hauptsatz} we have 
\[
\varphi(t)=\int_{-\infty}^{t}\varphi'(s)\d s=\left(I_{\nu}\varphi'\right)(t)\quad(t\in\R\mbox{ a.e.}).\tag*{\qedhere}
\]
\end{proof}
\begin{defn*}
For $\nu\ne0$ we define the\index{time derivative}\emph{ time derivative,
}$\td{\nu}$, on $\Lnu(\R;H)$ by 
\[
\td{\nu}\coloneqq I_{\nu}^{-1}.
\]
\end{defn*}
Note that by \prettyref{lem:convolution} and \prettyref{prop:int_nice},
$\td{\nu}$ is a closed linear operator for which $\cco(\R;H)\subseteq\dom(\td{\nu})$.
Since 
\[
\cco(\R;H)\supseteq\lin\set{\varphi\cdot x}{\varphi\in \cco(\R),\,x\in H}
\]
we infer that $\td{\nu}$ is densely defined by \prettyref{lem:dense sets}
and \prettyref{exer:C_cinfty dense}. Moreover, since $I_{\nu}\varphi'=\varphi$
for $\varphi\in\cco(\R;H)$ we get that 
\[
\td{\nu}\varphi=\varphi';
\]
that is, $\td{\nu}$ extends the classical derivative of continuously
differentiable functions. We shall discuss the actual domain of $\td{\nu}$
in the next chapter. 
\begin{prop}
\label{prop:C_cinfty_core}Let $\nu\ne0$. Then \index{mathcal{D}_H@$\calDH$}$\calDH\coloneqq\lin\set{\varphi\cdot x}{\varphi\in \cci(\R),\,x\in H}$
is a core for $\td{\nu}.$ Here, $\cci(\R)$ denotes the
space of smooth functions on $\R$ with compact
support.
\end{prop}

\begin{proof}
We first prove that 
\begin{equation}
\set{\varphi'}{\varphi\in \cci(\R)}\label{eq:der_test_fct}
\end{equation}
is dense in $\Lnu(\R)$. As $\cci(\R)$ is dense in $\Lnu(\R)$
(see \prettyref{exer:C_cinfty dense}), it suffices to approximate
functions in $\cci(\R)$. For this, let $f\in \cci(\R)$.
We now define 
\[
\varphi_{n}(t)\coloneqq\begin{cases}
                        \int_{-\infty}^{t}f(s) - f(s-n) \d s &\mbox{if } \nu>0,\\
                        \int_{-\infty}^{t}f(s) - f(s+n) \d s &\mbox{if } \nu<0
                       \end{cases}
\quad(t\in\R,n\in \N).
\]
Then $\varphi_{n}\in \cci(\R)$ for each $n\in\N$ and 
\[
\varphi_{n}'(t)=\begin{cases}
                 f(t)-f(t-n) &\mbox{if } \nu>0,\\
                 f(t)-f(t+n) &\mbox{if } \nu<0
                \end{cases}\quad (t\in \R,n\in \N).
\]
Consequently, 
\begin{align*}
\norm{\varphi_{n}'-f}_{\Lnu(\R)}^2&= \begin{cases}
                                      \int_\R |f(t-n)|^2  \e^{-2\nu t} \d t & \mbox{if } \nu>0,\\
                                      \int_\R |f(t+n)|^2 \e^{-2\nu t} \d t  &\mbox{if } \nu<0
                                     \end{cases}\\
                                   &= \norm{f}_{\Lnu(\R)}^2 \e^{-2|\nu|n} \to 0\quad (n\to \infty),
\end{align*}
which shows the density of \prettyref{eq:der_test_fct} in $\Lnu(\R)$. By \prettyref{lem:dense sets} we have that 
\[
\set{\varphi'\cdot x}{\varphi\in \cci(\R),\,x\in H}
\]
is total in $\Lnu(\R;H)$ and so $\td{\nu}[\calDH]$ is dense
in $\Lnu(\R;H)$. Now let $f\in\dom(\td{\nu})$ and $\varepsilon>0$.
By what we have shown above there exists some $\varphi\in\calDH$
such that 
\[
\norm{\td{\nu}\varphi-\td{\nu}f}_{\Lnu(\R;H)}\leq\varepsilon.
\]
Since $\partial_{t,\nu}^{-1}=I_{\nu}$ is bounded with $\norm{\partial_{t,\nu}^{-1}}\leq\frac{1}{\abs{\nu}}$, the latter implies
that
\[
\norm{\varphi-f}_{\Lnu(\R;H)}\leq\frac{\varepsilon}{\abs{\nu}},
\]
and hence, $\calDH$ is indeed a core for $\td{\nu}$. 
\end{proof}

\begin{cor}
\label{cor:nu=00003D0} For $\nu\in\R$ the mapping
\begin{align*}
\exp(-\nu\m):\Lnu(\R;H) & \to\L(\R;H)\\
f & \mapsto(t\mapsto\e^{-\nu t}f(t))
\end{align*}
is unitary, and for $\nu,\mu\ne0$ one has 
\[
\exp(-\nu\m)(\td{\nu}-\nu)\exp(-\nu\m)^{-1}=\exp(-\mu\m)(\td{\mu}-\mu)\exp(-\mu\m)^{-1}.
\]
\end{cor}

\begin{proof}
The proof is left as \prettyref{exer:derivative nu=00003D0}.
\end{proof}

By \prettyref{cor:nu=00003D0} we can now define $\td{0}$. Let $\nu\neq0$. Then
\[
 \td{0}\coloneqq \exp(-\nu\m)(\td{\nu}-\nu)\exp(-\nu\m)^{-1}.
\]
Note that in view of \prettyref{cor:nu=00003D0}, the assertion of \prettyref{prop:C_cinfty_core} now also holds for $\nu=0$.

Finally, we want to compute the adjoint of $\td{\nu}$. 

\begin{cor}
\label{cor:adjoint_time_derivative}Let $\nu\in\R$. The adjoint of
$\td{\nu}$ is given by 
\[
\td{\nu}^{\ast}=-\td{\nu}+2\nu.
\]
In particular, $\td{\nu}$ is a normal operator with $\Re\td{\nu}\coloneqq\frac{1}{2}\left(\overline{\td{\nu}+\td{\nu}^{\ast}}\right)=\nu$, and $\td{0}$ is skew-selfadjoint.
\end{cor}

\begin{proof}
Let $\nu\neq 0$ first.
Integrating by parts, one obtains 
\begin{align*}
\int_{\R}\scp{\td{\nu}\varphi(t)}{\psi(t)}\e^{-2\nu t}\d t & =\int_{\R}\scp{\varphi'(t)}{\psi(t)}\e^{-2\nu t}\d t\\
& =\int_{\R}\scp{\varphi(t)}{-\psi'(t)+2\nu\psi(t)}\e^{-2\nu t}\d t
\end{align*}
for $\varphi,\psi\in \cci(\R;H)$. Since $\cci(\R;H)$
is a core for $\td{\nu}$ by \prettyref{prop:C_cinfty_core}, the
latter shows 
\[
\td{\nu}\subseteq-\td{\nu}^{\ast}+2\nu.
\]
Since we know that $\td{\nu}$ is onto, it suffices to prove that $-\td{\nu}^\ast+2\nu$ is one-to-one, since this would imply equality in the latter operator inclusion. For doing so, we apply \prettyref{thm:ran-kernerl} to compute
\[
 \ker(-\td{\nu}^\ast +2\nu)=\ran(-\td{\nu}+2\nu)^\bot. 
\]
Moreover, we have that $-\td{\nu}+2\nu$ is unitarily equivalent to $-\td{-\nu}$ by \prettyref{cor:nu=00003D0} and since $\td{-\nu}$ is onto, so is $-\td{\nu}+2\nu$ and thus $\ker(-\td{\nu}^\ast +2\nu)=\Lnu(\R;H)^\bot=\{0\}$, which yields the assertion.

The case $\nu=0$ follows directly from the definition of $\td{0}$.
\end{proof}

\section{Comments}

Standard references for Bochner integration and related results are \cite{ABHN_2011,Diestel1977}.

Considering the derivative operator in an exponentially weighted space
goes back (at least) to Morgenstern \cite{Morgenstern1952}, where
ordinary differential equations were considered in a classical
setting. In fact, we shall return to this observation in the next
chapter when we devote our study to some implications of the already developed
concepts on ordinary and delay differential equations. 

A first occurence of the derivative operator in exponentially weighted
$L^{2}$-spaces can be found in \cite{Picard1989}, where a corresponding
spectral theorem has been focussed on. We will prove in a later chapter
that the spectral representation of the time-derivative as a multiplication operator can
be realised by a shifted variant of the Fourier transformation --
the so-called Fourier--Laplace transformation. 

In an applied context, the time derivative operator discussed here
has been introduced in \cite{PicPhy}. 

\section*{Exercises}
\addcontentsline{toc}{section}{Exercises}

\begin{xca}
\label{exer:delta-seq}A sequence $(\varphi_{n})_{n}$ in $\cci(\R^{d})$
is called a \emph{$\delta$-sequence,\index{delta-sequence,@$\delta$-sequence}}
if 

\begin{enumerate}

\item $\varphi_{n}\geq0$ for $n\in\N$,

\item $\spt\varphi_{n}\subseteq\ci{-\frac{1}{n}}{\frac{1}{n}}^{d}$
for $n\in\N$,

\item $\int_{\R^{d}}\varphi_{n}=1$ for $n\in\N$.

\end{enumerate}

Let $\varphi\in \cci(\R^{d})$ with $\spt\varphi\subseteq\ci{-1}{1}^{d}$,
$\varphi\geq0$ and $\int_{\R^{d}}\varphi=1$. Prove that $(\varphi_{n})_{n}$
given by $\varphi_{n}(x)\coloneqq n^d\varphi(nx)$ for $x\in \R^d$, $n\in\N$
defines a $\delta$-sequence. Moreover, give an example for such a
function $\varphi$. 
\end{xca}

\begin{xca}
\label{exer:C_cinfty dense} It is well-known
that $\set{\1_{I}}{I\text{ \ensuremath{d}-dimensional bounded interval}}$
is total in $\L(\R^{d})$.
\begin{enumerate}

\item Let $\varphi\in \cci(\R^{d})$, $f\in \L(\R^{d})$. Define as usual $f\ast \varphi\coloneqq \Bigl(x\mapsto \int_{\R^d} f(x-y) \varphi(y) \d y\Bigr)$.
Prove that $f\ast\varphi\in C^\infty(\R^d)$ with $\partial^{\alpha}\left(f\ast\varphi\right)=f\ast\partial^{\alpha}\varphi$
for all $\alpha\in\N_0^{d}$, where $\partial^{\alpha}\varphi=\partial_{1}^{\alpha_{1}}\cdots\partial_{d}^{\alpha_{d}}\varphi$.
Moreover, prove that $\spt f\ast\varphi\subseteq\spt f+\spt\varphi$.

\item Let $(\varphi_{n})_{n}$ be a $\delta$-sequence and $f\in\L(\R^{d})$.
Show that $f\ast\varphi_{n}\to f$ in $\L(\R^{d})$ as $n\to\infty$.\\
Hint: Prove that $\1_{I}\ast\varphi_{n}\to\1_{I}$ in $\L(\R^{d})$
for all $d$-dimensional intervals and use that $\norm{f\ast\varphi_{n}}_{2}\leq\norm{f}_{2}$
(see also \prettyref{lem:convolution}).

\item Prove that $\cci(\R^{d})$ is dense in $\L(\R^{d})$.

\end{enumerate}
\end{xca}

\begin{xca}
\label{exer:int-by-parts}Let $a<b$, $X_{0},X_{1},X_{2}$ be Banach
spaces, $f\from \oi{a}{b} \to X_{0}$ and $g\from \oi{a}{b} \to X_{1}$ both continuously differentiable,
$\ell\colon X_{0}\times X_{1}\to X_{2}$ bilinear and continuous. Prove
that $h\from \oi{a}{b}\to X_2$ given by
\[
h(t)\coloneqq\ell(f(t),g(t))\quad(t\in\oi{a}{b})
\]
is continuously differentiable with 
\[
h'(t)=\ell(f'(t),g(t))+\ell(f(t),g'(t))\quad(t\in\oi{a}{b}).
\]
If $f,f',g,g'$ have continuous extensions to $\ci{a}{b},$ prove
the integration by parts formula: 
\[
\int_{a}^{b}\ell(f'(t),g(t))\d t=\ell(f(b),g(b))-\ell(f(a),g(a))-\int_{a}^{b}\ell(f(t),g'(t))\d t.
\]
\end{xca}

\begin{xca}
\label{exer:norm_Inu}For $\nu\ne0$, show that $\norm{I_{\nu}}=\frac{1}{\abs{\nu}}$.
\end{xca}

\begin{xca}
\label{exer:derivative nu=00003D0} Prove \prettyref{cor:nu=00003D0}.
\end{xca}

\begin{xca}
\label{exer:spectrum_time_Der} Let $\nu\in\R$ and $H$ be a complex Hilbert
space. Prove that $\sigma(\td{\nu})\subseteq\set{\i t+\nu}{t\in\R}$,
where $\td{0}$ is defined in \prettyref{cor:adjoint_time_derivative}.\\
Hint: For $f\in\dom(\td{\nu}), z\in\C$ compute $\Re\scp{(z-\td{\nu})f}{f}_{\Lnu(\R;H)}$
by using \prettyref{cor:adjoint_time_derivative}. For proving the
surjectivity of $z-\td{\nu}$ for a suitable $z$, use the formula 
\[
\cran(z-\td{\nu})=\ker(z^{\ast}-\td{\nu}^{\ast})^{\bot}.
\]

Remark: Later we will see that, actually, $\sigma(\td{\nu})=\set{\i t+\nu}{t\in\R}$.
\end{xca}

\begin{xca}
\label{exer:inverse-depends-on-nu}Consider the differential equation
\[
\left(\td{\nu}^{2}-1\right)u=\1_{[-1,1]}.
\]
Since $\td{\nu}^{2}-1=\left(\td{\nu}-1\right)\left(\td{\nu}+1\right)$,
it follows by \prettyref{exer:spectrum_time_Der} that there is a
unique $u\in\Lnu(\R)$ solving this equation if $\nu\notin\{-1,1\}$.
Compute these solutions.

Hint: For $u\in\dom(\td{\nu})$ use the fact that $u$ is necessarily continuous (which we shall establish in the next lecture).
\end{xca}

\printbibliography[heading=subbibliography]

\chapter{Ordinary Differential Equations}

In this lecture, we discuss a first application of the time derivative
operator constructed in the previous lecture. More precisely, we analyse
well-posedness of ordinary differential equations and will at the
same time provide a Hilbert space proof of the classical Picard--Lindelöf
theorem. We shall furthermore see that the abstract theory developed
here also allows for more general differential equations to be considered.
In particular, we will have a look at so-called delay differential equations
with finite or infinite delay; neutral differential equations are
considered in the exercises section.

We start with some information on the time derivative and its
domain.

\section{The Domain of the time derivative and the Sobolev Embedding Theorem}

Let $H$ be a Hilbert space.
Readers familiar with the notion of Sobolev spaces might
have already realised that the domain of $\td{\nu}$ can be described as $\Lnu(\R;H)$-functions
with distributional derivative lying in $\Lnu(\R;H)$. In order to
stress this, we include the following result. Later on, we have the
opportunity to have a more detailed look at Sobolev spaces in more general
contexts.
\begin{prop}
\label{prop:char_dom(tdnu)}Let $\nu\in\R$ and $f,g\in\Lnu(\R;H)$.
Then the following conditions are equivalent: 
\begin{enumerate}
\renewcommand{\labelenumi}{{\upshape (\roman{enumi})}}
\item $f\in\dom(\td{\nu})$ and $\td{\nu}f = g$.
\item For all $\phi\in\cci(\R)$
we have 
\[
-\int_{\R}\phi'f=\int_{\R}\phi g.
\]
\end{enumerate}
\end{prop}

\begin{proof}
Assume that $f\in\dom(\td{\nu})$. By \prettyref{prop:C_cinfty_core}
and \prettyref{cor:adjoint_time_derivative}, we have that $\calDH=\lin\set{\varphi\cdot x}{\varphi\in \cci(\R),\,x\in H}\subseteq\dom(\td{\nu}^{*})$ (which also holds for $\nu=0$)
and 
\[
\scp{\td{\nu}f}{\psi\cdot x}_{\Lnu}=\scp{f}{\left(-\psi'+2\nu\psi\right)\cdot x}_{\Lnu}
\]
for all $x\in H$ and $\psi\in\cci(\R)$. Hence, we obtain for all
$\psi\in\cci(\R)$ 
\[
\int_{\R}\left(-\psi'+2\nu\psi\right)f\e^{-2\nu\cdot}=\int_{\R}\psi \td{\nu}f\e^{-2\nu\cdot};
\]
putting $\phi\coloneqq\e^{-2\nu\cdot}\psi$ and using that multiplication
by $\e^{-2\nu\cdot}$ is a bijection on $\cci(\R)$, we deduce the
claimed formula with $g=\td{\nu}f$.

On the other hand, the equation involving $g$ applied to $\phi=\e^{-2\nu\cdot}\psi$
for $\psi\in\cci(\R)$ implies that
\[
\int_{\R}\left(-\psi'+2\nu\psi\right)f\e^{-2\nu\cdot}=\int_{\R}\psi g\e^{-2\nu\cdot}.
\]
Testing this equation with $x\in H$ 
yields
\[
\scp{g}{\psi\cdot x}_{\Lnu}=\scp{f}{\left(-\psi'+2\nu\psi\right)\cdot x}_{\Lnu}=\scp{f}{\left(-\td{\nu}\psi\cdot x +2\nu\psi\cdot x\right)}_{\Lnu}.
\]
Since $\calDH$ is dense in $\dom(\td{\nu})$ by \prettyref{prop:C_cinfty_core}, we infer that
\[
\scp{g}{h}_{\Lnu}=\scp{f}{\left(-\td{\nu}h+2\nu h\right)}_{\Lnu}
\]
for all $h\in\dom(\td{\nu})$. Now, \prettyref{cor:adjoint_time_derivative},
yields
\[
\scp{g}{h}_{\Lnu}=\scp{f}{\td{\nu}^{*}h}_{\Lnu}\quad(h\in\dom(\td{\nu}^{*})).
\]
Thus, $f\in\dom(\td{\nu}^{**})=\dom(\td{\nu})$ and $\td{\nu}f=g$. 
\end{proof}

The next result confirms that functions in the domain of $\td{\nu}$
are continuous. This result was announced in \prettyref{exer:inverse-depends-on-nu} and  is known as the Sobolev embedding theorem. Here,
we make use of the explicit form of the domain of $\td{\nu}$ as being
the range space of the integral operator $I_{\nu}$. We define\index{$C_{nu}(R;H)$@$C_{\nu}(\R;H)$}
\[
C_{\nu}(\R;H)\coloneqq\set{f\colon\R\to H}{f\text{ continuous, }\norm{f}_{\nu,\infty}\coloneqq\sup_{t\in\R}\norm{\e^{-\nu t}f(t)}_H<\infty}
\]and regard it as being
endowed with the obvious norm. 
\begin{thm}[\index{Sobolev embedding theorem}Sobolev embedding theorem]
\label{thm:Sobolev_emb} Let $\nu\in\R$. Then every $f\in\dom(\td{\nu})$ has a continuous
representative, and the mapping 
\[
\dom(\td{\nu})\ni f\mapsto f\in C_{\nu}(\R;H)
\]
is continuous. 
\end{thm}

\begin{proof}
We restrict ourselves to the case when $\nu>0$; the remaining cases can be proved by invoking \prettyref{cor:nu=00003D0}.
Let $f\in\dom(\td{\nu})$. By definition, we find $g\in\Lnu(\R;H)$
such that $f=\td{\nu}^{-1}g=I_{\nu}g.$ Then for all $t\in\R$ we compute
\begin{align*}
 \int_{-\infty}^{t}\norm{g(\tau)}\d\tau &
 =\int_{-\infty}^{t}\norm{g(\tau)}\e^{-\nu\tau}\e^{\nu\tau}\d\tau
 \leq\sqrt{\int_{-\infty}^{t}\norm{g(\tau)}^{2}\e^{-2\nu\tau}\d\tau}\sqrt{\int_{-\infty}^{t}\e^{2\nu\tau}\d\tau}\\
 &\leq\norm{\td{\nu}f}_{\Lnu}\sqrt{\frac{1}{2\nu}}\e^{\nu t}.
\end{align*}
Thus, $g$ is integrable on $\loi{-\infty}{t}$ for all $t\in \R$ and dominated convergence implies that
\[f = \Bigl( t\mapsto \int_{-\infty}^t g(s)\d s\Bigr)\]
is continuous.
Moreover, for $t\in\R$ we obtain
\[\norm{f(t)} \leq \int_{-\infty}^{t}\norm{g(\tau)}\d\tau \leq \norm{\td{\nu}f}_{\Lnu}\sqrt{\frac{1}{2\nu}}\e^{\nu t}\]
which yields the claimed continuity.
\end{proof}

\begin{cor}
\label{cor:vanish_at_inf}For all $f\in\dom(\td{\nu})$, we have that $\norm{\e^{-\nu t}f(t)}_{H}\to0$
as $t\to\pm\infty$. 
\end{cor}

The proof is left as \prettyref{exer:vanish_at_inf}.

\section{The Picard--Lindelöf Theorem}

The prototype of the Picard--Lindelöf theorem will be formulated
for so-called uniformly Lipschitz continuous functions.

We first need a preparation.
\begin{defn*}
  Let $X$ be a Banach space. Then we define 
  \[S_{\rmc}(\R;X) \coloneqq\set{f\from \R\to X}{f\,\text{simple},\, \spt f\,\text{compact}}\]
  to be the set of \emph{simple functions from $\R$ to $X$ with compact support}.
\end{defn*}

\begin{lem}
\label{lem:simple_fcts_compact_spt}
  Let $X$ be a Banach space and $\nu,\eta\in\R$. Then $S_{\rmc}(\R;X)$ is dense in $\Lnu(\R;X)\cap \Lm{\eta}(\R,X)$; that is, for all $f\in \Lnu(\R;X)\cap\Lm{\eta}(\R;X)$ there exists $(f_n)_n$ in $S_{\rmc}(\R;X)$ such that $f_n\to f$ in both $\Lnu(\R;X)$ and $\Lm{\eta}(\R;X)$. 
	 In particular, $S_{\rmc}(\R,X)$ is dense in $\Lnu(\R;X)$.
\end{lem}

\begin{proof}
  Let $f\in\Lnu(\R;X)\cap\Lm{\eta}(\R;X)$.
  Then for all $n\in\N$ we have that $\1_{\ci{-n}{n}}f\in\Lnu(\R;X)\cap\Lm{\eta}(\R;X)$
  and $\1_{\ci{-n}{n}}f\to f$ in $\Lnu(\R;X)$ and in $\Lm{\eta}(\R;X)$
  as $n\to\infty$. 
  For $n\in\N$ let $(\tilde{f}_{n,k})_{k}$ be in
  $S(\mu_{2,\nu};X)$ such that $\tilde{f}_{n,k}\to\1_{\ci{-n}{n}}f$
  in $\Lnu(\R;X)$ as $k\to\infty$. We
  put $f_{n,k}\coloneqq\1_{\ci{-n}{n}}\tilde{f}_{n,k} \in S_{\rmc}(\R;X)$. Then
  $f_{n,k}\to\1_{\ci{-n}{n}}f$ in $\Lnu(\R;X)$ and in $\Lm{\eta}(\R;X)$
  as $k\to\infty$.
\end{proof}

In order to define the notion of uniformly Lipschitz continuous functions, we first need the Lipschitz semi-norm.
\begin{defn*}
  Let $X_0, X_1$ be normed spaces, and $F\from X_0\to X_1$ Lipschitz continuous. Then 
  \[
  \norm{F}_{\Lip}\coloneqq\sup_{\substack{x,y\in X_0 \\ x\neq y}}{\frac{\norm{F(x)-F(y)}}{\norm{x-y}}}
  \]
is the \emph{Lipschitz semi-norm} of $F$. 
\end{defn*}

\begin{defn*}
Let $H_{0},H_{1}$ be Hilbert spaces, $\mu\in\R$. Then a function $F\colon S_{\rmc}(\R;H_{0})\to\bigcap_{\nu\geq\mu}\Lnu(\R;H_{1})$
is called \emph{uniformly Lipschitz continuous}, if for all $\nu\geq\mu$
we have that $F$ considered in $\Lnu(\R;H_{0})\times\Lnu(\R;H_{1})$
is Lipschitz continuous, and for the unique Lipschitz continuous extensions
$F^{\nu}$, $\nu\geq\mu$, we have that
\[
\sup_{\nu\geq\mu}\norm{F^{\nu}}_{\Lip}<\infty.
\] 
\end{defn*}

\begin{rem}
  Another way to introduce uniformly Lipschitz continuous mappings is the following.
  Let $H_{0},H_{1}$ be Hilbert spaces, $\mu\in\R$. Let $(F^\nu)_{\nu\geq\mu}$ be a family of Lipschitz continuous mappings $F^\nu\from \Lnu(\R;H_0)\to \Lnu(\R;H_1)$ such that
  \[\sup_{\nu\geq\mu}\norm{F^{\nu}}_{\Lip}<\infty\]
  and the mappings are consistent in the sense that for all $\nu,\eta\geq\mu$ and $f\in \Lnu(\R;H_0)\cap \Lm{\eta}(\R;H_0)$ we have
  \[F^\nu(f) = F^\eta(f).\]
  Then, for $\nu\geq\mu$ and $f\in S_{\rmc}(\R;H_0)$ we have $F^\nu(f)\in \bigcap_{\eta\geq\mu}\Lm{\eta}(\R;H_{1})$ and $F^\nu|_{S_{\rmc}(\R;H_0)}$ is uniformly Lipschitz continuous.
\end{rem}

\begin{thm}[\index{Theorem of Picard--Lindelöf}Picard--Lindelöf -- Hilbert space version]
\label{thm:PicLind1} Let $H$ be a Hilbert space, $\mu\in\R$ and $F\colon S_{\rmc}(\R;H)\to\bigcap_{\nu\geq\mu}\Lnu(\R;H)$
uniformly Lipschitz continuous, with $L\coloneqq\sup_{\nu\geq\mu}\norm{F^{\nu}}_{\Lip}$.
Then for all $\nu>\max\{L,\mu\}$ the equation 
\[
\td{\nu}u_{\nu}=F^{\nu}(u_{\nu})
\]
admits a unique solution $u_{\nu}\in\dom(\td{\nu})$. Furthermore,
for all $\nu>\max\{L,\mu\}$ the following properties hold: 
\begin{enumerate}
\item
\label{thm:PicLind1:item:1}
If $F^{\nu}(u_{\nu})$ is continuous in a neighborhood of $a\in\R$,
then $u_{\nu}$ is continuously differentiable in a neighborhood of $a$. 
\item
\label{thm:PicLind1:item:2}
For all $a\in\R$, $\1_{\loi{-\infty}{a}}u_{\nu}$ is the unique fixed
point $v\in\Lnu(\R;H)$ of $\1_{\loi{-\infty}{a}}\td{\nu}^{-1}F^{\nu}$, that is, $v$ uniquely solves
\[
v=\1_{\loi{-\infty}{a}}\td{\nu}^{-1}F^{\nu}(v).
\]
\item 
\label{thm:PicLind1:item:4}
For all $\eta\geq\nu$ we have that $u_{\nu}=u_{\eta}$. 
\item
\label{thm:PicLind1:item:3}
For all $f\in\Lnu(\R;H)$ the equation 
\[
\td{\nu}v=F^{\nu}(v)+f
\]
admits a unique solution $v_{\nu,f}\in\dom(\td{\nu})$, and if $f,g\in\Lnu(\R;H)$
satisfy $f=g$ on $\loi{-\infty}{a}$ for some $a\in\mathbb{R},$
then $v_{\nu,f}=v_{\nu,g}$ on $\loi{-\infty}{a}$. 
\end{enumerate}
\end{thm}

\begin{proof}[Proof of \prettyref{thm:PicLind1} -- first part]
Define $\Phi\colon\Lnu(\R;H)\to\Lnu(\R;H)$ by 
\[
\Phi(u)=\td{\nu}^{-1}F^{\nu}(u).
\]
Since $\norm{\td{\nu}^{-1}} \leq \frac{1}{\nu}$ and $\nu>L$
it follows that $\Phi$ is a
contraction and thus admits a unique fixed point, which by definition
solves the equation in question. Moreover, we have that $u_{\nu}=\Phi(u_{\nu})=\td{\nu}^{-1}F^{\nu}(u_{\nu})\in\dom(\td{\nu})$.

Differentiability of $u_\nu$ as in \ref{thm:PicLind1:item:1} follows from \prettyref{exer:weak_diff+cont} and the continuity of $F^\nu(u_\nu)$.

For the unique existence asserted in \ref{thm:PicLind1:item:3}, note that the unique existence of $v_{\nu,f}$ follows from the above considerations after
realising that $\Phi(v)\coloneqq\td{\nu}^{-1}F^{\nu}(v)+\td{\nu}^{-1}f$
defines a contraction in $\Lnu(\R;H)$. For the remaining statements in \ref{thm:PicLind1:item:3} and the statements in \ref{thm:PicLind1:item:2} and \ref{thm:PicLind1:item:4},
we need some prerequisites. 
\end{proof}

\begin{defn*}
  Let $H_0, H_1$ be Hilbert spaces, and $F\from \Lnu(\R;H_{0})\to \Lnu(\R;H_{1})$ Lipschitz continuous. Then, $F$ is called \emph{\index{causal}causal} if for all $a\in\R$ and all $f,g\in \Lnu(\R;H_0)$ with
  $f=g$ on $\loi{-\infty}{a}$, we have that $F(f)=F(g)$ on $\loi{-\infty}{a}$.
\end{defn*}

\begin{rem}
In the following, we will frequently  make use of the following easy estimates: 
Let $\nu\in\R$, $a\in \R$. If $f\in \Lnu(\R;H)$ with $\spt f\subseteq \loi{-\infty}{a}$ then $f\in \bigcap_{\eta \leq \nu} \Lm{\eta}(\R;H)$ and
\[
 \norm{f}_{\Lm{\eta}(\R;H)}\leq \e^{(\nu-\eta)a} \norm{f}_{\Lnu(\R;H)}\quad (\eta \leq \nu).
\]
Likewise, if $\spt f\subseteq \roi{a}{\infty}$, we get $f\in \bigcap_{\rho \geq \nu}\Lm{\rho}(\R;H)$ with
\[
 \norm{f}_{\Lm{\rho}(\R;H)}\leq \e^{(\nu-\rho)a} \norm{f}_{\Lnu(\R;H)}\quad (\rho \geq \nu).
\]
\end{rem}

\begin{lem}
\label{lem:indepCaus}Let $H_{0},H_{1}$ be Hilbert spaces, $\mu\in\R$, $F\colon S_{\rmc}(\R;H_{0})\to\bigcap_{\nu\geq\mu}\Lnu(\R;H_{1})$
uniformly Lipschitz continuous. Then the following statements hold: 
\begin{enumerate}
\item\label{lem:indepCaus:item:1} $F^{\nu}$ is causal for all $\nu\geq \mu$. 
\item\label{lem:indepCaus:item:2} The mapping $\td{\nu}^{-1}F^{\nu}$ is causal if $\nu\geq\max\{\mu,0\}$ and $\nu\neq 0$. 
\item\label{lem:indepCaus:item:3} For all $\nu\geq\eta\geq\mu$, we have that $F^{\nu}=F^{\eta}$ on $\Lnu(\R;H_0)\cap\Lm{\mu}(\R;H_0)$. 
\end{enumerate}
\end{lem}

\begin{proof} 
\ref{lem:indepCaus:item:1} 
We divide the proof into three steps.

(i) Let $\nu\geq\mu$. In order to show causality of $F^\nu$, we first note that it suffices to have $F^\nu(f) = F^\nu(g)$ on $\loi{-\infty}{a}$ for all $f,g\in S_{\rmc}(\R;H_0)$ with $f=g$ on $\loi{-\infty}{a}$. Indeed, let $f,g\in\Lnu(\R;H)$ with $f=g$ on $\loi{-\infty}{a}$ for some $a\in\R$. By \prettyref{lem:simple_fcts_compact_spt} we find
$(f_{n})_{n}$ and $(\tilde{g}_{n})_{n}$ in $S_{\rmc}(\R;H_{0})$ such that
$f_{n}\to f$ and $\tilde{g}_{n}\to g$ in $\Lnu(\R;H_{0})$. Next,
$\1_{\loi{-\infty}{a}}f_{n}\to\1_{\loi{-\infty}{a}}f=\1_{\loi{-\infty}{a}}g$
as $n\to\infty$ in $\Lnu(\R;H_{0})$. Thus, putting $g_{n}\coloneqq\1_{\loi{-\infty}{a}}f_{n}+\1_{\oi{a}{\infty}}\tilde{g}_{n}$
for all $n\in\N$ we obtain that $g_{n}\to g$ in $\Lnu(\R;H_{0})$.
Since $F^\nu(f_n) = F^\nu(g_n)$ on $\loi{-\infty}{a}$ for all $n\in\N$ and $F^\nu\from \Lnu(\R;H_0)\to \Lnu(\R;H_1)$ is continuous, taking the limit $n\to\infty$ yields $F^\nu(f) = F^\nu(g)$ on $\loi{-\infty}{a}$.

(ii) Let $a\in\R$, $c\geq0$ and $f\in S_{\rmc}(\R;H_0)$ such that $f=0$ on $\loi{-\infty}{a}$, $g\in \bigcap_{\nu\geq \mu}\Lnu(\R;H_1)$ such that
$\norm{g}_{\Lnu(\R;H_1)}\leq c\norm{f}_{\Lnu(\R;H_0)}$ for all $\nu\geq\mu$. Then
\[\int_{-\infty}^a \norm{g(t)}_{H_1}^2 \e^{2\nu(a-t)}\d t \leq \int_\R \norm{g(t)}_{H_1}^2 \e^{2\nu(a-t)}\d t \leq c^2 \int_a^\infty \norm{f(t)}_{H_0}^2 \e^{2\nu(a-t)}\d t \to 0\]
as $\nu\to\infty$. Since $\e^{2\nu(a-t)} \to \infty$ as $\nu\to\infty$ for all $t<a$, the monotone convergence theorem implies $g=0$ on $\loi{-\infty}{a}$.

(iii) Let $f,g\in S_{\rmc}(\R;H_0)$ such that $f=g$ on $\loi{-\infty}{a}$ for some $a\in\R$. Then $f-g=0$ on $\loi{-\infty}{a}$. Since $F$ is uniformly Lipschitz continuous, with $L\coloneqq \sup_{\nu\geq\mu} \norm{F^\nu}_{\Lip}$ we obtain $\norm{F^\nu(f)-F^\nu(g)}_{\Lnu(\R;H_1)} \leq L\norm{f-g}_{\Lnu(\R;H_0)}$ for all $\nu\geq\mu$. 
By (ii) we conclude $F^\nu(f)=F^\nu(g)$ on $\loi{-\infty}{a}$ for all $\nu\geq\mu$, which by (i) yields the assertion.

The statement in \ref{lem:indepCaus:item:2} directly follows from \ref{lem:indepCaus:item:1}. 
Note that $\td{\nu}^{-1}F^\nu$ is uniformly Lipschitz continuous only for $\nu>0$.

Let us prove \ref{lem:indepCaus:item:3}.
Since $F^\nu(f)=F(f)=F^\eta(f)$ for $f\in S_{\rmc}(\R;H_0)$, the set $S_{\rmc}(\R;H_0)$ is dense in $\Lnu(\R;H_0)\cap \Lm{\mu}(\R;H_0)$ by \prettyref{lem:simple_fcts_compact_spt}, and $F^\nu$ and $F^\eta$ are Lipschitz-continuous, we obtain the assertion.
\end{proof}
\begin{proof}[Proof of \prettyref{thm:PicLind1} -- second part] The remaining part in \ref{thm:PicLind1:item:3}:
Let $f,g\in\Lnu(\R;H)$ with $f=g$ on $\loi{-\infty}{a}$. Since $\nu>L\geq 0$,
we compute using \prettyref{lem:indepCaus}\ref{lem:indepCaus:item:2} and causality of $\td{\nu}^{-1}$ that
\begin{align*}
\1_{\loi{-\infty}{a}}v_{\nu,f} & =\1_{\loi{-\infty}{a}}\td{\nu}^{-1}F^{\nu}\left(v_{\nu,f}\right)+\1_{\loi{-\infty}{a}}\td{\nu}^{-1}f\\
 & =\1_{\loi{-\infty}{a}}\td{\nu}^{-1}F^{\nu}\left(\1_{\loi{-\infty}{a}}v_{\nu,f}\right)+\1_{\loi{-\infty}{a}}\td{\nu}^{-1}\1_{\loi{-\infty}{a}}f\\
 & =\1_{\loi{-\infty}{a}}\td{\nu}^{-1}F^{\nu}\left(\1_{\loi{-\infty}{a}}v_{\nu,f}\right)+\1_{\loi{-\infty}{a}}\td{\nu}^{-1}\1_{\loi{-\infty}{a}}g.
\end{align*}
The same computation also yields that
\[
\1_{\loi{-\infty}{a}}v_{\nu,g}=\1_{\loi{-\infty}{a}}\td{\nu}^{-1}F^{\nu}\left(\1_{\loi{-\infty}{a}}v_{\nu,g}\right)+\1_{\loi{-\infty}{a}}\td{\nu}^{-1}\1_{\loi{-\infty}{a}}g.
\]
It is easy to see that $u\mapsto\1_{\loi{-\infty}{a}}\td{\nu}^{-1}F^{\nu}\left(u\right)+\1_{\loi{-\infty}{a}}\td{\nu}^{-1}\1_{\loi{-\infty}{a}}g$
defines a contraction in $\Lnu(\R;H)$. Hence, the contraction mapping
principle implies that $\1_{\loi{-\infty}{a}}v_{\nu,f}=\1_{\loi{-\infty}{a}}v_{\nu,g}$.

The  statement in \ref{thm:PicLind1:item:2} follows from the fact that $u\mapsto\1_{\loi{-\infty}{a}}\td{\nu}^{-1}F^{\nu}(u)$
defines a contraction and \prettyref{lem:indepCaus}\ref{lem:indepCaus:item:2}.

For the proof of \ref{thm:PicLind1:item:4}, we observe that for all $n\in\N$, we have $\1_{\loi{-\infty}{n}}u_{\eta}\in\Lnu(\R;H)\cap\Lm{\eta}(\R;H)$.
Hence, by \ref{thm:PicLind1:item:2} and \prettyref{lem:indepCaus}\ref{lem:indepCaus:item:3}, it follows that
\begin{align*}
\1_{\loi{-\infty}{n}}u_{\eta} & =\1_{\loi{-\infty}{n}}\td{\eta}^{-1}F^{\eta}\left(\1_{\loi{-\infty}{n}}u_{\eta}\right)
=\1_{\loi{-\infty}{n}}\td{\nu}^{-1}F^{\nu}\left(\1_{\loi{-\infty}{n}}u_{\eta}\right).
\end{align*}
As $\1_{\loi{-\infty}{n}}u_{\nu}$ satisfies the same fixed point
equation, we deduce $\1_{\loi{-\infty}{n}}u_{\eta}=\1_{\loi{-\infty}{n}}u_{\nu}$
for all $n\in\N,$ which yields the assertion. 
\end{proof}
As a first application of \prettyref{thm:PicLind1} we state and prove
the classical version of the Theorem of Picard--Lindelöf.

\begin{thm}[Picard--Lindel\"of -- classical version] \label{thm:PicLind_class} Let $H$ be a Hilbert space, $\Omega\subseteq \R\times H$ be open, $f\colon\Omega\to H$ continuous, $(t_{0},x_{0})\in\Omega$. 
Assume there exists $L\geq0$
such that for all $(t,x),(t,y)\in\Omega$ we have 
\[
\norm{f(t,x)-f(t,y)}\leq L\norm{x-y}.
\]
Then, there exists $\delta>0$ such that the initial value problem
\begin{equation}
\begin{cases}
u'(t)=f(t,u(t)) & (t\in\oi{t_{0}}{t_{0}+\delta}),\\
u(t_{0})=x_{0},
\end{cases}\label{eq:IVP_ODE}
\end{equation}
admits a unique continuously differentiable solution, $u\colon\ci{t_{0}}{t_{0}+\delta}\to H$, which satisfies
$\left(t,u(t)\right)\in\Omega$ for all $t\in\ci{t_{0}}{t_{0}+\delta}$. 
\end{thm}

\begin{proof}
First of all we observe that we may assume, without loss of generality, that
$x_{0}=0$. Indeed, to solve the initial value problem 
\[
\begin{cases}
v'(t)=f(t,v(t)+x_{0}) & (t\in\oi{t_{0}}{t_{0}+\delta}),\\
v(t_{0})=0,
\end{cases}
\]
for a continuously differentiable $v\colon\ci{t_{0}}{t_{0}+\delta}\to H$
is equivalent to solving the problem in \prettyref{thm:PicLind_class}
for $u$ by setting $u=v+\1_{\ci{t_{0}}{t_{0}+\delta}}x_{0}$. Appropriately
shifting the time coordinate, we may also assume that $t_{0}=0$.

Thus, let $(0,0)\in\Omega$. Let also $[0,\delta']\times\cball{0}{\varepsilon}\subseteq\Omega$
for some $\delta',\varepsilon>0$. Denote by $P\colon H\to H$
the orthogonal projection onto $\cball{0}{\varepsilon}$.
By \prettyref{exer:ProjC}, $P$ is Lipschitz continuous with Lipschitz
semi-norm bounded by $1$. We then define 
\begin{align*}
F\colon S_{\rmc}(\R; H) & \to\bigcap_{\nu\geq0}\Lnu(\R; H)\\
g & \mapsto\bigl(t\mapsto\1_{\roi{0}{\delta'}}(t)f(t,P(g(t)))\bigr)
\end{align*}
and will prove that $F$ is well-defined and uniformly Lipschitz continuous. Since the mapping $t\mapsto\1_{\roi{0}{\delta'}}(t)f(t,0)$
is supported on $\ci{0}{\delta'}$, we obtain for $\nu\geq0$ that $F(0)\in \Lnu(\R;H)$.
Moreover, for $\nu\geq0$ and $g,h\in S_{\rmc}(\R; H)$ we estimate
\begin{align*}
& \norm{F(g)-F(h)}_{\Lnu(\R;H)}^{2} \\
& =\int_{\R}\norm{F(g)(t)-F(h)(t)}^{2}\e^{-2\nu t}\d t
 =\int_{0}^{\delta'}\norm{f(t,P(g(t)))-f(t,P(h(t)))}^{2}\e^{-2\nu t}\d t\\
 & \leq L^{2}\int_{0}^{\delta'}\norm{P(g(t))-P(h(t))}^{2}\e^{-2\nu t}\d t
 \leq L^{2}\int_{0}^{\delta'}\norm{g(t)-h(t)}^{2}\e^{-2\nu t}\d t \\
 & \leq L^{2}\norm{g-h}_{\Lnu(\R;H)}^{2},
\end{align*}
which shows that $F$ is well-defined and uniformly Lipschitz continuous.

By \prettyref{thm:PicLind1}, there exists $v\in\dom(\td{\nu})$ with
$\nu>L$ such that 
\[
\td{\nu}v=F^{\nu}(v).
\]
We read off from $v=\td{\nu}^{-1}F^{\nu}(v)$ that $v=0$ on $\loi{-\infty}{0}$, and that $v$ is continuous by \prettyref{thm:Sobolev_emb}. 
Moreover, we obtain that
\[
v(t)=\int_{-\infty}^{t}\1_{\roi{0}{\delta'}}(\tau)f(\tau,P(v(\tau)))\d\tau=\int_{0}^{\min\{t,\delta'\}}f(\tau,P(v(\tau)))\d\tau,
\]
from which we read off that $v$ is continuously differentiable on
$\oi{0}{\delta'}$ since $f$ and $P$ are also continuous. The
same equality implies for $0<t\leq\delta\coloneqq\min\{\frac{\varepsilon}{M},\delta'\}$,
where $M\coloneqq\sup_{\left(t,x\right)\in\ci{0}{\delta'}\times\cball{0}{\varepsilon}}\norm{f(t,x)}$,
that 
\[
\norm{v(t)}\leq\int_{0}^{t}\norm{f(\tau,P(v(\tau)))}\d\tau\leq M\delta\leq\varepsilon.
\]
Thus, $(t,v(t))\in\ci{0}{\delta'}\times\cball{0}{\varepsilon}\subseteq\Omega$
for all $0\leq t\leq \delta$ and so $Pv(t)=v(t)$ for $0\leq t\leq\delta$.
Thus, $u\coloneqq v|_{\ci{0}{\delta}}$ satisfies \prettyref{eq:IVP_ODE}.

Finally, concerning uniqueness, let $\tilde{u}\colon\ci{0}{\delta}\to H$
be a continuously differentiable solution of \prettyref{eq:IVP_ODE}.
Let $\tilde{v}$ be the extension of $\tilde{u}$ by $0$ to the whole
of $\R$. Then we get that 
\begin{align*}
\1_{\loi{-\infty}{\delta}}\tilde{v} & =\1_{\loi{-\infty}{\delta}}\int_{0}^{\cdot}\1_{\roi{0}{\delta'}}(\tau) f(\tau,\tilde{v}(\tau))\d\tau\\
 & =\1_{\loi{-\infty}{\delta}}\int_{-\infty}^{\cdot}\1_{\roi{0}{\delta'}}(\tau)f(\tau,P(\tilde{v}(\tau)))\d\tau\\
 & =\1_{\loi{-\infty}{\delta}}\td{\nu}^{-1}F^{\nu}(\1_{\loi{-\infty}{\delta}}\tilde{v}).
\end{align*}
Since $\1_{\loi{-\infty}{\delta}}v$ is the unique solution of the equation
$w=\1_{\loi{-\infty}{\delta}}\td{\nu}^{-1}F^{\nu}(w)$, we obtain that $\1_{\loi{-\infty}{\delta}}\tilde{v}=\1_{\loi{-\infty}{\delta}}v$,
which yields $u=\tilde{u}$. 
\end{proof}

\begin{rem} The reason for the proof of the classical Picard--Lindel\"of theorem being seemingly complicated is two-fold. First of all, the Hilbert space solution theory is for $\L$-functions rather than continuous (or continuously differentiable) functions. The second, maybe more important point is that the Hilbert space Picard--Lindel\"of asserts a solution theory, which provides \emph{global} existence in the time variable. The main body of the proof of the classical Picard--Lindel\"of theorem presented here is therefore devoted to `localisation' of the abstract theorem. Furthermore, note that the method of proof for obtaining uniqueness and the admittance of the initial value rests on causality. This effect will resurface when we discuss partial differential equations.
\end{rem}

\section{Delay Differential Equations}

In this section, our study will not be as in depth as done for the local Picard--Lindelöf theorem. 
Of course, the solution theory afforded would not be a very good one, if it was only applicable to, arguably,  the easiest case of ordinary differential equations.
We shall see next that the developed theory  applies
to more elaborate examples.

In what follows, let $H$ be a Hilbert space over $\K$. We start out with a delay differential equation with  so-called `discrete delay'.
For this, we introduce, for $h\in\R$, the \emph{time-shift operator}
\begin{align*}
\tau_{h}\colon S_{\rmc}(\R;H) & \to\bigcap_{\nu\in\R}\Lnu(\R;H),\\
f & \mapsto f(\cdot+h).
\end{align*}

\begin{lem}
\label{lem:time-shift-norm}
Let $\mu\in\R$. The mapping $\tau_{h} \colon S_{\rmc}(\R;H) \to\bigcap_{\nu\geq \mu}\Lnu(\R;H)$ is uniformly Lipschitz continuous
if and only if $h\leq 0$.
Moreover, for $\nu\in\R$ we have 
\[
\norm{\tau_{h}}_{\Lnu(\R;H)}=\e^{h\nu}.
\]
\end{lem}

\begin{proof}
Let $f\in S_{\rmc}(\R;H)$. Then for $\nu\in\R$ we compute
\begin{align*}
\norm{\tau_{h}f}_{\Lnu(\R;H)}^{2} & =\int_{\R}\norm{f(t+h)}^{2}\e^{-2\nu t}\d t
 =\int_{\R}\norm{f(t)}^{2}\e^{-2\nu(t-h)}\d t \\
 & =\norm{f}_{\Lnu(\R;H)}^{2}\e^{2\nu h}.
\end{align*}
Since $\sup_{\nu\geq \nu} \e^{2\nu h} <\infty$ if and only if $h\leq 0$ we obtain the equivalence. Moreover, the above equality also also yields the norm of $\tau_h$ in $\Lnu(\R;H)$.
\end{proof}
We will reuse $\tau_{h}$ for the Lipschitz continuous extensions
to $\Lnu(\R;H)$. The well-posedness theorem for delay equations with discrete delay
is contained in the next theorem. We note here that we only formulate the respective
result for right-hand sides that are globally Lipschitz continuous.
With a localisation technique, as has already been carried out for the classical
Picard--Lindelöf theorem, it is also possible to obtain local results. 
\begin{thm}
\label{thm:discretedelay} Let $H$ be a Hilbert
space, $\mu\in\R$, $N\in\N$, $h_{1},\ldots,h_{N}\in\loi{-\infty}{0}$, and
\[
G\colon S_{\rmc}(\R;H^{N})\to\bigcap_{\nu\geq\mu}\Lnu(\R;H)
\]
uniformly Lipschitz. Then there exists an $\eta\in\R$ such that for
all $\nu\geq\eta$ the equation 
\[
\td{\nu}u=G^{\nu}\left(\tau_{h_{1}}u,\ldots,\tau_{h_{N}}u\right)
\]
admits a solution $u\in\dom(\td{\nu})$ which is unique in $\bigcup_{\nu\geq\eta}\Lnu(\R;H)$.
Moreover, for all $a\in\R$ the function $u_{a}\coloneqq\1_{\loi{-\infty}{a}}u$
satisfies 
\[
u_{a}=\1_{\loi{-\infty}{a}}\td{\nu}^{-1}G^{\nu}\left(\tau_{h_{1}}u_{a},\ldots,\tau_{h_{N}}u_{a}\right).
\]
\end{thm}

\begin{proof}
The assertion follows from \prettyref{thm:PicLind1} applied to $F\coloneqq G\circ\left(\tau_{h_{1}},\ldots,\tau_{h_{N}}\right)$. 
\end{proof}
Next, we formulate an initial value problem for a subclass of the
latter type of equations. 
\begin{thm}
\label{thm:IVP-discretedelay} 
Let $h>0$, $f\colon\Rge{0}\times H\times H\to H$
continuous, and $f(\cdot,0,0)\in\Lm{\mu}(\R; H)$ for some $\mu>0$.
Assume that there exists $L\geq0$ with 
\[
\norm{f(t,x,y)-f(t,u,v)}\leq L\norm{(x,y)-(u,v)}\quad\left((t,x,y),(t,u,v)\in\Rge{0}\times H\times H\right).
\]
Let $u_{0}\in C\left([-h,0]; H\right)$. Then there exists $\eta\in\R$
such that for all $\nu\geq\eta$ the initial value problem 
\begin{equation}
\begin{cases}
u'(t)=f(t,u(t),u(t-h)) & (t>0),\\
u(\tau)=u_{0}(\tau) & (\tau\in\ci{-h}{0})
\end{cases}\label{eq:dde}
\end{equation}
admits a unique continuous solution $u\colon\roi{-h}{\infty}\to H$,
continuously differentiable on $\oi{0}{\infty}$. 
\end{thm}

\begin{proof}
For $t<0$ let $f(t,\cdot,\cdot)\coloneqq 0$.
We define $F\colon S_{\rmc}(\R; H)\to\bigcap_{\nu\geq\mu}\Lnu(\R; H)$
by 
\begin{align*}
 & F(\phi)(t)\\
 & \coloneqq f\bigl(t,\phi(t)+\1_{\roi{0}{\infty}}(t)u_{0}(0),\phi(t-h)+\1_{\roi{0}{\infty}}(t-h)u_{0}(0)+\1_{\roi{0}{h}}(t)u_{0}(t-h)\bigr)
\end{align*}
for all $t\in\R$. It is easy to see that $F$ is uniformly Lipschitz continuous.
Thus, by \prettyref{thm:PicLind1}, we find $\eta\geq\mu$ such that
for all $\nu\geq\eta$ the equation 
\[
\td{\nu}v=F^{\nu}(v)
\]
admits a solution $v\in\bigcap_{\nu\geq\eta}\dom(\td{\nu})$ which is unique
in $\bigcup_{\nu\geq\eta}\Lnu(\R; H).$ Note that $\spt F^{\nu}(v)\subseteq\roi{0}{\infty}$.
Hence, $v=0$ on $\loi{-\infty}{0}.$ By \prettyref{thm:Sobolev_emb},
we obtain that $v(0)=0$. We claim that $u\coloneqq v+\1_{\roi{0}{\infty}}(\cdot)u_{0}(0)+\1_{\roi{-h}{0}}u_{0}$
is a solution of \prettyref{eq:dde}. First of all note that $u$
is continuous on $\roi{-h}{\infty}$. Next, for $h> t>0$ we have that $t-h< 0$
and thus $v(t-h)=0$ and so we see that
\begin{align*}
F^{\nu}(v)(t) & =f\left(t,v(t)+\1_{\roi{0}{\infty}}(t)u_{0}(0),v(t-h)+\1_{\roi{0}{\infty}}(t-h)u_{0}(0)+\1_{\roi{0}{h}}(t)u_{0}(t-h)\right)\\
 & =f(t,u(t),u_{0}(t-h)).
\end{align*}
Similarily, for $t\geq h$ we obtain
\[
F^{\nu}(v)(t)  =f(t,u(t),u(t-h))
\]
and thus, by continuity of $f$, $u_{0}$ and $u$, it follows that $v$ is
continuously differentiable on $\oi{0}{\infty}$ and 
\[
u'(t)=v'(t)=\td{\nu}v(t)=f(t,u(t),u(t-h)).
\]
It remains to show uniqueness. For this, let $w\colon\roi{-h}{\infty}\to H$
be a solution of \prettyref{eq:dde}. Then
\[
 w(t)=u_0(0)+\int_{0}^t f(s,w(s),w(s-h))\d s\quad (t\geq 0)
\]
and $w(t)=u_0(t)$ if $t\in \ci{-h}{0}$. We set  $\tilde{v}\coloneqq w-\1_{\roi{0}{\infty}}(\cdot)u_{0}(0)-\1_{\roi{-h}{0}}u_{0}$ and infer
\begin{align*}
 \tilde{v}(t)&= \int_{0}^t f(s,w(s),w(s-h))\d s\\ 
 &=\int_{-\infty}^t f\bigl(s,\tilde{v}(s)+\1_{\roi{0}{\infty}}(s)u_{0}(0),\\
 & \qquad\qquad\qquad\qquad\qquad\tilde{v}(s-h)+\1_{\roi{0}{\infty}}(s-h)u_{0}(0)+\1_{\roi{0}{h}}(s)u_{0}(s-h)\bigr)\d s
\end{align*}
for all $t\in \R$. For  $a\in\R$ we set $\tilde{v}_{a}\coloneqq\1_{\loi{-\infty}{a}}\tilde{v}\in\bigcap_{\nu\in\R}\Lnu(\R; H)$
and obtain, using the above formula for $\tilde{v}$,
\[
\tilde{v}_{a}=\1_{\loi{-\infty}{a}}\td{\nu}^{-1} F^\nu(\tilde{v}_a).
\]
 By uniqueness of the solution of 
\[
\1_{\loi{-\infty}{a}}v=\1_{\loi{-\infty}{a}}\td{\nu}^{-1}F^{\nu}\left(\1_{\loi{-\infty}{a}}v\right)
\]
it follows that $\tilde{v}_{a}=\1_{\loi{-\infty}{a}}v$ for all $a\in\R$
and, thus, $u=w$. 
\end{proof}
The equation to come involves the whole history of the unknown; that is, the unknown evaluated at $\loi{-\infty}{0}$.
For a mapping $u\colon\R\to H$ we define the mapping 
\[
u_{\left(\cdot\right)}\colon\R\ni t\mapsto\bigl(\Rle{0}\ni\theta\mapsto u(t+\theta)\in H\bigr).
\]

\begin{lem}
\label{lem:history}Let $\mu>0$. Then 
\begin{align*}
\Theta\colon S_{\rmc}(\R;H) & \to\bigcap_{\nu\geq\mu}\Lnu\bigl(\R;\L(\Rle{0};H)\bigr)\\
u & \mapsto u_{\left(\cdot\right)}
\end{align*}
is uniformly Lipschitz continuous. More precisely, for all $\nu>0$ we have
\[
\norm{\Theta^{\nu}}=\frac{1}{\sqrt{2\nu}}.
\]
\end{lem}

\begin{proof}
Let $u\in S_{\rmc}(\R;H)$. Then we compute 
\begin{align*}
\norm{\Theta u}_{\Lnu\bigl(\R;\L(\Rle{0};H)\bigr)}^{2} & =\int_{\R}\int_{\Rle{0}}\norm{u(t+\theta)}^{2}\d\theta\e^{-2\nu t}\d t
 =\int_{\R}\int_{\Rle{0}}\norm{u(t)}^{2}\e^{-2\nu(t-\theta)}\d\theta\d t\\
 & =\frac{1}{2\nu}\int_{\R}\norm{u(t)}^{2}\e^{-2\nu t}\d t.\tag*{{\qedhere}}
\end{align*}
\end{proof}
\begin{thm}
\label{thm:infinitedelay} Let $H$ be a Hilbert space, $\mu\in\R$ and let
$\Phi\colon S\bigl(\R;\L(\Rle{0};H)\bigr)\to\bigcap_{\nu\geq\mu}\Lnu(\R;H)$
be uniformly Lipschitz. Then, there exists $\eta>0$ such that for all
$\nu\geq\eta$ the equation 
\[
\td{\nu}u=\Phi^{\nu}(u_{\left(\cdot\right)})
\]
admits a solution $u\in\bigcap_{\nu\geq\eta}\dom(\td{\nu})$ unique in
$\bigcup_{\nu\geq\eta}\Lnu(\R;H).$ 
\end{thm}

\begin{proof}
This is another application of \prettyref{thm:PicLind1}. 
\end{proof}

\section{Comments}

In a way, the proof of \prettyref{thm:PicLind_class} is standard
PDE-theory in a nutshell; a solution theory for $L_{p}$-spaces is
used to deduce existence and uniqueness of solutions and a posteriori
regularity theory provides more information on the properties of the
solution.

Note that -- of course -- other proofs are available for the Picard--Lindelöf
theorem. We chose, however, to present this proof here in order to
provide a perspective on classical results. Furthermore, we mention
that in order to obtain unique existence for the solution, it suffices
to assume that $f$ satisfies a uniform Lipschitz condition with respect
to the second variable and that $f$ is measurable. Continuity
of $f$ is needed in order to obtain $C^{1}$-solutions.

A more detailed exposition and more examples of the theory applied
to delay differential equations can be found in \cite{KPSTW14_OD}
and -- in a Banach space setting -- \cite{PTW14_OD}.

There is also a way of dealing with delay differential equations by
expanding the state space the problem is formulated in. In this case,
it is possible to make use of the rich theory of $C_{0}$-semigroups.
We refer to \cite{Batkai2005} for this.

Causality is one of the main concepts for evolutionary equations. We have provided this notion for mappings defined on $\Lnu$-type spaces only. The situation becomes different, if one considers merely densely defined mappings. Then it is a priori unclear, whether for a Lipschitz continuous mapping the continuous extension is also causal. For this we refer to \prettyref{exer:causal_cont} below and to \cite{Jacob2000,W15_CB}, and \cite[Chapter 2]{W16_H} as well as to references mentioned there.

\section*{Exercises}
\addcontentsline{toc}{section}{Exercises}

\begin{xca}
\label{exer:weak_diff+cont}(a) Let $X$ be a Banach space, $u\colon\ci{a}{b}\to X$
continuous. Show that $v\colon\oi{a}{b}\to X$ given by 
\[
v(t)=\int_{a}^{t}u(\tau)\d\tau
\]
is continuously differentiable with $v'(t)=u(t)$.

(b) Let $H$ be a Hilbert space, and $\nu\in\R.$ Let $u\in\dom(\td{\nu})$
with $\td{\nu}u$ continuous. Show that $u$ is continuously differentiable
and $u'=\td{\nu}u$. 
\end{xca}

\begin{xca}\label{exer:vanish_at_inf}
Prove \prettyref{cor:vanish_at_inf}. 
\end{xca}

\begin{xca}\label{exer:12Hoelder}
Let $H$ be a Hilbert space.
Show that 
\[
\dom(\td{\nu})\hookrightarrow C_{\nu}^{1/2}(\R;H)\coloneqq\set{f\in C_{\nu}(\R;H)}{\e^{-\nu\cdot}f\text{\ is \ensuremath{\tfrac{1}{2}}-Hölder continuous}},
\]
where a function $g\colon\R\to H$ is said to be $\frac{1}{2}$\emph{-Hölder
continuous}, if 
\[
\sup_{\substack{s,t\in\R \\ t\neq s}}\frac{\norm{g(t)-g(s)}}{\abs{t-s}^{1/2}}<\infty.
\]
\end{xca}

\begin{xca}
\label{exer:ProjC}Let $H$ be a Hilbert space, $C\subseteq H$ closed
and convex. Show that the orthogonal projection, $P$, of $H$ onto $C$
defines a Lipschitz continuous mapping with Lipschitz semi-norm bounded
by $1$. 
\end{xca}

\begin{xca}
Let $h\colon\R\times\Rle{0}\times\R^{n}\to\R^{n}$ be continuous satisfying
\[
\norm{h(t,s,x)-h(t,s,y)}\leq L\norm{x-y}
\]
with $h(\cdot,\cdot,0)=0$. Let $R>0$ and $u_{0}\in C(\Rle{0};\R^{n})$ have
compact support. Show that the initial value problem 
\[
\begin{cases}
u'(t)=\int_{-R}^{0}h(t,s,u_{\left(t\right)}(s))\d s & (t>0),\\
u(t)=u_{0}(t) & (t\leq0)
\end{cases}
\]
admits a unique continuous solution $u\colon\R\to\R^{n}$, which is
continuously differentiable on $\Rg{0}$.

Hint: Modify $\Theta$ from \prettyref{lem:history}. 
\end{xca}

\begin{xca}
Let $H$ be a Hilbert space.
Show that for a uniformly Lipschitz continuous $\Phi\colon S\bigl(\R;\L(\Rle{0};H)^{2}\bigr)\to\bigcap_{\nu\geq\mu}\Lnu(\R;H)$
the equation 
\[
\td{\nu}u=\Phi^{\nu}\left(u_{(\cdot)},\left(\td{\nu}u\right)_{\left(\cdot\right)}\right)
\]
admits a unique solution $u\in\dom(\td{\nu})$ for $\nu$ large enough. 
\end{xca}

\begin{xca}\label{exer:causal_cont}
Let $D\subseteq\L(\R)$ be dense and suppose that $F\colon D\subseteq\L(\R)\to\L(\R)$
admits a Lipschitz continuous extension $F^{0}$.

\begin{enumerate}
\item Show that $F^{0}$ is causal if and only if for all $\phi\in S_{\rmc}(\R;\R)$ and all $a\in\R$ there exists $L\geq 0$ such that
\[\abs{\scp{\1_{\loi{-\infty}{a}}\cdot(F(f)-F(g))}{\phi}_{\L(\R)}} \leq L \norm{\1_{\loi{-\infty}{a}}\cdot(f-g)}_{L_{2}(\R)}\]
for all $f,g\in D$; that is,
the mapping 
\[
\left(D,\norm{\1_{\loi{-\infty}{a}}\cdot(\cdot-\cdot)}_{\L(\R)}\right)\ni f\mapsto F(f) \in \left(\L(\R),\abs{\scp{\1_{\loi{-\infty}{a}}\cdot(\cdot-\cdot)}{\phi}_{\L(\R)}}\right)\]
is Lipschitz continuous.
\item For $a\in\R$ let $\dom(F)\cap\dom(F\1_{\loi{-\infty}{a}})$
be dense in $\L(\R)$ and if $f,g\in D=\dom(F)$ and $f=g$ on $\loi{-\infty}{a}$ then also $F(f) = F(g)$ on $\loi{-\infty}{a}$. Show that $F^{0}$ is causal.
\item Assume that for all $f,g\in D$ that $f=g$ on $\loi{-\infty}{a}$
implies that $F(f)=F(g)$ on $\loi{-\infty}{a}$. Show that this is not sufficient for $F^0$ to be causal. 
Hint: Find a dense subspace $D=\dom(F)$ so that the first condition in (b) is not satisfied.
\end{enumerate}
\end{xca}

\printbibliography[heading=subbibliography]

\chapter{The Fourier--Laplace Transformation and Material Law Operators\label{chap:The-Fourier-Laplace-transformati}}

In this chapter we introduce the Fourier--Laplace transformation
and use it to define operator-valued functions of $\td{\nu}$; the
so-called material law operators. These operators will play a crucial
role when we deal with partial differential equations. In the equations
of classical mathematical physics, like the heat equation, wave equation
or Maxwell's equation, the involved material parameters, such as heat
conductivity or permeability of the underlying medium, are incorporated
within these operators and hence the name ``material law''. We start
our chapter by defining the Fourier transformation and proving Plancherel's
theorem in the Hilbert space-valued case, which states that the Fourier
transformation defines a unitary operator on $\L(\R;H)$.

Throughout, let $H$ be a complex Hilbert space.

\section{The Fourier Transformation}

We start by defining the Fourier transformation on $L_{1}(\R;H)$. 

\begin{defn*}
For $f\in L_{1}(\R;H)$ we define the \emph{\index{Fourier transform}Fourier transform $\hat{f}$ of $f$} by 
\[
\hat{f}(s)\coloneqq\frac{1}{\sqrt{2\pi}}\int_{\R}\e^{-\i st}f(t)\d t\quad(s\in\R).
\]
\end{defn*}

We also introduce 
\[
 \cb(\R;H)\index{C{b}(R;H)@$\cb(\R;H)$}\coloneqq \set{ f\colon \R\to H}{f\text{ continuous, bounded}}
\] endowed with the sup-norm, $\norm{\cdot}_\infty$.

\begin{lem}[\index{Lemma of Riemann--Lebesgue}Riemann--Lebesgue]
\label{lem:Riemann-Lebesgue} Let $f\in L_{1}(\R;H)$. Then $\hat{f}\in \cb(\R;H)$ and $\lim_{\abs{t}\to\infty} \bigl\|\hat{f}(t)\bigr\| = 0$.
Moreover, 
\[
\bigl\|\hat{f}\,\bigr\|_{\infty}\leq\frac{1}{\sqrt{2\pi}}\norm{f}_{1}.
\]
\end{lem}

\begin{proof}
First, note that $\hat{f}$ is continuous by dominated convergence
and bounded with 
\[
\bigl\|\hat{f}\,\bigr\|_{\infty}\leq\frac{1}{\sqrt{2\pi}}\norm{f}_{1}.
\]
This shows that the mapping 
\begin{align}
L_{1}(\R;H) & \to \cb(\R;H),\quad f \mapsto\hat{f}\label{eq:fourier_on_L1} 
\end{align}
defines a bounded linear operator. Moreover, for $\varphi\in \cco(\R;H)$ we compute
\begin{align*}
\hat{\varphi}(s) & =\frac{1}{\sqrt{2\pi}}\int_{\R}\e^{-\i st}\varphi(t)\d t
 =\frac{1}{\sqrt{2\pi}}\frac{1}{\i s}\int_{\R}\e^{-\i st}\varphi'(t)\d t
\end{align*}
for $s\ne0$ and thus, 
\[
\limsup_{\abs{s}\to\infty}\norm{\hat{\varphi}(s)}\leq\limsup_{\abs{s}\to\infty}\frac{1}{\abs{s}}\frac{1}{\sqrt{2\pi}}\norm{\varphi'}_{1}=0,
\]
which shows that $\lim_{\abs{s}\to\infty}\norm{\hat{\varphi}(s)}=0$. By the facts that
$\cco(\R;H)$ is dense in $L_{1}(\R;H)$ (see \prettyref{lem:dense sets}),
$\set{f\in \cb(\R;H)}{\lim_{\abs{t}\to\infty}\norm{f(t)}=0}$ is a closed subspace of $\cb(\R;H)$ and the operator
in \prettyref{eq:fourier_on_L1} is bounded, the assertion follows. 
\end{proof}
It is our main goal to extend the definition of the Fourier transformation
to functions in $\L(\R;H)$. For doing so, we make use of the Schwartz space
of rapidly decreasing functions. 
\begin{defn*}
We define 
\[
\mathcal{S}(\R;H)\coloneqq\set{f\in C^{\infty}(\R;H)}{\forall n,k\in\N:\,\bigl(t\mapsto t^{k}f^{(n)}(t)\bigr)\in \cb(\R;H)}
\]
to be the \emph{\index{Schwartz space}Schwartz space} of rapidly decreasing functions on $\R$ with values
in $H$. 
\end{defn*}
As usual we abbreviate $\mathcal{S}(\R)\coloneqq \mathcal{S}(\R;\K)$.
\begin{rem}
$\mathcal{S}(\R;H)$ is a Fréchet space with respect to the seminorms
\[
\mathcal{S}(\R;H)\ni f\mapsto\sup_{t\in\R}\norm{t^{k}f^{(n)}(t)}\quad(k,n\in\N).
\]
Moreover, $\mathcal{S}(\R;H)\subseteq\bigcap_{p\in\ci{1}{\infty}}L_{p}(\R;H)$.
Indeed, $\mathcal{S}(\R;H)\subseteq L_{\infty}(\R;H)$ by definition,
and for $1\leq p<\infty$ we have that 
\begin{align*}
\int_{\R}\norm{f(t)}^{p}\d t & =\int_{\R}\frac{1}{\left(1+\abs{t}\right)^{2p}}\norm{(1+\abs{t})^{2}f(t)}^{p}\d t\\
 & \leq\sup_{t\in\R}\norm{(1+\abs{t})^{2}f(t)}^{p}\:\int_{\R}\frac{1}{(1+\abs{t})^{2p}}\d t<\infty.
\end{align*}
\end{rem}

\begin{prop}
\label{prop:Fourier-Schwartz-space} For $f\in\mathcal{S}(\R;H)$
we have $\hat{f}\in\mathcal{S}(\R;H)$ and the mapping 
\[
\mathcal{S}(\R;H)\to\mathcal{S}(\R;H),\quad f\mapsto\hat{f}
\]
is bijective. Moreover, for $f,g\in L_{1}(\R;H)$ we have that 
\begin{equation}
\int_{\R}\scp{\hat{f}(t)}{g(t)}\d t=\int_{\R}\scp{f(t)}{\hat{g}(-t)}\d t.\label{eq:Fourier-adjoint}
\end{equation}
Additionally,  if $f,\hat{f}\in L_{1}(\R;H)$ then 
\begin{equation}
f(t)=\hat{\hat{f\,}}(-t)\quad(t\in\R).\label{eq:Fourier-inversion}
\end{equation}
\end{prop}

\begin{proof}
Let $f\in \mathcal{S}(\R;H)$.
By \prettyref{exer:derivative-under-integral} we have 
\begin{equation}
\hat{f}\:'(s)=\frac{1}{\sqrt{2\pi}}\int_{\R}(-\i t)\e^{-\i st}f(t)\d t=-\i\hat{\bigl(t\mapsto tf(t)\bigr)}(s)\quad(s\in\R)\label{eq:derivative-of-fourier}
\end{equation}
and 
\begin{equation}
s\hat{f}(s)=\frac{\i}{\sqrt{2\pi}}\intop_{\R}\left(-\i s\right)\e^{-\i st}f(t)\d t=-\i\hat{f'}(s)\quad(s\in\R).\label{eq:fourier-of-derivative}
\end{equation}
Using these formulas, one can show that $\hat{f}\in\mathcal{S}(\R;H).$
Since the bijectivity of the Fourier transformation on $\mathcal{S}(\R;H)$
would follow from \prettyref{eq:Fourier-inversion}, it suffices to
prove the formulas \prettyref{eq:Fourier-adjoint} and \prettyref{eq:Fourier-inversion}.
Let $f,g\in L_{1}(\R;H).$ Then we compute using \prettyref{prop:interchange_integral}
and Fubini's theorem 
\begin{align*}
\int_{\R}\scp{\hat{f}(t)}{g(t)}\d t & =\int_{\R}\frac{1}{\sqrt{2\pi}}\scp{\int_{\R}\e^{-\i st}f(s)\d s}{g(t)}\d t\\
 & =\int_{\R}\int_{\R}\frac{1}{\sqrt{2\pi}}\e^{\i st}\scp{f(s)}{g(t)}\d s\d t\\
 & =\int_{\R}\scp{f(s)}{\frac{1}{\sqrt{2\pi}}\int_{\R}\e^{\i st}g(t)\d t}\d s\\
 & =\int_{\R}\scp{f(s)}{\hat{g}(-s)}\d s,
\end{align*}
which yields \prettyref{eq:Fourier-adjoint}. For proving formula \prettyref{eq:Fourier-inversion},
we consider the function $\gamma$ defined by $\gamma(t)\coloneqq\e^{-\frac{t^{2}}{2}}$ for $t\in\R$.
Clearly, $\gamma\in\mathcal{S}(\R)$. We claim that $\hat{\gamma}=\gamma$.
Indeed, we observe that $\gamma$ solves the initial value problem
$y'+ty=0$ subject to $y(0)=1$; if we can show that $\hat{\gamma}$
solves the same initial value problem, then their equality would follow from the uniqueness
of the solution. First, we observe that $\hat{\gamma}(0)=\frac{1}{\sqrt{2\pi}}\int_{\R}\e^{-\frac{t^{2}}{2}}\d t=1.$
Second, we compute using the formulas \prettyref{eq:derivative-of-fourier}
and \prettyref{eq:fourier-of-derivative} that
\[
\hat{\gamma}'(s)=-\i\hat{\bigl(t\mapsto t\gamma(t)\bigr)}(s)=\i\hat{\gamma'}(s)=-s\hat{\gamma}(s)\quad(s\in\R).
\]
Altogether, we have shown that $\hat{\gamma}$ solves the same initial value problem
as $\gamma$ and hence, $\hat{\gamma}=\gamma$. Let now $f\in L_{1}(\R;H)$
with $\hat{f}\in L_{1}(\R;H),$ $a>0$ and $x\in H$. Then we compute using \prettyref{eq:Fourier-adjoint}
\begin{align*}
\scp{\int_{\R}\hat{f}(t)\gamma(at)\e^{\i st}\d t}{x} & =\int_{\R}\scp{\hat{f}(t)}{\gamma(at)x\e^{-\i st}}\d t
 =\int_{\R}\scp{f(t)}{\hat{\bigl(\gamma(a\cdot)x\e^{-\i s(\cdot)}\bigr)}(-t)}\d t\\
 & =\int_{\R}\scp{f(t)}{\frac{1}{\sqrt{2\pi}}\int_{\R}\gamma(ar)x\e^{-\i sr}\e^{\i tr}\d r}\d t\\
 & =\frac{1}{a}\int_{\R}\scp{f(t)}{\hat{\gamma}\left(\frac{s-t}{a}\right)x}\d t
 =\frac{1}{a}\int_{\R}\scp{f(t)}{\gamma\left(\frac{s-t}{a}\right)x}\d t\\
 & =\int_{\R}\scp{f(s-at)}{\gamma\left(t\right)x}\d t
 =\scp{\int_{\R}f(s-at)\gamma\left(t\right)\d t}{x}
\end{align*}
for each $s\in\R$. Since this holds for all $x\in H$ we get 
\[
\int_{\R}\hat{f}(t)\gamma(at)\e^{\i st}\d t=\int_{\R}f(s-at)\gamma\left(t\right)\d t\quad(s\in\R).
\]
Letting $a\to0$ in the latter equality, we obtain 
\begin{equation}
\int_{\R}\hat{f}(t)\e^{\i st}\d t=\lim_{a\to0}\int_{\R}f(s-at)\gamma\left(t\right)\d t \quad(s\in\R),\label{eq:limit_gamma}
\end{equation}
where we have used dominated convergence for the term on the left-hand
side. In order to compute the limit on the right-hand side, we first
observe that 
\[
\int_{\R}\norm{\int_{\R}f(s-at)\gamma\left(t\right)\d t}\d s\leq\int_{\R}\int_{\R}\norm{f(s-at)}\d s\:\gamma(t)\d t=\norm{f}_{1}\norm{\gamma}_{1},
\]
and hence, for each $a>0$ the operator
\begin{align*}
S_{a}\from L_{1}(\R;H)& \to L_{1}(\R;H),\\
f & \mapsto \left(s\mapsto\int_{\R}f(s-at)\gamma\left(t\right)\d t\right)
\end{align*}
is bounded by $\norm{\gamma}_{1}$. Moreover, since $S_{a}\psi\to\psi(\cdot)\norm{\gamma}_{1}$
as $a\to0$ for $\psi\in \cc(\R;H)$, we infer that 
\[
S_{a}f\to f(\cdot)\norm{\gamma}_{1}\quad(a\to0)
\]
for each $f\in L_{1}(\R;H)$. Hence, passing to a suitable sequence
$\left(a_{n}\right)_{n}$ in $\R_{>0}$ tending to $0$, we get 
\[
\lim_{n\to\infty}\left(S_{a_{n}}f\right)(s)\to f(s)\norm{\gamma}_{1}\quad(\text{a.e. } s\in\R).
\]
Using this identity for the right-hand side of \prettyref{eq:limit_gamma},
we get 
\[
\int_{\R}\hat{f}(t)\e^{\i st}\d t=f(s)\norm{\gamma}_{1}\quad(\text{a.e. } s\in\R),
\]
and since $\norm{\gamma}_{1}=\sqrt{2\pi}$, we derive \prettyref{eq:Fourier-inversion}. 
\end{proof}
With these preparations at hand, we are now able to prove the main
theorem of this section. 
\begin{thm}[Plancherel\index{Theorem of Plancherel}]
\label{thm:Plancherel} The mapping 
\[
\mathcal{F}\from \mathcal{S}(\R;H)\subseteq\L(\R;H)\to\L(\R;H),\:f\mapsto\hat{f}
\]
extends to a unitary operator on $\L(\R;H)$, again denoted by $\mathcal{F}$,
the \emph{\index{Fourier transformation}Fourier transformation. }Moreover, $\mathcal{F}^{\ast}=\mathcal{F}^{-1}$
is given by $f\mapsto\hat{f}(-\cdot)$. 
\end{thm}

\begin{proof}
Using \prettyref{eq:Fourier-adjoint} and \prettyref{eq:Fourier-inversion}
we obtain that 
\begin{align*}
\scp{\hat{f}}{\hat{g}}_{2} & =\int_{\R}\scp{\hat{f}(t)}{\hat{g}(t)}\d t
 =\int_{\R}\scp{f(t)}{\hat{\hat{g\,}}(-t)}\d t
 =\int_{\R}\scp{f(t)}{g(t)}\d t
 =\scp{f}{g}_{2}
\end{align*}
for all $f,g\in\mathcal{S}(\R;H)$ and thus, in particular, 
\begin{equation}
\norm{f}_{2}=\norm{\mathcal{F}f}_{2}.\label{eq:isometry}
\end{equation}
Moreover, $\dom(\mathcal{F})=\ran(\mathcal{F})=\mathcal{S}(\R;H)$
is dense in $\L(\R;H)$ and hence, the first assertion follows by \prettyref{exer:unitary_extension}. As $\mathcal{F}$ is unitary, we have $\mathcal{F}^*=\mathcal{F}^{-1}$, thus, by \prettyref{eq:Fourier-adjoint} applied to $f,g\in \mathcal{S}(\R;H)$, we read off (using \prettyref{prop:adjoint_core}) that $\mathcal{F}^{-1}=(f\mapsto \hat{f}(-\cdot))$, which yields all the claims of the theorem at hand.
\end{proof}
\begin{rem}
We emphasise that for $f\in\L(\R;H)$ the Fourier transform $\mathcal{F}f$
is not given by the integral expression for $L_{1}$-functions, simply
because the integral does not need to exist. However, by dominated
convergence 
\[
\mathcal{F}f=\lim_{R\to\infty}\frac{1}{\sqrt{2\pi}}\int_{-R}^{R}\e^{-\i t(\cdot)}f(t)\d t,
\]
where the limit is taken in $\L(\R;H).$ 
\end{rem}

\section{The Fourier--Laplace Transformation and its Relation to the Time
Derivative}

We now use the Fourier transformation to define an analogous transformation
on our exponentially weighted $\L$-type spaces; the so-called Fourier--Laplace
transformation. We recall from \prettyref{cor:nu=00003D0} that
for $\nu\in\R$ the mapping 
\[
\exp(-\nu\m)\from\Lnu(\R;H)\to\L(\R;H),\;f\mapsto\left(t\mapsto\e^{-\nu t}f(t)\right)
\]
is unitary. In a similar fashion, we obtain that
\[
\exp(-\nu\m)\from L_{1,\nu}(\R;H)\to L_{1}(\R;H),\;f\mapsto\left(t\mapsto\e^{-\nu t}f(t)\right)
\]
defines an isometry. 
\begin{defn*}
Let $\nu\in\R$. We define the \emph{\index{Fourier--Laplace transformation}Fourier--Laplace transformation} as
\[
\mathcal{L}_{\nu}\from \Lnu(\R;H)
\to\L(\R;H),\:f\mapsto\mathcal{F}\exp(-\nu\m)f.
\]
\end{defn*}
We can also consider the Fourier--Laplace transformation as a mapping from $L_{1,\nu}(\R;H)$ to $\cb(\R;H)$; that is,
\[
\mathcal{L}_{\nu}\from L_{1,\nu}(\R;H)
\to\cb(\R;H),\:f\mapsto\mathcal{F}\exp(-\nu\m)f.
\]

\begin{rem}
Note that $\mathcal{L}_{\nu}=\mathcal{F}\exp(-\nu\m)$ is unitary as an operator from $\Lnu(\R;H)$
to $\L(\R;H)$ since it is the composition of two unitary operators. For $\varphi\in \cci(\R;H)$, we have
the expression 
\[
\left(\mathcal{L}_{\nu}\varphi\right)(t)=\frac{1}{\sqrt{2\pi}}\int_{\R}\e^{-(\i t+\nu)s}\varphi(s)\d s\quad(t\in\R),
\]
which shows that $\mathcal{L}_{\nu}$ can be interpreted as a shifted
variant of the Fourier transformation, where the real part in the
exponent equals $\nu$ instead of zero. 
\end{rem}

Our next goal is to show that the Fourier--Laplace transformation provides
a spectral representation of our time derivative, $\td{\nu}$. 
%

\begin{defn*}
Let $V\from \R\to \K$ be measurable.
We define the \emph{\index{multiplication-by-$V$ operator}multiplication-by-$V$ operator} as 
\[
V(\m)\from \dom(V(\m))\subseteq\L(\R;H)\to\L(\R;H),\:f\mapsto\bigl(t\mapsto V(t)f(t)\bigr)
\]
with 
\[
\dom(V(\m))\coloneqq\set{f\in\L(\R;H)}{\bigl(t\mapsto V(t)f(t)\bigr)\in\L(\R;H)}.
\]
In particular, if $V$ is the identity on $\R$ we will just write $\m$ instead of $\id(\m)$ and call it the \emph{\index{multiplication-by-the-argument
operator, $\m$}multiplication-by-the-argument operator}.
\end{defn*}

\begin{rem}
Note that the multiplication-by-$V$ operator is a vector-valued analogue of the multiplication operator seen in \prettyref{thm:mult1} and \prettyref{thm:mult2}. The statements in these theorems generalise (easily) to the vector-valued situation at hand. 
Thus, as in \prettyref{thm:mult1}, one shows that $\m$ is selfadjoint.
Moreover, when $H\neq\{0\}$, in a similar fashion to the arguments carried out in \prettyref{thm:mult2} one shows that 
\[
\sigma(\m)=\R.
\]
\end{rem}

In order to avoid trivial cases, we shall assume throughout that $H\neq\{0\}$.

\begin{thm}
\label{thm:spectral-repr-derivative}Let $\nu\in\R$. Then 
\[
\td{\nu}=\mathcal{L}_{\nu}^{\ast}(\i\m+\nu)\mathcal{L}_{\nu}.
\]
In particular, 
\[
\sigma(\td{\nu})=\set{\i t+\nu}{t\in\R}.
\]
\end{thm}

\begin{proof}
We first prove the assertion for $\nu\ne0$ and show that 
\[
I_{\nu}=\mathcal{L}_{\nu}^{\ast}\left(\frac{1}{\i\m+\nu}\right)\mathcal{L}_{\nu}.
\]
The assertion will then follow by \prettyref{thm:mult1}. Note that $\frac{1}{\i\m+\nu}\in L(\L(\R;H))$
by \prettyref{thm:mult1}, and hence, both operators $I_{\nu}$ and $\mathcal{L}_{\nu}^{\ast}(\tfrac{1}{\i\m+\nu})\mathcal{L}_{\nu}$ are bounded and
defined on the whole of $\Lnu(\R;H).$ Thus, it suffices to prove the
equality on a dense subset of $\Lnu(\R;H)$, like $\cc(\R;H).$
We will just do the computation for the case when $\nu>0$. So, let $\varphi\in \cc(\R;H)$
and compute 
\begin{align*}
\left(\mathcal{L}_{\nu}I_{\nu}\varphi\right)(t) & =\frac{1}{\sqrt{2\pi}}\int_{\R}\e^{-\left(\i t+\nu\right)s}\int_{-\infty}^{s}\varphi(r)\d r\d s
 =\frac{1}{\sqrt{2\pi}}\int_{\R}\int_{r}^{\infty}\e^{-\left(\i t+\nu\right)s}\d s\:\varphi(r)\d r\\
 & =\frac{1}{\sqrt{2\pi}}\frac{1}{\i t+\nu}\int_{\R}\e^{-(\i t+\nu)r}\varphi(r)\d r
 =\frac{1}{\i t+\nu}\left(\mathcal{L}_{\nu}\varphi\right)(t)
\end{align*}
for $t\in\R.$ For $\nu<0$ the computation is analogous. In the case when
$\nu=0$ we observe that 
\begin{align*}
\td{0} & =\exp(-\nu\m)(\td{\nu}-\nu)\exp(-\nu\m)^{-1}
 =\exp(-\nu\m)\mathcal{L}_{\nu}^{\ast}(\i\m+\nu-\nu)\mathcal{L}_{\nu}\exp(-\nu\m)^{-1}\\
 &  =\mathcal{L}_{0}^{\ast}(\i\m)\mathcal{L}_{0}.\tag*{\qedhere}
\end{align*}
\end{proof}

\section{Material Law Operators}

Using the multiplication operator representation of $\td{\nu}$ via
the Fourier--Laplace transformation, we can assign a functional calculus
to this operator. We will do this in the following and define operator-valued
functions of $\td{\nu}.$ The class of functions used for this calculus
are the so-called material laws. We begin by defining this function class. 
\begin{defn*}
A mapping $M\from\dom(M)\subseteq\C\to L(H)$ is called a \emph{\index{material law}material law} if
\begin{enumerate}
\item $\dom(M)$ is open and $M$ is holomorphic (i.e., complex differentiable;
see also \prettyref{exer:holomorphic}),
\item there exists some $\nu\in\R$ such that $\C_{\Re>\nu}\subseteq\dom(M)$ and
\[
\norm{M}_{\infty,\C_{\Re>\nu}}\coloneqq\sup_{z\in\C_{\Re>\nu}}\norm{M(z)}<\infty.
\]
\end{enumerate}
Moreover, we set
\[
\sbb{M}\coloneqq\inf\set{\nu\in\R}{\C_{\Re>\nu}\subseteq\dom(M) \text{ and }\norm{M}_{\infty,\C_{\Re>\nu}}<\infty}
\]
to be the \emph{\index{abscissa of boundedness, $\sbb{\cdot}$}abscissa of boundedness} of $M$. 
\end{defn*}
\begin{example}
\label{exa:material-laws}

\begin{enumerate}
\item Polynomials in $z^{-1}$: Let $n\in\N_0$, $M_{0},\ldots,M_{n}\in L(H)$.
Then 
\[
M(z)\coloneqq\sum_{k=0}^{n}z^{-k}M_{k}\quad(z\in\C\setminus\{0\})
\]
defines a material law with 
\begin{align*}
\sbb{M} & =\begin{cases}
-\infty, & \text{ if }M_{1}=\ldots=M_{n}=0,\\
0, & \text{ otherwise.}
\end{cases}
\end{align*}
\item Series in $z^{-1}$: Let $(M_{k})_{k\in\N}$ in $L(H)$ such that $\sum_{k=0}^{\infty}\norm{M_{k}}r^{-k}<\infty$
for some $r>0$. Then 
\[
M(z)\coloneqq\sum_{k=0}^{\infty}z^{-k}M_{k}\quad(z\in\C\setminus\{0\})
\]
defines a material law with $\sbb{M}\leq r$.
\item Exponentials: Let $h\in\R,M_{0}\in L(H)$ where $M_{0}\ne0$ and set 
\[
M(z)\coloneqq M_{0}\e^{zh}\quad(z\in\C).
\]
Then $M$ is a material law if and only if $h\leq0$. In this case,
$\sbb{M}=-\infty.$
\item Laplace transforms: Let $\nu\in\R$ and $k\in L_{1,\nu}(\R)$ with
$\spt k\subseteq\Rge{0}$. Then 
\[
M(z)\coloneqq\sqrt{2\pi}(\mathcal{L}k)(z)\coloneqq \int_{0}^{\infty}\e^{-zt}k(t)\d t\quad(z\in\C_{\Re>\nu})
\]
defines a material law with $\sbb{M}\leq\nu$.
\item\label{exa:material-laws:item:5} Fractional powers: Let $M_{0}\in L(H)$, $M_{0}\ne0,$ $\alpha\in\R$
and set 
\[
M(z)\coloneqq M_{0}z^{-\alpha}\quad(z\in\C\setminus\Rle{0}),
\]
where we set 
\[
\left(r\e^{\i\theta}\right)^{-\alpha}\coloneqq r^{-\alpha}\e^{-\i\alpha\theta}\quad(r>0,\theta\in\oi{-\pi}{\pi}).
\]
Then $M$ is a material law if and only if $\alpha\geq0$ and 
\begin{align*}
\sbb{M} & =\begin{cases}
-\infty & \text{ if }\alpha=0,\\
0 & \text{ otherwise}.
\end{cases}
\end{align*}
\end{enumerate}
\end{example}

For material laws $M$ we now define the corresponding material law
operators in terms of the functional calculus induced by the spectral
representation of $\td{\nu}$. 
\begin{prop}
\label{prop:mat_law_function_of_td}Let $M\from\dom(M)\subseteq\C\to \bo(H)$ be
a material law. Then, for $\nu>\sbb{M}$, the operator 
\[
M(\i\m+\nu)\from\L(\R;H)\to\L(\R;H),\;f\mapsto\bigl(t\mapsto M(\i t+\nu)f(t)\bigr)
\]
is bounded. Moreover, we define the \emph{material law operator}\index{material law operator} 
\[
M(\td{\nu})\coloneqq\mathcal{L}_{\nu}^{\ast}M(\i\m+\nu)\mathcal{L}_{\nu}\in L(\Lnu(\R;H))
\]
and obtain 
\[
\norm{M(\td{\nu})}\leq\norm{M}_{\infty,\C_{\Re>\nu}}.
\]
\end{prop}

\begin{proof}
The proof is clear. 
\end{proof}

\begin{rem}\label{rem:mlohom}
  The set of material laws is an algebra and the mapping of assigning a material law to its corresponding material law operator is an algebra homomorphism in the following sense.
  For $j\in\{1,2\}$ let $M_j\from \dom(M_j)\subseteq\C \to \bo(H)$ be material laws, $\lambda\in\C$.
  Then $M_1+M_2$ (with domain $\dom(M_1)\cap \dom(M_2)$), $\lambda M_1$ and $M_1\cdot M_2$ (with domain $\dom(M_1)\cap \dom(M_2)$) are material laws as well.
  Moreover, $\sbb{M_1+M_2}, \sbb{M_1\cdot M_2}\leq \max\{\sbb{M_1},\sbb{M_2}\}$. Furthermore, if $M_2(z)$ is a multiplication operator for all $z\in \dom(M_2)$, then
  for $\nu>\max\{\sbb{M_1},\sbb{M_2}\}$ we have
  $(M_1M_2)(\td{\nu}) = M_1(\td{\nu}) M_2(\td{\nu}) = M_2(\td{\nu}) M_1(\td{\nu}) = (M_2M_1)(\td{\nu})$.
\end{rem}

\begin{example}
\label{exa:material_law_revisited}We now revisit the material laws presented
in \prettyref{exa:material-laws} and compute their corresponding operators,
$M(\td{\nu})$.
\begin{enumerate}
\item Let $n\in\N_0$, $M_{0},\ldots,M_{n}\in L(H)$ and 
\[
M(z)\coloneqq\sum_{k=0}^{n}z^{-k}M_{k}\quad(z\in\C\setminus\{0\}).
\]
Then, for $\nu>0$, one obviously has 
\[
M(\td{\nu})=\sum_{k=0}^{n}\td{\nu}^{-k}M_{k},
\]
due to \prettyref{thm:spectral-repr-derivative}.
\item Let $(M_{k})_{k\in\N}$ in $L(H)$ such that $\sum_{k=0}^{\infty}\norm{M_{k}}r^{-k}<\infty$
for some $r>0$ and 
\[
M(z)\coloneqq\sum_{k=0}^{\infty}z^{-k}M_{k}\quad(z\in\C\setminus\{0\}).
\]
Then, for $\nu>r$, one has 
\[
M(\td{\nu})=\sum_{k=0}^{\infty}\td{\nu}^{-k}M_{k}
\]
again on account of \prettyref{thm:spectral-repr-derivative}.
\item\label{exa:material_law_revisited:item:3} Let $h\leq0,M_{0}\in L(H)$ and 
\[
M(z)\coloneqq M_{0}\e^{zh}\quad(z\in\C).
\]
Then, for $\nu\in\R$, we have 
\[
M(\td{\nu})=M_{0}\tau_{h},
\]
where 
\[
\tau_{h}\from\Lnu(\R;H)\to\Lnu(\R;H),\;f\mapsto\bigl(t\mapsto f(t+h)\bigr).
\]
Indeed, for $\varphi\in \cc(\R;H)$ we compute 
\begin{align*}
\left(\mathcal{L}_{\nu}M_{0}\tau_{h}\varphi\right)(t) & =\frac{1}{\sqrt{2\pi}}\int_{\R}\e^{-(\i t+\nu)s}M_{0}\varphi(s+h)\d s \\
 & =M_{0}\frac{1}{\sqrt{2\pi}}\int_{\R}\e^{-(\i t+\nu)\left(s-h\right)}\varphi(s)\d s
 =M(\i t+\nu)\left(\mathcal{L}_{\nu}\varphi\right)(t)
\end{align*}
for all $t\in\R$, where we have used \prettyref{prop:interchange_integral}
in the second line. Hence, 
\[
M_{0}\tau_{h}\varphi=\mathcal{L}_{\nu}^{\ast}M(\i\m+\nu)\mathcal{L}_{\nu}\varphi=M(\td{\nu})\varphi
\]
and since $\cc(\R;H)$ is dense in $\Lnu(\R;H)$ the assertion follows.
\item Let $\nu\in\R$ and $k\in L_{1,\nu}(\R)$ with $\spt k\subseteq\Rge{0}$
and 
\[
M(z)\coloneqq\sqrt{2\pi} (\mathcal{L}k)(z) \quad(z\in\C_{\Re>\nu}).
\]
Then, by \prettyref{exer:convolution}, 
\[
M(\td{\mu})=k\ast
\]
for each $\mu>\nu.$
\item Let $M_{0}\in L(H)$, $\alpha>0$ and 
\[
M(z)\coloneqq M_{0}z^{-\alpha}\quad(z\in\C\setminus\Rle{0}).
\]
Then for $\nu>0$ we have 
\begin{equation}
\left(M(\td{\nu})f\right)(t)=M_{0}\int_{-\infty}^{t}\frac{1}{\Gamma(\alpha)}(t-s)^{\alpha-1}f(s)\d s\quad(\text{a.e. }t\in\R)\label{eq:fractional_int}
\end{equation}
for each $f\in\Lnu(\R;H)$; see \prettyref{exer:fractional}. This
formula gives rise to the definition 
\[
\left(\td{\nu}^{-\alpha}f\right)(t)\coloneqq\int_{-\infty}^{t}\frac{1}{\Gamma(\alpha)}(t-s)^{\alpha-1}f(s)\d s \quad(t\in\R),
\]
which is known as the \emph{\index{fractional integral}(Riemann--Liouville) fractional integral of order $\alpha$}.
\end{enumerate}
\end{example}

Throughout the previous examples, the operator $M(\td{\nu})$ did not
depend on the actual value of $\nu$. Indeed, this is true for all
material laws. In order to see this, we need the following lemma. 
\begin{lem}
\label{lem:F-L idenpendent of nu}Let $\mu,\nu\in\R$ with $\mu<\nu$, and set $U\coloneqq\set{z\in\C}{\Re z\in\oi{\mu}{\nu}}$.
Moreover, let $g\from\overline{U}\to H$ be continuous and holomorphic on
$U$ such that $g(\i\cdot+\nu),g(\i\cdot+\mu)\in\L(\R;H)$ and there
exists a sequence $(R_{n})_{n\in\N}$ in $\R_{\geq0}$ such that $R_{n}\to\infty$
and 
\begin{equation}
\int_{\mu}^{\nu}\norm{g(\pm \i R_{n}+\rho)}\d\rho\to0\quad(n\to\infty).\label{eq:integral_to_zero}
\end{equation}
Then 
\[
\mathcal{L}_{\mu}^{\ast}g(\i\cdot+\mu)=\mathcal{L}_{\nu}^{\ast}g(\i\cdot+\nu).
\]
\end{lem}

\begin{proof}
Let $t\in\R$. By Cauchy's integral theorem, we have that 
\[
\int_{\gamma_{R_{n}}}g(z)\e^{zt}\d z=0,
\]
where $\gamma_{R_{n}}$ is the rectangular closed path with corners
$\pm\i R_{n}+\mu,\pm\i R_{n}+\nu$. 
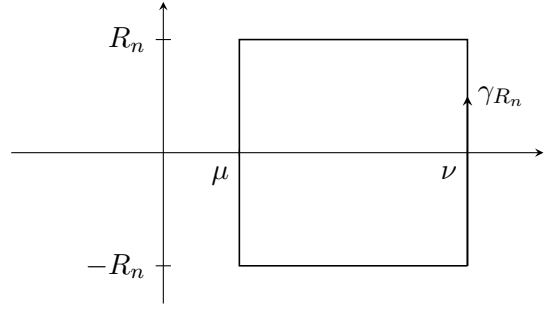
\begin{figure}[htb]
  \centering
  \begin{tikzpicture}[>=stealth,xscale=2]
    \draw[->] (-1,0)--(2.5,0);
    \draw[->] (0,-2)--(0,2);
    \def\R{1.5}
    \def\MU{0.5}
    \def\NU{2}
    \draw (\MU,0.05)--(\MU,-0.05) node[below left]{$\mu$};
    \draw (\NU,0.05)--(\NU,-0.05) node[below left]{$\nu$};
    \draw (0.05,-\R)--(-0.05,-\R) node[left]{$-R_n$};
    \draw (0.05,\R)--(-0.05,\R) node[left]{$R_n$};
    \draw[->,semithick] (\NU,-\R)--(\NU,\R/2);
    \draw[semithick] (\NU,-\R)--(\NU,\R)--(\MU,\R)--(\MU,-\R)--(\NU,-\R);
    \draw (\NU,\R/2) node[right]{$\gamma_{R_n}$};
  \end{tikzpicture}
  \caption{Curve $\gamma_{R_n}$.}
  \label{fig:Cauchy-curve_Rn}
\end{figure}
Thus, we have that 
\begin{equation}
\begin{split}
& \i\int_{-R_{n}}^{R_{n}}g(\i s+\nu)\e^{(\i s+\nu)t}\d s-\i\int_{-R_{n}}^{R_{n}}g(\i s+\mu)\e^{(\i s+\mu)t}\d s\\
& = -\int_{\mu}^{\nu}g(-\i R_{n}+\rho)\e^{(-\i R_{n}+\rho)t}\d\rho+\int_{\mu}^{\nu}g(\i R_{n}+\rho)\e^{(\i R_{n}+\rho)t}\d\rho.
\end{split}
\label{eq:path}
\end{equation}
Note that with the help of the formula for the inverse Fourier-transformation (see \prettyref{thm:Plancherel}) and $\mathcal{L}_\nu^*=(\mathcal{F}\exp(-\nu \m))^*=\exp(-\nu \m )^{-1}\mathcal{F}^*$ the left-hand side of \prettyref{eq:path} is nothing but
\[
\sqrt{2\pi}\i\left(\left(\mathcal{L}_{\nu}^{\ast}\1_{\ci{-R_{n}}{R_{n}}}g(\i\cdot+\nu)\right)(t)-\left(\mathcal{L}_{\mu}^{\ast}\1_{\ci{-R_{n}}{R_{n}}}g(\i\cdot+\mu)\right)(t)\right),
\]
and hence, there is a subsequence of $(R_{n})_{n}$ (which we do not
relabel) such that the left-hand side of \prettyref{eq:path} tends
to 
\[
\sqrt{2\pi}\i\left(\left(\mathcal{L}_{\nu}^{\ast}g(\i\cdot+\nu)\right)(t)-\left(\mathcal{L}_{\mu}^{\ast}g(\i\cdot+\mu)\right)(t)\right)
\]
for almost every $t\in\R$ as $n\to\infty$. As such, all we need to show is
that the right-hand side of \prettyref{eq:path} tends to $0$ as
$n\to\infty$, which obviously follows by \prettyref{eq:integral_to_zero}. 
\end{proof}
\begin{thm}
\label{thm:material law independent of nu}Let $M\from\dom(M)\subseteq\C\to L(H)$ be
a material law. Then, for $\mu,\nu>\sbb{M}$ and $f\in\Lnu(\R;H)\cap\Lm{\mu}(\R;H)$,
we have 
\[
M(\td{\nu})f=M(\td{\mu})f.
\]
Moreover, $M(\td{\nu})$ is causal for all $\nu>\sbb{M}$. 
\end{thm}

\begin{proof}
Let $\mu<\nu$. We prove the assertion for $f=\1_{\ci{a}{b}}\cdot x$
with $a<b$ and $x\in H$ first. For $\rho\in\R$ we compute 
\[
\left(\mathcal{L}_{\rho}f\right)(t)=\frac{1}{\sqrt{2\pi}}\int_{a}^{b}x\e^{-(\i t+\rho)s}\d s=\frac{1}{\sqrt{2\pi}}\frac{1}{\i t+\rho}\left(\e^{-(\i t+\rho)a}-\e^{-(\i t+\rho)b}\right)x\quad(t\in\R\setminus\{0\}).
\]
Moreover, we define 
\[
g(z)\coloneqq\frac{1}{\sqrt{2\pi}}M(z)x\frac{1}{z}\left(\e^{-za}-\e^{-zb}\right)\quad(z\in\C_{\Re\geq \mu}\setminus\{0\})
\]
and prove that $g$ satisfies the assumptions of \prettyref{lem:F-L idenpendent of nu}.
First, we note that $g$ is bounded on $\set{z\in \C}{\mu\leq \Re z\leq \nu}\setminus\{0\}$.
Indeed, we only need to prove that it is bounded near $0$ provided
that $\mu\leq0.$ To that end, we observe 
\[
\frac{1}{z}(\e^{-za}-\e^{-zb})=\e^{-za}\frac{1-\e^{-z(b-a)}}{z}\to b-a\quad(z\to0).
\]
Thus, $g$ is bounded near $0$. In particular, $z=0$ is
a removable singularity and, hence, $g$ can be extended holomorphically
to $\C_{\Re\geq\mu}$. Moreover, for $\rho\geq\mu$ we have that 
\[
\int_{\R}\norm{g(\i t+\rho)}^{2}\d t=\int_{-1}^{1}\norm{g(\i t+\rho)}^{2}\d t+\int_{|t|>1}\norm{g(\i t+\rho)}^{2}\d t.
\]
The first term on the right-hand side is finite since $g$ is bounded,
while the second term can be estimated by 
\[
\int_{|t|>1}\norm{g(\i t+\rho)}^{2}\d t \leq \norm{M}_{\infty,\C_{\Re>\mu}}^{2}\norm{x}^{2}\frac{(\e^{-\rho a}+\e^{-\rho b})^{2}}{2\pi}\int_{|t|>1}\frac{1}{t^{2}+\rho^{2}}\d t<\infty.
\]
This proves that $g(\i\cdot+\rho)\in\L(\R;H)$ for each $\rho\geq\mu$
and hence, particularly for $\rho=\mu$ and $\rho=\nu$. Finally,
for $\rho\geq\mu$ we have that 
\[
\norm{g(\i t+\rho)}\leq \frac{1}{\sqrt{2\pi}}\norm{M}_{\infty,\C_{\Re>\mu}}\norm{x}\frac{1}{\sqrt{t{^{2}}+\rho{^{2}}}}\left(\e^{-\rho a}+\e^{-\rho b}\right)\to0\quad\left(\abs{t}\to\infty\right),
\]
which together with the boundedness of $g$ yields \prettyref{eq:integral_to_zero}
by dominated convergence. This shows that $g$ satisfies the assumptions
of \prettyref{lem:F-L idenpendent of nu} and thus 
\[
M(\td{\nu})f=\mathcal{L}_{\nu}^{\ast}g(\i\cdot+\nu)=\mathcal{L}_{\mu}^{\ast}g(\i\cdot+\mu)=M(\td{\mu})f.
\]
By linearity, this equality extends to $S_{\rmc}(\R;H)$ and so, 
\[
F\from S_{\rmc}(\R;H)\to\bigcap_{\nu\geq\mu}\Lnu(\R;H),\;f\mapsto M(\td{\nu})f
\]
is well-defined. Moreover, $F$ is uniformly Lipschitz continuous
(with $\sup_{\nu\geq\mu}\norm{F^{\nu}}=\norm{M}_{\infty,\C_{\Re>\mu}}$)
and hence, the assertions follow from \prettyref{lem:indepCaus}. 
\end{proof}

\section{Comments}

The Fourier and the Fourier--Laplace transformation introduced in this
chapter are used to define an operator-valued functional calculus
for the time derivative, $\td{\nu}$.  This functional calculus can be defined since the Fourier--Laplace transformation provides the unitary transformation yielding the spectral representation of the time derivative as multiplication operator. This fact was already noticed in \cite{Picard1989}, which eventually led to evolutionary equations in \cite{PicPhy}.

We emphasise that we have used the fundamental
property that both $\mathcal{F}$ and $\mathcal{L}_{\nu}$ are unitary.
It is noteworthy that the Fourier transformation
is an isometric isomorphism on $L_{2}(\R;X)$ if and only if $X$ is a Hilbert space, see
\cite{Kwapien1972}. In the Banach space-valued case one has to further
restrict the class of functions used to define a functional calculus.
For the topic of functional calculus we refer to the 21st ISem \cite{ISEM21}
by Markus Haase and to his monograph, \cite{Haase2006}.

\section*{Exercises}
\addcontentsline{toc}{section}{Exercises}

\begin{xca}
\label{exer:derivative-under-integral} Let $(\Omega,\Sigma,\mu)$ be
a measure space, $X$ a Banach space and $I\subseteq\R$ an open interval.
Let $g\from I\times\Omega\to X$ such that $g(t,\cdot)\in L_{1}(\mu;X)$
for each $t\in I$, and define 
\begin{align*}
h & \from I\to X,\;t\mapsto\int_{\Omega}g(t,\omega)\d\mu(\omega).
\end{align*}
\begin{enumerate}
\item Assume that $g(\cdot,\omega)$ is continuous for $\mu$-almost every
$\omega\in\Omega$ and let $f\in L_{1}(\mu)$ such that 
\[
\norm{g(t,\omega)}\leq f(\omega)\quad(t\in I,\omega\in\Omega).
\]
Prove that $h$ is continuous.
\item Assume that $g(\cdot,\omega)$ is differentiable for $\mu$-almost
every $\omega\in\Omega$ and let $f\in L_{1}(\mu)$ such that 
\[
\norm{\partial_{t}g(t,\omega)}\leq f(\omega)\quad(t\in I,\omega\in\Omega).
\]
Prove that $h$ is differentiable with 
\[
h'(t)=\int_{\Omega}\partial_{t}g(t,\omega)\d\mu(\omega).
\]
\end{enumerate}
\end{xca}

\begin{xca}
\label{exer:unitary_extension}Let $H_{0},H_{1}$ be two Hilbert spaces
and $U\from \dom(U)\subseteq H_{0}\to H_{1}$ linear such that
\begin{itemize}
\item $\dom(U)$ is dense in $H_{0}$ and $\ran(U)$ is dense in $H_{1}$.
\item $\forall x\in\dom(U):\:\norm{Ux}_{H_{1}}=\norm{x}_{H_{0}}$.
\end{itemize}
Show that $U$ can be uniquely extended to a unitary operator between
$H_{0}$ and $H_{1}$. 
\end{xca}

\begin{xca}
\label{exer:holomorphic} Let $\Omega\subseteq\C$ be open, $X$ a complex
Banach space and $f\from \Omega\to X$. Prove that the following statements
are equivalent:
\begin{enumerate}
\renewcommand{\labelenumi}{{\upshape (\roman{enumi})}}
\item\label{exer:holomorphic:item:1} $f$ is holomorphic.
\item For all $x'\in X'$ the mapping $x'\circ f\from\Omega\to\C$ is holomorphic.
\item $f$ is locally bounded and $x'\circ f\from\Omega\to\C$ is holomorphic
for all $x'\in D$, where $D\subseteq X'$ is a norming set\footnote{$D\subseteq X'$ is called a norming set for $X$, if $\norm{x}=\sup_{x'\in D\setminus\{0\}}\frac{1}{\norm{x'}}\abs{x'(x)}$
for each $x\in X$. Note that $X'$ is norming for $X$ by the Hahn--Banach
theorem.} for $X$.
\item\label{exer:holomorphic:item:4} $f$ is analytic, i.e.~for each $z_{0}\in\Omega$ there is $r>0$
and $(a_{n})_{n}$ in $X$ with $\ball{z_{0}}{r}\subseteq\Omega$
and 
\[
f(z)=\sum_{n=0}^{\infty}a_{n}\left(z-z_{0}\right)^{n}\quad(z\in\ball{z_{0}}{r}).
\]
\end{enumerate}
Assume now that $X=L(X_{1},X_{2})$ for two complex Banach spaces
$X_{1},X_{2}$, let $D_{1}\subseteq X_{1}$ be dense and $D_{2}\subseteq X_{2}'$
norming for $X_{2}$. Prove that the statements 
\renewcommand{\theenumi}{{\upshape (\roman{enumi})}}
\ref{exer:holomorphic:item:1} to \ref{exer:holomorphic:item:4} are equivalent
\renewcommand{\theenumi}{{\upshape (\alph{enumi})}}
to
\begin{enumerate}
\renewcommand{\labelenumi}{{\upshape (\roman{enumi})}}
\addtocounter{enumi}{4}
\item $f$ is locally bounded and $\Omega\ni z\mapsto x_{2}'(f(z)(x_{1}))\in\C$
is holomorphic for all $x_1\in D_1$ and $x_2'\in D_2$.
\end{enumerate}
\end{xca}

\begin{xca}
\label{exer:convolution} Let $\nu\in\R$ and $k\in L_{1,\nu}(\R)$.
Prove that 
\[
\mathcal{L}_{\nu}\left(k\ast f\right)=\sqrt{2\pi}\left(\mathcal{L}_{\nu}k\right)\cdot\left(\mathcal{L}_{\nu}f\right)
\]
for $f\in\Lnu(\R;H)$. 
\end{xca}

\begin{xca}
\label{exer:fractional} Let $\alpha>0$ and define $g_{\alpha}(t)\coloneqq\1_{\roi{0}{\infty}}(t)t^{\alpha-1}$
for $t\in\R.$ Show that $g_{\alpha}\in L_{1,\nu}(\R)$ for each $\nu>0$
and that 
\[
\left(\mathcal{L}_{\nu}g_{\alpha}\right)(t)=\frac{1}{\sqrt{2\pi}}\Gamma(\alpha)(\i t+\nu)^{-\alpha}.
\]
Use this formula and \prettyref{exer:convolution} to derive \prettyref{eq:fractional_int}.

Hint: To compute the Fourier--Laplace transform of $g_{\alpha}$, derive
that $\mathcal{L}_{\nu}g_{\alpha}$ solves a first order ordinary
differential equation and use separation of variables to solve this
equation. 
\end{xca}

\begin{xca}
\label{exer:analytic_on_strip} Let $\mu,\nu\in\R$ with $\mu<\nu$ and $f\in\Lnu(\R;H)\cap\Lm{\mu}(\R;H)$.
Moreover, set $U\coloneqq\set{z\in\C}{\mu<\Re z<\nu}$. Show that $f\in\bigcap_{\mu<\rho<\nu}\Lm{\rho}(\R;H)\cap L_{1,\rho}(\R;H)$
and that 
\[
U\ni z\mapsto\left(\mathcal{L}_{\Re z}f\right)(\Im z)
\]
is holomorphic. 
\end{xca}

\begin{xca}
\label{exer:causality_autonomous} Let $H_{0},H_{1}$ be Hilbert spaces
and $T\from \Lnu(\R;H_{0})\to\Lnu(\R;H_{1})$ linear and bounded. We call
$T$ \emph{\index{autonomous}autonomous} if $T\tau_{h}=\tau_{h}T$ for each $h\in\R$ ($\tau_{h}$
denotes the translation operator defined in \prettyref{exa:material_law_revisited}).
Prove that for autonomous $T$, the following statements are equivalent:
\begin{enumerate}
\renewcommand{\labelenumi}{{\upshape (\roman{enumi})}}
\item $T$ is causal.
\item For all $f\in\Lnu(\R;H_{0})$ with $\spt f\subseteq\roi{0}{\infty}$
one has $\spt Tf\subseteq\roi{0}{\infty}$.
\end{enumerate}
Moreover, prove that for a material law $M$, the operator $M(\td{\nu})$
is autonomous for each $\nu>\sbb{M}$. 
\end{xca}

\printbibliography[heading=subbibliography]

\chapter{Solution Theory for Evolutionary Equations}
\label{chap:Solution_Theory}

In this chapter, we shall discuss and present the first major result
of this year's internet seminar: Picard's theorem on the solution theory for evolutionary equations which is
the main result of \cite{PicPhy}. In order to stress the applicability
of this theorem, we shall deal with applications first and provide
a proof of the actual result afterwards. With an initial interest in applications in mind, we start off with the introduction of some
vector-analytic operators.

\section{First Order Sobolev Spaces\label{sec:First-order-Sobolev}}

Throughout this section let $\Omega\subseteq\R^{d}$ be an open set.
\begin{defn*}
We define 
\begin{align*}
\grad_{\rmc}\colon \cci(\Omega)\subseteq\L(\Omega) & \to\L(\Omega)^{d}\\
\phi & \mapsto\left(\partial_{j}\phi\right)_{j\in\{1,\ldots,d\}},
\end{align*}
\begin{align*}
\dive_{\rmc}\colon\cci(\Omega)^{d}\subseteq\L(\Omega)^{d} & \to\L(\Omega)\\
\left(\phi_{j}\right)_{j\in\{1,\ldots,d\}} & \mapsto\sum_{j\in\{1,\ldots,d\}}\partial_{j}\phi_{j},
\end{align*}
and  if $d=3$,
\begin{align*}
\curl_{\rmc}\colon\cci(\Omega)^{3}\subseteq\L(\Omega)^{3} & \to\L(\Omega)^{3}\\
\left(\phi_{j}\right)_{j\in\{1,2,3\}} & \mapsto\begin{pmatrix}
\partial_{2}\phi_{3}-\partial_{3}\phi_{2}\\
\partial_{3}\phi_{1}-\partial_{1}\phi_{3}\\
\partial_{1}\phi_{2}-\partial_{2}\phi_{1}
\end{pmatrix}.
\end{align*}
Furthermore, we put
\[
\dive\coloneqq-\grad_{\rmc}^{*},\quad\grad\coloneqq-\dive_{\rmc}^{*},\quad\curl\coloneqq\curl_{\rmc}^{*}
\]
and
\[
\dive_{0}\coloneqq-\grad^{*},\quad\grad_{0}\coloneqq-\dive^{*},\quad\curl_{0}\coloneqq\curl^{*}.
\]
\end{defn*}
\begin{prop}
The relations $\dive,\dive_{0},\grad,\grad_{0},\curl$ and $\curl_{0}$ are
all densely defined, closed linear operators.
\end{prop}

\begin{proof}
The operators $\grad_{\rmc},\dive_{\rmc}$ and $\curl_{\rmc}$ are densely
defined by \prettyref{exer:Ccinfty_dense_hd}. Thus, $\dive,\grad$ and $\curl$
are closed linear operators by \prettyref{lem:ddc}. Moreover, it follows from integration
by parts that $\grad_{\rmc}\subseteq\grad$, $\dive_{\rmc}\subseteq\dive$
and $\curl_{\rmc}\subseteq\curl$. Thus, $\dive,\grad$ and $\curl$ are also
densely defined. This, in turn, implies that $\grad_{\rmc},\dive_{\rmc}$
and $\curl_{\rmc}$ are closable by \prettyref{lem:ddc} with respective closures $\grad_{0},\dive_{0}$ and $\curl_{0}$ by \prettyref{lem:ass}. 
\end{proof}
We shall describe the domains of these operators in
more detail in the next theorem.
\begin{thm}
\label{thm:Sobolev_space} If $f\in\L(\Omega)$ and $g=(g_{j})_{j\in\{1,\ldots,d\}}\in\L(\Omega)^{d}$ then the following statements hold:
\begin{enumerate}
\item $f\in\dom(\grad)$ and $g=\grad f$ if and only if
\[
\forall\phi\in\cci(\Omega),j\in\{1,\ldots,d\}\colon-\int_{\Omega}f\partial_{j}\phi=\int_{\Omega}g_{j}\phi.
\]
\item $f\in\dom(\grad_{0})$ and $g=\grad_{0}f$ if and only if there exists $(f_{k})_{k}$ in $\cci(\Omega)$ such that $f_{k}\to f$
in $\L(\Omega)$ and $\grad f_{k}\to g$ in $\L(\Omega)^{d}$ as $k\to \infty$.
\item $g\in\dom(\dive)$ and $f=\dive g$ if and only if
\[
\forall\phi\in\cci(\Omega)\colon-\int_{\Omega}g\cdot\grad\phi=\int_{\Omega}f\phi.
\]
\item $g\in\dom(\dive_{0})$ and $f=\dive_{0}g$ if and only if there exists $(g_{k})_{k}$ in $\cci(\Omega)^{d}$ such that $g_{k}\to g$
in $\L(\Omega)^{d}$ and $\dive g_{k}\to f$ in $\L(\Omega)$ as $k\to \infty$.
\end{enumerate}
If $d=3$ and $f,g\in \L(\Omega)^3$ then the following statements hold:
\begin{enumerate}
\addtocounter{enumi}{4}
\item $f\in\dom(\curl)$ and $g=\curl f$ if and only if
\[
\forall\phi\in\cci(\Omega)^{3}\colon\int_{\Omega}f\cdot \curl\phi=\int_{\Omega}g\cdot \phi.
\]
\item $f\in\dom(\curl_{0})$ and $g=\curl_{0}f$ if and only if there exists $(f_{k})_{k}$ in $\cci(\Omega)^3$ such that $f_{k}\to f$
in $\L(\Omega)^{3}$ and $\curl f_{k}\to g$ in $\L(\Omega)^{3}$ as $k\to \infty$.
\end{enumerate}
\end{thm}

All the statements in \prettyref{thm:Sobolev_space} are elementary
consequences of the integration by parts formula and the definitions
of the adjoint. We ask the reader to prove these statements in \prettyref{exer:Sobolev_elem}.
\begin{rem}
We remark here that, classically, the following notation has been introduced: \index{H(div,Omega), H_0(div,Omega)@$H(\dive,\Omega), H_0(\dive,\Omega)$}\index{H(curl,Omega), H_0(curl,Omega)@$H(\curl,\Omega),H_0(\curl,\Omega)$}\index{H^1(Omega), H^1_0(Omega)@$H^1(\Omega),H_0^1(\Omega)$}
\begin{align*}
H^{1}(\Omega) & \coloneqq\dom(\grad),\\
H_{0}^{1}(\Omega) & \coloneqq\dom(\grad_{0}),\\
H(\dive,\Omega) & \coloneqq\dom(\dive),\\
H(\curl,\Omega) & \coloneqq\dom(\curl).
\end{align*}
Following the rationale of appending zero as an index for $H_{0}^{1}(\Omega)$,
we shall also use
\begin{align*}
H_{0}(\dive,\Omega) & \coloneqq\dom(\dive_{0}),\\
H_{0}(\curl,\Omega) & \coloneqq\dom(\curl_{0}).
\end{align*}
We do, however, caution the reader that other authors also use $H_{0}(\dive,\Omega)$
and $H_{0}(\curl,\Omega)$ to denote the kernel of $\dive$ and $\curl$.
All the spaces just defined are so-called Sobolev spaces\index{Sobolev space}. We note
that for $d=3$ we clearly have $H^{1}(\Omega)^3\subseteq H(\dive,\Omega)\cap H(\curl,\Omega)$.
On the other hand, note that $H(\dive,\Omega)$ is neither a sub-
nor a superset of $H(\curl,\Omega)$.
\end{rem}

\begin{rem}
  We emphasise that $H_0^1(\Omega)\subseteq H^1(\Omega)$ is a proper inclusion for many open $\Omega$. The `$0$' in the index is a reminder of `$0$'-boundary conditions. In fact, the only difference between these two spaces lies in the behaviour of their elements at the boundary of $\Omega$. The space $H_0^1$ signifies all $H^1$-functions vanishing at $\partial\Omega$ in a generalised sense. The corresponding statements are true for the inclusions $H_0(\dive,\Omega)\subseteq H(\dive,\Omega)$ and $H_0(\curl,\Omega)\subseteq H(\curl,\Omega)$. The space $H_0(\dive,\Omega)$ describes $H(\dive,\Omega)$-vector fields with vanishing normal component and to lie in $H_0(\curl,\Omega)$ provides a handy generalisation of vanishing tangential component. We will anticipate these abstractions, when we apply the solution theory of evolutionary equations for particular cases. In a later chapter we will come back to this issue when we discuss inhomogeneous boundary value problems.
\end{rem}

For later use, we record the following relationships between the vector-analytical
operators introduced above.

\begin{prop}
\label{prop:ker_ran_vec_ana} Let $d=3$. We have the following inclusions:
\begin{align*}
\cran(\curl_{0}) & \subseteq\ker(\dive_{0}),\\
\cran(\grad_{0}) & \subseteq\ker(\curl_{0}),\\
\cran(\curl) & \subseteq\ker(\dive),\\
\cran(\grad) & \subseteq\ker(\curl).
\end{align*}
\end{prop}

\begin{proof}
It is elementary to show that for given $\psi\in\cci(\Omega)^{3}$ and
$\phi\in\cci(\Omega)$ we have $\dive_{0}\curl_{0}\psi=0$ as well
as $\curl_{0}\grad_{0}\phi=0$. Thus, we obtain $\ran(\curl_{\rmc})\subseteq\ker(\dive_{0})$ and $\ran(\grad_{\rmc})\subseteq\ker(\curl_{0})$.
Since $\ker(\dive_{0})$ and $\ker(\curl_{0})$ are closed, and $\cci(\Omega)^{3}$
and $\cci(\Omega)$  are cores for $\curl_{0}$ and $\grad_{0}$ respectively, we
obtain the first two inclusions. The last two inclusions follow from
the first two 
by taking into account the orthogonal decompositions
\[
\L(\Omega)^{3} =\cran(\grad)\oplus\ker(\dive_{0})=\ker(\curl)\oplus\cran(\curl_{0})
\]
and
\[
\L(\Omega)^{3} =\cran(\grad_{0})\oplus\ker(\dive)=\ker(\curl_{0})\oplus\cran(\curl)
\]
which follow from \prettyref{cor:abstrHelmH}.
\end{proof}

\section{Well-Posedness of Evolutionary Equations and Applications}

The solution theory of evolutionary equations\index{Solution Theory, evolutionary equations} is contained in the
next result, Picard's theorem. This result is central for all the derivations to come.
In fact, with the notation of \prettyref{thm:Solution_theory_EE},
we shall prove that for all (well-behaved) $F$ there is a unique solution
of
\[
\left(\td{\nu}M(\td{\nu})+A\right)U=F.
\]
The solution $U$ depends continuously and causally on the choice
of $F$. 

In order to formulate the result, for a Hilbert space $H$, $\nu\in\R$ and a given operator $A\from\dom(A)\subseteq H\to H$ we define its extended operator in $\Lnu(\R;H)$, again denoted by $A$, by 
\begin{align*}
\Lnu(\R;\dom(A))\subseteq\Lnu(\R;H) & \to\Lnu(\R;H)\\
f & \mapsto\bigl(t\mapsto Af(t)\bigr).
\end{align*}
We have collected some properties of extended operators in
\prettyref{exer:tensorprod} and \prettyref{exer:abstractFourier_exten}.

\begin{thm}[Picard]
\label{thm:Solution_theory_EE} Let $\nu_{0}\in\R$ and $H$ be a Hilbert
space. Let $M\colon\dom(M)\subseteq\C\to\bo(H)$ be a material law
with $\sbb{M}<\nu_{0}$ and let $A\colon\dom(A)\subseteq H\to H$
be skew-selfadjoint. Assume that
\[
\Re\scp{\phi}{zM(z)\phi}_H\geq c\norm{\phi}_H^{2}\quad(\phi\in H,z\in\C_{\Re\geq\nu_{0}})
\]
for some $c>0$. Then for all $\nu\geq\nu_{0}$ the operator $\td{\nu}M(\td{\nu})+A$ is closable and
\[
S_{\nu}\coloneqq\bigl(\overline{\td{\nu}M(\td{\nu})+A}\bigr)^{-1}\in\bo(\Lnu(\R;H)).
\]
Furthermore, $S_\nu$ is causal and satisfies $\norm{S_{\nu}}_{\bo(\Lnu)}\leq1/c$, and for all $F\in\dom(\td{\nu})$ we
have
\[
S_{\nu}F\in\dom(\td{\nu})\cap\dom(A).
\]
Furthermore, for $\eta,\nu\geq\nu_{0}$ and $F\in\Lnu(\R;H)\cap\Lm{\eta}(\R;H)$ we have that $S_{\nu}F=S_{\eta}F$.
\end{thm}

The property that $S_{\nu}F=S_{\eta}F$ for all $F\in\Lnu(\R;H)\cap\Lm{\eta}(\R;H)$
where $\eta,\nu\geq\nu_{0}$, for some $\nu_{0}\in\R$, will be referred
to as $S_{\nu}$ being \emph{eventually independent of $\nu$} in what follows.

\begin{rem}\label{rem:regularity}
 Recall that $F\in \dom(\td{\nu})$ implies $U\coloneqq S_\nu F\in \dom(\td{\nu})\cap \dom(A)$ in \prettyref{thm:Solution_theory_EE}. Since $M(\td{\nu})$ leaves the space $\dom(\td{\nu})$ invariant, this gives that $M(\td{\nu})U\in \dom(\td{\nu})$ and thus, $U$ solves the evolutionary equation literally; that is,
 \[
  (\td{\nu}M(\td{\nu})+A)U=F,
 \]
while for $F\in \Lnu(\R;H)$, in general, we just have
\[
 \bigl(\overline{\td{\nu}M(\td{\nu})+A}\bigr)U=F.
\]
\end{rem}

\begin{defn*}
 Let $H$ be a Hilbert space and $T\in L(H)$. If $T$ is selfadjoint, we write $T\geq c$ for some $c\in \R$ if 
\[
 \forall x\in H:\, \scp{x}{Tx}_H\geq c\norm{x}_H^2.
\]
Moreover, we define the \emph{real part of} $T$ \index{Real part of operator} by $\Re T\coloneqq \frac{1}{2}(T+T^\ast)$.
\end{defn*}

Note that if $H$ is a Hilbert space and $T\in L(H)$ then $\Re T$ is selfadjoint. Moreover,
\[\scp{x}{(\Re T)x}_H = \Re \scp{x}{Tx}_H \quad(x\in H).\]
Hence, in \prettyref{thm:Solution_theory_EE} the assumption on the material law can be rephrased as 
\[\Re zM(z)\geq c \quad(z\in\C_{\Re\geq\nu_{0}}).\]
The following operators will be prototypical examples needed for the applications
of the previous theorem.

\begin{prop}
\label{prop:block_op_realinv}Let $H_{0},H_{1}$ be Hilbert spaces.
\begin{enumerate}
\item\label{prop:block_op_realinv:item:1} Let $B\colon\dom(B)\subseteq H_{0}\to H_{1}$, $C\colon\dom(C)\subseteq H_{1}\to H_{0}$ be
densely defined linear operators. Then
\begin{align*}
\begin{pmatrix}
0 & C\\
B & 0
\end{pmatrix}\colon\dom(B)\times\dom(C)\subseteq H_{0}\times H_{1} &\to H_{0}\times H_{1}\\
(\phi,\psi) & \mapsto(C\psi,B\phi)
\end{align*}
is densely defined, and we have
\[
\begin{pmatrix}
0 & C\\
B & 0
\end{pmatrix}^{*}=\begin{pmatrix}
0 & B^{*}\\
C^{*} & 0
\end{pmatrix}.
\]
\item\label{prop:block_op_realinv:item:2} Let $a\in\bo(H_{0})$, and  $c>0$. Assume $\Re a\geq c$.
Then $a^{-1}\in\bo(H_{0})$ with $\norm{a^{-1}}\leq \frac{1}{c}$ and $\Re a^{-1}\geq c\norm{a}^{-2}$.
\end{enumerate}
\end{prop}

\begin{proof}
The proof of the first statement can be done in two steps. First, notice that the inclusion
$\begin{pmatrix}
0 & B^{*}\\
C^{*} & 0
\end{pmatrix}\subseteq\begin{pmatrix}
0 & C\\
B & 0
\end{pmatrix}^{*}$ follows immediately. If, on the other hand, $\begin{pmatrix}\phi\\\psi\end{pmatrix}\in\dom\left(\begin{pmatrix}
 0  &  C\\
 B  &  0 
\end{pmatrix}^{*}\right)$ with $\begin{pmatrix}
0 & C\\
B & 0
\end{pmatrix}^{*}\begin{pmatrix}\phi\\\psi\end{pmatrix}=\begin{pmatrix}\xi\\\zeta\end{pmatrix}$ we get for all $x\in\dom(B)$ that
\begin{align*}
\scp{Bx}{\psi}_{H_1} & =\scp{\begin{pmatrix}
0 & C\\
B & 0
\end{pmatrix}\begin{pmatrix}
x\\
0
\end{pmatrix}}{\begin{pmatrix}
\phi\\
\psi
\end{pmatrix}}_{H_0\times H_1}
=\scp{\begin{pmatrix}
x\\
0
\end{pmatrix}}{\begin{pmatrix}
0 & C\\
B & 0
\end{pmatrix}^{*}\begin{pmatrix}
\phi\\
\psi
\end{pmatrix}}_{H_0\times H_1} \\
& =\scp{\begin{pmatrix}
x\\
0
\end{pmatrix}}{\begin{pmatrix}
\xi\\
\zeta
\end{pmatrix}}_{H_0\times H_1}=\scp{x}{\xi}_{H_0}.
\end{align*}
Hence, $\psi\in\dom(B^{*})$ and $B^{*}\psi=\xi$. Similarly, we obtain
$\phi\in\dom(C^{*})$ and $C^{*}\phi=\zeta$.

For the second statement, we compute for all $\phi\in H_{0}$ using
the Cauchy--Schwarz inequality
\[
\norm{\phi}\norm{a\phi}_{H_0}\geq\abs{\scp{\phi}{a\phi}}_{H_0}\geq\Re\scp{\phi}{a\phi}_{H_0}\geq c\scp{\phi}{\phi}_{H_0} = c\norm{\phi}_{H_0}^2.
\]
Thus, $a$ is one-to-one. Since $\Re a=\Re a^{*}$ it follows that
$a^{*}$ is one-to-one, as well. Thus, we get that $a$ has dense
range by \prettyref{thm:ran-kernerl}. The inequality 
\[
\norm{a\phi}_{H_0}\geq c\norm{\phi}_{H_0}
\]
implies that $a^{-1}$ is bounded with $\norm{a^{-1}}\leq \frac{1}{c}$. Hence, as $a^{-1}$ is closed, $\dom(a^{-1})=\ran(a)$ is closed by \prettyref{lem:bbc} and hence, $\dom(a^{-1})=H_0$; that is, $a^{-1}\in\bo(H_{0})$.
To conclude, let $\psi\in H_{0}$ and put $\phi\coloneqq a^{-1}\psi$.
Then $\norm{\psi}_{H_0}=\norm{aa^{-1}\psi}_{H_0}\leq\norm{a}\norm{a^{-1}\psi}_{H_0}$
and so
\begin{align*}
\Re\scp{\psi}{a^{-1}\psi}_{H_0} & = \Re\scp{a\phi}{\phi}_{H_0}
= \Re\scp{\phi}{a\phi}_{H_0}
 \geq c\scp{\phi}{\phi}_{H_0}
 = c\scp{a^{-1}\psi}{a^{-1}\psi}_{H_0}\\
 &  \geq c\frac{1}{\norm{a}^{2}}\norm{\psi}_{H_0}^2.\tag*{{\qedhere}}
\end{align*}
\end{proof}

\subsection*{The Heat Equation}

\index{heat equation, evolutionary equation}
The first example we will consider is the heat equation in an open subset $\Omega\subseteq\R^d$. Under a heat source,
$Q\colon\R\times\Omega\to\R$, the heat distribution, $\theta\colon\R\times\Omega\to\R$,
satisfies the so-called heat-flux-balance\index{heat-flux-balance}
\[
\partial_{t}\theta+\dive q=Q.
\]
Here, $q\colon\R\times\Omega\to\R^{d}$ is the heat flux\index{heat flux} which is connected
to $\theta$ via Fourier's law\index{Fourier's law}
\[
q=-a\grad\theta,
\]
where $a\colon\Omega\to\R^{d\times d}$ is the heat conductivity,
which is measurable, bounded and uniformly strictly positive in the
sense that 
\[
\Re a(x)\geq c
\]
for all $x\in\Omega$ and some $c>0$ in the sense of positive definiteness. Moreover,
we assume that $\Omega$ is thermally isolated, which is modelled
by requiring that the normal component of $q$ vanishes at $\partial\Omega$; that
is, $q\in\dom(\dive_{0})$. Written as a block matrix and incorporating
the boundary condition, we obtain 
\[
\left(\partial_{t}\begin{pmatrix}
1 & 0\\
0 & 0
\end{pmatrix}+\begin{pmatrix}
0 & 0\\
0 & a^{-1}
\end{pmatrix}+\begin{pmatrix}
0 & \dive_{0}\\
\grad & 0
\end{pmatrix}\right)\begin{pmatrix}
\theta\\
q
\end{pmatrix}=\begin{pmatrix}
Q\\
0
\end{pmatrix}.
\]

\begin{thm}
\label{thm:wp_heat}For all $\nu>0$, the operator 
\[
\td{\nu}\begin{pmatrix}
1 & 0\\
0 & 0
\end{pmatrix}+\begin{pmatrix}
0 & 0\\
0 & a^{-1}
\end{pmatrix}+\begin{pmatrix}
0 & \dive_{0}\\
\grad & 0
\end{pmatrix}
\]
is densely defined and closable in $\Lnu\left(\R;\L(\Omega)\times\L(\Omega)^{d}\right)$.
The respective closure is continuously invertible with causal inverse
being eventually independent of $\nu$. 
\end{thm}

\begin{proof}
The assertion follows from \prettyref{thm:Solution_theory_EE} applied
to 
\[
M(z)=\begin{pmatrix}
1 & 0\\
0 & 0
\end{pmatrix}+z^{-1}\begin{pmatrix}
0 & 0\\
0 & a^{-1}
\end{pmatrix} \quad\text{and} \quad 
A=\begin{pmatrix}
0 & \dive_{0}\\
\grad & 0
\end{pmatrix}.
\]
Note that $M$ is a material law with $\sbb{M}=0$ by \prettyref{exa:material-laws}. Moreover, for $(x,y)\in \L(\Omega)\times\L(\Omega)^d$ and $z\in \C_{\Re\geq\nu}$ with $\nu>0$ we estimate 
\begin{align*}
 \Re \scp{(x,y)}{zM(z)(x,y)}_{\L(\Omega)\times\L(\Omega)^d} &\geq \Re z\norm{x}_{\L(\Omega)}^2+ c\norm{a}^{-2}\norm{y}_{\L(\Omega)^d}^2\\
 &\geq \min\{\nu ,c\norm{a}^{-2}\} \norm{(x,y)}^2_{\L(\Omega)\times \L(\Omega)^d}, 
\end{align*}
where we have used  \prettyref{prop:block_op_realinv}\ref{prop:block_op_realinv:item:2} in the first inequality.
Moreover, $A$ is skew-selfadjoint by \prettyref{prop:block_op_realinv}\ref{prop:block_op_realinv:item:1}.
\end{proof}

\begin{rem}
Assume that $Q\in\dom(\td{\nu})$. It then follows from \prettyref{thm:Solution_theory_EE}
that 
\begin{align}
\begin{pmatrix}
\theta\\
q
\end{pmatrix} & \coloneqq\overline{\left(\td{\nu}\begin{pmatrix}
1 & 0\\
0 & 0
\end{pmatrix}+\begin{pmatrix}
0 & 0\\
0 & a^{-1}
\end{pmatrix}+\begin{pmatrix}
0 & \dive_{0}\\
\grad & 0
\end{pmatrix}\right)}^{-1}\begin{pmatrix}
Q\\
0
\end{pmatrix}\nonumber \\
 & \in\dom\left(\td{\nu}\right)\cap\dom\left(\begin{pmatrix}
0 & \dive_{0}\\
\grad & 0
\end{pmatrix}\right).\label{eq:theta_reg}
\end{align}
Then, as in \prettyref{rem:regularity}, it follows that $\theta$ and $q$ satisfy the heat-flux-balance
and Fourier's law in the sense that $\theta\in\dom(\td{\nu})\cap\dom(\grad)$
and $q\in\dom(\dive_{0})$ and
\begin{align*}
\partial_{t}\theta+\dive_{0}q & =Q,\\
q & =-a\grad\theta.
\end{align*}
This regularity result is true even for $Q\in\Lnu(\R;\L(\Omega));$
see \cite{PTW16_MR}.
\end{rem}

\subsection*{The Scalar Wave Equation}

\index{wave equation, scalar, evolutionary equation}
The classical scalar wave equation in a medium $\Omega\subseteq\R^{d}$
(think, for instance, of a vibrating string ($d=1)$ or membrane ($d=2)$)
consists of the equation of the balance of momentum\index{balance of momentum} where the acceleration
of the (vertical) displacement, $u\colon\R\times\Omega\to\R$, is balanced
by external forces, $f\colon\R\times\Omega\to\R$, and the divergence
of the stress\index{stress}, $\sigma\colon\R\times\Omega\to\R^{d}$, in such a  way that
\[
\partial_{t}^{2}u-\dive\sigma=f.
\]
The stress is related to $u$ via the following so-called stress-strain
relation (here Hooke's law)\index{Hooke's law}
\[
\sigma=T\grad u,
\]
where the so-called \index{elasticity tensor}elasticity tensor, $T\colon\Omega\to\R^{d\times d}$, is bounded, measurable, and satisfies
\[
T(x)=T(x)^{*}\geq c
\]
for some $c>0$ uniformly in $x\in\Omega$. The quantity $\grad u$
is referred to as the strain. We think of $u$ as being fixed at $\partial\Omega$
(``clamped boundary condition'').\index{clamped boundary condition} This is modelled by $u\in\dom(\grad_{0})$. 

Using $v\coloneqq\partial_{t}u$ as an unknown, we can rewrite the
balance of momentum and Hooke's law as $2\times2$-block-operator
matrix equation
\[
\left(\partial_{t}\begin{pmatrix}
1 & 0\\
0 & T^{-1}
\end{pmatrix}-\begin{pmatrix}
0 & \dive\\
\grad_{0} & 0
\end{pmatrix}\right)\begin{pmatrix}
v\\
\sigma
\end{pmatrix}=\begin{pmatrix}
f\\
0
\end{pmatrix}.
\]
The solution theory of evolutionary equations for the wave equation
now reads as follows:
\begin{thm}
\label{thm:st_wave} Let $\Omega\subseteq\R^{d}$ be open, and $T$ as
indicated above. Then, for all $\nu>0$, 
\[
\td{\nu}\begin{pmatrix}
1 & 0\\
0 & T^{-1}
\end{pmatrix}-\begin{pmatrix}
0 & \dive\\
\grad_{0} & 0
\end{pmatrix}
\]
is densely defined and closable in $\Lnu\left(\R;\L(\Omega)\times\L(\Omega)^{d}\right)$.
The respective closure is continuously invertible with causal inverse
being eventually independent of $\nu$. 
\end{thm}

\begin{proof}
We apply \prettyref{thm:Solution_theory_EE} to $A=-\begin{pmatrix}
0 & \dive\\
\grad_{0} & 0
\end{pmatrix}$, which is skew-selfadjoint by \prettyref{prop:block_op_realinv}\ref{prop:block_op_realinv:item:1},
and $M(z)=\begin{pmatrix}
1 & 0\\
0 & T^{-1}
\end{pmatrix}$, which defines a material law with $\sbb{M}=-\infty$. The positive
definiteness constraint needed in \prettyref{thm:Solution_theory_EE}
is satisfied by \prettyref{prop:block_op_realinv}\ref{prop:block_op_realinv:item:2} on account of the selfadjointness
of $T$, which implies the same for $T^{-1}$. Indeed, for $\nu_0 >0$ and $z\in \C_{\Re\geq \nu_0}$ we estimate
\begin{align*}
 \Re \scp{(x,y)}{zM(z)(x,y)}_{\L(\Omega)\times \L(\Omega)^d} 
 &= \Re \scp{x}{z x}_{\L(\Omega)}+\Re \scp{y}{zT^{-1}y}_{\L(\Omega)^d}\\
 &\geq \nu_0 \norm{x}^2_{\L(\Omega)} + \nu_0 \frac{c}{\|T\|^2}\norm{y}_{\L(\Omega)^d}^2\\
 &\geq \nu_0 \min\{1,c/\|T\|^2\} \norm{(x,y)}_{\L(\Omega)\times \L(\Omega)^d}^2
\end{align*}
for each $(x,y)\in \L(\Omega)\times \L(\Omega)^d$, where we used the selfadjointness of $T^{-1}$ in the second line.
\end{proof}

\begin{rem}
Let $f\in\Lnu(\R;\L(\Omega))$, $\nu>0$, and define
\[
\begin{pmatrix}
u\\
\tilde{\sigma}
\end{pmatrix}=\left(\overline{\td{\nu}\begin{pmatrix}
1 & 0\\
0 & T^{-1}
\end{pmatrix}-\begin{pmatrix}
0 & \dive\\
\grad_{0} & 0
\end{pmatrix}}\right)^{-1}\begin{pmatrix}
\td{\nu}^{-1}f\\
0
\end{pmatrix}.
\]
By \prettyref{thm:Solution_theory_EE}, we obtain $\begin{pmatrix} u\\\tilde{\sigma}\end{pmatrix}\in\dom(\td{\nu})\cap\dom\left(\begin{pmatrix}
0 & \dive\\
\grad_{0} & 0
\end{pmatrix}\right).$ Hence, we have
\begin{align*}
\td{\nu}u-\dive\tilde{\sigma} & =\td{\nu}^{-1}f\\
\td{\nu}T^{-1}\tilde{\sigma} & =\grad_{0}u
\end{align*}
or
\begin{align*}
\td{\nu}u-\dive\tilde{\sigma} & =\td{\nu}^{-1}f\\
\tilde{\sigma} & =T\td{\nu}^{-1}\grad_{0}u.
\end{align*}
Thus, formally, after another time-differentiation and the setting of
$\sigma=\td{\nu}\tilde{\sigma}$ we obtain a solution of the wave equation,
$(u,\sigma)$. Notice, however, that differentiating $\dive\tilde{\sigma}$
cannot be done without any additional knowledge of the regularity
of $\tilde{\sigma}$. In fact, in order to arrive at the balance of
momentum equation, one would need to have $\dive\tilde{\sigma}\in\dom(\td{\nu})$.
However, one only has $\tilde{\sigma}\in\dom(\td{\nu})\cap\dom(\dive)$.
It is an elementary argument, see \cite[Lemma 4.6]{SW16_SD}, that
we in fact have $\dive\td{\nu}^{-1}=\overline{\td{\nu}^{-1}\dive},$
which suggests that, in general, $\dive\tilde{\sigma}\notin\dom(\td{\nu})$, see \prettyref{exer:commute_integral}.
\end{rem}

\subsection*{Maxwell's Equations}

\index{Maxwell's equations, evolutionary equation}
The final example in this lecture forms the archetypical evolutionary
equation -- Maxwell's equations in a medium $\Omega\subseteq\R^{3}$.
In order to see this (and to finally conclude the $2\times2$-block
matrix formulation historically due to the work of \cite{Minkowski1910,Schmidt1968,Leis1968}), we start out with \index{Faraday's law}Faraday's law of induction,
which relates the unknown electric field, $E\colon\R\times\Omega\to\R^{3}$,
to the magnetic induction, $B\colon\R\times\Omega\to\R^{3}$, via
\[
\partial_{t}B+\curl E=0.
\]
We assume that the medium is contained in a perfect conductor, which is reflected
in the so-called electric boundary condition\index{electric boundary condition} which asks for the vanishing of the
tangential component of $E$ at the boundary. This is modelled by $E\in\dom(\curl_{0})$.
The next constituent of Maxwell's equations is Ampere's law\index{Ampere's law}
\[
\partial_{t}D+J_{c}-\curl H=J_{0},
\]
which relates the unknown electric displacement\index{electric displacement}, $D\colon\R\times\Omega\to\R^{3},$
charge\index{charge}, $J_{c}\colon\R\times\Omega\to\R^{3}$, and magnetic field\index{magnetic field}, $H\colon\R\times\Omega\to\R^{3}$,
to the (given) external currents, \index{external current} $J_{0}\colon\R\times\Omega\to\R^{3}$.
Maxwell's equations are completed by constitutive relations specific
to each material at hand. Indeed, the (bounded, measurable) dielectricity\index{dielectricity $\varepsilon$},
$\varepsilon\colon\Omega\to\R^{3\times3}$, and the (bounded, measurable)
magnetic permeability\index{magnetic permeability $\mu$}, $\mu\colon\Omega\to\R^{3\times3}$, are symmetric matrix-valued
functions which couple the electric displacement \index{electric displacement} to the electric field \index{electric field}
and the magnetic field \index{magnetic field} to the magnetic induction \index{magnetic induction} via
\[
D=\varepsilon E,\text{ and } B=\mu H.
\]
Finally, Ohm's law \index{Ohm's law} relates the charge to the electric field via the
(bounded, measurable) electric conductivity\index{electric conductivity $\sigma$}, $\sigma\colon\Omega\to\R^{3\times3}$,
as
\[
J_{c}=\sigma E.
\]
All in all, in terms of $(E,H)$,  Maxwell's equations read 
\[
\left(\partial_{t}\begin{pmatrix}
\varepsilon & 0\\
0 & \mu
\end{pmatrix}+\begin{pmatrix}
\sigma & 0\\
0 & 0
\end{pmatrix}+\begin{pmatrix}
0 & -\curl\\
\curl_{0} & 0
\end{pmatrix}\right)\begin{pmatrix}
E\\
H
\end{pmatrix}=\begin{pmatrix}
J_{0}\\
0
\end{pmatrix}.
\]
For the time being, we shall assume that there exists $c>0$ and $\nu_{0}>0$
such that for all $\nu\geq \nu_{0}$ we have
\[
\nu\varepsilon(x)+\Re\sigma(x)\geq c,\quad\mu(x)\geq c\quad (x\in \Omega)
\]
in the sense of positive definiteness. Note that the latter condition
allows particularly for $\varepsilon=0$ on certain regions, if $\Re\sigma$
compensates for this. This situation is referred to as the eddy current
approximation in these regions. With the above preparations at hand, we may now formulate
the well-posedness result concerning Maxwell's equations.
\begin{thm}
\label{thm:st_maxwell} Let $\Omega\subseteq\R^{3}$ be open and $\nu\geq \nu_{0}$.
Then
\[
\td{\nu}\begin{pmatrix}
\varepsilon & 0\\
0 & \mu
\end{pmatrix}+\begin{pmatrix}
\sigma & 0\\
0 & 0
\end{pmatrix}+\begin{pmatrix}
0 & -\curl\\
\curl_{0} & 0
\end{pmatrix}
\]
is densely defined and closable in $\Lnu\left(\R;\L(\Omega)^{3}\times\L(\Omega)^{3}\right)$.
The respective closure is continuously invertible with causal inverse
being eventually independent of $\nu$. 
\end{thm}

\begin{proof}
The assertion follows from \prettyref{thm:Solution_theory_EE} applied
to the material law 
\[M(z)=\begin{pmatrix}
\varepsilon & 0\\
0 & \mu
\end{pmatrix}+z^{-1}\begin{pmatrix}
\sigma & 0\\
0 & 0
\end{pmatrix}\]
and the skew-selfadjoint operator 
\[A=\begin{pmatrix}
0 & -\curl\\
\curl_{0} & 0
\end{pmatrix}.\qedhere\]
\end{proof}
\begin{rem}
In the physics literature (see e.g.~\cite[Chapter 18]{Feynman1964}), Maxwell's equations are usually
complemented by Gauss' law,
\[
\dive_{0}B=0,
\]
as well as the introduction of the charge density, $\rho=\dive\varepsilon E$, and
the current, $J=J_{0}-J_{c}$, by the continuity equation
\[
\partial_{t}\rho=\dive J.
\]
We shall argue in the following that these equations are \emph{automatically}
satisfied if $(E,H)$ is a solution to Maxwell's equation. Indeed,
assuming $J_{0}\in\dom(\td{\nu})$, then, as a consequence of \prettyref{thm:Solution_theory_EE},
we have that 
\begin{align*}
\begin{pmatrix}
E\\
H
\end{pmatrix} & = \left(\overline{\td{\nu}\begin{pmatrix}
\varepsilon & 0\\
0 & \mu
\end{pmatrix}+\begin{pmatrix}
\sigma & 0\\
0 & 0
\end{pmatrix}+\begin{pmatrix}
0 & -\curl\\
\curl_{0} & 0
\end{pmatrix}}\right)^{-1}\begin{pmatrix}
J_{0}\\
0
\end{pmatrix}\\
 & \quad\in\dom\left(\td{\nu}\right)\cap\dom\left(\begin{pmatrix}
0 & -\curl\\
\curl_{0} & 0
\end{pmatrix}\right).
\end{align*}
Reformulating the latter equation yields
\begin{align*}
B & =\mu H=-\td{\nu}^{-1}\curl_{0}E,\\
\varepsilon E & =\td{\nu}^{-1}\left(-\sigma E+J_{0}+\curl H\right)=\td{\nu}^{-1}J+\td{\nu}^{-1}\curl H.
\end{align*}
Since $\curl_{0}E\in\ran(\curl_{0})$, we have by \prettyref{prop:interchange_integral}
that $\td{\nu}^{-1}\curl_{0}E\in\cran(\curl_{0})$. Thus, by \prettyref{prop:ker_ran_vec_ana},
we obtain
\[
\dive_{0}B=\dive_{0}\left(-\td{\nu}^{-1}\curl_{0}E\right)=0.
\]
Similarly, we deduce that 
\[
\rho=\dive\varepsilon E=\dive\td{\nu}^{-1}J.
\]
If, in addition, we have that $J\in\dom(\dive)$, we recover the continuity
equation. In general, the continuity equation is satisfied in the
integrated sense just derived.
\end{rem}

We shall keep the list of examples to that for now. In the course
of this internet seminar, we will see more (involved) examples. Furthermore,
we will study the boundary conditions more deeply and shall relate
the conditions introduced abstractly here to more classical formulations
involving trace spaces.

\section{Proof of Picard's Theorem}

In this section we shall prove the well-posedness theorem. For this,
we recall an elementary result from functional analysis. It is remindful
of the Lax--Milgram lemma.\index{Lax--Milgram lemma}
\begin{prop}
\label{prop:accr_inv}Let $H$ be a Hilbert space and $B\colon\dom(B)\subseteq H\to H$
densely defined and closed with $\dom(B)\supseteq\dom(B^{*})$. Assume
there exists $c>0$ such that 
\[
\Re\scp{\phi}{B\phi}_H\geq c\norm{\phi}_H^{2}\quad(\phi\in\dom(B)).
\]
Then $B^{-1}\in\bo(H)$ and $\norm{B^{-1}}\leq1/c$.
\end{prop}

\begin{proof}
The proof is a refinement of the argument in \prettyref{prop:block_op_realinv}.
In fact, the assumed inequality implies closedness of the range of $B$ as
well as continuous invertibility with $B^{-1}\colon\ran(B)\to H$.
The fact that $\ran(B)$ is dense in $H$ follows from the fact that
$\Re\scp{\phi}{B^{*}\phi}_H\geq c\norm{\phi}_H^{2}$ for all $\phi\in\dom(B^{*})\subseteq\dom(B)$
which, in turn, also follows from the assumed inequality.
\end{proof}
\begin{proof}[Proof of \prettyref{thm:Solution_theory_EE}]
Let $\nu \geq \nu_0$ and 
%
  $z\in\C_{\Re\geq\nu}$. Define $B(z)\coloneqq zM(z)+A$. Since
$M(z)\in\bo(H)$ it follows from \prettyref{thm:adj-sum} that $B(z)^{*}=\left(zM(z)\right)^{*}-A$
and $\dom(B(z))=\dom(B(z)^{*})=\dom(A)$. Moreover, for all $\phi\in\dom(A)$
we have
\[
\Re\scp{\phi}{B(z)\phi}_H=\Re\scp{\phi}{\left(zM(z)+A\right)\phi}_H=\Re\scp{\phi}{zM(z)\phi}_H\geq c\norm{\phi}_H^{2},
\]
due to the skew-selfadjointness of $A$. Thus, by \prettyref{prop:accr_inv} applied to $B(z)$ instead of
$B$, we deduce that 
\[
S\colon\C_{\Re\geq \nu}\ni z\mapsto B(z)^{-1}
\]
is bounded  and assumes values in $\bo(H)$ with norm bounded by $1/c$.
By \prettyref{exer:comp_ana}, we have that $S$ is holomorphic. Thus,
$S$ is a material law and $\norm{S(\td{\nu})}\leq1/c$ by \prettyref{prop:mat_law_function_of_td}.
Moreover, \prettyref{thm:material law independent of nu} implies
that $S(\td{\nu})$ is independent of $\nu$ and causal.

Next, if $f\in\dom(\td{\nu}),$ it follows that $\left(\i\m+\nu\right)\mathcal{L}_{\nu}f\in\L(\R;H)$.
Hence, for all $t\in\R$ we obtain 
\begin{align*}
AS(\i t+\nu)\mathcal{L}_{\nu}f(t) & =A\bigl(\left(\i t+\nu\right)M(\i t+\nu)+A\bigr)^{-1}\mathcal{L}_{\nu}f(t)\\
 & =\mathcal{L}_{\nu}f(t)-\left(\i t+\nu\right)M(\i t+\nu)S(\i t+\nu)\mathcal{L}_{\nu}f(t).
\end{align*}
Thus, by the boundedness of $M$ and $S$, we deduce $S(\i\cdot+\nu)\mathcal{L}_{\nu}f\in\L(\R;\dom(A))$.
This implies $S(\td{\nu})f\in\Lnu(\R;\dom(A))$ by \prettyref{exer:abstractFourier_exten}.
Similarly, but more easily, it follows that $\left(\i\cdot+\nu\right)S(\i\cdot+\nu)\mathcal{L}_{\nu}f\in\L(\R;H)$ also, which shows $S(\td{\nu})f\in \dom(\td{\nu})$.

We now define the operator $B(\i\m+\nu)$ by
\begin{align*}
 \dom(B(\i\m+\nu))\coloneqq \bigl\{f\in \L(\R;H)\,;\, & f(t)\in \dom(A) \mbox{ for a.e. } t\in \R,\\
 & (t\mapsto B(\i t+\nu)f(t))\in \L(\R;H)\bigr\}
\end{align*}
and
\[
 B(\i\m+\nu)f\coloneqq (t\mapsto B(\i t+\nu)f(t))\quad (f\in \dom(B(\i\m+\nu))).
\]
Then one easily sees that $B(\i\m+\nu)=S(\i\m+\nu)^{-1}$ and since $S(\i\m+\nu)$ is closed, it follows that $B(\i\m+\nu)$ is closed as well. Moreover
\[
 (\i\m+\nu)M(\i\m+\nu)+A\subseteq B(\i\m+\nu)
\]
and hence, the operator $(\i\m+\nu)M(\i\m+\nu)+A$ is closable, which also yields the closability of $\td{\nu}M(\td{\nu})+A$ by unitary equivalence. To complete the proof, we have to show that
\[
 \overline{(\i\m+\nu)M(\i\m+\nu)+A}=B(\i\m+\nu),
\]
as this equality would imply $S(\td{\nu})=\bigl(\overline{\td{\nu}M(\td{\nu})+A}\bigr)^{-1}$ by unitary equivalence. For showing the asserted equality, let $f\in \dom(B(\i\m+\nu))$. For $n\in \N$ we define $f_n\coloneqq \1_{\ci{-n}{n}} f$. Then $f_n\in \dom(\i\m+\nu)\cap \dom(A) \subseteq \dom\bigl((\i\m+\nu)M(\i\m+\nu)+A\bigr)$ for each $n\in \N$ and by dominated convergence, we have that $f_n\to f$ as $n\to \infty$ as well as
\[
 \bigl((\i\m+\nu)M(\i\m+\nu)+A\bigr)f_n=B(\i\m+\nu)f_n=\1_{\ci{-n}{n}} B(\i\m+\nu)f\to B(\i\m+\nu)f\quad (n\to \infty).
\]
This shows that $f\in \dom\bigl(\overline{(\i\m+\nu)M(\i\m+\nu)+A}\bigr)$ and hence, the assertion follows.
\end{proof}

\begin{rem}\label{rem:Aaccretive}
  Note that \prettyref{thm:Solution_theory_EE} can partly be generalised in the following way (with the same proof). Let $M\from \C_{\Re>\nu_0}\to \bo(H)$ be holomorphic and $A$ a closed, densely defined operator in $H$
  such that $zM(z)+A$ is boundedly invertible for all $z\in \C_{\Re>\nu_0}$ and that $\sup_{z\in \C_{\Re>\nu_0}}\norm{(zM(z)+A)^{-1}}_{\bo(H)}<\infty$. Then $S_\nu\in \bo(\Lnu(\R;H))$ is causal and eventually independent of $\nu$.
\end{rem}

\begin{rem}\label{rem:somlo}
  As the proof of \prettyref{thm:Solution_theory_EE} shows, for $\nu\geq\nu_0$ we have that $S\from \C_{\Re\geq \nu}\ni z\mapsto (zM(z)+A)^{-1}\in \bo(H)$ is a material law and $S_\nu = S(\td{\nu})$. Thus, the solution operator is a material law operator, and by \prettyref{rem:mlohom} applied to $S$ and $z\mapsto \frac{1}{z}1_{H}$ we obtain 
  \[S_\nu \td{\nu}\subseteq \td{\nu} S_\nu.\]
\end{rem}

\section{Comments}


The proof of \prettyref{thm:Solution_theory_EE} here is rather close to the strategy originally employed in \cite{PicPhy}, at least where existence and uniqueness are concerned. The causality part is a consequence of some observations detailed in \cite{KPSTW14_OD,W15_CB}. The original process of proving causality used the Theorem of Paley and Wiener, which we shall discuss later on.

The eddy current approximation \index{eddy current approximation} has enjoyed great interest in the mathematical and physical community,
in particular for the case when $\varepsilon=0$ everywhere. The reason
being that then Maxwell's equations are merely of parabolic type. We
shall refer to \cite{Pauly2018} and the references therein for an
extensive discussion. 

Both \prettyref{prop:accr_inv} and the Lax--Milgram\index{Lax--Milgram lemma}
lemma have been put into a general perspective in \cite{PTW15_WP_P}.

\section*{Exercises}
\addcontentsline{toc}{section}{Exercises}

\begin{xca}
\label{exer:tensorprod}Let $\left(\Omega,\Sigma,\mu\right)$ be a
$\sigma$-finite measure space and let $H_0,H_1$ be Hilbert spaces. Let $A\colon\dom(A)\subseteq H_0\to H_1$
be densely defined and closed. Show that the operator
\begin{align*}
A_{\mu}\colon\L(\mu;\dom(A))\subseteq\L(\mu;H_0) & \to\L(\mu;H_1)\\
f & \mapsto\bigl(\omega\mapsto Af(\omega)\bigr)
\end{align*}
is densely defined and closed. Moreover, show that $\left(A_{\mu}\right)^{*}=\left(A^{*}\right)_{\mu}$.
\end{xca}

\begin{xca}
\label{exer:abstractFourier_exten}In the situation of \prettyref{exer:tensorprod},
if $(\Omega_{1},\Sigma_{1},\mu_{1})$ is another $\sigma$-finite measure space and $\mathcal{F}\colon\L(\mu)\to\L(\mu_{1})$
is unitary, show that for $j\in \{0,1\}$ there exists a unique unitary operator $\mathcal{F}_{H_j}\colon\L(\mu;H_j)\to\L(\mu_{1};H_j)$
such that 
\[
 \mathcal{F}_{H_j} (\phi x)=(\mathcal{F}\phi) x\quad (\phi\in \L(\mu),x\in H_j).
\]
Furthermore, prove that
\[
\mathcal{F}_{H_1}A_{\mu}\mathcal{F}_{H_0}^{*}=A_{\mu_{1}}.
\]
\end{xca}

\begin{xca}
\label{exer:Ccinfty_dense_hd}Show that for $\Omega\subseteq\R^{d}$
open, the set $\cci(\Omega)\subseteq\L(\Omega)$ is dense. 
\end{xca}

\begin{xca}
\label{exer:Sobolev_elem}Prove \prettyref{thm:Sobolev_space}.
\end{xca}

\begin{xca}
\label{exer:comp_ana}Let $H$ be a Hilbert space, $A\colon\dom(A)\subseteq H\to H$
skew-selfadjoint, and $c>0$. Moreover, let $M\colon\dom(M)\subseteq\C\to\bo(H)$ be holomorphic with 
\[
 \Re M(z)\geq c \quad (z\in \dom (M)).
\]
Show that $\dom(M)\ni z\mapsto\left(M(z)+A\right)^{-1}$ is holomorphic.
\end{xca}

\begin{xca}
\label{exer:commute_integral} Let $C\colon \dom(C)\subseteq H_0 \to H_1$ be a densely defined and closed linear operator acting in Hilbert spaces $H_0$ and $H_1$. For $\nu>0$ show that
\[
    \overline{\td{\nu}^{-1}C}=C\td{\nu}^{-1}.
\]
Hint: Apply \prettyref{exer:abstractFourier_exten} and show $\overline{(\i\m+\nu)^{-1}C}=C(\i\m+\nu)^{-1}$ with a suitable approximation argument.
\end{xca}

\begin{xca}
\label{exer:H=00003DW+Kasuga}
\begin{enumerate}
\item Compute $H_{0}^{1}(\Omega)^{\bot}$ where the orthogonal complement
is computed in $H^{1}(\Omega)$.
\item Assume that
\[
D\coloneqq\set{\phi\in H^{1}(\Omega)}{\grad\phi\in\dom(\dive),\phi=\dive\grad\phi}\subseteq C^{\infty}(\Omega).
\] and show that $C^{\infty}(\Omega)\cap H^{1}(\Omega)\subseteq H^{1}(\Omega)$
is dense.
\end{enumerate}
Remark: The regularity assumption in (b) always holds and is known as Weyl's Lemma, see e.g. \cite[Corollary 8.11]{Gilbarg1983}, where the more general situation of an elliptic operator with smooth coefficients is treated. See also \cite[p.127]{Donoghue1969}, where the regularity is shown for harmonic distributions.
\end{xca}

\printbibliography[heading=subbibliography]

\chapter{Examples of Evolutionary Equations}

This chapter is devoted to a small tour through a variety of evolutionary
equations. More precisely, we shall look into the equations of poro-elastic
media, (time-)fractional elasticity, thermodynamic media with delay
as well as visco-elastic media. The discussion of these examples will
be similar to that of the examples in the previous chapter
in the sense that we shall present the equations first, reformulate
them suitably and then apply the solution theory to them. The
study of visco-elastic media within the framework of partial integro-differential
equations will be carried out in the exercises section.

\section{Poro-Elastic Deformations}

In this section we will discuss the equations of poro-elasticity\index{poro-elasticity}, which form a coupled system of equations. More precisely, the equations of (linearised)
elasticity are coupled with the diffusion equation. Before properly
writing these equations we introduce the following notation and differential operators.
\begin{defn*}
  Let $\K^{d\times d}_\rmsym\coloneqq \set{A \in \K^{d\times d}}{A=A^\top}\subseteq \K^{d\times d}$ the (closed) subspace of symmetric $d\times d$ matrices. 
  Let $\Omega\subseteq\R^d$ be open. Then define
  \begin{align*}
    \L(\Omega)_{\rmsym}^{d\times d}
    & \coloneqq \L(\Omega; \K^{d\times d}_\rmsym) \\
    & = \set{(\Phi_{jk})_{j,k\in\{1,\ldots,d\}}\in\L(\Omega)^{d\times d}}{\forall j,k\in\{1,\ldots,d\}\colon\Phi_{jk}=\Phi_{kj}}.
  \end{align*}
  Analogously, we set
  $\cci(\Omega)_{\rmsym}^{d\times d} \coloneqq \cci(\Omega;\K^{d\times d}_\rmsym)$.
\end{defn*}
Note that the symmetry of a $d\times d$ matrix here means that the matrix elements are symmetric with respect to the main diagonal. For $\K=\C$, this does not correspond to the symmetry of the associated linear operator (which would rather be selfadjointness).

\begin{defn*}
Let $\Omega\subseteq\mathbb{\R}^{d}$ be open. Then we define
\begin{align*}
\Grad_{\rmc}\colon\cci(\Omega)^{d}\subseteq\L(\Omega)^{d} & \to\L(\Omega)_{\rmsym}^{d\times d}\\
\left(\phi_{j}\right)_{j\in\{1,\ldots,d\}} & \mapsto\frac{1}{2}\left(\partial_{k}\phi_{j}+\partial_{j}\phi_{k}\right)_{j,k\in\{1,\ldots,d\}},
\end{align*}
and
\begin{align*}
\Dive_{\rmc}\colon\cci(\Omega)_{\rmsym}^{d\times d}\subseteq\L(\Omega)_{\rmsym}^{d\times d} & \to\L(\Omega)^{d}\\
\left(\Phi_{jk}\right)_{j,k\in\{1,\ldots,d\}} & \mapsto\left(\sum_{k=1}^{d}\partial_{k}\Phi_{jk}\right)_{j\in\{1,\ldots,d\}}.
\end{align*}
Similarly to the definitions in the previous chapter, we put $\Grad\coloneqq-\Dive_{\rmc}^{*}$,
$\Dive\coloneqq-\Grad_{\rmc}^{*}$ and $\Grad_{0}\coloneqq-\Dive^{*}$,
$\Dive_{0}\coloneqq-\Grad^{*}$, where (analogously to the scalar-valued case)
we observe that $\Grad_{\rmc}\subseteq-\Dive_{\rmc}^{*}$ motivating the
notation $\Grad$ and $\Grad_{0}$. 
\end{defn*}

\begin{rem}
Note that in the literature $\Grad u$ is also denoted by $\varepsilon(u)$
and is called the \emph{strain tensor}\index{strain tensor}. Due
to the (obvious) similarity to the scalar case, we refrain from using
$\varepsilon$ in this context and prefer $\Grad$ instead. Again,
the index $0$ in the operators refers to generalised Dirichlet (for $\Grad_0$) or
Neumann (for $\Dive_0$) boundary conditions.
\end{rem}

We are now properly equipped to formulate the equations of poro-elasticity\index{poro-elasticity};
see also \cite{Murad1996} and below for further details. In an elastic
body $\Omega\subseteq\R^{d}$, the displacement field, $u\colon\R\times\Omega\to\R^{d}$,
and the pressure field, $p\colon\R\times\Omega\to\R$, of a fluid
diffusing through $\Omega$ satisfy the following two energy balance
equations 
\begin{align*}
\partial_{t}\rho\partial_{t}u-\grad\partial_{t}\lambda\dive u-\Dive C\Grad u+\grad\alpha^{*}p & =f,\\
\partial_{t}(c_{0}p+\alpha\dive u)-\dive k\grad p & =g.
\end{align*}
The right-hand sides $f\colon\R\times\Omega\to\R^{d}$ and $g\colon\R\times\Omega\to\R$
describe some given external forcing. We assume homogeneous Neumann
boundary conditions for the diffusing fluid as well as homogeneous
Dirichlet (i.e.~clamped) boundary conditions for the elastic body.
The operator $\rho\in\bo(\L(\Omega)^{d})$ describes the density of
the medium $\Omega$ (usually realised as a multiplication operator by a bounded, measurable, scalar
function). The bounded linear operators $C\in\bo(\L(\Omega)_{\rmsym}^{d\times d})$
and $k\in\bo(\L(\Omega)^{d})$ are the elasticity tensor and the hydraulic
conductivity of the medium, whereas $c_{0},\lambda\in\bo(\L(\Omega))$
are the porosity of the medium and the compressibility of the fluid,
respectively. The operator $\alpha\in\bo(\L(\Omega))$ is the so-called
Biot--Willis constant. Note that in many applications $\rho,c_{0},\lambda$ and $\alpha$
are just positive real numbers, and $C$ and $k$ are strictly positive
definite tensors or matrices.

The reformulation of the equations for poro-elasticity involves several
`tricks'. One of these is to introduce the matrix trace as the operator
\begin{align*}
\trace\colon\L(\Omega)_{\rmsym}^{d\times d} & \to\L(\Omega)\\
(\Phi_{jk})_{j,k\in\{1,\ldots,d\}} & \mapsto\sum_{j=1}^{d}\Phi_{jj}.
\end{align*}
Note that the adjoint is given by $\trace^{*}f=\diag(f,f,\ldots,f)\in\L(\Omega)_{\rmsym}^{d\times d}$. 
It is then elementary to obtain $\trace\Grad\subseteq\dive$ as well
as $\grad=\Dive\trace^{*}$. Hence, we get formally
\begin{align*}
\partial_{t}\rho\partial_{t}u-\Dive\bigl(\left(\partial_{t}\trace^{*}\lambda\trace+C\right)\Grad u-\trace^{*}\alpha^{*}p\bigr) & =f,\\
\partial_{t}\left(c_{0}p+\alpha\trace\Grad u\right)-\dive k\grad p & =g.
\end{align*}
Next, we introduce a new set of unknowns
\begin{align*}
v & \coloneqq\partial_{t}u,\\
T & \coloneqq C\Grad u,\\
\omega & \coloneqq\lambda\trace\Grad v-\alpha^{*}p,\\
q & \coloneqq-k\grad p.
\end{align*}
Here, $v$ is the velocity, $T$ is the stress tensor and $q$ is the heat flux. The quantity $\omega$ is an additional variable, which helps to rewrite the system into the form of evolutionary equations.

In order to finalise the reformulation we shall assume some additional properties
on the coefficients involved. Throughout the rest of this section,
we assume that
\begin{align*}
\rho=\rho^{*} & \geq c,\\
c_{0}=c_{0}^{*} & \geq c,\\
\Re\lambda & \geq c,\\
\Re k & \geq c,\text{ and }\\
C=C^{*} & \geq c
\end{align*}
for some $c>0$, where all inequalities are thought of in the sense of positive definiteness (compare Chapter \ref{chap:Solution_Theory}). As a consequence, we obtain
\[
\trace\Grad v=\lambda^{-1}\omega+\lambda^{-1}\alpha^{*}p.
\]
Rewriting the defining equations for $T,$ $\omega,$ and $q$
together with the two equations we started out with, we obtain the
system
\begin{align*}
\partial_{t}\rho v-\Dive\left(T+\trace^{*}\omega\right) & =f,\\
\partial_{t}c_{0}p+\alpha\lambda^{-1}\omega+\alpha\lambda^{-1}\alpha^{*}p+\dive q & =g,\\
\lambda^{-1}\omega+\lambda^{-1}\alpha^{*}p-\trace\Grad v & =0,\\
\partial_{t}C^{-1}T-\Grad v & =0,\\
k^{-1}q+\grad p & =0.
\end{align*}
Note that at this stage of modelling we did assume that we can freely interchange the order of differenation, so that
$\Grad \partial_t u = \partial_t \Grad u$.
Introducing 
\begin{align}
M_{0} & \coloneqq\begin{pmatrix}
 \rho  &  0  &  0  &  0  &  0\\
 0  &  c_{0}  &  0  &  0  &  0\\
 0  &  0  &  0  &  0  &  0\\
 0  &  0  &  0  &  C^{-1}  &  0\\
 0  &  0  &  0  &  0  &  0 
\end{pmatrix},\quad 
M_{1} \coloneqq\begin{pmatrix}
0 & 0 & 0 & 0 & 0\\
0 & \alpha\lambda^{-1}\alpha^{*} & \alpha\lambda^{-1} & 0 & 0\\
0 & \lambda^{-1}\alpha^{*} & \lambda^{-1} & 0 & 0\\
0 & 0 & 0 & 0 & 0\\
0 & 0 & 0 & 0 & k^{-1}
\end{pmatrix},\label{eq:PoroM}\\
V & \coloneqq\begin{pmatrix}
1 & 0 & 0 & 0 & 0\\
0 & 1 & 0 & 0 & 0\\
0 & 0 & 1 & \trace & 0\\
0 & 0 & 0 & 1 & 0\\
0 & 0 & 0 & 0 & 1
\end{pmatrix}, \quad
A \coloneqq\begin{pmatrix}
0 & 0 & 0 & -\Dive & 0\\
0 & 0 & 0 & 0 & \dive_{0}\\
0 & 0 & 0 & 0 & 0\\
-\Grad_{0} & 0 & 0 & 0 & 0\\
0 & \grad & 0 & 0 & 0
\end{pmatrix}, \label{eq:PoroVA}
\end{align}
we obtain 
\[
\left(\partial_{t}M_{0}+M_{1}+VAV^{*}\right)\begin{pmatrix}
v\\
p\\
\omega\\
T\\
q
\end{pmatrix}=\begin{pmatrix}
f\\
g\\
0\\
0\\
0
\end{pmatrix}.
\]
This perspective enables us to prove well-posedness for the equations
of poro-elasticity by applying \prettyref{thm:Solution_theory_EE}.

\begin{thm}
\label{thm:wp-Poro} Put $H\coloneqq\L(\Omega)^{d}\times\L(\Omega)\times\L(\Omega)\times\L(\Omega)_{\rmsym}^{d\times d}\times\L(\Omega)^{d}$
and let $M_{0},M_{1},V\in\bo(H)$ and $A$ be given as in \prettyref{eq:PoroM} and \prettyref{eq:PoroVA}.
Then there exists $\nu_{0}>0$ such that for all $\nu\geq\nu_{0}$
the operator $\overline{\td{\nu}M_{0}+M_{1}+VAV^{*}}$
is continuously invertible on $\Lnu(\R;H)$. The inverse
$S_{\nu}$ of this operator is causal and eventually independent of $\nu$. Moreover,
$\sup_{\nu\geq\nu_{0}}\norm{S_{\nu}}<\infty$ and $F\in \dom(\td{\nu})$
implies $S_{\nu}F\in\dom(\td{\nu})\cap\dom(VAV^{*})$.
\end{thm}

We will provide two prerequisites for the proof. We ask for the details of the proof in \prettyref{exer:Poro}. 

\begin{prop}
\label{prop:UAU} Let $H_{0}$, $H_{1}$ be Hilbert spaces, $B\colon\dom(B)\subseteq H_{0}\to H_{0}$
skew-selfadjoint, $V\in \bo(H_{0},H_{1})$ bijective.
Then $\left(VBV^{*}\right)^{*}=-VBV^{*}$.
\end{prop}

The proof of \prettyref{prop:UAU} is left as (part of) \prettyref{exer:Poro}.

\begin{prop}
\label{prop:EuclidM0M1} Let $H$ be a Hilbert space, $N_{0},N_{1}\in\bo(H)$
with $N_{0}=N_{0}^{*}$. Assume there exist $c_0,c_1>0$ such that $\scp{x}{N_{0}x}\geq c_{0}\norm{x}^{2}$ for all $x\in \ran(N_0)$
and $\Re\scp{y}{N_{1}y}\geq c_{1}\norm{y}^{2}$ for all $y\in\ker(N_{0})$. 
Then for all $0<c_{1}'<c_{1}$ there exists $\nu_{0}>0$ such that
for all $\nu\geq\nu_{0}$ we have that
\[
\nu N_{0}+\Re N_{1}\geq c'_{1}.
\]
\end{prop}

\begin{proof}
Note that by the selfadjointness of $N_{0}$ we can decompose $H=\cran(N_{0})\oplus\ker(N_{0})$, see \prettyref{cor:abstrHelmH}. 
Let $z\in H$, and $x\in\cran(N_{0})$, $y\in\ker(N_{0})$ such that $z=x+y$. For $\varepsilon,\nu>0$ we estimate
\begin{align*}
 & \nu\scp{x+y}{N_{0}(x+y)}+\Re\scp{x+y}{N_{1}(x+y)}\\
 & =\nu\scp{x}{N_{0}x}+\Re\scp{y}{N_{1}y}+\Re\scp{x}{N_{1}x}+\Re\scp{x}{N_{1}y}+\Re\scp{y}{N_{1}x}\\
 & \geq\nu c_{0}\norm{x}^2+c_{1}\norm{y}^2-\norm{N_{1}}\norm{x}^2-2\norm{N_{1}}\norm{x}\norm{y}\\
 & \geq\nu c_{0}\norm{x}^2+c_{1}\norm{y}^2-\norm{N_{1}}\norm{x}^2-\frac{1}{\varepsilon}\norm{N_{1}}^{2}\norm{x}^{2}-\varepsilon\norm{y}^{2}\\
 & =\left(\nu c_{0}-\frac{1}{\varepsilon}\norm{N_{1}}^{2}-\norm{N_{1}}\right)\norm{x}^2+\left(c_{1}-\varepsilon\right)\norm{y}^2,
\end{align*}
where we have used the Peter--Paul inequality (i.e., Young's inequality for products of non-negative numbers).
For $0<c_{1}'<c_{1}$ we find $\varepsilon>0$ such that $c_{1}-\varepsilon>c_{1}'$.
Then we choose $\nu>\frac{1}{c_{0}}\left(c_{1}'+\frac{1}{\varepsilon}\norm{N_{1}}^{2}+\norm{N_{1}}\right)$.
With this choice of $\nu_{0}$ we deduce for all $\nu\geq\nu_{0}$ that
\[
\nu\scp{z}{N_{0}z}+\Re\scp{z}{N_{1}z}\geq c_{1}'\left(\norm{x}^{2}+\norm{y}^{2}\right)=c_{1}'\norm{z}^{2},
\]
which yields the assertion.
\end{proof}

\section{Fractional Elasticity\index{fractional elasticity}}

Let $\Omega\subseteq\R^d$ be open.
In order to better fit to the experimental data of visco-elastic solids (i.e., to incorporate solids that `memorise' previous force applied to them) the equations of linearised elasticity need to be extended in some way. The balance law for
the momentum, however, is still satisfied; that is, for the displacement
$u\colon\R\times\Omega\to\R^{d}$ we still have that
\[
\partial_{t}\rho\partial_{t}u-\Dive T=f,
\]
where $\rho\in\bo(\L(\Omega)^{d})$ models the density and $f\colon\R\times\Omega\to\R^{d}$
is a given external forcing term. The stress tensor, $T\colon\R\times\Omega\to\R_{\textnormal{sym}}^{d\times d}$,
does \textbf{not} follow the classical Hooke's law, which, if it did, would look like 
\[
T=C\Grad u
\]
for $C\in\bo(\L(\Omega)_{\textnormal{sym}}^{d\times d})$. Instead it is
amended by another material dependent coefficient $D\in\bo(\L(\Omega)_{\textnormal{sym}}^{d\times d})$
and a fractional time derivative; that is,
\[
T=C\Grad u+D\partial_{t}^{\alpha}\Grad u,
\]
for some $\alpha\in\ci{0}{1}$, where $\partial_t^\alpha \coloneqq \partial_t \partial_t^{\alpha-1}$, see \prettyref{exa:material-laws}\ref{exa:material-laws:item:5}. We shall simplify the present consideration
slightly and refer to \prettyref{exer:fracEla} instead for a more
involved example. Throughout this section, we shall assume that 
\[
C=0,\; D=D^{*}\geq c,\text{ and }\rho=\rho^{*}\geq c
\]
for some $c>0$. Thus, putting $v\coloneqq\partial_{t}u$ and assuming
the clamped boundary conditions again, we study well-posedness of
\begin{align}
\partial_{t}\rho v-\Dive T & =f,\label{eq:FEmomentum}\\
T & =D\partial_{t}^{\alpha}\Grad_{0}u.\label{eq:FEmatlaw}
\end{align}

In order to do that, we first rewrite the second equation. 
We will make use of the following proposition which will serve us to show bounded invertibility of $\partial_t^\alpha$ (in the space $\Lnu$), and which will also be employed to obtain well-posedness.

\begin{prop}
\label{prop:pdfrac}Let $\nu>0$, $z\in\C_{\Re\geq\nu}$, $\alpha\in[0,1]$. Then
\[
\Re z^{\alpha}\geq (\Re z)^\alpha \geq \nu^{\alpha}.
\]
\end{prop}

\begin{proof}
Let us prove the first inequality. Note that without loss of generality, we may assume that $\Re z=1$. Let $\varphi\coloneqq \arg z \in\oi{-\tfrac{\pi}{2}}{\tfrac{\pi}{2}}$.
Since $\ln\circ\cos$ is concave on $\oi{-\tfrac{\pi}{2}}{\tfrac{\pi}{2}}$ (as $(\ln\circ\cos)' = -\tan$ is decreasing) and $(\ln\circ\cos)(0) = 0$, we obtain
\[\ln\cos(\alpha \varphi) = \ln\cos(\alpha\varphi + (1-\alpha)0) \geq \alpha \ln\cos(\varphi) + (1-\alpha)\ln\cos(0) = \ln \bigl(\cos(\varphi)^\alpha\bigr),\]
and therefore $\cos(\alpha\varphi)\geq \cos(\varphi)^\alpha$.
Since $\Re z=1$ implies $\abs{z} = \frac{1}{\cos(\varphi)}$, we obtain
\[\Re z^\alpha = \frac{\cos(\alpha\varphi)}{(\cos\varphi)^\alpha} \geq 1 = (\Re z)^\alpha.\]
The second inequality follows from monotonicity of $x\mapsto x^\alpha$.
\end{proof}

Applying \prettyref{prop:pdfrac} and noting that $D$ is boundedly invertible we can reformulate  \prettyref{eq:FEmatlaw} as 
\[
\td{\nu}^{-\alpha}D^{-1}T-\Grad_{0}u=0,
\]
so that \prettyref{eq:FEmatlaw} and \prettyref{eq:FEmomentum} read
\[
\left(\td{\nu}\begin{pmatrix}
\rho & 0\\
0 & \td{\nu}^{-\alpha}D^{-1}
\end{pmatrix}-\begin{pmatrix}
0 & \Dive\\
\Grad_{0} & 0
\end{pmatrix}\right)\begin{pmatrix}
v\\
T
\end{pmatrix}=\begin{pmatrix}
f\\
0
\end{pmatrix}.
\]
A solution theory for the latter equation, thus, reads as follows,
where again $v\coloneqq\td{\nu}u$.
\begin{thm}
\label{thm:wp-FE} Put $H\coloneqq\L(\Omega)^{d}\times\L(\Omega)_{\textnormal{sym}}^{d\times d}.$
Then for all $\nu>0$ the operator
\[
\td{\nu}\begin{pmatrix}
\rho & 0\\
0 & \td{\nu}^{-\alpha}D^{-1}
\end{pmatrix}-\begin{pmatrix}
0 & \Dive\\
\Grad_{0} & 0
\end{pmatrix}
\]
is densely defined and closable in $\Lnu(\R;H)$. The inverse of the
closure is continuous, causal and eventually independent of $\nu$.
\end{thm}

\begin{proof}
The proof rests on \prettyref{thm:Solution_theory_EE}. Since $\begin{pmatrix}
0 & \Dive\\
\Grad_{0} & 0
\end{pmatrix}$ is skew-selfadjoint by \prettyref{prop:block_op_realinv}\ref{prop:block_op_realinv:item:1}, it suffices to confirm the
positive definiteness condition for the material law. For this let
$z\in\C_{\Re\geq\nu}$ and compute for $x\in\L(\Omega)_{\textnormal{sym}}^{d\times d}$,
using \prettyref{prop:pdfrac} and \prettyref{prop:block_op_realinv}\ref{prop:block_op_realinv:item:2},
\[
\Re\scp{x}{zz^{-\alpha}D^{-1}x}=\Re\scp{x}{z^{1-\alpha}D^{-1}x}\geq\nu^{1-\alpha}\scp{x}{D^{-1}x} \geq \nu^{1-\alpha}\frac{c}{\norm{D}^2}\norm{x}^2.
\]
This yields the assertion.
\end{proof}

\section{The Heat Equation with Delay}

Let $\Omega\subseteq\R^{d}$ be open. In this section we concentrate
on a generalisation of the heat equation discussed in the previous
chapter. Although we keep the heat-flux-balance in the sense that 
\[
\partial_{t}\theta+\dive q=Q,
\]
with $q\colon\R\times\Omega\to\R^{d}$ being the heat flux and $\theta\colon\R\times\Omega\to\R$
being the heat, we shall now modify Fourier's law to the extent that
\[
q=-a\grad\theta-b\tau_{-h}\grad\theta
\]
for some $a,b\in\bo(\L(\ensuremath{\Omega)^{d})}$ with $\Re a\geq c$
for some $c>0$, and $h>0$. We shall again assume homogeneous Neumann
boundary conditions for $q$. Written in the now standard block
operator matrix form, this modified heat equation reads
\[
\left(\td{\nu}\begin{pmatrix}
1 & 0\\
0 & 0
\end{pmatrix}+\begin{pmatrix}
0 & 0\\
0 & \left(a+b\tau_{-h}\right)^{-1}
\end{pmatrix}+\begin{pmatrix}
0 & \dive_{0}\\
\grad & 0
\end{pmatrix}\right)\begin{pmatrix}
\theta\\
q
\end{pmatrix}=\begin{pmatrix}
Q\\
0
\end{pmatrix}.
\]
In order to actually justify the existence of the operator $\left(a+b\tau_{-h}\right)^{-1}$
 as a bounded linear operator, we provide the following lemma.
\begin{lem}
\label{lem:delayInv}
\begin{enumerate}
\item
There exists $\nu_{0}>0$ such that for all
$\nu\geq\nu_{0}$ the operator $a+b\tau_{-h}$ is continuously invertible
in $\Lnu(\R;\L(\Omega)^{d})$.
\item For all $0<c'<c/\norm{a}^{2}$ there is $\nu_{1}\geq\nu_{0}$
such that for all $z\in\C_{\Re\geq\nu_{1}}$ we have $\Re\left(a+b\e^{-zh}\right)^{-1}\geq c'$. 
\end{enumerate}
\end{lem}

\begin{proof}
Note that $a$ is invertible with $\norm{a^{-1}}\leq\frac{1}{c}$ and $\Re a^{-1}\geq \frac{c}{\norm{a}^2}$ by \prettyref{prop:block_op_realinv}\ref{prop:block_op_realinv:item:2}.
\begin{enumerate}
\item
By \prettyref{exa:material_law_revisited}\ref{exa:material_law_revisited:item:3}, for all $\nu>0$ we obtain
\[
\norm{b\tau_{-h}}_{\bo(\Lnu)}\leq\norm{b}_{\bo(\L(\Omega)^{d})}\sup_{t\in\R}\abs{\e^{-\left(it+\nu\right)h}}=\norm{b}_{\bo(\L(\Omega)^{d})}\e^{-h\nu}.
\]
Thus, we find $\nu_{0}>0$ such that for all $\nu\geq\nu_{0}$ we
obtain $\norm{b\tau_{-h}a^{-1}}_{\bo(\Lnu)}\leq\frac{1}{c}\norm{b\tau_{-h}}_{\bo(\Lnu)}<1.$
Thus, 
\[
a+b\tau_{-h}=\left(1+b\tau_{-h}a^{-1}\right)a
\]
is continuously invertible by a Neumann series argument.
\item 
Let $0<c'<c/\norm{a}^{2}$. We choose $\nu_{1}\geq\nu_{0}$ such that $\norm{b\e^{-zh}a^{-1}}_{\bo(\L(\Omega)^{d})}\leq\min\{\frac{1}{2},\varepsilon\}$
for all $z\in\C_{\Re\geq\nu_{1}}$, where $0<\varepsilon\leq\frac{1}{2}c\left(\frac{c}{\norm{a}^{2}}-c'\right)$.
For $z\in\C_{\Re\geq\nu_{1}}$ we compute
\begin{align*}
\Re\left(a+b\e^{-zh}\right)^{-1} & =\Re a^{-1}\left(1+b\e^{-zh}a^{-1}\right)^{-1}
 =\Re\left(a^{-1}\sum_{k=0}^{\infty}\left(-b\e^{-zh}a^{-1}\right)^{k}\right)\\
 & =\Re\left(a^{-1}+\sum_{k=1}^{\infty}a^{-1}\left(-b\e^{-zh}a^{-1}\right)^{k}\right)\\
 & \geq\frac{c}{\norm{a}^{2}}-\norm{\sum_{k=1}^{\infty}a^{-1}\left(-b\e^{-zh}a^{-1}\right)^{k}}\\
 & \geq\frac{c}{\norm{a}^{2}}-\sum_{k=1}^{\infty}\norm{a^{-1}\left(-b\e^{-zh}a^{-1}\right)^{k}}\\
 & \geq\frac{c}{\norm{a}^{2}}-\frac{1}{c}\sum_{k=1}^{\infty}\norm{\left(-b\e^{-zh}a^{-1}\right)}^{k}\\
 & =\frac{c}{\norm{a}^{2}}-\frac{1}{c}\frac{\norm{\left(-b\e^{-zh}a^{-1}\right)}}{1-\norm{\left(-b\e^{-zh}a^{-1}\right)}}
 \geq\frac{c}{\norm{a}^{2}}-\frac{1}{c}2\varepsilon\geq c'.\tag*{{\qedhere}}
\end{align*}
\end{enumerate}
\end{proof}
With this lemma we are in the position to provide the well-posedness
for the modified heat equation, as well.
\begin{thm}
\label{thm:wp-heatdelay} Let $H=\L(\Omega)\times\L(\Omega)^{d}$. There exists $\nu_{0}>0$ such that for
all $\nu\geq\nu_{0}$ the operator
\[
\td{\nu}\begin{pmatrix}
1 & 0\\
0 & 0
\end{pmatrix}+\begin{pmatrix}
0 & 0\\
0 & \left(a+b\tau_{-h}\right)^{-1}
\end{pmatrix}+\begin{pmatrix}
0 & \dive_{0}\\
\grad & 0
\end{pmatrix}
\]
is densely defined and closable with continuously invertible closure
in $\Lnu(\R;H)$. The inverse
of the closure is causal and eventually independent of $\nu$. 
\end{thm}

\begin{proof}
The proof rests on \prettyref{thm:Solution_theory_EE} and \prettyref{lem:delayInv}.
\end{proof}

\section{Dual Phase Lag Heat Conduction\index{dual phase lag heat conduction}\label{sec:Dual-phase-lag}}

The last example is concerned with a different modification of Fourier's
law. The heat-flux balance
\begin{equation}
\partial_{t}\theta+\dive q=Q\label{eq:dualHB}
\end{equation}
is accompanied by the modified Fourier's law
\begin{equation}
\bigl(1+s_{q}\partial_{t}+\frac{1}{2}s_{q}^{2}\partial_{t}^{2}\bigr)q=-(1+s_{\theta}\partial_{t})\grad\theta,\label{eq:dualFL}
\end{equation}
where $s_{q}\in\R,s_{\theta}>0$ are given numbers, which are called
`phases'. 
\begin{rem}
The modified Fourier's law in \prettyref{eq:dualFL} is an attempt
to resolve the problem of inifinite propagation speed which stems
from a truncated Taylor series expansion of a model given by 
\[
\tau_{s_{q}}q=-\tau_{s_{\theta}}\grad\theta.
\]
Note that it can be shown that such a model would even be ill-posed,
see \cite{Dreher2009}.
\end{rem}

Let us turn back to the system \prettyref{eq:dualHB} and \prettyref{eq:dualFL}.
Notice, since $s_{\theta}>0$, and due to a strictly positive real
part of the derivative in our functional analytic setting, we deduce
that $(1+s_{\theta}\td{\nu})$ is continuously invertible for $\nu\geq 0$. Thus,
we obtain
\[
\td{\nu}\bigl(\td{\nu}^{-1}+s_{q}+\frac{1}{2}s_{q}^{2}\td{\nu}\bigr)(1+s_{\theta}\td{\nu})^{-1}q=-\grad\theta
\]
The block operator matrix formulation of the dual phase lag heat conduction
model is thus
\[
\left(\td{\nu}\begin{pmatrix}
1 & 0\\
0 & \bigl(\td{\nu}^{-1}+s_{q}+\frac{1}{2}s_{q}^{2}\td{\nu}\bigr)(1+s_{\theta}\td{\nu})^{-1}
\end{pmatrix}+\begin{pmatrix}
0 & \dive_{0}\\
\grad & 0
\end{pmatrix}\right)\begin{pmatrix}
\theta\\
q
\end{pmatrix}=\begin{pmatrix}
Q\\
0
\end{pmatrix}.
\]
\begin{thm}
\label{thm:wp-duallag} Let $H=\L(\Omega)\times\L(\Omega)^{d}$. Assume $s_{q}\in\R\setminus\{0\}$, $s_{\theta}>0$. Then there
exists $\nu_{0}>0$ such that for all $\nu\geq\nu_{0}$ the operator
\[
\td{\nu}\begin{pmatrix}
1 & 0\\
0 & \bigl(\td{\nu}^{-1}+s_{q}+\frac{1}{2}s_{q}^{2}\td{\nu}\bigr)(1+s_{\theta}\td{\nu})^{-1}
\end{pmatrix}+\begin{pmatrix}
0 & \dive_{0}\\
\grad & 0
\end{pmatrix}
\]
is densely defined and closable with continuously invertible closure
in $\Lnu(\R;H)$. The inverse
of the closure is causal and eventually independent of $\nu$. 
\end{thm}

The proof of \prettyref{thm:wp-duallag} is again based on \prettyref{thm:Solution_theory_EE}.
Thus, we shall only record
the decisive observation in the next result. For this, we define 
\[
M(z)\coloneqq\frac{z^{-1}+s_{q}+\frac{1}{2}s_{q}^{2}z}{1+s_{\theta}z}\in\C \quad(z\in\C\setminus\{0,-\tfrac{1}{s_\theta}\}).
\]

\begin{lem}
Let $s_{q}\in\R\setminus\{0\}, s_{\theta}>0$. Then there exists $\nu_{0}\in\R$
and $c>0$ such that for all $z\in\C_{\Re\geq\nu_{0}}$ we have
\[
\Re zM(z)\geq c.
\]
\end{lem}

\begin{proof}
We put $\sigma\coloneqq \tfrac{s_{q}}{s_{\theta}}$. Let $z\in\C\setminus\{0,-\tfrac{1}{s_{\theta}}\}$.
We compute
\[
zM(z)=\frac{1+s_{q}z+\frac{1}{2}s_{q}^{2}z^{2}}{1+s_{\theta}z}=\frac{1}{2}s_{q}z\sigma+\sigma\left(1-\frac{1}{2}\sigma\right)+\frac{1-\sigma\left(1-\frac{1}{2}\sigma\right)}{1+s_{\theta}z}
\]
and therefore 
\[
\Re zM(z)=\frac{1}{2}s_{q}\sigma\Re z+\sigma\left(1-\frac{1}{2}\sigma\right)+\frac{\bigl(1-\sigma\left(1-\frac{1}{2}\sigma\right)\bigr)\left(1+s_{\theta}\Re z\right)}{\abs{1+s_{\theta}z}^{2}}.
\]
By assumption 
\[
0<\frac{s_{q}^{2}}{s_{\theta}}=s_{q}\sigma,
\]
and since 
\[\frac{\bigl(1-\sigma\left(1-\frac{1}{2}\sigma\right)\bigr)\left(1+s_{\theta}\Re z\right)}{\abs{1+s_{\theta}z}^{2}}\to 0\]
as $\Re z\to\infty$, we obtain
\[
\Re zM(z)\geq\frac{1}{2}s_{q}\sigma\Re z-\delta
\]
for some $\delta>0$ and all $z\in\C$ with $\Re z$ large
enough.
\end{proof}

\section{Comments}

The equations of poro-elasticity have been proposed in \cite{Murad1996}
and were mathematically studied in \cite{Showalter2000,McGhee2010}.

Equations of fractional elasticity are discussed in \cite{PTW15_FI,W14_FE,CW17_1D,Nolte2003}.
The well-posedness conditions stated here and in \prettyref{exer:fracEla} can
be generalised as it is outlined in \cite{PTW15_FI} to the case
where both $C$ and $D$ are non-negative, selfadjoint operators
so that $C$ and $D$ satisfy the conditions imposed on $N_{1}$ and
$N_{0}$ in \prettyref{prop:EuclidM0M1}. We refrained from presenting
this argument here, as it seemed too technical for the time being.
Note however that the proof is neither fundamentally different nor
considerably less elementary. 

The heat equation with delay has also been studied in \cite{Khusainov2015}
with an entirely different strategy; the dual phase lag models have
been dealt with in \cite{MPTW14_TE,Tzou1995}.

Other ideas to rectify infinite propagation speed of the heat equation
can be found in \cite{Andreu2006}, where nonlinear models for heat
conduction are being discussed.

The visco-elastic equations discussed in \prettyref{exer:integro-diif} are
studied with convolution operators more general than below in \cite{Trostorff2015};
see also \cite{Cannarsa2003,Trostorff2018,Dafermos1970,Pruess2009}.

\section*{Exercises}
\addcontentsline{toc}{section}{Exercises}

\begin{xca}
\label{exer:Poro} 
\begin{enumerate}
\item
\label{exer:Poro_a}
Prove \prettyref{prop:UAU}.
\item
Prove \prettyref{thm:wp-Poro}.
\item
Let $\Omega\subseteq\R^d$ be open, $\nu>0$, $f\in H_{\nu}^{1}(\R;\L(\Omega)^{d})$ and $g\in H_{\nu}^{1}(\R;\L(\Omega))$.
With the help of \prettyref{thm:wp-Poro} show that for large enough
$\nu>0$ there exist a unique $u\in\dom\left(\td{\nu}^2\right)\cap\dom\left(\grad\lambda\dive\td{\nu}\right)\cap\dom\left(\Dive C\Grad_{0}\right)$
and $p\in\dom(\td{\nu})\cap\dom(\grad\alpha^{*})\cap\dom(\dive_{0}k\grad)$
such that 
\begin{align*}
\td{\nu}\rho\td{\nu}u-\grad\lambda\dive\td{\nu}u-\Dive C\Grad_{0}u+\grad\alpha^{*}p & =f\\
\td{\nu}c_{0}p+\alpha\dive\td{\nu}u-\dive_{0}k\grad p & =g.
\end{align*}
\end{enumerate}
\end{xca}

\begin{xca}
\label{exer:fracEla}Let $\Omega\subseteq\R^d$ be open, $C,D\in\bo(\L(\Omega)_{\rmsym}^{d\times d})$,
$D=D^{*}\geq c$ for some $c>0$ and $\alpha\in[\tfrac{1}{2},1]$. Show that
there exists $\nu_{0}>0$ such that for all $\nu\geq\nu_{0}$ the
system 
\begin{align*}
\td{\nu}\rho v-\Dive T & =f,\\
T & =\left(C+D\td{\nu}^{\alpha}\right)\Grad_{0}u,
\end{align*}
where $v=\td{\nu}u$, admits a unique solution $(v,T)\in\Lnu(\R;\L(\Omega)^{d}\times\L(\Omega)_{\rmsym}^{d\times d})$
for all $f\in H_{\nu}^{1}(\R;\L(\Omega)^{d})$.
\end{xca}

The following exercises are devoted to showing the well-posedness
of certain equations in visco-elasticity\index{visco-elasticity},
where the `viscous part' is modelled by convolution with certain integral
kernels. The proof of the required positive definiteness property
requires some preliminary results. We assume the reader to be equipped
with the basics from the theory of functions of one complex variable.

For $U\subseteq\C$ open, $u\colon U\to\C$ holomorphic,
define $f_{\Re u}\colon\tilde{U}\to\R$ by $f_{\Re u}(x,y)\coloneqq\Re u(x+\i y)$
for $(x,y)\in\tilde{U}\coloneqq\set{(x,y)\in\R^{2}}{x+\i y\in U}$.
We put
\[
H_{\Re}(U)\coloneqq\set{f_{\Re u}}{u\colon U\to\C\text{ holomorphic}}.
\]
\begin{xca}
\label{exer:mvp}
Let $U\subseteq\C$ be open.
\begin{enumerate}
\item\label{exer:mvp:item:1}
Let $f\in H_{\Re}(U)$. Show that $f$ satisfies the \emph{mean
value property}; that is, for all $(x,y)\in\tilde{U}$ and $r>0$
with $\overline{\ball{(x,y)}{r}}\subseteq\tilde{U}$ we have
\[
f(x,y)=\frac{1}{2\pi}\int_{0}^{2\pi}f(x+r\cos\theta,y+r\sin\theta)\d\theta.
\]
\item\label{exer:mvp:item:2} Let $U\coloneqq \C_{\Im >0}$ and $f\in H_{\Re}(U)\cap C(\R\times \R_{\geq 0})$. Moreover, assume that $f(x,0)=0$ for each $x\in \R$ and $f(x,y)\to 0$ as $|(x,y)|\to \infty$. Show that $f=0$ on $\R\times \R_{\geq 0}$.
\end{enumerate}
\end{xca}

\begin{xca}
\label{exer:Poisson}
In this exercise we show a version of \emph{Poisson's formula}.\index{Poisson's formula}  Let $U\coloneqq \C_{\Im>0}$ and $f\in H_{\Re}(U)\cap C(\R\times \R_{\geq 0})$. 
\begin{enumerate}
\item Assume that $f(\cdot,0)\in L_p(\R)$ for some $1\leq p<\infty$. Show that $\C_{\Im>0}\ni z\mapsto\frac{1}{\pi}\int_{\R}\frac{\Im z'+\i\left(\Re z-x'\right)}{(\Re z-x')^{2}+(\Im z)^{2}}f(x',0)\d x'$
is holomorphic.
\item Assume that $f(\cdot,0)\in L_\infty(\R)$. Show that $\frac{1}{\pi} \int_\R \frac{y}{(x-x')^2+y^2)} f(x',0)\d x'\to f(x_0,0)$ as $x\to x_0$ and $y\to 0+$.
\item (Poisson's formula)
Assume that $f(\cdot,0)\in L_p(\R)$ for some $1\leq p<\infty$ and $f(x,y)\to 0$ as $|(x,y)|\to \infty$ in $\R\times \R_{\geq 0}$. Show that
\[
f(x,y)=\frac{1}{\pi}\int_{\R}\frac{y}{(x-x')^{2}+y^{2}}f(x',0)\d x'\quad ((x,y)\in \R\times \R_{>0}).
\]
Hint: Apply \prettyref{exer:mvp}\ref{exer:mvp:item:2}.
\end{enumerate}
\end{xca}

\begin{xca}
\label{exer:convInt}Let $\nu_{0}\in\R$ and $k\in L_{1,\nu_{0}}(\R;\R)$
with $\spt k\subseteq\R_{\geq0}$. \begin{enumerate}
\item\label{exer:convInt:item:1}
Show that for all $(x,\nu)\in\R\times\Rg{\nu_{0}}$ we have
\[
\Im(\mathcal{L}k)(\i x+\nu)=\frac{1}{\pi}\int_{\R}\frac{\nu-\nu_{0}}{(x-x')^{2}+(\nu-\nu_{0})^{2}}\Im(\mathcal{L}k)(\i x'+\nu_{0})\d x'.
\]
Hint: Approximate $k$ by functions in $\cci(\R_{\geq 0};\R)$ and use Poisson's formula (see \prettyref{exer:Poisson}).
\item\label{exer:convInt:item:2}
Assume there exists $d\geq0$ such that for all $x\in\R$ 
\[
x\Im(\mathcal{L}k)(\i x+\nu_{0})\leq d.
\]
Show that for all $\nu\geq\nu_{0}$ and $x\in\R$ we have
\[
x\Im(\mathcal{L}k)(\i x+\nu)\leq4d.
\]
Hint: Use the formula in \ref{exer:convInt:item:1} and split the integral into positive
and negative part of $\R$; use symmetry of $(\mathcal{L}k)$ under conjugation
due to the realness of $k$.
\end{enumerate}
\end{xca}

\begin{xca}
\label{exer:integro-diif} Let $\Omega\subseteq\R^{d}$ be open, $\nu_{0}\in\R$
and $k\in L_{1,\nu_{0}}(\R;\R)$ with $\spt k\subseteq\R_{\geq0}$.
Assume there exists $d\geq0$ such that 
\[
x\Im(\mathcal{L}k)(\i x+\nu_{0})\leq d \quad(x\in \R).
\]
Show that there exists $\nu_{1}\geq\nu_{0}$ such that for all $\nu\geq\nu_1$ the operator
\[
\td{\nu}\begin{pmatrix}
1 & 0\\
0 & \left(1-k*\right)^{-1}
\end{pmatrix}+\begin{pmatrix}
0 & \Dive\\
\Grad_{0} & 0
\end{pmatrix}
\]
is well-defined, densely defined and closable in $\Lnu(\R;H)$ with $H=\L(\Omega)^{d}\times\L(\Omega)_{\rmsym}^{d\times d}$.
Further, show that its closure is continuously invertible, and that the corresponding inverse
is causal and eventually independent of $\nu$.
\end{xca}

\begin{xca}
\label{exer:examplConv} Let $\nu_{0}\in\R$ and $k\in L_{1,\nu_{0}}(\R;\R)$
with $\spt k\subseteq\R_{\geq0}$.
\begin{enumerate}
\item\label{exer:examplConv:item:1}
Assume that $k$ is absolutely continuous with $k'\in L_{1,\nu_{0}}(\R;\R)$. Show that there exist $\nu_{1}\geq\nu_{0}$
and $d\geq0$ with
\[
x\Im(\mathcal{L}k)(\i x+\nu_{1})\leq d\quad(x\in\R).
\]

\item\label{exer:examplConv:item:2}
Assume that $k(t)\geq0$ for all $t\in\R$ and that $k(t)\leq k(s)$, whenever
$s\leq t$. Show that there exist $\nu_{1}\geq\nu_{0}$ 
\[
x\Im(\mathcal{L}k)(\i x+\nu_{1})\leq 0\quad(x\in\R).
\]
Hint: For part (b) use the explicit formula for $\Im (\mathcal{L}k)$ as an integral and the periodicity of $\sin$.

NB: The condition in \ref{exer:examplConv:item:1} is a standard assumption for convolution
kernels in the framework of visco-elastic equations; the condition
in \ref{exer:examplConv:item:2} is from \cite{Pruess2009}.
\end{enumerate}
\end{xca}

\printbibliography[heading=subbibliography]

\chapter{Causality and a Theorem of Paley and Wiener}

In this chapter we turn our focus back to causal operators. In Chapter \ref{chap:The-Fourier-Laplace-transformati} we
 found out that material laws provide a class of causal and autonomous
bounded operators. In this chapter we will present another proof of
this fact, which rests on a result which characterises functions in $\L(\R;H)$
with support contained in the non-negative reals; the celebrated Theorem of Paley and Wiener. With
the help of this theorem, which is interesting in its own right, the
proof of causality for material laws becomes very easy. At a first
glance it seems that holomorphy of a material law is
a rather strong assumption. In the second part of this chapter, however, we
shall see that in designing autonomous and causal solution operators,
there is no way of circumventing holomorphy.

In the following, let $H$ be a Hilbert space, and we consider $\Lnu(\R_{\geq0};H)$ as the subspace of functions in $\Lnu(\R;H)$ vanishing on $\oi{-\infty}{0}$.

\section{A Theorem of Paley and Wiener}

We start with the following lemma, for which we need the notion of locally integrable functions.
We define
\begin{align*}
  L_{1,\loc}(\R;H) & \coloneqq\set{f}{\forall\, K\subseteq \R\text{ compact}: \1_K f \in L_1(\R;H)} \\
  & \,= \set{f}{\forall\, \varphi\in\cci(\R): \varphi f \in L_1(\R;H)}.
\end{align*}
\begin{lem}
\label{lem:pos_support}Let $f\in L_{1,\loc}(\R;H)$. Then we have $f\in L_{2}(\R_{\geq0};H)$
if and only if $f\in\bigcap_{\nu>0}\Lnu(\R;H)$ with $\sup_{\nu>0}\norm{f}_{\Lnu(\R;H)}<\infty$.
In the latter case we have that 
\[
\norm{f}_{\L(\R_{\geq0};H)}=\lim_{\nu\to0\rlim}\norm{f}_{\Lnu(\R;H)}=\sup_{\nu>0}\norm{f}_{\Lnu(\R;H)}.
\]
\end{lem}

\begin{proof}
Let $f\in\L(\R_{\geq0};H)$ and $\nu>0$. Then we estimate 
\begin{align*}
\int_{\R}\norm{f(t)}_{H}^{2}\e^{-2\nu t}\d t & =\int_{\R_{\geq0}}\norm{f(t)}_{H}^{2}\e^{-2\nu t}\d t
\leq\int_{\R_{\geq0}}\norm{f(t)}_{H}^{2}\d t
=\norm{f}_{\L(\R_{\geq0};H)}^{2},
\end{align*}
which proves that $f\in\Lnu(\R;H)$ with $\norm{f}_{\Lnu(\R;H)}\leq\norm{f}_{\L(\R_{\geq0};H)}$
for each $\nu>0$. Moreover, $\norm{f}_{\Lnu(\R;H)}\to\norm{f}_{\L(\R_{\geq0};H)}$
as $\nu\to0$ by monotone convergence and since clearly $\norm{f}_{\Lnu(\R;H)}\leq\norm{f}_{\Lm{\mu}(\R;H)}$
for $0<\mu\leq\nu$ we obtain 
\[
\norm{f}_{\L(\R_{\geq0};H)}=\lim_{\nu\to0\rlim}\norm{f}_{\Lnu(\R;H)}=\sup_{\nu>0}\norm{f}_{\Lnu(\R;H)}.
\]
Assume now that $f\in\bigcap_{\nu>0}\Lnu(\R;H)$ with $C\coloneqq\sup_{\nu>0}\norm{f}_{\Lnu(\R;H)}<\infty$. This inequality yields
\[
     \sup_{\nu\in \roi{0}{\infty}}\int_{\oi{-\infty}{0}} \norm{f(t)}^2\e^{-2\nu t} \d t\leq C.
\]Hence, the monotone convergence theorem yields that $g(t)\coloneqq \lim_{\nu\to\infty}\norm{f(t)}^2\e^{-2\nu t}$ for $t\in \oi{-\infty}{0}$ defines a function $g\in L_1\oi{-\infty}{0}$. Thus, $[g=\infty]$ is a set of measure zero and thus $[f=0]=\oi{-\infty}{0}\setminus [g=\infty]$ has full measure in $\oi{-\infty}{0}$ implying that $\spt f\subseteq\R_{\geq0}$.

Finally, from 
\[
      \sup_{\nu\in \roi{0}{\infty}}\int_{\oi{0}{\infty}} \norm{f(t)}^2\e^{-2\nu t} \d t\leq C.
\]
we infer again by the monotone convergence theorem that $t\mapsto \lim_{\nu\to 0}\norm{f(t)}^2\e^{-2\nu t}=\norm{f(t)}^2$ defines a function in $L_1(0,\infty)$, showing the remaining assertion.
\end{proof}

For the proof of the Paley--Wiener theorem we need a suitable space
of holomorphic functions on the right half-plane, the so-called \index{Hardy space}\emph{Hardy
space $\cHt(\C_{\Re>\nu};H)$}, which we introduce in the following.

\begin{defn*}
For $\nu\in\R$ we define the Hardy space
\[
\cHt(\C_{\Re>\nu};H)\coloneqq\set{g\from\C_{\Re>\nu}\to H}{g\text{ holomorphic, }\sup_{\rho>\nu}\intop_{\R}\norm{g(\i t+\rho)}_{H}^{2}\d t<\infty}
\]
and equip it with the norm $\norm{\cdot}_{\cHt(\C_{\Re>\nu};H)}$ defined by
\[
\norm{g}_{\cHt(\C_{\Re>\nu};H)}\coloneqq\sup_{\rho>\nu}\left(\intop_{\R}\norm{g(\i t+\rho)}_{H}^{2}\d t\right)^{\frac{1}{2}}.
\]
\end{defn*}

We motivate the Theorem of Paley--Wiener first. For this, let $f\in \Lnu(\R_{\geq0};H)$ and define its \emph{Laplace transform}\index{Laplace transform}
as
\begin{equation}\label{eq:Lapl}
  \C_{\Re>\nu}\ni  z\mapsto  \mathcal{L}f(z) \coloneqq \frac{1}{\sqrt{2\pi}}\int_{0}^\infty f(t)\e^{-z t}\d t.
\end{equation}Note that $\mathcal{L}f(z)=\mathcal{L}_{\Re z}f(\Im z)$ for all $z\in \C_{\Re>\nu}$ due to the support constraint on $f$. Moreover, it is not difficult to see that the integral on the right-hand side of \prettyref{eq:Lapl} exists as $\bigl(t\mapsto\e^{-\rho t}f(t)\bigr)\in L_{1}(\R_{\geq0};H)\cap\L(\R_{\geq0};H)$
for all $\rho>\nu$.
Hence, $\mathcal{L}f\from \C_{\Re>\nu}\to H$ is holomorphic (cf.~\prettyref{exer:analytic_on_strip}). Moreover, by
\prettyref{lem:pos_support}
\begin{align*}
\sup_{\rho>\nu}\norm{\mathcal{L}f(\i\cdot+\rho)}_{\L(\R;H)} & =\sup_{\rho>\nu}\norm{\mathcal{L}_{\rho}f}_{\L(\R;H)}
 =\sup_{\rho>\nu}\norm{f}_{\Lm{\rho}(\R;H)}
 =\sup_{\rho>0}\norm{\e^{-\nu\cdot}f}_{\Lm{\rho}(\R;H)}\\
 & =\norm{\e^{-\nu\cdot}f}_{\L(\R;H)} = \norm{f}_{\Lnu(\R;H)},
\end{align*}
which proves that 
\begin{align*}
\mathcal{L}\from \Lnu(\R_{\geq0};H) & \to\cHt(\C_{\Re>\nu};H)\\
f & \mapsto\bigl(z\mapsto\left(\mathcal{L}_{\Re z}f\right)(\Im z)\bigr)
\end{align*}
is well-defined and isometric. It turns out that $\mathcal{L}$ is actually surjective, see \prettyref{cor:Lapalce_unitary} below. The difficult surjectivity statement is contained in the following Theorem of Paley--Wiener\index{Theorem of Paley--Wiener}, \cite{Paley_Wiener}.
We mainly follow the proof given in \cite[19.2 Theorem]{rudin1987real}.

\begin{thm}[Paley--Wiener]
\label{thm:Paley-Wiener}
Let $g\in\cHt(\C_{\Re>0};H)$. Then there exists an $f\in\L(\R_{\geq0};H)$ such that
\[
\mathcal{L}_{\nu}f=g(\i\cdot+\nu)\quad(\nu>0).
\]
\end{thm}

\begin{proof}
For $\nu>0$ we set $g_{\nu}\coloneqq g(\i\cdot+\nu)\in\L(\R;H)$
and $f_{\nu}\coloneqq\mathcal{F}^{\ast}g_{\nu}\in\L(\R;H)$. Moreover,
we set $f\coloneqq\e^{(\cdot)}f_{1}$. We first prove that $f\in\bigcap_{\nu>0}\Lnu(\R;H)$
with $\sup_{\nu>0}\norm{f}_{\Lnu(\R;H)}<\infty.$ For doing so, let
$a>0$, $\rho>0$ and $x\in\R$. Applying Cauchy's integral theorem
to the function $z\mapsto\e^{zx}g(z)$ and the curve $\gamma$, as indicated in Figure \ref{fig:Cauchy-curve}, we obtain 
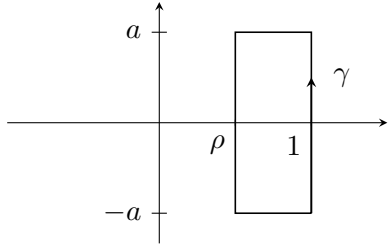
\begin{figure}[htb]
  \centering
  \begin{tikzpicture}[>=stealth,xscale=2,yscale=0.8]
    \draw[->] (-1,0)--(1.5,0);
    \draw[->] (0,-2)--(0,2);
    \def\A{1.5}
    \def\RHO{0.5}
    \draw (\RHO,0.05)--(\RHO,-0.05) node[below left]{$\rho$};
    \draw (1,0.05)--(1,-0.05)node[below left]{$1$};
    \draw (0.05,-\A)--(-0.05,-\A) node[left]{$-a$};
    \draw (0.05,\A)--(-0.05,\A) node[left]{$a$};
    \draw[->,semithick] (1,-\A)--(1,\A/2);
    \draw[semithick] (1,-\A)--(1,\A)--(\RHO,\A)--(\RHO,-\A)--(1,-\A);
    \draw (1.2,\A/2) node{$\gamma$};
  \end{tikzpicture}
  \caption{Curve $\gamma$.}
  \label{fig:Cauchy-curve}
\end{figure}
\begin{align}
0 & = \i\intop_{-a}^{a}\e^{(\i t+1)x}g(\i t+1)\d t-\intop_{\rho}^{1}\e^{(\i a+\kappa)x}g(\i a+\kappa)\d\kappa \label{eq:Cauchy} \\
& -\i\intop_{-a}^{a}\e^{(\i t+\rho)x}g(\i t+\rho)\d t+\intop_{\rho}^{1}\e^{(-\i a+\kappa)x}g(-\i a+\kappa)\d\kappa.\nonumber 
\end{align}
Moreover, since 
\begin{align*}
\intop_{\R}\norm{\intop_{\rho}^{1}\e^{(\pm\i a+\kappa)x}g(\pm\i a+\kappa)\d\kappa}_{H}^{2}\d a
& \leq \intop_{\R}\abs{\intop_{\rho}^{1}\abs{\e^{(\pm\i a+\kappa)x}}^2\d \kappa \intop_\rho^1 \norm{g(\pm\i a+\kappa)}_H^2\d\kappa}\d a\\
& \leq \abs{\intop_{\rho}^{1}\e^{2\kappa x}\d \kappa}  \abs{\intop_\rho^1 \intop_{\R}\norm{g(\pm\i a+\kappa)}_H^2 \d a \d\kappa}\\
& \leq\abs{\int_{\rho}^{1}\e^{2\kappa x}\d\kappa} \abs{1-\rho}\norm{g}_{\cHt(\C_{\Re>0};H)}^{2} <\infty,
\end{align*}
we infer that $\left(a\mapsto\intop_{\rho}^{1}\e^{(\pm\i a+\kappa)x}g(\pm\i a+\kappa)\d\kappa\right)\in\L(\R;H)$
and thus, we find a sequence $(a_{n})_{n\in\mathbb{N}}$ in $\R_{>0}$
such that $a_{n}\to\infty$ and 
\[
\intop_{\rho}^{1}\e^{(\pm\i a_{n}+\kappa)x}g(\pm\i a_{n}+\kappa)\d\kappa\to0
\]
as $n\to\infty$. Hence, using \prettyref{eq:Cauchy} with $a$ replaced
by $a_{n}$ and letting $n$ tend to infinity, we derive that
\[
\intop_{-a_{n}}^{a_{n}}\e^{(\i t+1)x}g(\i t+1)\d t-\intop_{-a_{n}}^{a_{n}}\e^{(\i t+\rho)x}g(\i t+\rho)\d t\to0\quad(n\to\infty).
\]
Noting that for each $\mu>0$ we have
\[
\intop_{-a_{n}}^{a_{n}}\e^{(\i t+\mu)x}g(\i t+\mu)\d t=\sqrt{2\pi}\e^{\mu x}\mathcal{F}^{\ast}(\1_{[-a_{n},a_{n}]}g_{\mu})(x) \quad(x\in\R)
\]
and that $\1_{[-a_{n},a_{n}]}g_{\mu}\to g_{\mu}$ in $\L(\R;H)$
as $n\to\infty$, we may choose a subsequence (again denoted by $(a_n)$)
such that
\begin{align*}
0= & \lim_{n\to\infty}\left(\,\intop_{-a_{n}}^{a_{n}}\e^{(\i t+1)x}g(\i t+1)\d t-\intop_{-a_{n}}^{a_{n}}\e^{(\i t+\rho)x}g(\i t+\rho)\d t\right)\\
= & \lim_{n\to\infty}\left(\sqrt{2\pi}\e^{x}\mathcal{F}^{\ast}(\1_{[-a_{n},a_{n}]}g_{1})(x)-\sqrt{2\pi}\e^{\rho x}\mathcal{F}^{\ast}(1_{[-a_{n},a_{n}]}g_{\rho})(x)\right)\\
= & \sqrt{2\pi}\bigl(\e^{x}f_{1}(x)-\e^{\rho x}f_{\rho}(x)\bigr)
\end{align*}
for almost every $x\in\R$. Hence, $f=\e^{(\cdot)}f_{1}=\exp(\rho\m)f_{\rho}$
for each $\rho>0$ and thus, 
\[
\intop_{\R}\norm{f(t)}_{H}^{2}\e^{-2\rho t}\d t=\intop_{\R}\norm{f_{\rho}(t)}_{H}^{2}\d t<\infty
\]
which shows $f\in\bigcap_{\rho>0}\Lm{\rho}(\R;H)$ with 
\[
\sup_{\rho>0}\norm{f}_{\Lm{\rho}(\R;H)}=\sup_{\rho>0}\norm{f_{\rho}}_{\L(\R;H)}=\sup_{\rho>0}\norm{g_{\rho}}_{\L(\R;H)}=\norm{g}_{\cHt(\mathbb{C}_{\Re>0};H)}.
\]
Thus, $f\in\L(\R_{\geq0};H)$ with $\norm{f}_{\L(\R_{\geq0};H)}=\norm{g}_{\cHt(\mathbb{C}_{\Re>0};H)}$
by \prettyref{lem:pos_support}. Moreover, 
\[
\mathcal{L}_{\nu}f=\mathcal{F}\exp(-\nu\m)f=\mathcal{F}\exp(-\nu m)\exp(\nu \m)f_{\nu}=\mathcal{F}f_{\nu}=g_{\nu}=g(\i\cdot+\nu)
\]
for each $\nu>0$, which shows the representation formula for $g$.
\end{proof}

Summarising the results of \prettyref{thm:Paley-Wiener} and the arguments carried out just before \prettyref{thm:Paley-Wiener}, we obtain the following statement.

\begin{cor}
\label{cor:Lapalce_unitary}Let $\nu\in\R$. Then the mapping 
\begin{align*}
\mathcal{L}\from \Lnu(\R_{\geq0};H) & \to\cHt(\C_{\Re>\nu};H)\\
f & \mapsto\bigl(z\mapsto\left(\mathcal{L}_{\Re z}f\right)(\Im z)\bigr)
\end{align*}
is an isometric isomorphism. In particular, $\cHt(\C_{\Re>\nu};H)$ is
a Hilbert space.
\end{cor}

\begin{proof}
We have argued already that  $\mathcal{L}$ is well-defined and isometric. Thus, we show that $\mathcal{L}$ is onto, next. For this, let $g\in\cHt(\C_{\Re>\nu};H)$ and define $\tilde{g}(z)\coloneqq g(z+\nu)$
for $z\in\C_{\Re>0}$. Then $\tilde{g}\in\cHt(\C_{\Re>0};H)$ and
thus, \prettyref{thm:Paley-Wiener} yields the existence of $\tilde{f}\in\L(\R_{\geq0};H)$
with 
\[
g(\i\cdot+\rho)=\tilde{g}(\i\cdot+\rho-\nu)=\mathcal{L}_{\rho-\nu}\tilde{f}=\mathcal{L}_{\rho}\bigl(\e^{\nu\cdot}\tilde{f}\,\bigr)\quad(\rho>\nu).
\]
Hence, setting $f\coloneqq\e^{\nu\cdot}\tilde{f}\in\Lnu(\R_{\geq0};H)$,
we obtain $\mathcal{L}f=g$.
\end{proof}
We can now provide an alternative proof of \prettyref{thm:material law independent of nu}
by proving causality with the help of the Theorem of Paley--Wiener.
\begin{prop}
\label{prop:M_causal_via_PW}Let $M\from\dom(M)\subseteq\C\to L(H)$ be
a material law. Then for $\nu>\sbb{M}$ we have $M(\td{\nu})\in \bo(\Lnu(\R;H))$ and $M(\td{\nu})$
is causal and autonomous (see \prettyref{exer:causality_autonomous}).
\end{prop}

\begin{proof}
Let $\nu>\sbb{M}$. Then $M\from\C_{\Re\geq\nu}\to L(H)$ is bounded and
holomorphic on $\C_{\Re>\nu}$. Hence, by unitary equivalence, $M(\td{\nu})\in L(\Lnu(\R;H))$.
Moreover, $M(\td{\nu})$ is autonomous by \prettyref{exer:causality_autonomous}.
Thus, for causality it suffices to check that $\spt M(\td{\nu})f\subseteq\R_{\geq0}$
whenever $f\in\Lnu(\R_{\geq0};H)$. So let $f\in\Lnu(\R_{\geq0};H)$.
Then $\mathcal{L}f\in\cHt(\C_{\Re>\nu};H)$ by \prettyref{cor:Lapalce_unitary}
and since $M$ is bounded and holomorphic on $\C_{\Re>\nu}$, we infer also
that
\[
\bigl(z\mapsto M(z)\left(\mathcal{L}f\right)(z)\bigr)\in\cHt(\C_{\Re>\nu};H).
\]
Again by \prettyref{cor:Lapalce_unitary} there exists $g\in\Lnu(\R_{\geq0};H)$
such that 
\[
\mathcal{L}g(z)=M(z)\left(\mathcal{L}f\right)(z)\quad(z\in\C_{\Re>\nu}).
\]
Thus, in particular
\[
\mathcal{L}_{\rho}g=M(\i\m+\rho)\mathcal{L}_{\rho}f\quad(\rho>\nu).
\]
Since $f,g\in\Lnu(\R_{\geq0};H)$ we infer that $\mathcal{L}_{\rho}g\to\mathcal{L}_{\nu}g$
and $\mathcal{L}_{\rho}f\to\mathcal{L}_{\nu}f$ in $\L(\R;H)$ as
$\rho\to\nu$ by dominated convergence. Moreover, $M(\i\m+\rho)\to M(\i\m+\nu)$
strongly on $\L(\R;H)$ as $\rho\to\nu$ (cf.~\prettyref{exer:strong_conv_material law}).
Hence, we derive 
\[
\mathcal{L}_{\nu}g=M(\i\m+\nu)\mathcal{L}_{\nu}f,
\]
and thus, $g=M(\td{\nu})f$ which shows the causality.
\end{proof}

\section{A Representation Result}

In this section we argue that our solution theory needs holomorphy
as a central property for the material law. There are two key properties for rendering $T\in L(\Lm{\nu_{0}}(\R;H))$ a material law operator. The first one is causality\index{causal} (i.e., $\1_{\loi{-\infty}{a}}(\m)T\1_{\loi{-\infty}{a}}(\m)= \1_{\loi{-\infty}{a}}(\m)T$ for all $a\in \R$) and, secondly, $T$ needs to be autonomous\index{autonomous} (i.e., $\tau_h T =T\tau_h$ for all $h\in \R$ where $\tau_hf=f(\cdot +h)$). The main theorem of this
section reads as follows:
\begin{thm}
\label{thm:representaion} Let $\nu_{0}\in\R$ and let $T\in L(\Lm{\nu_{0}}(\R;H))$
be causal and autonomous. Then $T|_{\Lm{\nu_{0}}\cap\Lnu}$ has a
unique extension $T_{\nu}\in L(\Lnu(\R;H))$ for each $\nu>\nu_{0}$
and there exists a unique $M\from\C_{\Re>\nu_{0}}\to L(H)$ holomorphic and
bounded such that $T_{\nu}=M(\td{\nu})$ for each $\nu>\nu_{0}$.
\end{thm}

We consider the following (shifted) variant of \prettyref{thm:representaion} first.
\begin{thm}
\label{thm:weiss}Let $T\in L(\L(\R;H))$ be causal and autonomous.
Then there exists $M\from\C_{\Re>0}\to L(H)$ holomorphic and bounded such
that 
\[
\left(\mathcal{L}Tf\right)(z)=M(z)\left(\mathcal{L}f\right)(z)\quad(f\in\L(\R_{\geq0};H),z\in\C_{\Re>0}).
\]
\end{thm}

\begin{proof}
For $s >0$ and $x\in H$ we define $f_{x,s}\coloneqq \1_{\oi{0}{s}}x$ and compute
\begin{equation}\label{eq:hendrik0}
   \mathcal{L}f_{x,s}(z) = \frac{1}{\sqrt{2\pi}}\int_0^s \e^{-z t}x \d t =\frac{1}{\sqrt{2\pi}} \frac{1-\e^{-z s}}{z}x \quad(z\in \C_{\Re>0}).
\end{equation}
We define $M\colon \C_{\Re>0}\to L(H)$ via
\[
    M(z)x\coloneqq \frac{\sqrt{2\pi}z}{1-\e^{-z}}\mathcal{L}Tf_{x,1}(z),
\]which is well-defined since $\spt Tf_{x,1} \subseteq [0,\infty)$ (use causality of $T$); $M(z)\in L(H)$, since $T$ is bounded. Also, $M(\cdot)x$ is evidently holomorphic for every $x\in H$ as a product of two holomorphic mappings and thus by \prettyref{exer:holomorphic}, $M$ is holomorphic itself. Next, we show that for all $z\in \C_{\Re>0}$ and $f\in \L(\R_{\geq 0};H)$, we have
\begin{equation}\label{eq:hendrik}
  \left(\mathcal{L}Tf\right)(z)=M(z)\left(\mathcal{L}f\right)(z).
\end{equation}
By definition of $M$, the equality is true for $f$ replaced by $f_{x,1}$, $x\in H$. Next, observe that $\lin\set{\1_{\oi{a}{a+1/n}}x}{a\geq 0, n\in \N, x\in H}$ is dense in $\L(\R_{\geq0};H)$. Hence, for \prettyref{eq:hendrik}, it suffices to show
\begin{equation}\label{eq:hendrik2}
   \left(\mathcal{L}T\1_{\oi{a}{a+1/n}}x\right)(z)=M(z)\left(\mathcal{L}\1_{\oi{a}{a+1/n}}x\right)(z)
\end{equation} for all $a\geq 0$, $n\in \N$, $x\in H$, and $z\in \C_{\Re>0}$. Next, using that $T$ is autonomous in the situation of \prettyref{eq:hendrik2}, we see $\left(T\1_{\oi{a}{a+1/n}}x\right)=\left(T\tau_{-a}\1_{\oi{0}{1/n}}x\right)=\tau_{-a}\left(T\1_{\oi{0}{1/n}}x\right)$ and, by a straightforward computation, $(\mathcal{L}\tau_{-a}f)(z)=\e^{-za}\mathcal{L}f(z)$ for all $f\in \L(\R_{\geq 0};H)$. Thus,
\[
  \left(\mathcal{L}T\1_{\oi{a}{a+1/n}}x\right)(z) = \e^{-za}\left(\mathcal{L}T\1_{\oi{0}{1/n}}x\right)(z),
\]which yields that it suffices to show \prettyref{eq:hendrik2} for $a=0$ only, that is, for $f=f_{x,1/n}$. Furthermore, we compute for $n\in\N$ and $z\in \C_{\Re>0}$
\begin{align*}
  \mathcal{L}Tf_{x,1}(z)& = \sum_{k=0}^{n-1} (\mathcal{L}T\1_{\oi{k/n}{(k+1)/n}}x)(z)\\
  & =  \sum_{k=0}^{n-1} \e^{-zk/n}(\mathcal{L}T\1_{\oi{0}{1/n}}x)(z) =  \frac{1-\e^{-z}}{1-\e^{-z/n}}(\mathcal{L}Tf_{x,1/n})(z).
\end{align*}
Thus, using \prettyref{eq:hendrik0} for $s=1/n$, we deduce from the definition of $M$,
\begin{align*}
   \mathcal{L}Tf_{x,1/n}(z) &= \frac{1-\e^{-z/n}}{\sqrt{2\pi}z}\frac{\sqrt{2\pi}z}{1-\e^{-z}}  \mathcal{L}Tf_{x,1}(z) \\
   & =  \frac{1-\e^{-z/n}}{\sqrt{2\pi}z}M(z)x \\
   & = M(z) \mathcal{L}f_{x,1/n}(z).
\end{align*}
Hence, \prettyref{eq:hendrik} holds for all $f\in \L(\R_{\geq0};H)$. It remains to show boundedness of $M$. For this, let $z\in \C_{\Re>0}$ and $x\in H$. Set $f\coloneqq \1_{\roi{0}{\infty}}\e^{-z^*}x$ as well as $c\coloneqq 2\Re z\sqrt{2\pi}$. Then
\[
  \mathcal{L}f(z) = \frac{1}{\sqrt{2\pi}}\int_0^\infty \e^{-zt-z^*t}x\d t = \frac{x}{c}.
\]
By virtue of \prettyref{eq:hendrik}, we get $\mathcal{L}Tf(z)=M(z)\mathcal{L}f(z)$ and thus $M(z)x=c\mathcal{L}Tf(z)$. This leads to
\begin{align*}
   \norm{M(z)x} & \leq \frac{c}{\sqrt{2\pi}}\int_0^\infty \norm{\e^{-zt}Tf(t)}\d t\leq \frac{c}{\sqrt{2\pi}} \norm{\1_{\roi{0}{\infty}}\e^{-z(\cdot)}}_{\L(\R)}\norm{Tf}_{\L(\R)} \\
   & \leq \frac{c}{\sqrt{2\pi}} \norm{\1_{\roi{0}{\infty}}\e^{-z(\cdot)}}_{\L(\R)}^2\norm{T}_{\bo(\L(\R;H))}\norm{x}_H = \norm{T}_{\bo(\L(\R;H))}\norm{x}_H,
\end{align*}
where we used that $\norm{f}_{\L(\R;H)}= \norm{\1_{\roi{0}{\infty}}\e^{-z(\cdot)}}_{\L(\R)}\norm{x}_H$. Thus, $\norm{M(z)}\leq \norm{T}$, which yields boundedness of $M$ and the assertion of the theorem.
\end{proof}

We can now prove our main result of this section.

\begin{proof}[Proof of \prettyref{thm:representaion}]
We just prove the existence of a function $M$. The proof of its
uniqueness is left as \prettyref{exer:Uniqueness_material_law}.\\
We first prove the assertion for $\nu_{0}=0$. So, let $T\in L(\L(\R;H))$ be
causal and autonomous. According to \prettyref{thm:weiss} we find
$M\from\C_{\Re>0}\to L(H)$ holomorphic and bounded such that 
\[
\left(\mathcal{L}Tf\right)(z)=M(z)\left(\mathcal{L}f\right)(z)\quad(f\in\L(\R_{\geq0};H),z\in\C_{\Re>0}).
\]
Let now $\varphi\in \cci(\R;H)$ and set $a\coloneqq\inf\spt\varphi$.
Then $\tau_{a}\varphi\in\L(\R_{\geq0};H)$, and for $\nu>0$ we compute
\begin{align}
\mathcal{L}_{\nu}T\varphi & =\mathcal{L}_{\nu}\tau_{-a}T\tau_{a}\varphi
 =\e^{-(\i\m+\nu)a}\mathcal{L}_{\nu}T\tau_{a}\varphi
 =\e^{-(\i\m+\nu)a}M(\i\m+\nu)\mathcal{L}_{\nu}\tau_{a}\varphi \nonumber\\
 & =M(\i\m+\nu)\mathcal{L}_{\nu}\varphi.\label{eq:unitar_equiv_on_dense_set}
\end{align}
 The latter implies 
\begin{align*}
\norm{T\varphi}_{\Lnu(\R;H)} & =\norm{\mathcal{L}_{\nu}T\varphi}_{\L(\R;H)}
 =\norm{M(\i\m+\nu)\mathcal{L}_{\nu}\varphi}_{\L(\R;H)}
 \leq\norm{M}_{\infty,\C_{\Re>0}}\norm{\varphi}_{\Lnu(\R;H)}
\end{align*}
and hence, $T|_{\cci(\R;H)}$ has a unique continuous extension
$T_{\nu}\in L(\Lnu(\R;H))$. Using \prettyref{eq:unitar_equiv_on_dense_set}
we obtain 
\[
T_\nu=\mathcal{L}_{\nu}^{\ast}M(\i\m+\nu)\mathcal{L}_{\nu}=M(\td{\nu})
\]
by approximation.

Let now $\nu_{0}\in\R$. Then the operator 
\[
\tilde{T}\coloneqq\e^{-\nu_{0}\m}T\e^{\nu_{0}\m}\in L(\L(\R;H))
\]
is causal and autonomous as well. Thus, $\tilde{T}|_{\cci(\R;H)}$
has continuous extensions $\tilde{T}_{\rho}\in L(\Lm{\rho}(\R;H))$
for each $\rho>0$ and there is $\tilde{M}\from\C_{\Re>0}\to L(H)$ holomorphic
and bounded such that $\tilde{T}_{\rho}=\tilde{M}(\td{\rho})$ for
each $\rho>0$. Using $T|_{\cci(\R;H)}=\e^{\nu_{0}\m}\tilde{T}|_{\cci(\R;H)}\e^{-\nu_{0}\m}$,
we derive that $T|_{\cci(\R;H)}$ has the unique continuous
extension $T_{\nu}=\e^{\nu_{0}\m}\tilde{T}_{\nu-\nu_{0}}\e^{-\nu_{0}\m}\in L(\Lnu(\R;H))$
for each $\nu>\nu_{0}$ and 
\begin{align*}
\mathcal{L}_{\nu}T_{\nu} & =\mathcal{L}_{\nu}\e^{\nu_{0}\m}\tilde{T}_{\nu-\nu_{0}}\e^{-\nu_{0}\m}
 =\mathcal{L}_{\nu-\nu_{0}}\tilde{T}_{\nu-\nu_{0}}\e^{-\nu_{0}\m}
 =\tilde{M}(\i\m+\nu-\nu_{0})\mathcal{L}_{\nu-\nu_{0}}\e^{-\nu_{0}\m}\\
 & =\tilde{M}(\i\m+\nu-\nu_{0})\mathcal{L}_{\nu}.
\end{align*}
Hence, 
\[
T_{\nu}=M(\td{\nu})
\]
for the holomorphic and bounded function $M$ given by $M(z)\coloneqq\tilde{M}(z-\nu_{0})$
for $z\in\C_{\Re>\nu_{0}}$.
\end{proof}

\section{Comments}

The stated Theorem of Paley and Wiener is of course not the only theorem
characterising properties of the support of $L_{2}$-functions in
terms of their Fourier or Laplace transform. For instance, a similar
result holds for functions having compact support, see e.g. \cite[19.3 Theorem]{rudin1987real} and \prettyref{exer:PW_compact}.
These theorems provide a nice connection between $L_{2}$-functions
and spaces of holomorphic functions in the form of the so-called Hardy spaces. In this
chapter we just introduced the Hardy space $\cHt$ and it is not surprising
that there are also the Hardy spaces $\mathcal{H}_{p}$ for $1\leq p\leq\infty$.
We refer to \cite{Duren1970} for this topic.

The representation result presented in the second part of this chapter
was originally proved by Fourès and Segal in 1955, \cite{Foures1955}.
In this article the authors prove an analogous representation result
for causal operators on $\L(\R^{d};H)$, where causality is defined
with respect to a closed and convex cone on $\R^{d}$. The quite elementary
proof of \prettyref{thm:weiss} for $d=1$ presented here was kindly communicated to us by Hendrik Vogt. 

\section*{Exercises}
\addcontentsline{toc}{section}{Exercises}

\begin{xca}
\label{exer:dense in L1}Let $\Lambda\subseteq\R_{>0}$ be a set with
an accumulation point in $\R_{>0}$. Prove that $\{\left(x\mapsto\e^{-\lambda x}\right)\,;\,\lambda\in\Lambda\}$
is a total set in $L_{1}(\R_{\geq0})$.

Hint: Use that the set is total if and only if 
\[
\forall f\in L_{\infty}(\R_{\geq0}):\:\left(\forall\lambda\in\Lambda:\,\intop_{\R_{\geq0}}\e^{-\lambda x}f(x)\d x=0\Rightarrow f=0\right).
\]
\end{xca}


\begin{xca}
\label{exer:strong_conv_material law} Let $M\from\dom(M)\subseteq\C\to L(H)$
be a material law. Moreover, let $\nu>\sbb{M}$. Show that $\lim_{\rho\to\nu\rlim}M(\i\m+\rho)=M(\i\m+\nu)$
where the limit is meant in the strong operator topology on $\L(\R;H)$. 
\end{xca}


\begin{xca}
\label{exer:Uniqueness_material_law} Prove the uniqueness statement
in \prettyref{thm:representaion}.
\end{xca}

\begin{xca}
\label{exer:noncausal_M} Give an example of a continuous and bounded
function $M\from\C_{\Re>0}\to L(H)$ such that the corresponding operator
$M(\td{\nu})$ is not causal for any $\nu>0$.
\end{xca}

\begin{xca}
\label{exer:PW_distributional} Prove the following distributional
variant of the Paley--Wiener theorem:
Let $\nu_0>0$, $k\in\N$, $f\from\C_{\Re>\nu_0}\to\C$, and set $h(z)\coloneqq\frac{1}{z^{k}}f(z)$
for $z\in\C_{\Re>\nu_0}$. We assume that $h\in\cHt(\C_{\Re>\nu_0};\C)$.
For $\nu>\nu_0$ we define the distribution $u\from \cci(\R)\to\C$
by
\[
u(\psi)\coloneqq\scp{\mathcal{L}_{\nu}^{\ast}h(\i\cdot+\nu)}{(\td{\nu}^{\ast})^k\psi}_{\Lnu(\R;\C)}\quad(\psi\in\cci(\R;\C)).
\]
Prove that $\spt u\subseteq\R_{\geq0}$, where 
\[
\spt u\coloneqq\R\setminus\bigcup\set{U\subseteq\R\text{ open}}{\forall\,\psi\in \cci(U;\C): u(\psi)=0}.
\]

What is $u$ if $f=\1_{\C_{\Re>\nu_0}}$?
\end{xca}

\begin{xca}
 \label{exer:PW_compact1}
 Let $g\in \L(\R),a>0$ such that $\spt g\subseteq \ci{-a}{a}$. Show that $f\coloneqq \mathcal{F}g$ extends to a holomorphic function $\tilde{f}\from\C\to \C$ with $\tilde{f}(\i t)=f(t)$ for each $t\in \R$ such that
 \[
  \exists C\geq 0\,\forall z\in \C:\, |f(z)|\leq C\e^{a|\Re z|}.
 \]
\end{xca}

\begin{xca}
 \label{exer:PW_compact}
 Let $f:\C\to \C$ be holomorphic such that
 \begin{enumerate}
  \item\label{exer:PW_compact:item:1} $\exists C\geq 0,\, a>0\, \forall z\in \C:\, |f(z)|\leq C\e^{a|\Re z|}$,
  \item\label{exer:PW_compact:item:2} $f(\i \cdot)\in \L(\R)$.
 \end{enumerate}
Prove that $g\coloneqq \mathcal{F}^\ast f(\i\cdot)$ satisfies $\spt g\subseteq \ci{-a}{a}$.

Hint: Apply \prettyref{thm:Paley-Wiener} to the function $h:\C_{\Re>0} \to \C$ given by 
\[
 h(z)\coloneqq \e^{-z a} \frac{f(z)}{z+1}\quad (z\in \C_{\Re>0})
\]
to derive that $\spt g\subseteq \R_{\geq -a}$.

Note: The assertion even holds true if one replaces condition \ref{exer:PW_compact:item:1} by
\[
 \exists C\geq 0,\, a>0\, \forall z\in \C:\, |f(z)|\leq C\e^{a|z|}.
\]
\end{xca}


\printbibliography[heading=subbibliography]

\chapter{Initial Value Problems and Extrapolation Spaces}

Up until now we have dealt with evolutionary equations of the form
\[
\bigl(\overline{\td{\nu}M(\td{\nu})+A}\bigr)U=F
\]
for some given $F\in\Lnu(\R;H)$ for some Hilbert space $H$, a skew-selfadjoint
operator $A$ in $H$ and a material law $M$ defined on a suitable
half space satisfying an appropriate positive definiteness condition with
$\nu\in\R$ chosen suitably large. Under these conditions, we
 established that the solution operator, $S_{\nu}\coloneqq\bigl(\overline{\td{\nu}M(\td{\nu})+A}\bigr)^{-1}\in\bo(\Lnu(\R;H))$,
is eventually independent of $\nu$ and causal; that is, if $F=0$
on $\loi{-\infty}{a}$ for some $a\in\R$, then so too is $U$. 

Solving for $U\in\Lnu(\R;H)$ for some non-negative $\nu$ penalises
$U$ having support on $\Rle{0}$. This might be interpreted as an
implicit \emph{initial condition at $-\infty$.} In this chapter,
we shall study how to obtain a solution for initial value problems with an
initial condition at $0$, based on the solution theory developed
in the previous chapters.

\section{What are Initial Values?}

This section is devoted to the motivation of the framework to follow in the
subsequent section. Let us consider the following, arguably easiest
but not entirely trivial, initial value problem: find a `causal' $u\colon\R\to\R$
such that for $u_{0}\in\R$ we have
\begin{equation}\label{eq:IVP_initial}
\begin{cases}
u'(t)=0 & (t>0),\\
u(0)=u_{0}.
\end{cases}
\end{equation}
First of all note that there is no condition for $u$ on $\Rl{0}$.
Since, there is no source term or right-hand side supported on $\Rl{0}$,
causality would imply that $u=0$ on $\oi{-\infty}{0}$. Moreover,
$u=c$ for some constant $c\in\R$ on $\oi{0}{\infty}$. Thus, in
order to match with the initial condition, 
\[
u(t)=u_{0}\1_{\roi{0}{\infty}}(t)\quad(t\in\R).
\]
Notice also that $u$ is not continuous. Hence, by the Sobolev embedding
theorem (\prettyref{thm:Sobolev_emb}), $u\notin\bigcup_{\nu>0}\dom(\td{\nu})$.
\begin{prop}
Let $H$ be a Hilbert space, $u_{0}\in H$. Define 
\begin{align*}
\delta_{0}u_{0}\colon\cci(\R;H) & \to\K\\
f & \mapsto\scp{u_{0}}{f(0)}_{H}.
\end{align*}
Then, for all $\nu\in\R_{>0}$, $\delta_{0}u_{0}$ extends to a continuous
linear functional on $\dom(\td{\nu})$. Re-using the notation for
this extension, for all $f\in\dom(\td{\nu})$ we have
\begin{equation}
\left(\delta_{0}u_{0}\right)(f)=-\scp{\1_{\roi{0}{\infty}}u_{0}}{\left(\td{\nu}-2\nu\right)f}_{\Lnu(\R;H)}.
\label{eq:delta_ibp}
\end{equation}
\end{prop}

\begin{proof}
The equality \prettyref{eq:delta_ibp} is obvious for $f\in\cci(\R;H)$ as it is a
direct consequence of the fundamental theorem of calculus (look at the right-hand side first). The continuity
of $\delta_{0}u_{0}$ follows from the Cauchy--Schwarz inequality
applied to the right-hand side of \prettyref{eq:delta_ibp}. Note that
$\1_{\roi{0}{\infty}}u_{0}\in\Lnu(\R;H)$.
\end{proof}
Recall from \prettyref{cor:adjoint_time_derivative} that 
\[
\td{\nu}^{*}=-\td{\nu}+2\nu.
\]
Hence, if we \emph{formally} apply this formula to \prettyref{eq:delta_ibp},
we obtain
\[
\scp{\td{\nu}\1_{\roi{0}{\infty}}u_{0}}{f}=\scp{\1_{\roi{0}{\infty}}u_{0}}{\td{\nu}^{*}f}=\left(\delta_{0}u_{0}\right)(f).
\]
Therefore, in order to use the introduced time derivative operator
for the above initial value problem, we need to extend the time derivative
to a broader class of functions than just $\dom(\td{\nu})$. To utilise the adjoint operator in this way will be central to the construction to
follow. It will turn out that indeed
\[
\td{\nu}\1_{\roi{0}{\infty}}u_{0}=\delta_{0}u_{0}.
\]
Moreover, we shall show below that 
\[
\td{\nu}u=\delta_{0}u_{0}
\]
considered on the full time-line $\R$ is one possible replacement
of the initial value problem \prettyref{eq:IVP_initial}.

\section{Extrapolating Operators}

Since we are dealing with functionals, let us recall the definition
of the dual space.
Throughout this section let $H,H_{0},H_{1}$ be Hilbert spaces.

\begin{defn*}
The space
\[
H'\coloneqq\set{\varphi\from H\to\K}{\varphi\text{ linear and bounded}}
\]
is called the \emph{dual space of $H$}. We equip $H'$ with the linear
structure
\[
(\lambda \odot\varphi+\psi)(x)\coloneqq\lambda^{\ast}\varphi(x)+\psi(x)\quad(\lambda\in\K,\varphi,\psi\in H',x\in H).
\]
\end{defn*}
\begin{rem}
Note that $H'$ is a Hilbert space itself, since by the Riesz representation
theorem for each $\varphi\in H'$ we find a unique element $R_{H}\varphi\in H$
such that 
\[
\forall x\in H:\:\varphi(x)=\scp{R_{H}\varphi}{x}.
\]
Due to the linear structure on $H'$, the so induced mapping $R_{H}\from H'\to H$ (which is one-to-one and onto)
becomes linear and 
\[
H'\times H'\ni(\varphi,\psi)\mapsto\scp{R_{H}\varphi}{R_{H}\psi}
\]
defines an inner product on $H'$, which induces the usual norm on
functionals. 
\end{rem}

From now on we will identify elements $x\in H$ with their representatives in $H'$; that is, we identify $x$ with $R_H^{-1}x$.

Let $C\colon\dom(C)\subseteq H_{0}\to H_{1}$ be linear,
densely defined and closed. We recall that in this case $\dom(C)$
endowed with the graph inner product
\[
(u,v)\mapsto\scp{u}{v}_{H_{0}}+\scp{Cu}{Cv}_{H_{1}}
\]
becomes a Hilbert space. Clearly, $\dom(C)\hookrightarrow H_{0}$
is continuous with dense range. Moreover, we see that $\dom(C)\ni x\mapsto Cx\in H_{1}$
is continuous. We define 
\begin{align*}
  C^{\diamond}\colon H_{1} & \to\dom(C)' \eqqcolon H^{-1}(C),\\
 (C^{\diamond}\phi)(x)& \coloneqq\scp{\phi}{Cx}_{H_{1}}\quad(\phi\in H_{1},x\in\dom(C)).
\end{align*}
Note that $C^{\diamond}$ is related to the dual operator $C'$ of $C$ considered as a bounded operator from $\dom(C)$ to $H_1$ by
\[
 C^\diamond =C'R_{H_1}^{-1}.
\]

\begin{prop}
\label{prop:CdiaCadj} With the notions and definitions from this
section, the following statements hold:
\begin{enumerate}
\item\label{prop:CdiaCadj:item:1} $C^{\diamond}$ is continuous and linear.

\item\label{prop:CdiaCadj:item:2} $C^{*}\subseteq C^{\diamond}$.

\item\label{prop:CdiaCadj:item:3} $\ker(C^{*})=\ker(C^{\diamond})$.

\item\label{prop:CdiaCadj:item:4} $C\subseteq\left(C^{*}\right)^{\diamond}\colon H_{0}\to\dom(C^{*})'=H^{-1}(C^{*})$.

\item\label{prop:CdiaCadj:item:5} $H_{0}\cong H_{0}'\hookrightarrow H^{-1}(C)$ densely and continuously.
\end{enumerate}
\end{prop}

\begin{proof}
\ref{prop:CdiaCadj:item:1} Let $\phi,\psi\in H_1$, $\lambda\in\K$. Then
\[C^\diamond(\lambda\phi+\psi)(x) = \lambda^*(C^\diamond\phi)(x) + (C^\diamond\psi)(x) = (\lambda\odot C^\diamond\phi + C^\diamond\psi)(x) \quad(x\in \dom(C)).\]
To show continuity, let $\phi\in H_{1}$ and $x\in\dom(C)$. Then
\[
\abs{\scp{\phi}{Cx}_{H_{1}}}\leq\norm{\phi}_{H_{1}}\norm{Cx}_{H_{1}}\leq\norm{\phi}_{H_{1}}\norm{x}_{\dom(C)}.
\]
Hence, $\norm{C^{\diamond}}=\sup_{\phi\in H_{1},\norm{\phi}_{H_1}\leq1}\norm{C^{\diamond}\phi}_{\dom(C)'}\leq1.$

\ref{prop:CdiaCadj:item:2} Let $\phi\in\dom(C^{\ast})$. Then we have for all $x\in\dom(C)$
\[
\left(C^{\diamond}\phi\right)(x)=\scp{\phi}{Cx}_{H_1}=\scp{C^{*}\phi}{x}_{H_0}=\left(C^{*}\phi\right)(x).
\]
We obtain $C^{\diamond}\phi=C^{*}\phi$. (Note that a functional on
$H_{0}$ is uniquely determined by its values on $\dom(C)$.)

\ref{prop:CdiaCadj:item:3} Using \ref{prop:CdiaCadj:item:2}, we are left with showing $\ker(C^{\diamond})\subseteq\ker(C^{*})$.
So, let $\phi\in\ker(C^{\diamond})$. Then for all $x\in\dom(C)$
we have
\[
0=\left(C^{\diamond}\phi\right)(x)=\scp{\phi}{Cx}_{H_1},
\]
which leads to $\phi\in\dom(C^{*})$ and $\phi\in\ker(C^{*})$.

\ref{prop:CdiaCadj:item:4} is a direct consequence of \ref{prop:CdiaCadj:item:2} applied to
$C^{*}$.

\ref{prop:CdiaCadj:item:5} Since $\dom(C)\hookrightarrow H_{0}$ is dense and continuous,
we obtain that $H_{0}'\hookrightarrow\dom(C)'$ is so as well; cf.\ \prettyref{exer:dual_embedding}.
\end{proof}
We will also write $C_{-1}\coloneqq\left(C^{*}\right)^{\diamond}$ for the so-called \emph{\index{extrapolated operator}extrapolated operator} of $C$. Then $(C^*)_{-1}=C^{\diamond}$. We will
record the index $-1$ at the beginning, but in order to avoid too
much clutter in the notation we will drop this index again, bearing
in mind that $C_{-1}\supseteq C$ and $\left(C^{*}\right)_{-1}\supseteq C^{*}$.

\begin{example}
We have shown that for all $\nu\in\R$ the operator $\td{\nu}$ is
densely defined and closed. Then for $f\in\Lnu(\R)$ we have for all
$\phi\in\cci(\R)$
\begin{align*}
\left((\td{\nu})_{-1}f\right)(\phi) & =\scp{f}{\td{\nu}^{*}\phi}_{\Lnu}
 =\scp{f}{\left(-\td{\nu}+2\nu\right)\phi}_{\Lnu}
 =-\int_{\R}\scp{f}{(\e^{-2\nu\cdot}\phi)'}_{\C}.
\end{align*}
Hence, $(\td{\nu})_{-1}f$ acts as the `usual' distributional derivative
taking into account the exponential weight in the scalar product.

With this observation we deduce that for $\nu>0$ we have
\[
\left({\td{\nu}}\right)_{-1}\1_{\roi{0}{\infty}}={\td{\nu}}\1_{\roi{0}{\infty}}=\delta_{0}.
\]
Hence, the initial value problem from the beginning reads: find $u$
such that
\[
(\td{\nu})_{-1}u=\delta_{0}u_{0}.
\]
\end{example}

\begin{example}
Let $\Omega\subseteq\R^{d}$ be open. Consider $\grad_{0}\colon H_{0}^{1}(\Omega)\subseteq\L(\Omega)\to\L(\Omega)^{d}$.
We compute $\dive_{-1}\colon\L(\Omega)^{d}\to H^{-1}(\Omega)$ with
$H^{-1}(\Omega)\coloneqq H_{0}^{1}(\Omega)'$. For $q\in\L(\Omega)^{d}$
we obtain for all $\phi\in H_{0}^{1}(\Omega)$
\begin{align*}
\left(\dive_{-1}q\right)(\phi) & =\scp{q}{\dive^{*}\phi}_{\L(\Omega)^d}
 =-\scp{q}{\grad_{0}\phi}_{\L(\Omega)^d}.
\end{align*}
Also, with similar arguments, we see that
\[
\left(\left(\grad\right)_{-1}f\right)(q)=-\scp{f}{\dive_{0}q}_{\L(\Omega)}
\]
for all $f\in\L(\Omega)$ and $q\in H_{0}(\dive,\Omega)$.
\end{example}

We consider a case of particular interest within the framework
of evolutionary equations.
\begin{prop}
\label{prop:A_neg}Let $A\colon\dom(C)\times\dom(C^{*})\subseteq H_{0}\times H_{1}\to H_{0}\times H_{1}$
be given by
\[
A\begin{pmatrix}
\phi\\
\psi
\end{pmatrix}=\begin{pmatrix}
0 & C^{*}\\
-C & 0
\end{pmatrix}\begin{pmatrix}
\phi\\
\psi
\end{pmatrix}=\begin{pmatrix}
C^{*}\psi\\
-C\phi
\end{pmatrix}.
\]
Then $A_{-1}\colon H_{0}\times H_{1}\to H^{-1}(C)\times H^{-1}(C^{*})$
acts as
\[
A_{-1}\begin{pmatrix}
\phi\\
\psi
\end{pmatrix}=\begin{pmatrix}
0 & (C^{*})_{-1}\\
-C_{-1} & 0
\end{pmatrix}\begin{pmatrix}
\phi\\
\psi
\end{pmatrix}=\begin{pmatrix}
(C^{*})_{-1}\psi\\
-C_{-1}\phi
\end{pmatrix}.
\]
\end{prop}

Next, we will look at the solution theory when carried over to
distributional right-hand sides. 

An immediate consequence of the introduction of extrapolated operators, however, is
that we are now in the position to omit the closure bar for the operator sum in an evolutionary
equation, which we will see in an abstract version in \prettyref{thm:ComputeClosureDistr} and for evolutionary equations in \prettyref{thm:NoClosureEEDistr}.
The main advantage is that we can calculate an operator sum much easier than the closure of it. 
The price we have to pay is that we have to work in a larger space $H^{-1}$ of an operator in $\Lnu(\R;H)$ rather than in the original Hilbert space $\Lnu(\R;H)$.
Put differently, this provides another notion of ``solutions'' for evolutionary equations.
For this, we need to introduce the set 
\[
\Fun(H)\coloneqq\set{\phi\colon\dom(\phi)\subseteq H\to\K}{\phi\text{ linear}}
\]
of not necessearily everywhere defined linear functionals on $H$.
Any $u\in H$ is thus identified with an element in $\Fun(H)$ via
$\psi\mapsto\scp{u}{\psi}_{H}$. Note that we can add and scalarly multiply elements in $\Fun(H)$ with respect to the same addition and multiplication
defined on $H'$ and with their natural domains. As usual, we will use the
$\subseteq$-sign for extension/restriction of mappings.

\begin{thm}
\label{thm:ComputeClosureDistr}Let $A\colon\dom(A)\subseteq H\to H$,
$B\colon\dom(B)\subseteq H\to H$ be densely defined and closed such that $A+B$
is closable, and assume that there exists $\left(T_{n}\right)_{n\in\N}$
in $\bo(H)$ such that $T_{n}\to\idop_{H}$ in the strong operator topology
with $\ran(T_{n})\subseteq\dom(B)$ and
\begin{align*}
T_{n}A & \subseteq AT_{n}, \quad T_{n}B \subseteq BT_{n}\text{ for all }n\in\N.
\end{align*}
Then
 $T_n^\ast  A^\ast \subseteq A^\ast T_n^\ast$ and $T_n^\ast B^\ast \subseteq B^\ast T_n^\ast$ for each $n\in \N$ and $\ran(T_n^\ast)\subseteq \dom(B^\ast)$.
Moreover, for $x,f\in H$ the following conditions are equivalent:
\begin{enumerate}
\renewcommand{\labelenumi}{{\upshape (\roman{enumi})}}
 \item $x\in\dom(\overline{A+B})$ and $(\overline{A+B})x=f$.
 \item $A_{-1}x+B_{-1}x\subseteq f\in\Fun(H)$.
\end{enumerate}
\end{thm}

\begin{proof}
Let $n\in \N$.  Taking adjoints in the inclusion $T_nA\subseteq AT_n$, we derive $(AT_n)^\ast\subseteq (T_nA)^\ast$. By \prettyref{thm:adj-comp} and \prettyref{rem:adj-comp} we obtain
 \[
  T_n^\ast A^\ast \subseteq \overline{T_n^\ast A^\ast}=(AT_n)^\ast\subseteq (T_nA)^\ast=A^\ast T_n^\ast.
 \]
The same argument shows the claim for $B^\ast$. Moreover, since $BT_n$ is a closed linear operator defined on the whole space $H$, it follows that $BT_n\in L(H)$ by the closed graph theorem. Hence, $(BT_n)^\ast$ is bounded by \prettyref{lem:bdds} and since $(BT_n)^\ast \subseteq (T_nB)^\ast=B^\ast T_n^\ast$, we derive that $\dom(B^\ast T_n^\ast)=H$, showing that $\ran(T_n^\ast)\subseteq \dom(B^\ast)$.\\
We now prove the asserted equivalence.

(i)$\Rightarrow$(ii): By definition, there exists $(x_{n})_{n}$
in $\dom(A)\cap \dom(B)$ such that $x_{n}\to x$ in $H$ and $Ax_{n}+Bx_{n}\to f.$
By continuity, we obtain $A_{-1}x_{n}\to A_{-1}x$ and $B_{-1}x_{n}\to B_{-1}x$
in $H^{-1}(A^{*})$ and $H^{-1}(B^{*})$, respectively. Thus, we have
\begin{align*}
 (A_{-1}x+B_{-1}x)(y)=\lim_{n\to \infty}(A_{-1}x_n+B_{-1}x_n)(y)&=\lim_{n\to \infty}\scp{Ax_n+Bx_n}{y}=\scp{f}{y},
\end{align*}
for each $y\in \dom(A^\ast)\cap \dom(B^\ast)$, which shows the asserted inclusion.

(ii)$\Rightarrow$(i): For $n\in\N$ we put $x_{n}\coloneqq T_{n}x$.
Then $x_{n}\in\dom(B)$ and for all $y\in\dom(A^{*})\cap\dom(B^{*})$,
we obtain 
\begin{align*}
\scp{T_{n}f-Bx_{n}}{y} & =\scp{T_{n}f}{y}-\scp{T_{n}x}{B^{*}y}
 =\scp{f}{T_{n}^{*}y}-\scp{x}{T_{n}^{*}B^{\ast}y}\\
 &=\scp{f}{T_{n}^{*}y}-\scp{x}{B^{*}T_{n}^{*}y}= f(T_n^\ast y)-(B_{-1}x)(T_n^*y) = (A_{-1}x)(T_n^\ast y) \\
  &=\scp{x}{A^{*}T_{n}^{*}y}
 =\scp{x}{T_{n}^{*}A^{*}y}
 =\scp{x_{n}}{A^{*}y},
\end{align*}
where we have used that $T_n^\ast y\in \dom(A^\ast)\cap \dom(B^\ast)$. Let now $y\in \dom(A^\ast)$. Then $T_k^\ast y \in \dom(A^\ast)\cap \dom(B^\ast)$ for each $k\in \N$ and thus, by what we have shown above
\[
 \scp{T_k(T_nf-Bx_n)}{y}=\scp{T_nf-Bx_n}{T_k^\ast y}=\scp{x_n}{A^* T_k^\ast y}=\scp{x_n}{T_k^\ast A^\ast y}=\scp{T_kx_n}{A^\ast y}
\]
for each $k\in \N$. Letting $k$ tend to infinity, we derive
\[
 \scp{T_n f-Bx_n}{y}=\scp{x_n}{A^\ast y}.
\]
Since this holds for each $y\in \dom(A^\ast)$, this implies that $x_{n}\in\dom(A)$ and $Ax_{n}+Bx_{n}=T_{n}f$.
Letting $n\to\infty,$ we deduce $x_{n}\to x$ and $Ax_{n}+Bx_{n}\to f$;
that is, (i).
\end{proof}

\begin{lem}
\label{lem:ComputingDualNorm}Let $T\colon\dom(T)\subseteq H\to H$ be
densely defined and closed with $0\in\rho(T)$. Then $T_{-1}\colon H\to H^{-1}(T^{*})$
is an isomorphsim. In particular, the norms $\norm{(T_{-1})^{-1}\cdot}_{H}$
and $\norm{\cdot}_{H^{-1}(T^{*})}$ are equivalent.
\end{lem}

\begin{proof}
Note that since $0\in\rho(T)$ we obtain $\{0\}=\ker(T)=\ker(\left(T^{*}\right)^{\diamond})=\ker(T_{-1})$, see \prettyref{prop:CdiaCadj}\ref{prop:CdiaCadj:item:3}.
Thus, $T_{-1}$ is one-to-one. Next, let $f\in H^{-1}(T^{*})$. Since
$0\in\rho(T)$, we obtain $0\in\rho(T^{*})$ by \prettyref{exer:resolvent_adj}, which implies that $\scp{T^{*}\cdot}{T^{*}\cdot}$
defines an equivalent scalar product on $\dom(T^{*})$. Thus, by the
Riesz representation theorem, we find $\phi\in\dom(T^{*})$ such that
for all $\psi\in\dom(T^{*})$ we have
\[
f(\psi)=\scp{T^{*}\phi}{T^{*}\psi}=\left(\left(T^{*}\right)^{\diamond}\left(T^{*}\phi\right)\right)(\psi).
\]
Hence, $f\in\ran(\left(T^{*}\right)^{\diamond})=\ran(T_{-1})$, thus proving
that $T_{-1}$ is onto.
\end{proof}
The following  alternative description of $H^{-1}(T^{*})$ is content of \prettyref{exer:completion}.

\begin{prop}\label{prop:extrapolation_completion}
Let $T\colon\dom(T)\subseteq H\to H$ be densely defined and closed with
$0\in\rho(T)$. Then 
\[
H^{-1}(T^{*})\cong\tilde{\left(H,\norm{T^{-1}\cdot}_H\right)},
\]
where $\cong$ means isomorphic as Banach spaces and $\tilde{(\cdot)}$ denotes
the completion.
\end{prop}

\begin{prop}
\label{prop:ContDistributionalExtn}Let $B\in\bo(H)$. Assume that $T\colon\dom(T)\subseteq H\to H$ is
densely defined and closed with $0\in\rho(T)$ and $T^{-1}B=BT^{-1}$.
Then $B$ admits a unique continuous extension $\overline{{B}}\in\bo(H^{-1}(T^{*}))$.
\end{prop}

\begin{proof}
By \prettyref{prop:CdiaCadj}\ref{prop:CdiaCadj:item:5}, $\dom(B)=H$ is dense in $H^{-1}(T^{*})$.
Thus, it suffices to show that $B\colon H\subseteq H^{-1}(T^{*})\to H^{-1}(T^{*})$
is continuous. For this, let $\phi\in H$ and compute for all $q\in\dom(T^{*})$
\begin{align*}
\abs{\left(B\phi\right)(q)} & =\abs{\scp{B\phi}{q}}
 =\abs{\scp{B\phi}{\left(T^{*}\right)^{-1}T^{*}q}}
 =\abs{\scp{T^{-1}B\phi}{T^{*}q}}
 =\abs{\scp{BT^{-1}\phi}{T^{*}q}}\\
 & \leq\norm{B}\norm{T^{-1}\phi}\norm{q}_{\dom(T^{*})}.
\end{align*}
The statement now follows upon invoking \prettyref{lem:ComputingDualNorm}.
\end{proof}
The abstract notions and concepts just developed will be applied to
evolutionary equations next.

\section{Evolutionary Equations in Distribution Spaces}

In this section, we will specialise the results from the previous
section and provide an extension of the solution theory in $\Lnu(\R;H)$.
For this, and throughout this whole section, we let $H$ be a Hilbert space, $\mu\in\R$ and
$M\colon\C_{\Re>\mu}\to\bo(H)$ be a material law.
Furthermore, let $\nu>\max\{\sbb{M},0\}$ and $A\colon\dom(A)\subseteq H\to H$ be
skew-selfadjoint.  In order to keep track
of the Hilbert spaces involved, we shall put
\begin{align*}
H_{\nu}^1(\R;H) & \coloneqq \dom(\td{\nu}),\\
H_{\nu}^{-1}(\R;H)& \coloneqq\dom(\td{\nu})'\cong\dom(\td{\nu}^{*})'.
\end{align*}

\begin{prop}\label{prop:DB}
 Let $H$ be a Hilbert space. Let $D\colon \dom(D)\subseteq H\to H$ be densely defined and closed and $B\in \bo(H)$. Assume that $DB$ is densely defined. Then for all $\phi\in H$,  $(DB)_{-1}(\phi) = (D_{-1}B)(\phi)$ on $\dom(D^*)$. 
\end{prop}
\begin{proof}
First of all, note that $(DB)^*=\overline{B^*D^*}$, by \prettyref{thm:adj-comp}. Next, let $\phi\in H$ and $x\in \dom(D^*)$. Then
\begin{align*}
  ((DB)_{-1}\phi)(x)  &=\scp{\phi}{(DB)^*x}  =\scp{\phi}{\overline{B^*D^*}x} \\
& =\scp{\phi}{B^*D^*x} =\scp{B\phi}{D^*x}  = (D_{-1}B\phi)(x).\qedhere
\end{align*}

\end{proof}
The first application of the theory developed in the previous section
reads as follows.
\begin{thm}
\label{thm:NoClosureEEDistr}Let $U,F\in\Lnu(\R;H)$. Then the following
statements are equivalent:
\begin{enumerate}
\renewcommand{\labelenumi}{{\upshape (\roman{enumi})}}
\item $U\in\dom(\overline{\td{\nu}M(\td{\nu})+A})$ and $\bigl(\overline{\td{\nu}M(\td{\nu})+A}\bigr)U=F$.
\item  $\td{\nu}M(\td{\nu})U+AU\subseteq F$ where the left-hand
side is considered as an element of $H_{\nu}^{-1}(\R;H)\cap\Lnu(\R;H^{-1}(A))\subseteq\Fun(\Lnu(\R;H))$.
\end{enumerate}
\end{thm}

Before we come to the proof, we state the following lemma, whose proof is left as \prettyref{exer:proof_Lemma}.

\begin{lem}
\label{lem:resolvTD} 
\begin{enumerate}
 \item Let $B:\dom(B)\subseteq H\to H$ and $C:\dom(C)\subseteq H\to H$ be densely defined closed linear operators. Moreover, let $\lambda,\mu\in \rho(C)$ be in the same connected component of $\rho(C)$ and
 \[
  (\mu-C)^{-1}B\subseteq B(\mu-C)^{-1}. 
 \]
Then $(\lambda-C)^{-1}B\subseteq B(\lambda -C)^{-1}$. 
\item For $\nu>0$ we have $(1+\varepsilon \td{\nu})^{-1}\to 1_{\Lnu(\R;H)}$ and $(1+\varepsilon \td{\nu}^\ast)^{-1}\to 1_{\Lnu(\R;H)}$ strongly as $\varepsilon \to 0\rlim$.
\end{enumerate}
\end{lem}

\begin{proof}[Proof of \prettyref{thm:NoClosureEEDistr}]
For $n\in \N$ we set $T_n\coloneqq (1+\frac{1}{n}\td{\nu})^{-1}$. By \prettyref{lem:resolvTD} we obtain $T_n,T_n^\ast\to 1_{\Lnu(\R;H)}$ strongly in $\Lnu(\R;H)$ as $n\to \infty$. Moreover, by Hille's theorem (see \prettyref{prop:interchange_integral}) we have $\td{\nu}^{-1}A\subseteq A\td{\nu}^{-1}$ and thus, $T_nA\subseteq AT_n$ for each $n\in \N$ by \prettyref{lem:resolvTD}, which also yields $T_n^\ast A\subseteq AT_n^\ast$ for each $n\in \N$ by \prettyref{thm:ComputeClosureDistr}.
The latter, together with the strong convergence of $(T_n)$ and $(T_n^\ast)$, yields that $T_n,T_n^\ast \to 1_{\Lnu(\R;\dom(A))}$ strongly in $\Lnu(\R;\dom(A))$ as $n\to \infty$.

Consider the inclusion
\begin{equation}\label{eq:-1-1}
   (\td{\nu}M(\td{\nu}))_{-1}U+A_{-1}U\subseteq F.
\end{equation}
We show next that \prettyref{eq:-1-1} is equivalent to (ii) by applying \prettyref{prop:DB} to the case $ D=\td{\nu}, B=M(\td{\nu})$. For this assume that \prettyref{eq:-1-1} holds.
By \prettyref{prop:DB}, we deduce that 
\[
    ((\td{\nu}M(\td{\nu}))_{-1}U+A_{-1}U)|_{\dom(\td{\nu}^*)\cap \dom(A)}=    ((\td{\nu})_{-1}M(\td{\nu})U+A_{-1}U)|_{\dom(\td{\nu}^*)\cap \dom(A)}
\]
Thus, \prettyref{eq:-1-1} implies (ii). 

On the other hand, assume that (ii) holds. Let $\phi\in \dom((\td{\nu}M(\td{\nu}))^*)\cap \Lnu(\R;\dom(A))$. Then, for $n\in \N$, $\phi_{n}\coloneqq T_n^\ast \phi \to \phi$ as $
n\to \infty$ in $\Lnu(\R;\dom(A))$ and
 \[
(\td{\nu}M(\td{\nu}))^*\phi_{n}=T_n^\ast(\td{\nu}M(\td{\nu}))^*\phi \to (\td{\nu}M(\td{\nu}))^*\phi\quad(n\to \infty)
\] in $\Lnu(\R;H)$. By (ii) we obtain
\[
   ((\td{\nu})_{-1}M(\td{\nu})U+A_{-1}U)(\phi_n)=F(\phi_{n}).
\]
Using \prettyref{prop:DB}, we infer
\[
   ((\td{\nu}M(\td{\nu}))_{-1}U+A_{-1}U)(\phi_n)=F(\phi_{n}).
\]Letting $n\to \infty$, we deduce \prettyref{eq:-1-1}.

We are now in the position to apply \prettyref{thm:ComputeClosureDistr} from above to the
case $\Lnu(\R;H)$ being the Hilbert space, $A$ the operator in $\Lnu(\R;H)$, $B=\td{\nu}M(\td{\nu})$, and $T_{n}=\left(1+\frac{1}{n}\td{\nu}\right)^{-1}$. We need to establish the commutativity properties next. The relation $T_n A\subseteq AT_n$ was already shown above.
Next, we infer $\ran(T_n)\subseteq \dom(\td{\nu})\subseteq \dom (B)$ and 
\[
T_{n}B\subseteq BT_n
\]
for all $n\in\N$ by using the Fourier--Laplace transformation, see also \prettyref{thm:spectral-repr-derivative}.
The closability of $A+B$ is implied by \prettyref{thm:Solution_theory_EE}.
\end{proof}

Assume now that there exists $c>0$ such that
\[
\Re zM(z)\geq c \quad(z\in\C_{\Re\geq\nu}).
\]
We recall from \prettyref{thm:Solution_theory_EE} that the operator
$
\overline{\td{\nu}M(\td{\nu})+A}
$
is continuously invertible in $\Lnu(\R;H)$.

\begin{thm}
\label{thm:extSolTh} The operator $S_{\nu}\coloneqq\bigl(\overline{\td{\nu}M(\td{\nu})+A}\bigr)^{-1}\in\bo(\Lnu(\R;H))$
admits a continuous extension to $L(H_{\nu}^{-1}(\R;H))$. 
\end{thm}

\begin{proof}
We apply \prettyref{prop:ContDistributionalExtn} to $\Lnu(\R;H)$ being the Hilbert space,
$T=\td{\nu}$ and $B=S_{\nu}$. For this, it remains to prove that
$T^{-1}S_{\nu}=S_{\nu}T^{-1}$. This however follows from the fact
that $z\mapsto S(z)\coloneqq\left(zM(z)+A\right)^{-1}$ is a material
law and $S(\td{\nu}) = S_\nu$.
\end{proof}

\section{Initial Value Problems for Evolutionary Equations}

Let $H$ be a Hilbert space, $\mu\in\R$,
$M\colon\C_{\Re>\mu}\to\bo(H)$ a material law, $\nu>\max\{\sbb{M},0\}$ and $A\colon\dom(A)\subseteq H\to H$
skew-selfadjoint. 
In this section we shall focus on the implementation of initial value
problems for evolutionary equations. A priori there is no explicit
initial condition implemented in the theory established in $\Lnu(\R;H)$.
Indeed, choosing $\nu>0$ we have only an implicit exponential decay
condition at $-\infty$. For initial values at $0$, we would rather want to
solve the following type of equation. In the situation of the previous
section, for a given initial value $U_{0}\in H$ we seek to solve 
the initial value problem
\begin{equation}
\begin{cases}
\left(\td{\nu}M(\td{\nu})+A\right)U=0 & \text{ on }\oi{0}{\infty},\\
U(0\rlim)=U_{0}.
\end{cases}\label{eq:IVP_PDE}
\end{equation}
In this generality the initial value problem cannot be solved. Indeed,
for $U\in\Lnu(\R;H)$ evaluation at $0$ is not well-defined.
A way to overcome this difficulty is to weaken the attainment of the
initial value. For this, we specialise to the case when
\[
M(\td{\nu})=M_{0}+\td{\nu}^{-1}M_{1}
\]
with $M_0, M_1\in \bo(H)$.

We start with the following lemma, which will also be useful in the next chapter.

\begin{lem}\label{lem:ivp}
 Let $U_{0}\in\dom(A)$, $U\in\Lnu(\R;H)$ such that $M_{0}U-\1_{\roi{0}{\infty}}M_{0}U_{0}:\R\to H^{-1}(A)$ is continuous,
$\spt U\subseteq\roi{0}{\infty}$ and 
\begin{align*}
\begin{cases}
\td{\nu}M_{0}U+M_{1}U+AU=0 & \text{on }\oi{0}{\infty},\\
M_{0}U(0\rlim)=M_{0}U_{0} & \text{in } H^{-1}(A),
\end{cases}
\end{align*}
where the first equality is meant in the sense that
\[
 \forall \varphi\in H_\nu^{1}(\R;H)\cap \Lnu(\R;\dom(A)),\, \spt\varphi \subseteq \roi{0}{\infty}:  \bigl(\td{\nu}M_{0}U+M_{1}U+AU\bigr)(\varphi)=0.
\]
Then $U-\1_{\roi{0}{\infty}}U_0\in \dom(\overline{\td{\nu}M_0+M_1+A})$ and 
\[
 \bigl(\overline{\td{\nu}M_{0}+M_{1}+A}\bigr)(U-\1_{\roi{0}{\infty}}U_{0})=-(M_{1}+A)U_{0}\1_{\roi{0}{\infty}}.
\]
\end{lem}

\begin{proof}
 We apply \prettyref{thm:NoClosureEEDistr} for showing the claim; that is, we show that
\[
 \bigl((\td{\nu}M_{0}+M_{1})(U-\1_{\roi{0}{\infty}}U_0)+A(U-\1_{\roi{0}{\infty}}U_{0})\bigr)(\psi)=(-(M_{1}+A)U_{0}\1_{\roi{0}{\infty}})(\psi)
\]
for each $\psi \in H_\nu^1(\R;H)\cap \Lnu(\R;\dom(A))$. Note that by continuity, it suffices to show the equality for $\psi\in \cci(\R;\dom(A))$. So, let $\psi \in \cci(\R;\dom(A))$ and for $n\in \N$ we  define the function $\varphi_n\in H_\nu^1(\R)$ by   
\[
 \varphi_n(t)\coloneqq \begin{cases}
                        0 & \text{if } t\leq 0,\\
                        nt & \text{if } t\in \oi{0}{1/n},\\
                        1 &\text{if } t \geq 1/n.
                       \end{cases}
\]
Note that $\varphi_n \psi \in  H_\nu^1(\R;H)\cap \Lnu(\R;\dom(A))$ and $\spt (\varphi_n\psi) \subseteq \roi{0}{\infty}$ for each $n\in \N$. Thus, we obtain
\begin{align*}
 &\bigl((\td{\nu}M_{0}+M_{1}+A)(U-\1_{\roi{0}{\infty}}U_0)\bigr)(\psi)\\
 &= \bigl((\td{\nu}M_{0}+M_{1}+A)U\bigr)(\psi)- \bigl((\td{\nu}M_{0}+M_{1}+A)(\1_{\roi{0}{\infty}}U_0)\bigr)(\psi)\\
 &= \bigl((\td{\nu}M_{0}+M_{1}+A)U\bigr)(\varphi_n \psi)+ \bigl((\td{\nu}M_{0}+M_{1}+A)U\bigr)((1-\varphi_n)\psi)-\\
 &\quad -\bigl((\td{\nu}M_{0}+M_{1}+A)(\1_{\roi{0}{\infty}}U_0)\bigr)(\psi)\\
 &= \bigl((\td{\nu}M_{0}+M_{1}+A)U\bigr)((1-\varphi_n)\psi)- (\delta_0 M_0 U_0)(\psi)-\bigl((M_1+A)(\1_{\roi{0}{\infty}} U_0)\bigr)(\psi) 
\end{align*}
for each $n\in \N$. Thus, the claim follows if we can show that 
\[
 \bigl((\td{\nu}M_{0}+M_{1}+A)U\bigr)((1-\varphi_n)\psi)- (\delta_0 M_0 U_0)(\psi)\to 0\quad (n\to \infty).
\]
For doing so, we first observe that for all $n\in\N$ we have
\[
 (\delta_0 M_0 U_0)(\psi)=(\delta_0 M_0 U_0)((1-\varphi_n)\psi)=(\td{\nu} M_0 \1_{\roi{0}{\infty}} U_0)((1-\varphi_n)\psi),
\]
since $\varphi_n(0)=0$. Moreover,
\[
 \bigl((M_{1}+A)U\bigr)((1-\varphi_n)\psi)=\scp{U}{(1-\varphi_n)(M_1^\ast+A^\ast)\psi}_{\Lnu}\to 0\quad (n\to \infty), 
\]
since $1-\varphi_n(\m) \to \1_{\loi{-\infty}{0}}(\m)$ strongly in $\Lnu(\R;H)$ and $\spt U\subseteq \roi{0}{\infty}$. Thus, it remains to show that
\[
 \bigl(\td{\nu}M_{0}(U-\1_{\roi{0}{\infty}}U_0)\bigr)((1-\varphi_n)\psi)\to 0\quad (n\to \infty).
\]
We compute
\begin{align*}
  & \bigl(\td{\nu}M_{0}(U-\1_{\roi{0}{\infty}}U_0)\bigr)((1-\varphi_n)\psi) \\
  & = \scp{M_0(U-\1_{\roi{0}{\infty}}U_0)}{\td{\nu}^\ast((1-\varphi_n)\psi)}_{\Lnu}\\
  & =\scp{M_0(U-\1_{\roi{0}{\infty}}U_0)}{n\1_{\ci{0}{1/n}} \psi}_{\Lnu}
  -\scp{M_0(U-\1_{\roi{0}{\infty}}U_0)}{(1-\varphi_n) \td{\nu}\psi}_{\Lnu}\\
  &  +2\nu\scp{M_0(U-\1_{\roi{0}{\infty}}U_0)}{(1-\varphi_n)\psi}_{\Lnu}.
\end{align*}
Note that the last two terms on the right-hand side tend to $0$ as $n\to \infty$ since, as above, $1-\varphi_n(\m) \to \1_{\loi{-\infty}{0}}(\m)$ strongly in $\Lnu(\R;H)$ and $\spt U\subseteq \roi{0}{\infty}$. For the first term, we observe that
\begin{align*}
 & \abs{\scp{M_0(U-\1_{\roi{0}{\infty}}U_0)}{n\1_{\ci{0}{1/n}} \psi}_{\Lnu}} \\
 &\leq n\int_{0}^{1/n} \abs{\scp{M_0 (U(t)-U_0)}{\psi(t)}_{H}} \e^{-2\nu t} \d t\\
 &\leq n\int_{0}^{1/n} \norm{M_0 (U(t)-U_0)}_{H^{-1}(A)}\norm{\psi(t)}_{\dom(A^\ast)} \e^{-2\nu t} \d t 
 \to 0\quad (n\to \infty),
\end{align*}
by the fundamental theorem of calculus, since $(M_0U)(t)\to M_0U_0$ in $H^{-1}(A)$ as $t\to 0\rlim$.
\end{proof}

Assume now additionally that there exists $c>0$ such that
\[
 zM_0+M_1\geq c\quad (z\in \C_{\Re\geq \nu}).
\]
Then we can actually prove a stronger result than in the previous lemma.

\begin{thm}
\label{thm:IVPs_PDEs}Let $U_{0}\in\dom(A)$, $U\in\Lnu(\R;H)$. Then
the following statements are equivalent:
\begin{enumerate}
\renewcommand{\labelenumi}{{\upshape (\roman{enumi})}}
\item $M_{0}U-\1_{\roi{0}{\infty}}M_{0}U_{0}\from\R\to H^{-1}(A)$ is continuous,
$\spt U\subseteq\roi{0}{\infty}$ and 
\begin{align*}
\begin{cases}
\td{\nu}M_{0}U+M_{1}U+AU=0 & \text{on }\oi{0}{\infty},\\
M_{0}U(0\rlim)=M_{0}U_{0} & \text{in } H^{-1}(A),
\end{cases}
\end{align*}
where the first equality is meant as in \prettyref{lem:ivp}.
\item $U-\1_{\roi{0}{\infty}}U_{0}\in \dom (\overline{\td{\nu}M_{0}+M_{1}+A})$, and $\bigl(\overline{\td{\nu}M_{0}+M_{1}+A}\bigr)(U-\1_{\roi{0}{\infty}}U_{0})=-(M_{1}+A)U_{0}\1_{\roi{0}{\infty}}$.
\item $U=S_{\nu}\delta_{0}M_{0}U_{0}$, with $S_{\nu}\in\bo(H_{\nu}^{-1}(\R;H))$ as in \prettyref{thm:extSolTh}.
\end{enumerate}
Moreover, in either case we have $M_{0}U-\1_{\roi{0}{\infty}}M_{0}U_{0}\in H_\nu^1(\R;H^{-1}(A))$.
\end{thm}

\begin{proof}
(i)$\Rightarrow$(ii): This was shown in \prettyref{lem:ivp}.

(ii)$\Rightarrow$(iii): We have that
\[
 U-\1_{\roi{0}{\infty}}U_0 =-S_\nu((M_1+A)\1_{\roi{0}{\infty}}U_0).
\]
Applying $\td{\nu}^{-1}$ to both sides of this equality we infer that
\begin{align*}
 \td{\nu}^{-1}(U-\1_{\roi{0}{\infty}}U_0)&=-S_\nu((M_1+A)\td{\nu}^{-1}\1_{\roi{0}{\infty}}U_0)\\
 &= -\td{\nu}^{-1}\1_{\roi{0}{\infty}}U_0+S_\nu(\td{\nu}M_0 \td{\nu}^{-1}\1_{\roi{0}{\infty}}U_0),
\end{align*}
which gives 
\[
 \td{\nu}^{-1}U=S_\nu(\td{\nu}M_0 \td{\nu}^{-1}\1_{\roi{0}{\infty}}U_0)=S_\nu(M_0\1_{\roi{0}{\infty}}U_0).
\]
Applying $\td{\nu}$ to both sides and taking into account \prettyref{thm:extSolTh}, we derive the claim.

(iii)$\Rightarrow$(ii): We do the argument in the proof of (ii)$\Rightarrow$(iii) backwards. First, we apply $\td{\nu}^{-1}$ to
 $U=S_\nu(\delta_0 M_0 U_0)$, which yields
 \begin{align*}
  \td{\nu}^{-1}U & = \td{\nu}^{-1}S_\nu(\delta_0 M_0 U_0) = S_\nu(M_0 \1_{\roi{0}{\infty}} U_0) = S_\nu(\td{\nu}M_0 \td{\nu}^{-1}\1_{\roi{0}{\infty}}U_0).
 \end{align*}
 Thus,
 \begin{align*}
 \td{\nu}^{-1}(U-\1_{\roi{0}{\infty}}U_0)
 &= S_\nu(\td{\nu}M_0 \td{\nu}^{-1}\1_{\roi{0}{\infty}}U_0)-\td{\nu}^{-1}\1_{\roi{0}{\infty}}U_0\\
 &=-S_\nu((M_1+A)\td{\nu}^{-1}\1_{\roi{0}{\infty}}U_0).
 \end{align*}
 An application of $\td{\nu}$ yields the claim.

(ii),(iii)$\Rightarrow$(i): 
Since $U=S_\nu(\delta_0 M_0 U_0)$, we derive that
\[
 (\td{\nu}M_0+M_1+A)U\subseteq \delta_0M_0U_0,
\]
which in particular yields $(\td{\nu}M_0+M_1+A)U=0$ on $\oi{0}{\infty}$. 
By (ii) we infer 
\[U-\1_{\roi{0}{\infty}}U_0 =-S_\nu((M_1+A)\1_{\roi{0}{\infty}}U_0),\]
which shows that $\spt (U-\1_{\roi{0}{\infty}}U_0)\subseteq \roi{0}{\infty}$ due to causality and hence, $\spt U\subseteq \roi{0}{\infty}$. It remains to show that $M_0(U-\1_{\roi{0}{\infty}}U_0)\in H_\nu^1(\R;H^{-1}(A))$, since this would imply the continuity of $M_0(U-\1_{\roi{0}{\infty}}U_0)$ with values in $H^{-1}(A)$ by \prettyref{thm:Sobolev_emb} and thus, 
\[
 M_0(U-\1_{\roi{0}{\infty}}U_0)(0\rlim)=M_0(U-\1_{\roi{0}{\infty}}U_0)(0\llim)=0 \text{ in } H^{-1}(A)
\]
since the function is supported on $\roi{0}{\infty}$ only. We compute 
\begin{align*}
 M_0(U-\1_{\roi{0}{\infty}}U_0) &=-M_0 S_\nu((M_1+A)\1_{\roi{0}{\infty}}U_0)\\
 &= \td{\nu}M_0 S_\nu(\td{\nu}^{-1}(M_1+A)\1_{\roi{0}{\infty}}U_0)\\
 &= \td{\nu}^{-1}(M_1+A)\1_{\roi{0}{\infty}}U_0-(M_1+A)S_\nu(\td{\nu}^{-1}(M_1+A)\1_{\roi{0}{\infty}}U_0),
\end{align*}
and since the right-hand side belongs to $H_\nu^1(\R;H^{-1}(A))$, the assertion follows.
\end{proof}

\begin{rem}
By \prettyref{thm:extSolTh}, we always have $U=S_{\nu}\delta_{0}M_{0}U_{0}\in H_{\nu}^{-1}(\R;H)$.
This then serves as our generalisation for the initial value problem
even if $U_{0}\notin\dom(A)$. 
\end{rem}

The upshot of \prettyref{thm:IVPs_PDEs}(ii) is that, provided $U_0\in \dom(A)$,  we can reformulate
initial value problems with the help of our theory as evolutionary equations with $\Lnu$-right-hand
sides. Thus, we do not need the detour to extrapolation spaces
for being able to solve the initial value problem \prettyref{eq:IVP_PDE}
(with an adapted initial condition as in (i)) in this situation.

Also note that  it may seem that $U$ does depend on the `full
information' of $U_{0}$ as it is indicated in (ii). In fact, $U$
only depends on the values of $U_{0}$ orthogonal to the kernel of
$M_{0}$ as it is seen in (iii). We conclude this chapter with two
examples; the first one is the heat equation, the second example considers
Maxwell's equations.

\begin{example}[Initial Value Problems for the Heat Equation]
We recall the setting for the heat equation outlined in \prettyref{thm:wp_heat}.
This time, we will use homogeneous Dirichlet boundary conditions for the heat
distribution $\theta$. Let $\Omega\subseteq\R^{d}$ be open and bounded,
$a\in L_{\infty}(\Omega)^{d\times d}$ with $\Re a(x)\geq c>0$ for
a.e.~$x\in\Omega$ for some $c>0$. In this case, we have 
\begin{align*}
M_{0} & =\begin{pmatrix}
1 & 0\\
0 & 0
\end{pmatrix},\quad
M_{1} =\begin{pmatrix}
0 & 0\\
0 & a^{-1}
\end{pmatrix},\quad
A =\begin{pmatrix}
0 & \dive\\
\grad_{0} & 0
\end{pmatrix}.
\end{align*}
For the unknown heat distribution, $\theta$, we ask it to have the
initial value $\theta_{0}\in\dom(\grad_{0})$. Let $\nu>0$ and $V\in\Lnu\bigl(\R;\L(\Omega)\times\L(\Omega)^{d}\bigr)$ be the unique solution of
\begin{align*}
\bigl(\overline{\td{\nu}M_{0}+M_{1}+A}\bigr)V & =-\left(M_{1}+A\right)\1_{\roi{0}{\infty}}\begin{pmatrix}
\theta_{0} \\
0
\end{pmatrix} =-\1_{\roi{0}{\infty}}\begin{pmatrix}
 0\\
 \grad_{0}\theta_{0}
\end{pmatrix}.
\end{align*}
Then $(\theta,q)\coloneqq U\coloneqq V+\1_{\roi{0}{\infty}}\begin{pmatrix}
\theta_{0} \\
0 
\end{pmatrix}\in\Lnu\bigl(\R;\L(\Omega)\times\L(\Omega)^{d}\bigr)$ satisfies (ii) from \prettyref{thm:IVPs_PDEs}. Hence, on $\oi{0}{\infty}$
we have
\[
\begin{pmatrix}
\td{\nu}\theta\\
a^{-1}q
\end{pmatrix}+\begin{pmatrix}
\dive q\\
\grad_{0}\theta
\end{pmatrix}=0
\]
and the initial value is attained in the sense that 
\[
\left(M_{0}\left(\theta,q\right)\right)(0\rlim)=\begin{pmatrix}
\theta(0\rlim)\\
0
\end{pmatrix}=\begin{pmatrix}
\theta_{0}\\
0
\end{pmatrix}\quad\text{in}\quad H^{-1}(A)=H^{-1}(\grad_{0})\times H^{-1}(\dive),
\]
which follows from \prettyref{prop:A_neg} where we computed $H^{-1}(A)$. Let us
have a closer look at the attainment of the initial value. As a particular
consequence of strong convergence in $H^{-1}(\grad_{0})$, we obtain
 for all $\phi\in\dom(\dive)$ 
 \[
\scp{\theta(t)}{\dive\phi}\to\scp{\theta_{0}}{\dive\phi}
\]
as $t\to0\rlim$. Since $\grad_{0}$ is one-to-one and has closed range,
we see that $\dive$ has dense and closed range. Hence $\dive$
is onto. This implies that for all $\psi\in\L(\Omega)$ 
\[
\scp{\theta(t)}{\psi}\to\scp{\theta_{0}}{\psi} \quad(t\to0\rlim).
\]
\end{example}
We deduce that the initial value is attained weakly. This might seem
a bit unsatisfactory, however, we shall see stronger assertions for more particular
cases in the next chapter. 

Next, we have a look at Maxwell's equations.

\begin{example}[Initial Value Problems for Maxwell's equations]
We briefly recall the situation of Maxwell's equations from \prettyref{thm:st_maxwell}.
Let $\varepsilon,\mu,\sigma\colon\Omega\to\R^{3\times3}$ satisfy
the assumptions in \prettyref{thm:st_maxwell} and let $\left(E_{0},H_{0}\right)\in\dom(\curl_{0})\times\dom(\curl)$.
Let $(\hat{E},\hat{H})\in\Lnu(\R;\L(\Omega)^{6})$ satisfy
\begin{align*}
 & \left(\overline{\td{\nu}\begin{pmatrix}
\varepsilon & 0\\
0 & \mu
\end{pmatrix}+\begin{pmatrix}
\sigma & 0\\
0 & 0
\end{pmatrix}+\begin{pmatrix}
0 & -\curl\\
\curl_{0} & 0
\end{pmatrix}}\right)\begin{pmatrix}
\hat{E}\\
\hat{H}
\end{pmatrix}\\
 & =-\left(\begin{pmatrix}
\sigma & 0\\
0 & 0
\end{pmatrix}+\begin{pmatrix}
0 & -\curl\\
\curl_{0} & 0
\end{pmatrix}\right)\1_{\roi{0}{\infty}}\begin{pmatrix}
E_{0}\\
H_{0}
\end{pmatrix} =\1_{\roi{0}{\infty}}\begin{pmatrix}
-\sigma E_{0}+\curl H_{0}\\
-\curl_{0}E_{0}
\end{pmatrix}.
\end{align*}
Then, as we have argued for the heat equation,
\[
\begin{pmatrix}
E\\
H
\end{pmatrix}\coloneqq\begin{pmatrix}
\hat{E}\\
\hat{H}
\end{pmatrix}+\1_{\roi{0}{\infty}}\begin{pmatrix}
E_{0}\\
H_{0}
\end{pmatrix}
\]
satisfies a corresponding initial value problem. We note here
that although often the second component in the right-hand side is
set to $0$, as there are `no magnetic monopoles', in the theory of evolutionary equations the
second component of the right-hand side does appear as an initial
value in disguise.
\end{example}

\section{Comments}

There are many ways to define spaces generalising the action of an
operator to a bigger class of elements; both in a concrete setting
and in abstract situations; see e.g.~\cite{Cojuhari2010,Engel2000}.
People have also taken into account simultaneous extrapolation
spaces for operators that commute, see e.g.~\cite{Palais1965,Picard2000}. 

These spaces are particularly useful for formulating initial value
problems as was exemplified above; see also the concluding chapter
of \cite{A11} for more insight. Yet there is more to it as one can in fact generalise the
equation under consideration or even force the attainment of the initial
value in a stronger sense. These issues, however, imply that either
the initial value is attained in a much weaker sense, or that there are
other structural assumptions needed to be imposed on the material law $M$ (as well as on
the operator $A$). 

In fact, quite recently, it was established that a particular proper subclass of evolutionary equations can be put into the framework of  $C_0$-semigroups. The conditions required to allow for statements in this
direction are, on the other hand, rather hard to check in practice;
see \cite{Trostorff2018}.   


\section*{Exercises}
\addcontentsline{toc}{section}{Exercises}

\begin{xca}
Let $H_0$ be a Hilbert space, $T\in\bo(H_{0})$. Compute $H^{-1}(T)$ and $H^{-1}(T^{*})$.
\end{xca}

\begin{xca}\label{exer:dual_embedding}
 Let $H_0,H_1$ be Hilbert spaces such that $H_0\hookrightarrow H_1$ is dense and continuous. Prove that $H_1'\hookrightarrow H_0'$ is dense and continuous as well.
\end{xca}

\begin{xca}
Prove the following statement which generalises \prettyref{prop:ContDistributionalExtn} from
above: Let $H_0$ be a Hilbert space, $A\in\bo(H_{0})$. Assume that $T\colon\dom(T)\subseteq H_{0}\to H_{0}$ is
densely defined and closed with $0\in\rho(T)$ and $T^{-1}A=AT^{-1}+T^{-1}BT^{-1}$
for some bounded $B\in\bo(H_{0})$. Then $A$ admits a unique continuous
extension, $\overline{{A}}\in\bo(H^{-1}(T^{*}))$.
\end{xca}

\begin{xca}
Let $H_0$ be a Hilbert space, $N\colon\dom(N)\subseteq H_{0}\to H_{0}$ be a \emph{normal} operator;
that is, $N$ is densely defined and closed and $NN^{*}=N^{*}N$.
Show that $H^{-1}(N)\cong H^{-1}(N^{*})$ and deduce $H^{-1}(\td{\nu})\cong H^{-1}(\td{\nu}^{*})$.
\end{xca}

\begin{xca} \label{exer:completion}
Prove \prettyref{prop:extrapolation_completion}.
\end{xca}

\begin{xca}
Let $H_0$ be a Hilbert space, $n\in\N$ and $T\colon\dom(T)\subseteq H_{0}\to H_{0}$ be a densely
defined, closed linear operator with $0\in\rho(T)$. We define $H^{n}(T)\coloneqq\dom(T^{n})$
and $H^{-n}(T)\coloneqq H^{-1}(T^{n})$. Show that
for all $k\in\N$ and $\ell\in\Z$ we have that $H^{k+\ell}(T)\hookrightarrow H^{\ell}(T)$
continuously and densely. Also show that $\mathcal{D}\coloneqq\bigcap_{n\in\N}\dom(T^{n})$
is dense in $H^{\ell}(T)$ and dense in $H^{-\ell}(T^\ast)$ for all $\ell\in \N$ and that $T|_{\mathcal{{D}}}$
can be continuously extended to a topological isomorphism $H^{\ell}(T)\to H^{\ell-1}(T)$  and to an isomorphism $H^{-\ell+1}(T^\ast)\to H^{-\ell}(T^\ast)$ for each $\ell\in \N$.
\end{xca}

\begin{xca}
\label{exer:proof_Lemma}
Prove \prettyref{lem:resolvTD}. 

Hint: Prove a similar equality with $\td{\nu}^{-1}$ formally replaced
by $z\in\partial\ball{r}{r}\subseteq\C$ and deduce the assertion
with the help of \prettyref{thm:spectral-repr-derivative}.
\end{xca}

%

\printbibliography[heading=subbibliography]

\chapter{Differential Algebraic Equations}

Let $H$ be a Hilbert space and $\nu\in\R$.
We saw in the previous chapter how initial value problems can
be formulated within the framework of evolutionary equations. More
precisely, we have studied problems of the form 
\begin{equation}
\begin{cases}
\left(\td{\nu}M_{0}+M_{1}+A\right)U=0 & \text{ on }\oi{0}{\infty},\\
M_{0}U(0\rlim)=M_{0}U_{0}
\end{cases}\label{eq:IV}
\end{equation}
for $U_{0}\in H$, $M_{0},M_{1}\in\bo(H)$ and $A\from\dom(A)\subseteq H\to H$ skew-selfadjoint; that is, we
have considered material laws of the form 
\[
M(z)\coloneqq M_{0}+z^{-1}M_{1}\quad(z\in\C\setminus\{0\}).
\]
Here, the initial value is attained in a weak sense as an
equality in the extrapolation space $H^{-1}(A)$. The first line
is also meant in a weak sense since the left-hand side turned out
to be a functional in $H_{\nu}^{-1}(\R;H)\cap\Lnu(\R;H^{-1}(A))$.
In \prettyref{thm:IVPs_PDEs} it was shown that the latter problem
can be rewritten as 
\[
\left(\td{\nu}M_{0}+M_{1}+A\right)U=\delta_{0}M_{0}U_{0}.
\]
In this chapter we aim to inspect initial value problems a little
closer but in the particularly simple case when $A=0$. However, we want to impose the initial condition for $U$ and not just 
$M_{0}U$. Thus, we want to deal with the problem 
\begin{equation}
\begin{cases}
\left(\td{\nu}M_{0}+M_{1}\right)U=0 & \text{ on }\oi{0}{\infty},\\
U(0\rlim)=U_{0}
\end{cases}\label{eq:DAE}
\end{equation}
for two bounded operators $M_{0},M_{1}$ and an initial value $U_{0}\in H$.
This class of differential equations is known as \emph{differential
algebraic equations} since the operator $M_{0}$ is allowed to have
a non-trivial kernel. Thus, \prettyref{eq:DAE} is a coupled problem
of a differential equation
(on $(\ker M_{0})^{\bot}$) and an algebraic equation (on $\ker M_{0}$). We begin by treating these equations in
the finite-dimensional case; that is, $H=\C^{n}$ and $M_{0},M_{1}\in\C^{n\times n}$
for some $n\in\N$.

\section{The finite-dimensional Case}

Throughout this section let $n\in\N$ and $M_{0},M_{1}\in\C^{n\times n}$.
\begin{defn*}
We define the \emph{\index{spectrum, matrix pair}spectrum of the
matrix pair} $(M_{0},M_{1})$ by 
\[
\sigma(M_{0},M_{1})\coloneqq\set{z\in\C}{\det(zM_{0}+M_{1})=0},
\]
and the \emph{resolvent set of the matrix pair\index{resolvent set, matrix pair} $(M_{0},M_{1})$
}by 
\[
\rho(M_{0},M_{1})\coloneqq\C\setminus\sigma(M_{0},M_{1}).
\]
\end{defn*}

\begin{rem}
\begin{enumerate}
\item It is immediate that $\sigma(M_{0},M_{1})$ is closed since
the mapping $z\mapsto\det(zM_{0}+M_{1})$ is continuous.
\item Note in particular that the spectrum (the set of eigenvalues) of
a matrix $A$ corresponds in this setting to the spectrum of the matrix pair $(1,-A)$.
\end{enumerate}
\end{rem}

In contrast to the case of the spectrum of one matrix, it may happen
that $\sigma(M_{0},M_{1})=\C$ (for example we can choose $M_{0}=0$ and $M_{1}$ singular).
More precisely, we have the following result.

\begin{lem}
\label{lem:spectrum_matrix_pair}The set $\sigma(M_{0},M_{1})$ is
either finite or equals the whole complex plane $\C$. If $\sigma(M_{0},M_{1})$
is finite then $\card{\sigma(M_{0},M_{1})}\leq n.$
\end{lem}

\begin{proof}
The function $z\mapsto\det(zM_{0}+M_{1})$ is a polynomial of order
less than or equal to $n$. If it is constantly zero, then $\sigma(M_{0},M_{1})=\C$
and otherwise $\card{\sigma(M_{0},M_{1})}\leq n$.
\end{proof}

\begin{defn*}
The matrix pair $(M_{0},M_{1})$ is called \emph{regular}\index{regular, matrix pair}
if $\sigma(M_{0},M_{1})\ne\C$.
\end{defn*}
The main problem in solving an initial value problem of the form \prettyref{eq:DAE}
is that one cannot expect a solution for each initial value $U_{0}\in\C^n$
as the following simple example shows.
\begin{example}
Let $M_{0}=\begin{pmatrix}
1 & 1\\
0 & 0
\end{pmatrix}, \,M_{1}=\begin{pmatrix}
1 & 0\\
0 & 1
\end{pmatrix}$ and let $U_{0}\in\C^{2}$. We assume that there exists a solution
$U\from\R_{\geq0}\to\C^{2}$ satisfying \prettyref{eq:DAE}; that is, 
\begin{align*}
U_{1}'(t)+U_{2}'(t)+U_{1}(t) & =0\quad(t>0),\\
U_{2}(t) & =0\quad(t>0),\\
U(0\rlim) & =U_{0}.
\end{align*}
The second and third equation yield that the second coordinate of $U_{0}$
has to be zero. Then, for $U_{0}=(x,0)\in\C^{2}$ the unique solution
of the above problem is given by 
\[
U(t)=\bigl(U_{1}(t),U_{2}(t)\bigr)=(x\e^{-t},0) \quad (t\geq0).
\]
\end{example}

\begin{defn*}
We call an initial value $U_{0}\in\C^{n}$ \emph{consistent} \index{consistent initial value} for \prettyref{eq:DAE},
if there exists $\nu>0$ and $U\in C(\R_{\geq0};\C^{n})\cap\Lnu(\R_{\geq0};\C^{n})$ such that \prettyref{eq:DAE} holds. We denote the
set of all consistent initial values for \prettyref{eq:DAE} by 
\[
\IV(M_{0},M_{1})\coloneqq\set{U_{0}\in\C^{n}}{U_{0}\text{ consistent}}.
\]
\end{defn*}
\begin{rem}
It is obvious that $\IV(M_{0},M_{1})$ is a subspace of $\C^{n}$.
In particular, $0\in\IV(M_{0},M_{1})$.
\end{rem}

It is now our goal to determine the space $\IV(M_{0},M_{1})$. One
possibility for doing so uses the so-called \emph{quasi-Weierstraß normal
form}.
\begin{prop}[quasi-Weierstraß normal form\index{quasi-Weierstraß normal form}]
\label{prop:Weierstrass_normal_form} Assume that $(M_{0},M_{1})$
is regular. Then there exist invertible matrices $P,Q\in\C^{n\times n}$
such that 
\[
PM_{0}Q=\begin{pmatrix}
1 & 0\\
0 & N
\end{pmatrix},\quad PM_{1}Q=\begin{pmatrix}
C & 0\\
0 & 1
\end{pmatrix},
\]
where $C\in\C^{k\times k}$ and $N\in\C^{(n-k)\times(n-k)}$ for some
$k\in\{0,\ldots,n\}$. Moreover, the matrix $N$ is nilpotent; that is, there exists $\ell\in\N$ such that $N^{\ell}=0$.
\end{prop}

\begin{proof}
Since $(M_{0},M_{1})$ is regular we find $\lambda\in\C$ such that
$\lambda M_{0}+M_{1}$ is invertible.
We set $P_{1}\coloneqq(\lambda M_{0}+M_{1})^{-1}$ and obtain 
\begin{align*}
M_{0,1} & \coloneqq P_{1}M_{0}=(\lambda M_{0}+M_{1})^{-1}M_{0},\\
M_{1,1} & \coloneqq P_{1}M_{1}=(\lambda M_{0}+M_{1})^{-1}M_{1}=1-\lambda M_{0,1}.
\end{align*}
Let now $P_{2}\in\C^{n\times n}$ such that 
\[
M_{0,2}\coloneqq P_{2}M_{0,1}P_{2}^{-1}=\begin{pmatrix}
J & 0\\
0 & \tilde{N}
\end{pmatrix}
\]
for some invertible matrix $J\in\C^{k\times k}$ and a nilpotent matrix
$\tilde{N}\in\C^{(n-k)\times(n-k)}$ (use the Jordan normal form
of $M_{0,1}$ here). Then 
\[
M_{1,2}\coloneqq P_{2}M_{1,1}P_{2}^{-1}=\begin{pmatrix}
1-\lambda J & 0\\
0 & 1-\lambda\tilde{N}
\end{pmatrix}.
\]
Now, by the nilpotency of $\tilde{N}$, the matrix $(1-\lambda\tilde{N})$ is
invertible by the Neumann series. We set 
\[
P_{3}\coloneqq\begin{pmatrix}
J^{-1} & 0\\
0 & (1-\lambda\tilde{N})^{-1}
\end{pmatrix}
\]
and obtain 
\[
P_{3}M_{0,2}=\begin{pmatrix}
1 & 0\\
0 & (1-\lambda\tilde{N})^{-1}\tilde{N}
\end{pmatrix},\text{ and }\quad P_{3}M_{1,2}=\begin{pmatrix}
J^{-1}-\lambda & 0\\
0 & 1
\end{pmatrix}.
\]
Note that $(1-\lambda\tilde{N})^{-1}\tilde{N}$ is nilpotent, since
the matrices commute and $\tilde{N}$ is nilpotent. Thus, the assertion
follows with $N\coloneqq(1-\lambda\tilde{N})^{-1}\tilde{N}$, $C\coloneqq J^{-1}-\lambda$,
$P=P_{3}P_{2}P_{1}$, and $Q=P_{2}^{-1}$.
\end{proof}

It is clear that the matrices $P$, $Q$, $C$ and $N$ in the previous proposition
are not uniquely determined by $M_0$ and $M_1$. However, the size of $N$ and $C$
as well as the degree of nilpotency of $N$ are determined by $M_{0}$
and $M_{1}$ as the following proposition shows.

\begin{prop}
\label{prop:index_and_uniqueness}Let $P,Q\in\C^{n\times n}$ be invertible
such that 
\[
PM_{0}Q=\begin{pmatrix}
1 & 0\\
0 & N
\end{pmatrix},\quad PM_{1}Q=\begin{pmatrix}
C & 0\\
0 & 1
\end{pmatrix},
\]
where $C\in\C^{k\times k}$, $N\in\C^{(n-k)\times(n-k)}$ for some
$k\in\{0,\ldots,n\}$, and $N$ is nilpotent. Then $(M_{0},M_{1})$
is regular and
\begin{enumerate}
\item\label{prop:index_and_uniqueness:item:1} $k$ is the degree of the polynomial $z\mapsto\det(zM_{0}+M_{1})$.
\item\label{prop:index_and_uniqueness:item:2} $N^{\ell}=0$ if and only if
\[
\sup_{|z|\geq r}\norm{z^{-\ell+1}(zM_{0}+M_{1})^{-1}}<\infty
\]
for one (or equivalently all) $r>0$ such that $\ball{0}{r}\supseteq\sigma(M_{0},M_{1})$.
\end{enumerate}
\end{prop}

\begin{proof}
First, note that
\[
\det(zM_{0}+M_{1})=\frac{1}{\det P\,\det Q}\det\begin{pmatrix}
z+C & 0\\
0 & zN+1
\end{pmatrix}=\frac{1}{\det P\,\det Q}\det(z+C)\quad(z\in\C).
\]
Hence, $(M_{0},M_{1})$ is regular and 
\[
k=\deg\det((\cdot)+C)=\deg\det((\cdot)M_{0}+M_{1}),
\]
which shows \ref{prop:index_and_uniqueness:item:1}. Moreover, we have $\rho(M_{0},M_{1})=\rho(-C)$ and
\[
\left(zM_{0}+M_{1}\right)^{-1}=Q\begin{pmatrix}
(z+C)^{-1} & 0\\
0 & (zN+1)^{-1}
\end{pmatrix}P\quad(z\in\rho(M_{0},M_{1})),
\]
and hence, for $r>0$ with $\ball{0}{r}\supseteq\sigma(M_{0},M_{1})$
we have
\[
\norm{(zM_{0}+M_{1})^{-1}}\leq K_1\norm{(zN+1)^{-1}} \quad(|z|\geq r)
\]
for some $K_1\geq0$, since $\sup_{|z|\geq r}\norm{(z+C)^{-1}}<\infty$.
Now let $\ell\in\N$ such that $N^{\ell}=0$. Then
\[
\norm{(zN+1)^{-1}}=\norm{\sum_{k=0}^{\ell-1}(-1)^{k}z^{k}N^{k}}\leq K_2\abs{z}^{\ell-1}\quad(\abs{z}\geq r)
\]
for some constant $K_2\geq0$ and thus, 
\[
\norm{(zM_{0}+M_{1})^{-1}}\leq K_1K_2\abs{z}^{\ell-1}\quad(|z|\geq r).
\]
Assume on the other hand that 
\[
\sup_{|z|\geq r}\norm{z^{-\ell+1}(zM_{0}+M_{1})^{-1}}<\infty
\]
for some $\ell\in\N$ and $r>0$ with $\sigma(M_{0},M_{1})\subseteq\ball{0}{r}$.
Then there exist $\tilde{K}_{1},\tilde{K}_{2}\geq0$ such that
\[
\norm{(zN+1)^{-1}}\leq\norm{\begin{pmatrix}
(z+C)^{-1} & 0\\
0 & (zN+1)^{-1}
\end{pmatrix}}\leq\tilde{K}_{1}\norm{\left(zM_{0}+M_{1}\right)^{-1}}\leq\tilde{K_{2}}\abs{z}{}^{\ell-1}
\]
for all $z\in\C$ with $\abs{z}\geq r$.
Now, let $p\in\N$ be minimal such that $N^{p}=0$. We show that $p\leq\ell$
by contradiction. Assume $p>\ell$. Then we compute 
\begin{align*}
0& =\lim_{n\to\infty}\frac{1}{n^{\ell}}(nN+1)^{-1}N^{p-\ell-1}=\lim_{n\to\infty}\sum_{k=0}^{p-1}(-1)^{k}n^{k-\ell}N^{k+p-\ell-1}\\
 & = \lim_{n\to\infty}\sum_{k=0}^{\ell-1}(-1)^{k}n^{k-\ell}N^{k+p-\ell-1} + (-1)^{\ell}N^{p-1} \\
 & = (-1)^{\ell}N^{p-1},
\end{align*}
which contradicts the minimality of $p$.
\end{proof}
\begin{thm}
\label{thm:iv_matrix_case}Let $(M_{0},M_{1})$ be regular
and $P,Q\in\C^{n\times n}$ be chosen according to \prettyref{prop:Weierstrass_normal_form}. Let $k=\deg\det((\cdot)M_{0}+M_{1})$.
Then 
\[
\IV(M_{0},M_{1})=\set{U_{0}\in\C^{n}}{Q^{-1}U_{0}\in\C^{k}\times\{0\}}.
\]
Moreover, for each $U_{0}\in\IV(M_{0},M_{1})$
the solution $U$ of \prettyref{eq:DAE} is unique and satisfies $U\in C(\R_{\geq0};\C^{n})\cap C^{1}(\R_{>0};\C^{n})$
as well as 
\begin{align*}
M_{0}U'(t)+M_{1}U(t) & =0\quad(t>0),\\
U(0\rlim) & =U_{0}.
\end{align*}
\end{thm}

\begin{proof}
Let $C\in\C^{k\times k}$ and $N\in\C^{(n-k)\times(n-k)}$ be nilpotent
as in \prettyref{prop:Weierstrass_normal_form}. Obviously $U$ is
a solution of \prettyref{eq:DAE} if and only if $V\coloneqq Q^{-1}U$
is continuous on $\R_{\geq0}$ and solves 
\begin{align}
\left(\td{\nu}\begin{pmatrix}
1 & 0\\
0 & N
\end{pmatrix}+\begin{pmatrix}
C & 0\\
0 & 1
\end{pmatrix}\right)V & =0\quad\text{on }\oi{0}{\infty},\label{eq:auxP}\\
V(0\rlim) & =Q^{-1}U_{0}\eqqcolon V_{0}.\nonumber 
\end{align}
Clearly, if $Q^{-1}U_{0}=(x,0)\in\C^{k}\times\{0\}$ then $V$ given by $V(t)\coloneqq(\e^{-tC}x,0)$ for $t\geq0$ is a solution of \prettyref{eq:auxP} for $\nu>0$ large enough.
On the other hand, if $V$ given by $V(t)=(V_{1}(t),V_{2}(t))\in\C^{k}\times\C^{n-k}$ ($t\geq 0$) is
a solution of \prettyref{eq:auxP} then we have 
\[
\td{\nu}NV_{2}+V_{2}=0\quad\text{on }\oi{0}{\infty}.
\]
Since $N$ is nilpotent, there exists $\ell\in\N$ with $N^{\ell}=0$.
Hence, 
\[
N^{\ell-1}V_{2}(t) = -N^{\ell-1}\td{\nu}NV_{2}(t) = \td{\nu}N^\ell V_{2}(t) =0\quad(t>0),
\]
which in turn implies $\td{\nu}N^{\ell-1}V_{2}=0$ on $\oi{0}{\infty}$.
Using again the differential equation, we infer $N^{\ell-2}V_{2}(t)=0$
for $t>0$. Inductively, we deduce $V_{2}(t)=0$ for $t>0$
and by continuity $V_{2}(0\rlim)=0$, which yields $V_{0}=Q^{-1}U_{0}\in\C^{k}\times\{0\}$.
The uniqueness follows from \prettyref{prop:uniqueness_DAE} below.
\end{proof}

\section{The infinite-dimensional Case}

Let now $M_{0},M_{1}\in L(H)$.
Again, it is our aim to determine the space of consistent initial
values for the problem 
\begin{equation}
\begin{cases}
\left(\td{\nu}M_{0}+M_{1}\right)U=0 & \text{on }\oi{0}{\infty},\\
U(0\rlim)=U_{0}.
\end{cases}\label{eq:DAE_infinite}
\end{equation}

Here, consistent initial values are defined as in the finite-dimensional
setting:
\begin{defn*}
We call an initial value $U_{0}\in H$ \emph{consistent} \index{consistent initial value} for \prettyref{eq:DAE_infinite},
if there exist $\nu>0$ and $U\in C(\R_{\geq0};H)\cap\Lnu(\R_{\geq0};H)$ such that \prettyref{eq:DAE_infinite} holds. We denote
the set of all consistent initial values for \prettyref{eq:DAE_infinite} by 
\[
\IV(M_{0},M_{1})\coloneqq\set{U_{0}\in H}{U_{0}\text{ consistent}}.
\]
\end{defn*}
Before we try to determine $\IV(M_{0},M_{1})$ we prove a regularity
result for solutions of \prettyref{eq:DAE_infinite}.
\begin{prop}
\label{prop:regularity_DAE}Let $\nu>0$, $U_0\in H$ and $U\in C(\R_{\geq0};H)\cap\Lnu(\R_{\geq0};H)$
be a solution of \prettyref{eq:DAE_infinite}.
Then $M_{0}(U-\1_{\roi{0}{\infty}}U_{0})\in H_{\nu}^{1}(\R;H)$ and 
\[
\td{\nu}M_{0}\left(U-\1_{\roi{0}{\infty}}U_{0}\right)+M_{1}U=0.
\]
\end{prop}

\begin{proof}
  We extend $U$ to $\R$ by $0$.
  First, observe that $M_0 (U - \1_{\roi{0}{\infty}} U_0)\from \R\to H$ is continuous, since $U$ is continuous and $U(0\rlim) = U_0$.
  By \prettyref{lem:ivp} (with $A=0$), we obtain $U-\1_{\roi{0}{\infty}} U_0\in \dom\bigl(\overline{\td{\nu}M_0 + M_1}\bigr)$ and $\bigl(\overline{\td{\nu}M_0+M_1}\bigr)(U - \1_{\roi{0}{\infty}} U_0) = -M_1U_0 \1_{\roi{0}{\infty}}$. Since $\td{\nu}$ is closed and $M_0$ is bounded, $\td{\nu}M_0$ is closed as well. Since $M_1$ is bounded, therefore also
  $\td{\nu}M_0+M_1$ is closed. Thus,
  $U-\1_{\roi{0}{\infty}}U_0\in \dom(\td{\nu}M_0 + M_1) = \dom(\td{\nu} M_0)$ and therefore $M_0(U-\1_{\roi{0}{\infty}}U_0)\in \dom(\td{\nu})$, and
  \[\td{\nu}M_0(U - \1_{\roi{0}{\infty}} U_0) + M_1U = 0. \qedhere\]
\end{proof}

We now come back to the space $\IV(M_{0},M_{1})$. Since we
are now dealing with an infinite-dimensional setting, we cannot use
normal forms to determine $\IV(M_{0},M_{1})$ without dramatically restricting
the class of operators. Thus, we follow a different approach
 using so-called Wong sequences.
\begin{defn*}
We set 
\[
\IV_{0}\coloneqq H
\]
and for $k\in\N_0$ we set 
\[
\IV_{k+1}\coloneqq M_{1}^{-1}[M_{0}[\IV_{k}]].
\]
The sequence $(\IV_{k})_{k\in\N_0}$ is called the \emph{Wong sequence\index{Wong sequence}}
associated with $(M_{0},M_{1})$.
\end{defn*}
\begin{rem}
By induction, we infer $\IV_{k+1}\subseteq\IV_{k}$ for each $k\in\N_0$.
\end{rem}

As in the matrix case, we denote by 
\[
\rho(M_{0},M_{1})\coloneqq\set{z\in\C}{(zM_{0}+M_{1})^{-1}\in \bo(H)}
\]
the \emph{resolvent set of $(M_{0},M_{1})$}.
\begin{lem}
\label{lem:properties-Wong}Let $k\in\N_0$. Then:
\begin{enumerate}
\item\label{lem:properties-Wong:item:1} $M_{1}(zM_{0}+M_{1})^{-1}M_{0}=M_{0}(zM_{0}+M_{1})^{-1}M_{1}$
for each $z\in\rho(M_{0},M_{1})$.
\item\label{lem:properties-Wong:item:2} $(zM_{0}+M_{1})^{-1}M_{0}[\IV_{k}]\subseteq\IV_{k+1}$ for
each $z\in\rho(M_{0},M_{1})$.
\item\label{lem:properties-Wong:item:3} If $x\in\IV_{k}$ we find $x_{1},\ldots,x_{k+1}\in H$ such
that for each $z\in\rho(M_{0},M_{1})\setminus\{0\}$
\[
(zM_{0}+M_{1})^{-1}M_{0}x=\frac{1}{z}x+\sum_{\ell=1}^{k}\frac{1}{z^{\ell+1}}x_{\ell}+\frac{1}{z^{k+1}}(zM_{0}+M_{1})^{-1}x_{k+1}.
\]
\item\label{lem:properties-Wong:item:4} If $\rho(M_{0},M_{1})\ne\varnothing$ then $M_{1}^{-1}[M_{0}[\overline{\IV_{k}}]]\subseteq\overline{\IV_{k+1}}$.
\end{enumerate}
\end{lem}

\begin{proof}
The proof of the statements \ref{lem:properties-Wong:item:1} to \ref{lem:properties-Wong:item:3} are left as \prettyref{exer:lemma_wong_sequence}. We now prove \ref{lem:properties-Wong:item:4}. If $k=0$ there is nothing to show. So assume that
the statement holds for some $k\in\N_0$ and let $x\in M_{1}^{-1}\left[M_{0}\left[\overline{\IV_{k+1}}\right]\right]$.
Since $\overline{\IV_{k+1}}\subseteq\overline{\IV_{k}}$, we infer
$x\in M_{1}^{-1}\left[M_{0}\left[\overline{\IV_{k}}\right]\right]\subseteq\overline{\IV_{k+1}}$
by induction hypothesis. Hence, we find a sequence $(w_{n})_{n\in\N}$
in $\IV_{k+1}$ with $w_{n}\to x$. Let now $z\in\rho(M_{0},M_{1})$.
Then, by \ref{lem:properties-Wong:item:2}, we have $(zM_{0}+M_{1})^{-1}M_{0}w_{n}\subseteq\IV_{k+2}$ for
each $n\in\N$ and hence, $(zM_{0}+M_{1})^{-1}M_{0}x\in\overline{\IV_{k+2}}$.
Moreover, since $M_{1}x\in M_{0}\left[\overline{\IV_{k+1}}\right]$,
we find a sequence $(y_{n})_{n\in\N}$ in $\IV_{k+1}$ with $M_{0}y_{n}\to M_{1}x$.
Setting now 
\[x_{n}\coloneqq(zM_{0}+M_{1})^{-1}zM_{0}x+(zM_{0}+M_{1})^{-1}M_{0}y_{n}\in\overline{\IV_{k+2}}\]
(where, again, we have used \ref{lem:properties-Wong:item:2}) for $n\in\N$, we derive 
\begin{align*}
x_{n} & =(zM_{0}+M_{1})^{-1}zM_{0}x+(zM_{0}+M_{1})^{-1}M_{0}y_{n}
 =x-\left(zM_{0}+M_{1}\right)^{-1}(M_{1}x-M_{0}y_{n}) \\
 & \to x
\end{align*}
as $n\to\infty$ and thus, $x\in\overline{\IV_{k+2}}$.
\end{proof}

The importance of the Wong sequence becomes apparent if we consider
solutions of \prettyref{eq:DAE_infinite}.

\begin{lem}
\label{lem:initial_values_Wong-sequence}Assume that $\rho(M_{0},M_{1})\ne\varnothing$.
Let $\nu>0$ and $U\in\Lnu(\R_{\geq0};H)\cap C(\R_{\geq0};H)$
be a solution of \prettyref{eq:DAE_infinite}. Then $U(t)\in\bigcap_{k\in\N_0}\overline{\IV_{k}}$
for each $t\geq0$.
\end{lem}

\begin{proof}
We prove the claim, $U(t)\in \overline{\IV_k}$ for all $t\geq 0$ and $k\in \N_0$, by induction. For $k=0$ there is nothing to show.
Assume now that $U(t)\in\overline{\IV_{k}}$ for each $t\geq0$ and
some $k\in\N_0$. By \prettyref{prop:regularity_DAE} we know that 
\[
\td{\nu}M_{0}(U-\1_{\roi{0}{\infty}}U_{0})+M_{1}U=0
\]
and thus, in particular, 
\[
M_{0}U(t)-M_{0}U_{0}+\int_{0}^{t}M_{1}U(s)\d s=0\quad(t\geq0).
\]
Let now $t\geq0$ and $h>0$. Then we infer 
\[
M_{0}U(t+h)-M_{0}U(t)+M_{1}\int_{t}^{t+h}U(s)\d s=0
\]
and hence, 
\[
\int_{t}^{t+h}U(s)\d s\in M_{1}^{-1}\bigl[M_{0}[\overline{\IV_{k}}]\bigr]\subseteq\overline{\IV_{k+1}}
\]
by \prettyref{lem:properties-Wong}\ref{lem:properties-Wong:item:4}. Since $U$ is continuous,
the fundamental theorem of calculus implies $U(t)\in\overline{\IV_{k+1}}$,
which yields the assertion.
\end{proof}

In particular, the space of consistent initial values has to be a
subspace of $\bigcap_{k\in\N_0}\overline{\IV_{k}}$. We now impose an
additional constraint on the operator pair $(M_{0},M_{1})$, which is equivalent to being regular
in the finite-dimensional setting (cf.~\prettyref{prop:index_and_uniqueness}).

\begin{defn*}
We call the operator pair $(M_{0},M_{1})$ \emph{regular} if there
exists $\nu_{0}\geq0$ such that
\begin{enumerate}
\item $\C_{\Re>\nu_{0}}\subseteq\rho(M_{0},M_{1})$, and
\item\label{def:regular_pair:2} there exist $C\geq0$ and $\ell\in\N$ such that for all $z\in\C_{\Re>\nu_{0}}$ we have $\norm{(zM_{0}+M_{1})^{-1}}\leq C\abs{z}^{\ell-1}$.
\end{enumerate}
Moreover, we call the smallest $\ell\in\N$ satisfying \ref{def:regular_pair:2} the \emph{index
of $(M_{0},M_{1})$\index{index of operator pair}}, which is denoted by $\ind(M_{0},M_{1})$.
\end{defn*}

\begin{rem}
Note that for matrices $M_{0}$ and $M_{1}$ the index equals the
degree of nilpotency of $N$ in the quasi-Weierstraß normal form by \prettyref{prop:index_and_uniqueness}.
\end{rem}

From now on, we will require that $(M_{0},M_{1})$ is regular. First, we
prove an important result on the Wong sequence in this case.
\begin{prop}
\label{prop:Wong-terminates} Let $(M_{0},M_{1})$ be regular, $k\in\N_0$, and $k\geq\ind(M_{0},M_{1})$. Then
\[
\overline{\IV_{k}}=\overline{\IV_{\ind(M_{0},M_{1})}}.
\]
\end{prop}

\begin{proof}
We show that $\overline{\IV_{k}}=\overline{\IV_{k+1}}$ for each $k\geq\ind(M_{0},M_{1}).$
Since the inclusion ``$\supseteq$'' holds trivially, it suffices
to show $\IV_{k}\subseteq\overline{\IV_{k+1}}$. For doing so, let
$k\geq\ind(M_{0},M_{1})$ and $x\in\IV_{k}$. By \prettyref{lem:properties-Wong}\ref{lem:properties-Wong:item:3} we find $x_{1},\ldots,x_{k+1}\in H$ such that 
\[
(zM_{0}+M_{1})^{-1}M_{0}x=\frac{1}{z}x+\sum_{\ell=1}^{k}\frac{1}{z^{\ell+1}}x_{\ell}+\frac{1}{z^{k+1}}(zM_{0}+M_{1})^{-1}x_{k+1}
\]
for each $z\in\C_{\Re>\nu_{0}}$. Since $k\geq\ind(M_{0},M_{1}),$
we derive 
\[
z(zM_{0}+M_{1})^{-1}M_{0}x\to x\quad(\Re z\to\infty),
\]
and since the elements on the left-hand side belong to $\IV_{k+1}$,
by \prettyref{lem:properties-Wong}\ref{lem:properties-Wong:item:2}, the assertion immediately follows.
\end{proof}

We now prove that in case of a regular operator pair $(M_{0},M_{1})$
the solution of \prettyref{eq:DAE_infinite} for a consistent initial
value $U_{0}$ is uniquely determined.

\begin{prop}
\label{prop:uniqueness_DAE} Let $(M_{0},M_{1})$ be regular, $U_{0}\in\IV(M_{0},M_{1})$, and $\nu>0$
such that a solution $U\in C(\R_{\geq0};H)\cap\Lnu(\R_{\geq0};H)$ of \prettyref{eq:DAE_infinite} exists. Then this solution is unique.
In particular
\[
(\mathcal{L}_{\rho}U)(t)=\frac{1}{\sqrt{2\pi}}\bigl((\i t+\rho)M_{0}+M_{1}\bigr)^{-1}M_{0}U_{0}\quad(\text{a.e. }t\in\R)
\]
for each $\rho>\max\{\nu,\nu_{0}\}$. 
\end{prop}

\begin{proof}
By \prettyref{prop:regularity_DAE} we have $M_{0}(U-\1_{\roi{0}{\infty}}U_{0})\in H_{\nu}^{1}(\R;H)$
and 
\[
\td{\nu}M_{0}(U-\1_{\roi{0}{\infty}}U_{0})+M_{1}U=0.
\]
Applying the Fourier--Laplace transformation, $\mathcal{L}_{\rho}$, for $\rho>\max\{\nu,\nu_{0}\}$
the latter yields 
\[
(\i t+\rho)M_{0}\bigl(\mathcal{L}_{\rho}U(t)-\frac{1}{\sqrt{2\pi}}\frac{1}{\i t+\rho}U_{0}\bigr)+M_{1}\mathcal{L}_{\rho}U(t)=0\quad(\text{a.e. }t\in\R)
\]
which in turn yields 
\[
\mathcal{L}_{\rho}U(t)=\frac{1}{\sqrt{2\pi}}\bigl((\i t+\rho)M_{0}+M_{1}\bigr)^{-1}M_{0}U_{0}\quad(\text{a.e. }t\in\R)
\]
and, in particular, proves the uniqueness of the solution. 
\end{proof}

\begin{rem}
The formula in \prettyref{prop:uniqueness_DAE} shows that $U\in\Lm{\nu_{0}}(\R;H)$
for all solutions $U$ of \prettyref{eq:DAE_infinite} for a consistent
initial value $U_{0}$ and hence, we also have $M_{0}(U-\1_{\roi{0}{\infty}}U_{0})\in H_{\nu_{0}}^{1}(\R;H)$.
\end{rem}

One interesting consequence of the latter proposition is the following.
\begin{cor}
\label{cor:M_0-injective}Let $(M_{0},M_{1})$ be regular. Then the
operator $M_{0}\from \IV(M_{0},M_{1})\to H$ is injective.
\end{cor}

\begin{proof}
Let $U_{0}\in\IV(M_{0},M_{1})$ with $M_{0}U_{0}=0$. By \prettyref{prop:uniqueness_DAE},
the solution $U$ of \prettyref{eq:DAE_infinite} with $U(0\rlim)=U_{0}$
satisfies 
\[
\mathcal{L}_{\rho}U(t)=\bigl((\i t+\rho)M_{0}+M_{1}\bigr)^{-1}M_{0}U_{0}=0
\]
and hence, $U=0$, which in turn implies $U_{0}=U(0\rlim)=0$.
\end{proof}

We now want to determine the space $\IV(M_{0},M_{1})$ in terms of
the Wong sequence.
\begin{prop}
\label{prop:IV_k-in-IV} Let $(M_{0},M_{1})$ be regular. Then
\[
\IV_{\ind(M_{0},M_{1})}\subseteq\IV(M_{0},M_{1})\subseteq\overline{\IV_{\ind(M_{0},M_{1})}}.
\]
\end{prop}

\begin{proof}
The second inclusion follows from \prettyref{lem:initial_values_Wong-sequence}
and \prettyref{prop:Wong-terminates}. Let now $U_{0}\in\IV_{\ind(M_{0},M_{1})}$
and set 
\[
V(z)\coloneqq\frac{1}{\sqrt{2\pi}}(zM_{0}+M_{1})^{-1}M_{0}U_{0}\quad(z\in\C_{\Re>\nu_{0}}).
\]
Let $k\coloneqq\ind(M_{0},M_{1})$. By \prettyref{lem:properties-Wong}\ref{lem:properties-Wong:item:3} we find $x_{1},\ldots,x_{k+1}\in H$ such that 
\[
V(z)=\frac{1}{\sqrt{2\pi}}\left(\frac{1}{z}U_{0}+\sum_{\ell=1}^{k}\frac{1}{z^{\ell+1}}x_{\ell}+\frac{1}{z^{k+1}}\left(zM_{0}+M_{1}\right)^{-1}x_{k+1}\right)\quad(z\in\C_{\Re>\nu_{0}}).
\]
In particular, we read off that $V\in\cHt(\C_{\Re>\nu_{0}};H)$ and
hence, by the Theorem of Paley--Wiener (more precisely by \prettyref{cor:Lapalce_unitary})
there exists $U\in\Lm{\nu_{0}}(\R_{\geq0};H)$ such that 
\[
\left(\mathcal{L}_{\rho}U\right)(t)=V(\i t+\rho)\quad(\text{a.e. }t\in\R,\rho>\nu_{0}).
\]
Moreover, 
\[
zV(z)-\frac{1}{\sqrt{2\pi}}U_{0}=\frac{1}{\sqrt{2\pi}}\left(\sum_{\ell=1}^{k}\frac{1}{z^{\ell}}x_{\ell}+\frac{1}{z^{k}}\left(zM_{0}+M_{1}\right)^{-1}x_{k+1}\right)\quad(z\in\C_{\Re>\nu_{0}})
\]
and hence $\left(z\mapsto zV(z)-\frac{1}{\sqrt{2\pi}}U_{0}\right)\in\cHt(\C_{\Re>\nu_{0}};H)$ as well.
Since 
\begin{align*}
\left(\mathcal{L}_{\rho}\td{\rho}(U-\1_{\roi{0}{\infty}}U_{0})\right)(t) & =(\i t+\rho)\left(\mathcal{L}_{\rho}U\right)(t)-\frac{1}{\sqrt{2\pi}}U_{0}\\
 & =(\i t+\rho)V(\i t+\rho)-\frac{1}{\sqrt{2\pi}}U_{0}\quad(\text{a.e. }t\in\R,\rho>\nu_{0}),
\end{align*}
we infer $U-\1_{\roi{0}{\infty}}U_{0}\in H_{\nu_{0}}^{1}(\R;H)$ and,
thus, $U-\1_{\roi{0}{\infty}}U_{0}$ is continuous by \prettyref{thm:Sobolev_emb}.
Hence, $U\in C(\R_{\geq0};H)$ and since $\spt U\subseteq\R_{\geq0}$
we derive $U(0\rlim)=U_{0}$. Finally, by the definition of $V$,
\[
M_{0}\left(zV(z)-\frac{1}{\sqrt{2\pi}}U_{0}\right)=-\frac{1}{\sqrt{2\pi}}M_{1}(zM_{0}+M_{1})^{-1}M_{0}U_{0}=-M_{1}V(z)\quad(z\in\C_{\Re>\nu_{0}}).
\]
Hence, 
\[
\td{\nu_{0}}M_{0}(U-\1_{\roi{0}{\infty}}U_{0})+M_{1}U=0,
\]
from which we see that $U$ solves \prettyref{eq:DAE_infinite}.
\end{proof}

Finally, we treat the case when $\IV(M_{0},M_{1})$ is closed. 

\begin{thm}
\label{thm:IV-closed}
Let $(M_{0},M_{1})$ be regular and $\IV(M_{0},M_{1})$ closed. Then the
operator $S\from\IV(M_{0},M_{1})\to C(\R_{\geq0};H)$, which assigns to each
initial state, $U_0\in \IV(M_{0},M_{1})$, its corresponding solution, $U\in C(\R_{\geq0};H)$, of \prettyref{eq:DAE_infinite} is bounded in the
sense that 
\[
S_{n}\from \IV(M_{0},M_{1})\to C([0,n];H),\quad U_{0}\mapsto SU_{0}|_{[0,n]}
\]
is bounded for each $n\in\N$.
\end{thm}

\begin{proof}
By \prettyref{prop:IV_k-in-IV} we infer that $\IV(M_{0},M_{1})=\overline{\IV_{k}}$
with $k\coloneqq\ind(M_{0},M_{1})$. Let $\nu>\nu_{0}\geq0$. By \prettyref{prop:uniqueness_DAE} and \prettyref{cor:Lapalce_unitary},
there exists $C\geq0$ such that
\[
\norm{\td{\nu}^{-k}Su_{0}}_{\Lm{\nu}(\Rge{0};H)}=\norm{\left(z\mapsto z^{-k}(zM_{0}+M_{1})^{-1}M_{0}U_{0}\right)}_{\cHt(\C_{\Re>\nu};H)}\leq C\sqrt{\frac{\pi}{\nu}}\norm{M_{0}U_{0}}_{H}
\]
for each $U_{0}\in\IV(M_{0},M_{1})$, where
we have used the regularity of $(M_{0},M_{1})$ and 
\[
\norm{(z\mapsto z^{-1}M_{0}U_{0})}_{\cHt(\C_{\Re>\nu};H)}=\sqrt{\frac{\pi}{\nu}}\norm{M_{0}U_{0}}_{H}.
\]
In particular, $S\from \IV(M_{0},M_{1})\to H^{-1}(\td{\nu}^k)$ is bounded. Since $\Lm{\nu_{0}}(\R_{\geq0};H)\hookrightarrow H^{-1}(\td{\nu}^k)$
continuously, we infer that $S\from \IV(M_{0},M_{1})\to\Lm{\nu_{0}}(\R_{\geq0};H)$
is bounded by the closed graph theorem. Hence, also 
\[
S_{n}\from \IV(M_{0},M_{1})\to\L([0,n];H),\quad U_{0}\mapsto SU_{0}|_{[0,n]}
\]
is bounded for each $n\in\N$ and since $C([0,n];H)\hookrightarrow\L([0,n];H)$
continuously, we infer that $S_{n}$ is bounded with values in $C([0,n];H)$
again by the closed graph theorem.
\end{proof}

\begin{rem}
 The variant of the closed graph theorem used in the proof above is the following: Let $X,Y$ be Banach spaces and $Z$ a Hausdorff topological vector space (e.g.~a Banach space) such that $Y\hookrightarrow Z$ continuously. Let $T\from X\to Z$ be linear and continuous with $T[X]\subseteq Y$. Then $T\in \bo(X,Y)$. Indeed, by the closed graph theorem it suffices to show that $T:X\to Y$ is closed. For doing so, let $(x_n)_n$ be a sequence in $X$ with $x_n\to x$ and $Tx_n \to y$ for some $x\in X, y\in Y$. Then $Tx_n\to Tx$ in $Z$ by the continuity of $T$ and $Tx_n\to y$ in $Z$ be the continuous embedding. Hence, $y=Tx$ and thus, $T$ is closed.
\end{rem}

\section{Comments}

The theory of differential algebraic equations in finite dimensions is a very
active field. The main motivation for studying these equations comes
from the modelling of electrical circuits and from control theory
(see e.g.~\cite{Dai1989} and \prettyref{exer:concrete_example}). The main reference for the statements
presented in the first part of this lecture is the book by Kunkel
and Mehrmann \cite{Kunkel2006}. Of course, also in the finite-dimensional
case Wong sequences can be used to determine the consistent initial
values, see \prettyref{exer:Wong-sequence-matrix}. For instance, in \cite{Berger2012} the connection between
Wong sequences and the quasi-Weierstraß normal form for matrix pairs
is studied. Of course, the theory is not restricted to linear and
homogeneous problems. Indeed, in the non-homogeneous case it turns
out that the set of consistent initial values also depends on the
given right-hand side.

The theory of differential algebraic equations in infinite dimensions is less well studied than the finite-dimensional case.
We refer to \cite{Thaller1996}, where the theory of $C_{0}$-semigroups
is used to deal with such equations. Moreover, we refer to \cite{Reis2005,Reis2007}, where sequences of projectors are used to decouple the system. Moreover, there exist several references in the Russian literature, where the equations are called Sobolev type equations (see e.g. \cite{Sviridyuk2003}). The results on infinite-dimensional
problems presented here are based on \cite{TW17_ID,TW17_HI,Trostorff2019}. In \cite{TW17_ID} the focus was on systems with index $0$ with an emphasis on exponential stability and dichotomy.

We also add the following remark concerning the result in \prettyref{thm:IV-closed}.
By \prettyref{cor:M_0-injective} we know that $M_{0}\from \IV(M_{0},M_{1})\to H$
is injective. If $\IV(M_{0},M_{1})$ is closed, it follows that the
operator $C\from \dom(C)\subseteq\IV(M_{0},M_{1})\to\IV(M_{0},M_{1})$
given by 
\begin{align*}
\dom(C) & \coloneqq\set{U_{0}\in\IV(M_{0},M_{1})}{M_{1}U_{0}\in M_{0}\left[\IV(M_{0},M_{1})\right]},\\
CU_{0} & \coloneqq M_{0}^{-1}M_{1}U_{0}\quad(U_{0}\in\dom(C))
\end{align*}
is well-defined and closed. Using this operator, $C$, \prettyref{thm:IV-closed}
states that if $\IV(M_{0},M_{1})$ is closed then $-C$ generates a
$C_{0}$-semigroup on $\IV(M_{0},M_{1})$. The precise statement can
be found in \cite[Theorem 5.7]{Trostorff2019}. Moreover, $C$ is
bounded if $\IV_{\ind(M_{0},M_{1})}$ is closed (cf.~\prettyref{exer:IV_k_closed}).

\section*{Exercises}
\addcontentsline{toc}{section}{Exercises}

\begin{xca}
\label{exer:Wong-sequence-matrix}Let $M_{0},M_{1}\in\C^{n\times n}$
such that $(M_{0},M_{1})$ is regular and define the Wong sequence
$(\IV_{j})_{j\in\N_0}$ associated with $(M_{0},M_{1})$. Moreover,
let $P,Q\in\C^{n\times n}$, $C\in\C^{k\times k},$ and $N\in\C^{(n-k)\times(n-k)}$
be as in the quasi-Weierstraß normal form for $(M_{0},M_{1})$ with
$N$ nilpotent (cf.~\prettyref{prop:Weierstrass_normal_form}). We
decompose a vector $x\in\C^{n}$ into $\check{x}\in\C^{k}$ and $\hat{x}\in\C^{n-k}$
such that $x=(\check{x},\hat{x})$. Prove that 
\[
x\in\IV_{j}\Leftrightarrow \hat{Q^{-1}x} \in\ran N^{j}\quad(j\in\N_0).
\]
Moreover, show that for each $z\in\rho(M_{0},M_{1})$ we have
\[
\IV_{j}=\ran\left((zM_{0}+M_{1})^{-1}M_{0}\right)^{j}\quad(j\in\N_0).
\]
\end{xca}

\begin{xca}
\label{exer:Drazin-inverse} Let $E\in\C^{n\times n}$. We set $k\coloneqq\ind(E,1)$,
where $1$ denotes the identity matrix in $\C^{n\times n}$. A matrix
$X\in\C^{n\times n}$ is called a \emph{Drazin inverse of $E$} if
the following properties hold:
\begin{itemize}
\item $EX=XE$,
\item $XEX=X,$
\item $XE^{k+1}=E^{k}$.
\end{itemize}
Prove that each matrix $E\in\C^{n\times n}$ has a unique Drazin inverse.

Hint: For the existence consider the quasi-Weierstraß form for $(E,1)$.
\end{xca}

\begin{xca}
\label{exer:IV-via-Drazin} Let $M_{0},M_{1}\in\C^{n\times n}$ with
$(M_{0},M_{1})$ regular and $M_{0}M_{1}=M_{1}M_{0}.$ Denote by $\Drazin{M_{0}}$
the Drazin inverse of $M_{0}$ (see \prettyref{exer:Drazin-inverse}).
Prove:
\begin{enumerate}
\item $\Drazin{M_{0}}M_{1}=M_{1}\Drazin{M_{0}}$,
\item $\ran \Drazin{M_{0}}M_{0}=\IV(M_{0},M_{1})$,
\item For all $U_{0}\in\IV(M_{0},M_{1})$ the solution $U$ of \prettyref{eq:DAE}
is given by 
\[
U(t)=\e^{-t\Drazin{M_{0}}M_{1}}U_{0}\quad(t\geq0).
\]
\end{enumerate}
\end{xca}

\begin{xca}
\label{exer:M_0 and M_1 commute} Let $M_{0},M_{1}\in\C^{n\times n}$
with $(M_{0},M_{1})$ regular. Prove that there exist two matrices
$E,A\in\C^{n\times n}$ with $(E,A)$ regular and $EA=AE$ such that
\begin{itemize}
\item $\IV(E,A)=\IV(M_{0},M_{1})$,
\item $U$ solves the initial value problem \prettyref{eq:DAE} for the
matrices $M_{0},M_{1}$ if and only if $U$ solves the initial value
problem \prettyref{eq:DAE} for the matrices $E,A$ with the same
initial value $U_{0}\in\IV(M_{0},M_{1})$.
\end{itemize}
\end{xca}

\begin{xca}
\label{exer:concrete_example} We consider the following electrical
circuit (see Figure \ref{fig:circuit}) with a resistor with resistance
$R>0$, an inductor with inductance $L>0$ and a capacitor with capacitance
$C>0$. 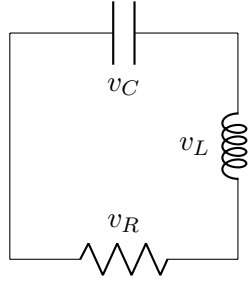
\begin{figure}
\begin{center}
\begin{circuitikz}
\draw (0,0)
to[R,R=$v_R$](3,0)
to[L,L=$v_L$](3,3)
to[C,C=$v_C$](0,3)
to[short](0,0);
\end{circuitikz}
\caption{Electrical circuit}\label{fig:circuit}
\end{center}
\end{figure}We denote the respective voltage drops by $v_{R},v_{L}$ and $v_{C}$.
Moreover, the current is denoted by $i$. The constitutive relations
for resistor, inductor and capacitor are given by 
\begin{align*}
Ri & =v_{R},\\
Li' & =v_{L},\\
Cv_{C}' & =i,
\end{align*}
respectively. Moreover, by Kirchhoff's second law we have 
\[
v_{R}+v_{C}+v_{L}=0.
\]
Write these equations as a differential algebraic equation and compute
the index and the space of consistent initial values. Moreover, compute
the solution for each consistent initial value for $R=2$ and $C=L=1$.
\end{xca}

\begin{xca}
\label{exer:lemma_wong_sequence} Prove the assertions \ref{lem:properties-Wong:item:1} to \ref{lem:properties-Wong:item:3} in
\prettyref{lem:properties-Wong}.
\end{xca}

\begin{xca}
\label{exer:IV_k_closed} Let $M_{0},M_{1}\in L(H)$.
\begin{enumerate}
\item Assume that $\rho(M_{0},M_{1})\ne\varnothing$. Prove that for
each $k\in\N$ the space $\IV_{k}$ is closed if and only
if $M_{0}\left[\IV_{k-1}\right]$ is closed.
\item Assume that $(M_{0},M_{1})$ is regular with $\ind(M_{0},M_{1})\geq1$.
Prove that if $\IV_{\ind(M_{0},M_{1})}$ is closed then the operator
\[
M_{0}|_{\IV_{\ind(M_{0},M_{1})}} \from \IV_{\ind(M_{0},M_{1})}\to M_{0}\left[\IV_{\ind(M_{0},M_{1})-1}\right]
\]
is an isomorphism.
\end{enumerate}
\end{xca}

\printbibliography[heading=subbibliography]

\chapter{Exponential Stability of Evolutionary Equations}

In this chapter we study the exponential stability of evolutionary
equations. Roughly speaking, exponential stability of a well-posed
evolutionary equation 
\[
\left(\td{\nu}M(\td{\nu})+A\right)U=F
\]
means that exponentially decaying right-hand sides $F$ lead to exponentially
decaying solutions $U$. The main problem in defining the notion of
exponential decay for a solution of an evolutionary equation is the
lack of continuity with respect to time, so a pointwise definition would not make sense
in this framework. Instead, we will use our exponentially weighted
spaces $\Lnu(\R;H)$, but this time for negative $\nu$, and define
the exponential stability by the invariance of these spaces under
the solution operator associated with the evolutionary equation under
consideration. 

\section{The Notion of Exponential Stability\label{sec:The-notion-of}}

Throughout this section, let $H$ be a Hilbert space, $M\from\dom(M)\subseteq\C\to\bo(H)$ a
material law and $A\from\dom(A)\subseteq H\to H$
a skew-selfadjoint operator. Moreover, we assume that there exist
$\nu_{0}>\sbb{M}$ and $c>0$ such that 
\[
\Re zM(z)\geq c\quad(z\in\C_{\Re\geq \nu_{0}}).
\]
By Picard's theorem (\prettyref{thm:Solution_theory_EE}) we know that for $\nu\geq \nu_{0}$ the operator
\[
S_{\nu}\coloneqq\bigl(\overline{\td{\nu}M(\td{\nu})+A}\bigr)^{-1}\in\bo(\Lnu(\R;H))
\]
is causal and independent of the particular choice of $\nu$. We now define the notion of exponential stability.
\begin{defn*}
We call the solution operators $(S_\nu)_{\nu\geq \nu_0}$
\emph{exponentially stable\index{exponentially stable} with decay
rate $\rho_{0}>0$} if for all $\rho\in\roi{0}{\rho_{0}}$ and $\nu\geq\nu_0$
we have 
\[
S_{\nu}F\in\Lm{-\rho}(\R;H)\quad(F\in\Lnu(\R;H)\cap\Lm{-\rho}(\R;H)).
\]
\end{defn*}
\begin{rem}
\label{rem:exp_decay_negative_weight}We emphasise that the definition
of exponential stability does not mean that the evolutionary equation
is just solvable for some negative weights. Indeed, if we consider
$H=\C$, $A=0$ and $M(z)=1$ for $z\in\C$
we obtain that the corresponding evolutionary equation 
\begin{equation}
\label{eq:dtnuUeqF}
\td{\nu}U=F
\end{equation}
is well-posed for each $\nu\ne0$. However, we also place a demand for
causality on our solution operator. Thus, we only have to consider parameters $\nu>0$.
We obtain the solution $U$ by
\[
U(t)=\int_{-\infty}^{t}F(s)\d s.
\]
As it turns out, the problem \prettyref{eq:dtnuUeqF} is not exponentially stable. Indeed, for $F\coloneqq\1_{[0,1]}\in\bigcap_{\nu\in\R}\Lnu(\R)$ the solution $U$ is given by
\[
U(t)=\begin{cases}
0 & \text{ if }t<0,\\
t & \text{ if }0\leq t\leq1,\\
1 & \text{ if }t>1,
\end{cases}
\]
which does not belong to the space $\Lm{-\rho}(\R)$ for any $\rho>0$. 
\end{rem}

We first show that the aforementioned notion of exponential stability
also yields a pointwise exponential decay of solutions if we assume
more regularity for our source term $F$. 
\begin{prop}
\label{prop:decay_regular_rhs} Let $(S_\nu)_{\nu\geq\nu_0}$ be exponentially stable with decay
rate $\rho_{0}>0$, $\nu\geq\nu_{0}$, $\rho\in\roi{0}{\rho_{0}}$ and $F\in\dom(\td{\nu})\cap\dom(\td{-\rho})$.
Then $U\coloneqq S_{\nu}F$
is continuous and satisfies 
\[
U(t)\e^{\rho t}\to0\quad(t\to\infty).
\]
\end{prop}

\begin{proof}
We first note that $\td{\nu}F=\td{-\rho}F$ by \prettyref{exer:derivative_invariant}.
Moreover, since $S_\nu$ is a material law operator (i.e., $S_\nu = S(\td{\nu})$ for some material law $S$; see \prettyref{rem:somlo}) we have
\[
S_{\nu}\td{\nu}\subseteq\td{\nu}S_{\nu}.
\]
Thus, in particular, we have
\[
S_{\nu}\td{\nu}F=\td{\nu}S_{\nu}F=\td{\nu}U;
\]
that is, $U\in\dom(\td{\nu})$. Moreover, since $\td{\nu}F=\td{-\rho}F\in\Lm{-\rho}(\R;H)$,
we infer also $U,\td{\nu}U\in\Lm{-\rho}(\R;H)$ by exponential
stability. By \prettyref{exer:derivative_invariant} this yields 
$U\in\dom(\td{-\rho})$ with $\td{-\rho}U=\td{\nu}U$. The assertion
now follows from the Sobolev embedding theorem (\prettyref{thm:Sobolev_emb}
and \prettyref{cor:vanish_at_inf}).
\end{proof}

\section{A Criterion for Exponential Stability of Parabolic-type Equations}

In this section we will prove a useful criterion for exponential stability
of a certain class of evolutionary equations. The easiest example
we have in mind is the heat equation with homogeneous Dirichlet
boundary conditions, which can be written as an evolutionary equation
of the form (cf.~\prettyref{thm:wp_heat})
\[
\left(\td{\nu}\begin{pmatrix}
1 & 0\\
0 & 0
\end{pmatrix}+\begin{pmatrix}
0 & 0\\
0 & a^{-1}
\end{pmatrix}+\begin{pmatrix}
0 & \dive\\
\grad_{0} & 0
\end{pmatrix}\right)\begin{pmatrix}
\theta\\
q
\end{pmatrix}=\begin{pmatrix}
Q\\
0
\end{pmatrix}
\]
in $\Lnu(\R;H)$, where $H=\L(\Omega)\oplus\L(\Omega)^{d}$ with $\Omega\subseteq\R^{d}$
open, and $a\in L(\L(\Omega)^{d})$ with 
\[
\Re a\geq c
\]
for some $c>0$ which models the heat conductivity, and $\nu>0$.

\begin{thm}
\label{thm:exp_stability_parabolic}Let $H_{0},H_{1}$ be Hilbert
spaces and $C\from\dom(C)\subseteq H_{0}\to H_{1}$ a densely defined closed
linear operator which is boundedly invertible. Moreover, let $M_{0}\in\bo(H_{0})$
be selfadjoint with 
\[
M_{0}\geq c_{0}
\]
for some $c_{0}>0$ and $M_{1}\from\dom(M_{1})\subseteq\C\to\bo(H_{1})$ be
a material law satisfying $\sbb{M_1}< -\rho_1$ for some $\rho_{1}>0$ and 
\[
\exists\, c_{1}>0\:\forall z\in\C_{\Re>-\rho_{1}}: \Re M_{1}(z)\geq c_{1}.
\]
Then
\[
S_\nu\coloneqq \left(\overline{\td{\nu}\begin{pmatrix}
M_{0} & 0\\
0 & 0
\end{pmatrix}+\begin{pmatrix}
0 & 0\\
0 & M_{1}(\td{\nu})
\end{pmatrix}+\begin{pmatrix}
0 & -C^{\ast}\\
C & 0
\end{pmatrix}}\right)^{-1} \in\bo\bigl(\Lnu(\R;H_0\oplus H_1)\bigr)
\]
for each $\nu>0$. Moreover, for all $\nu_0>0$ the family
$(S_\nu)_{\nu\geq \nu_0}$ is exponentially stable with decay rate $\rho_{0}\coloneqq\min\left\{ \rho_{1},c_{1}/\bigl(\norm{M_{1}}^2_{\infty,\C_{\Re>-\rho_{1}}}\norm{M_{0}}\norm{C^{-1}}^{2}\bigr)\right\}$.
\end{thm}

In order to prove this theorem we need a preparatory result.

\begin{lem}
\label{lem:pointwise_inverse}
Assume the hypotheses of \prettyref{thm:exp_stability_parabolic}.
Then for each $z\in\C_{\Re>-\rho_{0}}$ the operator 
\[
T(z)\coloneqq\begin{pmatrix}
zM_{0} & 0\\
0 & M_{1}(z)
\end{pmatrix}+\begin{pmatrix}
0 & -C^{\ast}\\
C & 0
\end{pmatrix}\from \dom(C)\times\dom(C^{\ast})\subseteq H_{0}\oplus H_{1}\to H_{0}\oplus H_{1}
\]
is boundedly invertible. Moreover, 
\[
\sup_{z\in\C_{\Re\geq-\rho}}\norm{T(z)^{-1}}<\infty
\]
for each $\rho<\rho_{0}$.
\end{lem}

\begin{proof}
Let $z\in\C_{\Re\geq-\rho}$ for some $\rho<\rho_{0}$. We note that
$M_{1}(z)$ is boundedly invertible with $\norm{M_{1}(z)^{-1}}\leq 1/c_{1}$ (see \prettyref{prop:block_op_realinv}\ref{prop:block_op_realinv:item:2}) and $(C^\ast)^{-1}=(C^{-1})^\ast \in \bo(H_0,H_1)$ (see \prettyref{lem:adj} and \prettyref{lem:bdds}). The beginning of the proof deals with a reformulation of $T(z)$. For this, let $u, f\in H_{0}$, $v,g\in H_{1}$. Then, by definition, $(u,v)\in\dom(T(z))=\dom(C)\times\dom(C^*)$ and $T(z)(u,v)=(f,g)$ if and only if $v\in \dom(C^*)$ and $u\in \dom(C)$ together with
\begin{align*}
 zM_0 u - C^*v &= f\\
 Cu + M_1 (z)v &=g.
\end{align*}
Since both $C^*$ and $M_1(z)$ are continuously invertible, we obtain equivalently $u\in \dom(C)$ together with
\begin{align*}
  z(C^*)^{-1}M_0 u - v &= (C^*)^{-1}f \\
   M_1 (z)^{-1}Cu + v &=M_1 (z)^{-1}g.
\end{align*}
Adding the latter two equations and retaining the first equation, we obtain the following equivalent system subject to the condition $u\in \dom(C)$
\begin{align*}
     v &=z(C^*)^{-1}(M_0 u - f)\in \dom(C^*), \\
    (z(C^*)^{-1}M_0C^{-1}  +   M_1 (z)^{-1})Cu  &=M_1 (z)^{-1}g + (C^*)^{-1}f.
\end{align*}
We now inspect the operator $S(z)\coloneqq z(C^{-1})^{\ast}M_{0}C^{-1}+M_{1}(z)^{-1}\in\bo(H_{1})$.
By \prettyref{prop:block_op_realinv} for $x\in H_{1}$ we estimate 
\begin{align*}
\Re\scp{x}{S(z)x} & =\Re\scp{C^{-1}x}{zM_{0}C^{-1}x}+\Re\scp{x}{M_{1}(z)^{-1}x}\\
 & \geq-\rho\norm{M_{0}}\norm{C^{-1}}^{2}\norm{x}^{2}+\frac{c_{1}}{\norm{M_{1}(z)}^{2}}\norm{x}^{2}\\
 & \geq\underbrace{\bigl(\frac{c_{1}}{\norm{M_{1}}_{\infty,\C_{\Re>-\rho_{1}}}^{2}}-\rho\norm{M_{0}}\norm{C^{-1}}^{2}\bigr)}_{\eqqcolon\mu}\norm{x}^{2}.
\end{align*}
Since $\rho<\rho_{0}$ and by the definition of $\rho_{0}$
we infer that $\mu>0$. Hence, $S(z)$ is boundedly invertible with
\[
\norm{S(z)^{-1}}\leq\frac{1}{\mu}.
\]
We now set 
\begin{align*}
u & \coloneqq C^{-1}S(z)^{-1}\bigl((C^{\ast})^{-1}f+M_{1}(z)^{-1}g\bigr)\in\dom(C),\\
v & \coloneqq z(C^*)^{-1}(M_0 u - f)\in \dom(C^*).
\end{align*}
By the first part of the proof we have that $\left(u,v\right)$
is the unique solution of $T(z)(u,v)=(f,g)$. Moreover, we can estimate
\begin{align*}
\norm{u} & \leq\norm{C^{-1}}\frac{1}{\mu}\Bigl(\norm{(C^{\ast})^{-1}}\norm{f}+\frac{1}{c_{1}}\norm{g}\Bigr),\text{ and }\\
\norm{v} & \leq\frac{1}{c_1}(\norm{g}+\norm{Cu}) \leq \frac{1}{c_1}\Bigl(\norm{g}+\frac{1}{\mu} \bigl(\norm{(C^{\ast})^{-1}}\norm{f}+\frac{1}{c_{1}}\norm{g}\bigr)\Bigr),
\end{align*}
which proves that $T(z)$ is boundedly invertible with 
\[
\sup_{z\in\C_{\Re\geq-\rho}}\norm{T(z)^{-1}}<\infty.\tag*{\qedhere}
\]
\end{proof}

\begin{proof}[Proof of \prettyref{thm:exp_stability_parabolic}]
Let $H\coloneqq H_{0}\oplus H_{1}$.
We set 
\[
M(z)\coloneqq\begin{pmatrix}
M_{0} & 0\\
0 & z^{-1}M_{1}(z)
\end{pmatrix}\quad(z\in\dom(M_{1})\setminus\{0\}).
\]
Let $\nu>0$. Then
\[
\forall z\in\C_{\Re\geq\nu}: \Re zM(z)\geq\min\{\nu c_{0},c_{1}\}
\]
and hence, the first assertion of the theorem follows from \prettyref{thm:Solution_theory_EE}.

Next, we focus on exponential stability. For this note that $(S_\nu)_{\nu\geq\nu_0}$ is exponentially stable if $S_\nu \in \bo(\Lnu(\R;H)\cap \Lm{-\rho}(\R;H))$ for all $\rho\in \roi{0}{\rho_0}$ and $\nu\geq\nu_0$. For this, by \prettyref{lem:simple_fcts_compact_spt},  it suffices to show the corresponding estimate on $ S_{\rmc}(\R;X)$. Therefore, we let $F\in S_{\rmc}(\R;H)\subseteq \Lnu(\R;H)$ and define
\[
U\coloneqq\left(\overline{\td{\nu}\begin{pmatrix}
M_{0} & 0\\
0 & 0
\end{pmatrix}+\begin{pmatrix}
0 & 0\\
0 & M_{1}(\td{\nu})
\end{pmatrix}+\begin{pmatrix}
0 & -C^{\ast}\\
C & 0
\end{pmatrix}}\right)^{-1}F.
\] Since the estimate $\norm{U}_{\Lnu}\leq \tilde{C} \norm{F}_{\Lnu}$ for some $\tilde{C}\geq 0$ is a consequence of the continuity of the solution operator, it is left to show that for all $\rho\in \roi{0}{\rho_0}$ there exists $C\geq 0$ such that $\norm{U}_{\Lm{-\rho}}\leq C\norm{F}_{\Lm{-\rho}}$. 

For this, we observe that since $F$ has compact support, we find  $n\in \N$ such that $\spt F\subseteq \R_{\geq-n}$. Then $U\in\Lnu(\R_{\geq-n};H)$ by causality. Since \[\left(\overline{\td{\nu}\begin{pmatrix}
M_{0} & 0\\
0 & 0
\end{pmatrix}+\begin{pmatrix}
0 & 0\\
0 & M_{1}(\td{\nu})
\end{pmatrix}+\begin{pmatrix}
0 & -C^{\ast}\\
C & 0
\end{pmatrix}}\right)^{-1}\] is a material law operator (induced by the material law $T(\cdot)^{-1}$ from \prettyref{lem:pointwise_inverse}), it commutes with the material law operator $\tau_{-n}$ (see \prettyref{rem:mlohom} for the commutation and \prettyref{rem:somlo} as well as \prettyref{exa:material_law_revisited} for the material law operator properties). Thus, we obtain
\[
\tau_{-n}U=\left(\overline{\td{\nu}\begin{pmatrix}
M_{0} & 0\\
0 & 0
\end{pmatrix}+\begin{pmatrix}
0 & 0\\
0 & M_{1}(\td{\nu})
\end{pmatrix}+\begin{pmatrix}
0 & -C^{\ast}\\
C & 0
\end{pmatrix}}\right)^{-1}\tau_{-n}F
\]
and  $\tau_{-n}F\in\Lm{-\rho}(\R_{\geq0};H)$ by assumption.
Hence, \prettyref{cor:Lapalce_unitary} yields $\mathcal{L}\tau_{-n}F\in\cHt(\C_{\Re>-\rho};H)$. Let $T(\cdot)$ be as in \prettyref{lem:pointwise_inverse}. We note here that thus \[
T(\cdot)^{-1}(\td{\nu})=\left(\overline{\td{\nu}\begin{pmatrix}
M_{0} & 0\\
0 & 0
\end{pmatrix}+\begin{pmatrix}
0 & 0\\
0 & M_{1}(\td{\nu})
\end{pmatrix}+\begin{pmatrix}
0 & -C^{\ast}\\
C & 0
\end{pmatrix}}\right)^{-1}.\]
Since $T(\cdot)^{-1}$ has a bounded analytic extension to $\C_{\Re>-\rho}$, we obtain $T(\cdot)^{-1}\mathcal{L}\tau_{-n}F \in\cHt(\C_{\Re>-\rho};H)$.
Since $\tau_{-n}U\in \Lnu(\R_{\geq0};H)$ and 
\[
\bigl(\mathcal{L}\tau_{-n}U \bigr)(z)=T(z)^{-1}\bigl(\mathcal{L}\tau_{-n}F \bigr)(z)\quad(z\in\C_{\Re>\nu}),
\]
we conclude that $\tau_{-n}U \in \Lm{-\rho}(\R_{\geq0};H)$ by \prettyref{cor:Lapalce_unitary}. The boundedness of $T(\cdot)^{-1}$ on $\C_{\Re>-\rho}$ implies the desired continuity estimate. In order to see this, let $C\coloneqq\sup_{z\in\C_{\Re\geq-\rho}}\norm{T(z)^{-1}}$. Using that $\mathcal{L}_{-\rho}$ is unitary from $\Lm{-\rho}(\R;H)$ to $\L(\R;H)$,
we compute 
\begin{align*}
\norm{U}_{\Lm{-\rho}} =\norm{\mathcal{L}_{-\rho}U}_{\L} 
 =\norm{T(\i\m-\rho)^{-1}\mathcal{L}_{-\rho}F}_{\L}
  \leq C\norm{\mathcal{L}_{-\rho}F}_{\L}
 =C\norm{F}_{\Lm{-\rho}},
\end{align*}
which is the desired continuity estimate.
\end{proof}

\section{Three Exponentially Stable Models for Heat Conduction}
\label{sec:Three_Exponentially_Stable_Models_for_Heat_Conduction}

\subsection*{The Classical Heat Equation}

We recall the classical heat equation (cf.~\prettyref{thm:wp_heat})
on an open subset $\Omega\subseteq\R^{d}$ consisting of two equations,
the heat-flux balance 
\[
\partial_t \theta+\dive q=f
\]
and Fourier's law 
\[
q=-a\grad\theta,
\]
where $f$ is a given source term and $a\in\bo(\L(\Omega)^{d})$ is
an operator modelling the heat conductivity of the underlying medium.
We will impose Dirichlet boundary conditions which will be incorporated
in our equation by replacing the operator $\grad$ by $\grad_{0}$
in Fourier's law (cf.~\prettyref{sec:First-order-Sobolev}).

In order to apply \prettyref{thm:exp_stability_parabolic} we need
that $\grad_{0}$ is boundedly invertible in some sense. This can
be shown using \emph{Poincar\'{e}'s inequality\index{Poincar\'{e}'s inequality}}.
\begin{prop}[Poincaré inequality]
\label{prop:Poincare} Let $\Omega\subseteq\R^{d}$ be open and contained
in a slab; that is, there exist 
$e\in\R^d$ with $\norm{e}=1$ and $a,b\in\R$, $a<b$ such that
\[
\Omega\subseteq\set{x\in\R^{d}}{a<\scp{e}{x}<b}.
\]
Then for each $u\in\dom(\grad_{0})$ we have
\[
\norm{u}_{\L(\Omega)}\leq(b-a)\norm{\grad_{0}u}_{\L(\Omega)^{d}}.
\]
\end{prop}

\begin{proof}
Without loss of generality, let $e=(1,0,\ldots,0)$.
Recall that, by definition, $\cci(\Omega)$ is a core for
$\grad_{0}$. Thus, it suffices to prove the assertion for functions
in $\cci(\Omega)$. Let $\varphi\in \cci(\Omega)$.
We identify $\varphi$ with its extension by $0$ to the whole of
$\R^{d}$. By the fundamental theorem of calculus, we may compute
\[
\varphi(x)=\int_{a}^{x_1}\partial_{1}\varphi(s,x_2,\ldots,x_{d})\d s\quad (x\in\Omega).
\]
Hence, by the Cauchy--Schwarz inequality and Tonelli's theorem
\begin{align*}
\int_{\Omega}\abs{\varphi(x)}^{2}\d x & =\int_{\Omega}\abs{\int_{a}^{x_1}\partial_{1}\varphi(s,x_{2},\ldots,x_{d})\d s}^{2}\d x \\
 & \leq\int_{\Omega}(b-a)\int_{a}^{b}(\partial_{1}\varphi(s,x_{2},\ldots,x_{d}))^{2}\d s\d x
 = (b-a)^{2}\int_{\Omega}\abs{\partial_{1}\varphi(x)}^{2}\d x\\
 & \leq(b-a)^{2}\norm{\grad_{0}\varphi}_{\L(\Omega)^{d}}^{2},
\end{align*}
which shows the assertion. 
\end{proof}
\begin{cor}
\label{cor:range(grad) closed} Under the assumptions of \prettyref{prop:Poincare}
the operator $\grad_{0}$ is one-to-one and $\ran(\grad_{0})$ is
closed.
\end{cor}

\begin{proof}
The injectivity follows immediately from Poincar\'{e}'s inequality. To
prove the closedness of $\ran(\grad_{0})$, let $(u_{k})_{k\in\N}$
in $\dom(\grad_{0})$ with $\grad_{0}u_{k}\to v$ in $\L(\Omega)^{d}$
for some $v\in\L(\Omega)^{d}$. By Poincar\'{e}'s inequality, we infer
that $(u_{k})_{k\in\N}$ is a Cauchy-sequence in $\L(\Omega)$ and
hence convergent to some $u\in\L(\Omega)$. By the closedness of $\grad_{0}$
we obtain $u\in\dom(\grad_{0})$ and $v=\grad_{0}u\in\ran(\grad_{0}).$ 
\end{proof}

We need another auxiliary result which is interesting in its own
right.

\begin{lem}
\label{lem:embeddings_projections} Let $H$ be a Hilbert space and
$V\subseteq H$ a closed subspace. We denote by 
\[
\iota_{V}\from V\to H,\quad x\mapsto x
\]
the canonical embedding of $V$ into $H$. Then $\iota_{V}\iota_{V}^{\ast}\from H\to H$
is the orthogonal projection on $V$ and $\iota_{V}^{\ast}\iota_{V}\from V\to V$
is the identity on $V$. 
\end{lem}

\begin{proof}
The proof is left as \prettyref{exer:embeddings}.
\end{proof}

We now come to the exponential stability of the heat equation. First,
we need to formulate both the heat-flux balance and Fourier's law
as a suitable evolutionary equation. For doing so, we assume that
$\Omega\subseteq\R^{d}$ is open and contained in a slab. Then $\ran(\grad_{0})$
is closed by \prettyref{cor:range(grad) closed}. It is clear that
we can write Fourier's law as 
\[
q=-a\grad_{0}\theta=-a\iota_{\ran(\grad_{0})}\iota_{\ran(\grad_{0})}^{\ast}\grad_{0}\theta.
\]
Hence, defining $\tilde{q}\coloneqq\iota_{\ran(\grad_{0})}^{\ast}q$
and $\tilde{a}\coloneqq\iota_{\ran(\grad_{0})}^{\ast}a\iota_{\ran(\grad_{0})}\in\bo(\ran(\grad_{0}))$,
we arrive at 
\[
\tilde{q}=-\tilde{a}\iota_{\ran(\grad_{0})}^{\ast}\grad_{0}\theta.
\]
Moreover, since $\ran(\grad_{0})^{\bot}=\ker(\dive)$, we derive from
the heat-flux balance 
\begin{align*}
f & =\partial_t\theta+\dive q
 =\partial_t\theta+\dive\iota_{\ran(\grad_{0})}\tilde{q}
\end{align*}
and hence, assuming that $\tilde{a}$ is invertible, we may write both
equations with the unknowns $(\theta,\tilde{q})$ as an evolutionary
equation in $\Lnu(\R;H)$ for $\nu>0$, where $H\coloneqq\L(\Omega)\oplus\ran(\grad_{0})$. This yields 
\begin{equation}
\left(\td{\nu}\begin{pmatrix}
1 & 0\\
0 & 0
\end{pmatrix}+\begin{pmatrix}
0 & 0\\
0 & \tilde{a}^{-1}
\end{pmatrix}+\begin{pmatrix}
0 & \dive\iota_{\ran(\grad_{0})}\\
\iota_{\ran(\grad_{0})}^{\ast}\grad_{0} & 0
\end{pmatrix}\right) \begin{pmatrix}
\theta\\
\tilde{q}
\end{pmatrix}=\begin{pmatrix}
f\\
0
\end{pmatrix}.\label{eq:heat_modified}
\end{equation}

For notational convenience, we set
\begin{equation}
C\coloneqq\iota_{\ran(\grad_{0})}^\ast\grad_{0}\from \dom(\grad_{0})\subseteq\L(\Omega)\to\ran(\grad_{0}).\label{eq:C}
\end{equation}

\begin{lem}
\label{lem:C_good} Let $\Omega\subseteq\R^{d}$ be open and contained
in a slab and $C$ as above. Then $C$ is densely defined, closed
and boundedly invertible. Moreover 
\[
C^{\ast}=-\dive\iota_{\ran(\grad_{0})}.
\]
\end{lem}

\begin{proof}
The proof is left as \prettyref{exer:closed_range}.
\end{proof}
\begin{prop}
\label{prop:exp_stab_heat}Let $\Omega\subseteq\R^{d}$ be open and contained
in a slab, $a\in\bo(\L(\Omega)^{d})$, and $c_1>0$ such that
\[
\Re a\geq c_{1}.
\]
Then $\tilde{a}\coloneqq\iota_{\ran(\grad_{0})}^{\ast}a\iota_{\ran(\grad_{0})}$
is boundedly invertible and the solution operators associated with \prettyref{eq:heat_modified} are exponentially
stable.
\end{prop}

\begin{proof}
For $x\in\ran(\grad_{0})$ we have 
\begin{align*}
\Re\scp{x}{\tilde{a}x}_{\ran(\grad_{0})} & =\Re\scp{\iota_{\ran(\grad_{0})}x}{a\iota_{\ran(\grad_{0})}x}_{\L(\Omega)^{d}}\\
 & \geq c_{1}\norm{\iota_{\ran(\grad_{0})}x}_{\L(\Omega)^{d}}^{2}
 =c_{1}\norm{x}_{\ran(\grad_{0})}^{2},
\end{align*}
and thus, $\tilde{a}$ is boundedly invertible. Hence, \prettyref{eq:heat_modified}
is an evolutionary equation of the form considered in \prettyref{thm:exp_stability_parabolic}
with $M_{0}\coloneqq \idop$, $M_{1}(z)\coloneqq\tilde{a}^{-1}$ for $z\in\C$
and $C$ given by \prettyref{eq:C}. Since $\Re\tilde{a}^{-1}\geq\frac{c_{1}}{\norm{\tilde{a}}^{2}}$,
\prettyref{thm:exp_stability_parabolic} is applicable and we derive
the exponential stability.
\end{proof}

\subsection*{The Heat Equation with Additional Delay}

Again we consider the heat equation, but now we replace Fourier's
law by
\[
q=-a_{1}\grad_{0}\theta-a_{2}\tau_{-h}\grad_{0}\theta
\]
for some operators $a_{1},a_{2}\in\bo(\L(\Omega)^{d})$ and $h>0$.
As above, we assume that $\Omega\subseteq\R^d$ is open and contained in a slab. We may
introduce $\tilde{q}\coloneqq\iota_{\ran(\grad_{0})}^{\ast}q$ and
$\tilde{a}_{j}\coloneqq\iota_{\ran(\grad_{0})}^{\ast}a_{j}\iota_{\ran(\grad_{0})}\in\bo(\L(\Omega)^{d})$
for $j\in\{1,2\}$. Moreover, we assume that there exists $c>0$ such that
\[
\Re a_{1}\geq c.
\]
By \prettyref{lem:delayInv} there exists $\nu_0>0$ such that the operator $\tilde{a}_{1}+\tilde{a}_{2}\tau_{-h}$
is boundedly invertible in $\Lnu(\R;\ran(\grad_{0}))$ and its inverse is uniformly strictly positive definite for each $\nu\geq \nu_0$. Hence, we may write the heat equation with additional
delay as an evolutionary equation of the form 
\begin{equation}
\left(\td{\nu}\begin{pmatrix}
1 & 0\\
0 & 0
\end{pmatrix}+\begin{pmatrix}
0 & 0\\
0 & (\tilde{a}_{1}+\tilde{a}_{2}\tau_{-h})^{-1}
\end{pmatrix}+\begin{pmatrix}
0 & -C^{\ast}\\
C & 0
\end{pmatrix}\right)\begin{pmatrix}
\theta\\
\tilde{q}
\end{pmatrix}=\begin{pmatrix}
f\\
0
\end{pmatrix}\label{eq:heat_delay_exp_stability}
\end{equation}
with $C$ given by \prettyref{eq:C}. 
\begin{prop}
Let $\Omega\subseteq\R^{d}$ be open and contained in a slab, $h>0$, $a_{1},a_{2}\in\bo(\L(\Omega)^{d})$, and $c>0$ such that
\[
\Re a_{1}\geq c
\]
and  $\norm{a_{2}}<c$. Then the solution operators $(S_\nu)_{\nu\geq \nu_0}$ associated with \prettyref{eq:heat_delay_exp_stability}
are exponentially stable.
\end{prop}

\begin{proof}
Note that $\norm{\tilde{a}_2}\leq \norm{a_2}<c$. We choose
\[
0<\rho_{1}<\frac{1}{h}\log\frac{c}{\norm{\tilde{a}_{2}}}.
\]
Then we estimate for $z\in\C_{\Re>-\rho_{1}}$
\[
\Re\scp{x}{\bigl(\tilde{a}_{1}+\tilde{a}_{2}\e^{-zh}\bigr)x}_{\ran(\grad_{0})}\geq(c-\norm{\tilde{a}_{2}}\e^{\rho_{1}h})\norm{x}_{\ran(\grad_{0})}^{2}.
\]
By the choice of $\rho_{1}$, we infer  $\tilde{c}\coloneqq(c-\norm{\tilde{a}_{2}}\e^{\rho_{1}h})>0$.
Hence, 
\[
M_{1}(z)\coloneqq\bigl(\tilde{a}_{1}+\tilde{a}_{2}\e^{-hz}\bigr)^{-1}\quad(z\in\C_{\Re>-\rho_{1}})
\]
is well-defined and satisfies 
\[
\Re M_{1}(z)\geq c_{1}\quad(z\in\C_{\Re>-\rho_{1}})
\]
for some $c_{1}>0$ by \prettyref{prop:block_op_realinv}. Thus, \prettyref{thm:exp_stability_parabolic}
is applicable and yields the exponential stability of \prettyref{eq:heat_delay_exp_stability}. 
\end{proof}

\subsection*{A Dual Phase Lag Model}

In this last variant of heat conduction, we replace Fourier's law
by 
\[
(1+s_{q}\partial_t)q=(1+s_{\theta}\partial_t)\grad_{0}\theta,
\]
where $s_{q},s_{\theta}>0$ are the so-called ``phases'' (cf.~\prettyref{sec:Dual-phase-lag},
where a different type of dual phase lag model is studied). The latter
equation can be reformulated as 
\[
(1+s_{q}\td{\nu})(1+s_{\theta}\td{\nu})^{-1}q=\grad_{0}\theta
\]
for $\nu>0$. Assuming that $\Omega\subseteq\R^d$ is open and contained in a slab, and defining $\tilde{q}\coloneqq\iota_{\ran(\grad_{0})}^{\ast}q$,
the dual phase lag model may be written as 
\begin{equation}
\left(\td{\nu}\begin{pmatrix}
1 & 0\\
0 & 0
\end{pmatrix}+\begin{pmatrix}
0 & 0\\
0 & (1+s_{q}\td{\nu})(1+s_{\theta}\td{\nu})^{-1}
\end{pmatrix}+\begin{pmatrix}
0 & -C^{\ast}\\
C & 0
\end{pmatrix}\right)\begin{pmatrix}
\theta\\
\tilde{q}
\end{pmatrix}=\begin{pmatrix}
f\\
0
\end{pmatrix}\label{eq:dual_phase_exp}
\end{equation}
with $C$ given by \prettyref{eq:C}.
\begin{prop}
Let $\Omega\subseteq\R^{d}$ be open and contained in a slab, $\nu_0>0$. Moreover,
let $s_{\theta}>s_{q}>0$. Then the solution operators $(S_\nu)_{\nu\geq \nu_0}$ associated with \prettyref{eq:dual_phase_exp} are
exponentially stable.
\end{prop}

\begin{proof}
Again, we note that \prettyref{eq:dual_phase_exp} is of the form
considered in \prettyref{thm:exp_stability_parabolic} with $M_{0}\coloneqq \idop$
and 
\[
M_{1}(z)\coloneqq\frac{1+s_{q}z}{1+s_{\theta}z}\quad(z\in\C\setminus\{-s_{\theta}^{-1}\}).
\]
Setting $\mu\coloneqq\frac{s_{q}}{s_{\theta}}<1$ we compute 
\[
\Re M_{1}(z)=\Re\left(\mu+\frac{(1-\mu)}{1+s_{\theta}z}\right)=\mu+(1-\mu)\frac{1+s_{\theta}\Re z}{|1+s_{\theta}z|^{2}}\geq\mu\quad(z\in\C_{\Re>-s_{\theta}^{-1}}).
\]
Thus, \prettyref{thm:exp_stability_parabolic} is applicable and hence,
the claim follows.
\end{proof}

\section{Comments}

The results of this chapter are based on the results obtained in \cite[Section 2]{Trostorff2018}.
There, Laplace transform techniques are used to characterise the exponential
stability of evolutionary equations in a slightly more general setting.
Moreover, besides parabolic-type equations, also hyperbolic-type equations
are considered, such as the damped wave equation or the equations of
visco-elasticity.

The exponential stability of partial differential equations is a well-studied
field. In particular, in the framework of $C_{0}$-semigroups, where
the exponential stability is defined via pointwise estimates due
to the continuity of solutions, several results are known. We just
mention Datko's theorem \cite{Datko1972} (see also \cite[Theorem 5.1.2]{ABHN_2011}),
which states that a $C_{0}$-semigroup is exponentially stable if
and only if the solution operator associated with the equation 
\[
\left(\td{\nu}+A\right)U=F
\]
leaves $L_{p}(\R_{\geq0};H)$ invariant for some (or equivalently
all) $p\in\roi{1}{\infty}$. As it turns out, the latter is equivalent
to the invariance of $\Lm{-\rho}(\R;H)$ for some $\rho>0$ and thus,
our notion of exponential stability coincides with the usual one used
in the theory of $C_{0}$-semigroups. Another important theorem on the exponential
stability of $C_{0}$-semigroups on Hilbert spaces is the Theorem
of Gearhart--Pr\"u{\ss} \cite{Pruess1984} (see also \cite[Chapter 5, Theorem 1.11]{Engel2000}),
where the exponential stability of a $C_{0}$-semigroup is characterised
in terms of the resolvent of its generator.

Besides the exponential stability, which is the only type of stability
studied so far within the current framework, different kinds of
asymptotic behaviours were considered for $C_{0}$-semigroups. We
just mention the celebrated Arendt--Batty--Lyubich--Vu theorem \cite{Arendt1988,Lyubich1988}
on strong stability of $C_{0}$-semigroups or the Theorem of Borichev--Tomilov
\cite{Borichev2010} on the polynomial stability of $C_{0}$-semigroups
on Hilbert spaces.

\section*{Exercises}
\addcontentsline{toc}{section}{Exercises}

\begin{xca}
\label{exer:derivative_invariant} Let $H$ be a Hilbert space, $\nu,\rho\in\R$
and $u\in L_{1,\loc}(\R;H)$. Prove the following statements:
\begin{enumerate}
\item If $u\in\dom(\td{\nu})\cap\dom(\td{\rho})$ then $\td{\nu}u=\td{\rho}u$. 
\item If $u\in\dom(\td{\nu})$ such that $u,\td{\nu}u\in\Lm{\rho}(\R;H)$
then $u\in\dom(\td{\rho})$.
\end{enumerate}
\end{xca}

\begin{xca}
\label{exer:embeddings}
Prove \prettyref{lem:embeddings_projections}.
\end{xca}

\begin{xca}
\label{exer:closed_range}Let $H_{0},H_{1}$ be Hilbert spaces and
$A\from\dom(A)\subseteq H_{0}\to H_{1}$ a densely defined closed
linear operator. Moreover, we assume that $A$ has closed
range. Show that the adjoint of the operator $\iota_{\ran(A)}^{\ast}A\from \dom(A)\subseteq H_{0}\to\ran(A)$
is given by $A^{\ast}\iota_{\ran(A)}$. If additionally $A$ is one-to-one, show that $\iota_{\ran(A)}^\ast A$ is boundedly invertible.
\end{xca}

\begin{xca}
\label{exer:integro} Let $\Omega\subseteq\R^{d}$ be open and contained
in a slab. We consider the heat conduction with a memory term given
by the equations 
\begin{align}
\td{\nu}\theta+\dive q & =f,\nonumber \\
q & =-(\idop-k\ast)\grad_{0}\theta,\label{eq:heat_integro}
\end{align}
where $k\in L_{1,-\rho_{1}}(\R_{\geq0};\R)$ for some $\rho_{1}>0$
with
\[
\int_{0}^{\infty}\abs{k(t)}\d t<1.
\]
Write \prettyref{eq:heat_integro} as a suitable evolutionary equation
and prove that this equation is exponentially stable. 
\end{xca}

\begin{xca}
\label{exer:classical_exp_decay} Let $A\in\C^{n\times n}$ for some
$n\in\N$ and consider the evolutionary equation 
\[
(\td{\nu}+A)U=F.
\]
Prove that the solution operators associated with this problem are exponentially stable if and only if $A$ has only eigenvalues with strictly positive real part. 
\end{xca}

\begin{xca}
\label{exer:Korn} Let $\Omega\subseteq\R^{d}$ be open. 

\begin{enumerate}
\item Let $\varphi\in \cci(\Omega)^{d}$. Prove \emph{Korn's
inequality \index{Korn's inequality}}
\[
\norm{\Grad\varphi}_{\L(\Omega)_{\rmsym}^{d\times d}}^{2}\geq\frac{1}{2}\sum_{j=1}^{d}\norm{\grad\varphi_{j}}_{\L(\Omega)^{d}}^{2}.
\]
\item Use Korn's inequality to prove that for $u\in\L(\Omega)^{d}$
we have 
\[
u\in\dom(\Grad_{0})\quad\Longleftrightarrow\quad\forall j\in\{1,\ldots,d\}:\:u_{j}\in\dom(\grad_{0}).
\]
Moreover, show that in either case 
\[
\frac{1}{2}\sum_{j=1}^{d}\norm{\grad_{0}u_{j}}_{\L(\Omega)^{d}}^{2}\leq\norm{\Grad_{0}u}_{\L(\Omega)_{\rmsym}^{d\times d}}^{2}\leq\sum_{j=1}^{d}\norm{\grad_{0}u_{j}}_{\L(\Omega)^{d}}^{2}.
\]
\item Let now $\Omega$ be contained in a slab. Prove
that $\Grad_{0}$ is one-to-one and has closed range. 
\end{enumerate}
\end{xca}

\begin{xca}
  Let $\Omega \subseteq \R^d$ open and $a\in \bo(\L(\Omega)^d)$ with $\Re a\geq c>0$.
  \begin{enumerate}
   \item Let $\nu>0$ and $f\in \Lnu(\R;\L(\Omega))$. 
   Moreover, assume that $\Omega$ is contained in a slab and define $\tilde{a}\coloneqq \iota_{\ran(\grad_0)}^\ast a \iota_{\ran(\grad_0)}$.
   Let $\theta\in \Lnu(\R;\L(\Omega))$, $q\in \Lnu(\R;\L(\Omega)^d)$ satisfy
   \[
    \left(\td{\nu}\begin{pmatrix}
1 & 0\\
0 & 0
\end{pmatrix}+\begin{pmatrix}
0 & 0\\
0 & {a}^{-1}
\end{pmatrix}+\begin{pmatrix}
0 & \dive\\
\grad_{0} & 0
\end{pmatrix}\right) \begin{pmatrix}
\theta\\
q
\end{pmatrix}=\begin{pmatrix}
f\\
0
\end{pmatrix}
   \]
   and $\tilde{\theta}\in \Lnu(\R;\L(\Omega))$, $\tilde{q}\in \Lnu(\R;\ran(\grad_0))$ satisfy
   \[
    \left(\td{\nu}\begin{pmatrix}
1 & 0\\
0 & 0
\end{pmatrix}+\begin{pmatrix}
0 & 0\\
0 & \tilde{a}^{-1}
\end{pmatrix}+\begin{pmatrix}
0 & \dive\iota_{\ran(\grad_{0})}\\
\iota_{\ran(\grad_{0})}^{\ast}\grad_{0} & 0
\end{pmatrix}\right) \begin{pmatrix}
\tilde{\theta}\,\\
\tilde{q}
\end{pmatrix}=\begin{pmatrix}
f\\
0
\end{pmatrix}.
   \]
Show that $(\theta,\iota_{\ran(\grad_0)}^\ast q)=(\tilde{\theta},\tilde{q}\,)$.
\item Let $\Omega$ be bounded and consider the evolutionary equation
\[
 \left(\td{\nu}\begin{pmatrix}
1 & 0\\
0 & 0
\end{pmatrix}+\begin{pmatrix}
0 & 0\\
0 & a^{-1}
\end{pmatrix}+\begin{pmatrix}
0 & \dive_0\\
\grad & 0
\end{pmatrix}\right) \begin{pmatrix}
\theta\\
q
\end{pmatrix}=\begin{pmatrix}
f\\
0
\end{pmatrix}.
\]
Show that the associated solution operators are not exponentially stable.
  \end{enumerate}

\end{xca}

\printbibliography[heading=subbibliography]

\chapter{Boundary Value Problems and Boundary Value Spaces}

This chapter is devoted to the study of inhomogeneous boundary value
problems. For this, we shall reformulate the boundary value problem
again into a form which fits within the general framework of evolutionary equations.
In order to have an idea of the type of boundary values which make sense
to study, we start off with a section that deals with the boundary
values of functions in the domain of the gradient operator defined
on a half space in $\R^d$ (for $d=1$ we have $\L(\R^{d-1})=\K$).

\section{The Boundary Values of Functions in the Domain of the Gradient}

In this section we let $\Omega\coloneqq\R^{d-1}\times\Rg{0}$ and $f\in H^{1}(\Omega)$; our aim is
to make sense of the function $\R^{d-1}\ni\check{{x}}\mapsto f(\check{{x}},0)$.
Note that this makes no sense if we only assume $f\in \L(\Omega)$
since $\R^{d-1}\times\{0\}=\partial \Omega$ is a set of
($d$-dimensional) Lebesgue-measure zero. However, if we assume $f$
to be weakly differentiable, something more can be said and the boundary
values can be defined by means of a continuous extension of the so-called
trace map. In order to properly formulate this, we need the following density
result.
\begin{thm}
\label{thm:CinfbdryDense} The set $\mathcal{D}\coloneqq\set{\phi\colon\Omega\to\K}{\exists\psi\in\cci(\R^{d})\colon\psi|_{\Omega}=\phi}$
is dense in the space $H^{1}(\Omega)$.
\end{thm}

We will need a density result for $H^1(\R^d)$ first.

\begin{lem}
\label{lem:CinfwholespaceDense} $\cci(\R^{d})$ is dense in $H^{1}(\R^{d})$.
\end{lem}

\begin{proof}
Let $f\in H^{1}(\R^{d})$. We first show that $f$ can be approximated
by functions with compact support. For this let $\phi\in\cci(\R^{d})$
with the properties $0\leq\phi\leq1$, $\phi=1$ on $\ball{0}{1/2}$
and $\phi=0$ on $\R\setminus\ball{0}{1}$. For all $k\in\N$
we put $\phi_{k}\coloneqq\phi(\cdot/k)$ and $f_{k}\coloneqq\phi_{k}f\in\L(\R^{d})$.
Then $f_k$ has support contained in $\cball{0}{k}$. The dominated
convergence theorem implies that $f_{k}\to f$ in $\L(\R^{d})$
as $k\to\infty$. Next, let $\psi\in\cci(\R^{d})^{d}$ and compute
for all $k\in\N$
\begin{align*}
-\scp{f_{k}}{\dive\psi} & =-\scp{\phi_{k}f}{\dive\psi}
 =-\scp{f}{\phi_{k}\dive\psi}
 =-\scp{f}{\dive\left(\phi_{k}\psi\right)-\left(\grad\phi_{k}\right)\cdot\psi}\\
 & =-\scp{f}{\dive\left(\phi_{k}\psi\right)}+\scp{f\grad\phi_{k}}{\psi}
 =\scp{(\grad f)\phi_{k}+\frac{1}{k}f(\grad\phi)(\cdot/k)}{\psi},
\end{align*}
which shows that $f_{k}\in\dom(\grad)=H^{1}(\R^{d})$ and $\grad f_{k}=(\grad f)\phi_{k}+\frac{1}{k}f\left(\grad\phi\right)(\cdot/k)$.
From this expression of $\grad f_k$ we observe $\grad f_k\to \grad f$ in $\L(\R^d)^d$ by dominated convergence. Hence, $f_{k}\to f$ in $\dom(\grad)=H^1(\R^d)$. 

To conclude the proof of this lemma it suffices to revisit \prettyref{exer:C_cinfty dense}.
For this, let $(\psi_{k})_{k}$ in $\cci(\R^{d})$ be a $\delta$-sequence.
Then, by \prettyref{exer:C_cinfty dense}, we infer $\psi_{k}\ast f\to f$ in $\L(\R^{d})$
as $k\to\infty$ and hence, by \prettyref{exer:convSmooth}, it follows also that
$\grad\left(\psi_{k}\ast f\right)=\psi_{k}\ast\grad f\to\grad f$
(note the component-wise definition of the convolution). A combination
of the first part of this proof together with an estimate for the
support of the convolution (see again \prettyref{exer:C_cinfty dense})
yields the assertion.
\end{proof}
\begin{proof}[Proof of \prettyref{thm:CinfbdryDense}]
Let $f\in H^{1}(\Omega)$. The approximation of $f$ by functions in $\mathcal{D}$ is done in two steps.
First, we shift $f$ in the negative $e_d$-direction to avoid the boundary, and then we convolve the shifted $f$ to obtain smooth approximants in $\mathcal{D}$.

Let $\tilde{f}\in\L(\R^{d})$ be the extension of $f$ by zero.
Put $e_{d}\coloneqq(\delta_{jd})_{j\in\{1,\ldots d\}}$, the $d$-th
unit vector. Then for all $\tau>0$ we have $\Omega+\tau e_{d}\subseteq\Omega$
and, thus by \prettyref{exer:H1translate}, we deduce  $f_{\tau}\coloneqq\tilde{f}(\cdot+\tau e_{d})|_{\Omega}\to f$ in $H^{1}(\Omega)$ as $\tau\to 0$.
Thus, it suffices to approximate $f_{\tau}$ for $\tau>0$.

Let $\tau>0$ and let $\left(\psi_{k}\right)_{k}$ in $\cci(\R^d)$ be a $\delta$-sequence.
Then $\psi_{k}\ast\tilde{f}(\cdot+\tau e_{d})\in H^{1}(\R^{d})$,
by \prettyref{exer:convSmooth}. Define $f_{k,\tau}\coloneqq\bigl(\psi_{k}\ast\tilde{f}(\cdot+\tau e_{d})\bigr)|_{\Omega}$.
Then we obtain that $f_{k,\tau}\to f_{\tau}$ in
$H^{1}(\Omega)$ as $k\to\infty$. Indeed, the only thing left to
prove is that $\grad f_{k,\tau}\to\grad f_{\tau}$ in $\L(\Omega)^{d}$
as $k\to\infty$. For this, we denote by $g$ the extension of $\grad f$
by $0$. Since $g\in\L(\R^{d})^{d}$ it suffices to show that $\grad f_{k,\tau}=\psi_{k}\ast g_{\tau}$
on $\Omega$ for all large enough $k\in\N,$ where $g_{\tau}=g(\cdot+\tau e_{d})$.
Let $k>\frac{1}{\tau}$. Then for all $x\in\Omega$
and $y\in\spt\psi_{k}\subseteq\ci{-1/k}{1/k}^{d}$ we infer $x-y+\tau e_{d}\in\Omega$.
In particular, $f(\cdot-y+\tau e_{d})\in H^{1}(\Omega)$ and $\grad f(\cdot-y+\tau e_{d})=g(\cdot-y+\tau e_{d})$.
Take $\eta\in\cci(\Omega)^{d}$ and compute
\begin{align*}
-\scp{f_{k,\tau}}{\dive\eta}_{\L(\Omega)} & =-\int_{\Omega}\int_{\R^{d}}\psi_{k}(x-y)\tilde{f}(y+\tau e_{d})^*\d y\dive\eta(x)\d x\\
 & =-\int_{\Omega}\int_{\R^{d}}\psi_{k}(y)\tilde{f}(x-y+\tau e_{d})^*\d y\dive\eta(x)\d x\\
 & =-\int_{\Omega}\int_{\ci{-1/k}{1/k}^{d}}\psi_{k}(y)f(x-y+\tau e_{d})^*\d y\dive\eta(x)\d x\\
 & =-\int_{\ci{-1/k}{1/k}^{d}}\psi_{k}(y)\scp{f(\cdot-y+\tau e_{d})}{\dive\eta}_{\L(\Omega)}\d y\\
 & =\int_{\ci{-1/k}{1/k}^{d}}\psi_{k}(y)\scp{g(\cdot-y+\tau e_{d})}{\eta}_{\L(\Omega)^{d}}\d y\\
 & =\scp{\psi_{k}\ast g_{\tau}}{\eta}_{\L(\Omega)^d}.
\end{align*}
As $\psi_{k}\ast\tilde{f}(\cdot+\tau e_{d})\in H^{1}(\R^{d})$, we
conclude the proof using \prettyref{lem:CinfwholespaceDense}.
\end{proof}
With these preparations at hand, we can define the boundary trace
of $H^{1}(\Omega)$. 
\begin{thm}
\label{thm:trace} The operator
\begin{align*}
\gamma\colon\mathcal{D}\subseteq H^{1}(\Omega) & \to\L(\R^{d-1})\\
f & \mapsto\bigl(\R^{d-1}\ni\check{x}\mapsto f(\check{x},0)\bigr)
\end{align*}
is continuous, densely defined and, thus, admits a unique continuous
extension to $H^{1}(\Omega)$ again denoted by $\gamma$. Moreover, we have
\[
\norm{\gamma f}_{\L(\R^{d-1})}\leq\left(2\norm{f}_{\L(\Omega)}\norm{\grad f}_{\L(\Omega)^{d}}\right)^{\frac{1}{2}}\leq\norm{f}_{H^{1}(\Omega)} \quad(f\in H^1(\Omega)).
\]
 
\end{thm}

\begin{proof}
Note that $\gamma$ is densely defined by \prettyref{thm:CinfbdryDense}. Let $f\in\cci(\R^{d})$ and
$\check{x}\in\R^{d-1}$. Let $R>0$ be such that $\spt f\subseteq\ball{0}{R}$.
Then
\begin{align*}
\int_{\R^{d-1}}\abs{f(\check{x},0)}^{2}\d\check{x} & = -\int_{\R^{d-1}}\int_{0}^{R}\partial_{d}\abs{f(\check{x},\hat{x})}^{2}\d\hat{x}\d\check{x} \\
& = -\int_{\Omega}\bigl(f(x)^{*}\partial_{d}f(x)+\partial_{d}f^{*}(x)f(x)\bigr)\d x\\
 & \leq2\norm{f}_{\L(\Omega)}\norm{\grad f}_{\L(\Omega)^{d}}.
\end{align*}
The remaining inequality follows from $2ab\leq a^{2}+b^{2}$ for all $a,b\in\R$.
\end{proof}

Except for one spatial dimension, where the boundary trace can be obtained
by point evaluation, the boundary trace $\gamma$ does not map onto
the whole of $\L(\R^{d-1})$. Hence, in order to define the space of
all possible boundary values for a function in $H^{1}$ one uses a
quotient construction: we set
\[
H^{1/2}(\R^{d-1})\coloneqq\set{\gamma f}{f\in H^{1}(\Omega)}
\]
 and endow $H^{1/2}(\R^{d-1})$ with the norm 
\[
\norm{\gamma f}_{H^{1/2}(\R^{d-1})}\coloneqq\inf\set{\norm{g}_{H^{1}(\Omega)}}{g\in H^{1}(\Omega),\gamma g=\gamma f}.
\]
It is not difficult to see that $H^{1/2}(\R^{d-1})$ is unitarily
equivalent to $\left(\ker\gamma\right)^{\bot}$, where the orthogonal
complement is computed with respect to the scalar product in $H^{1}(\Omega)$.
Thus, $H^{1/2}(\R^{d-1})$ is a Hilbert space.

\begin{rem}
 The norm defined on the space $H^{1/2}(\R^{d-1})$ given above is not the standard norm defined on this space. Indeed, following \cite[Section 2.3.8]{Necas} the usual norm is given by 
 \[
  \left(\norm{u}_{\L(\R^{d-1})}^2+\int_{\R^{d-1}}\int_{\R^{d-1}} \frac{|u(x)-u(y)|^2}{|x-y|^{d}}\d x\d y\right)^{1/2}
 \]
for $u\in H^{1/2}(\R^{d-1})$. However, this norm turns out to be equivalent to the norm given above, see e.g. \cite[Section 4]{Trostorff2014}. \\
As the notation of this space suggests, it is can also defined as an interpolation space between $H^1(\R^{d-1})$ and $\L(\R^{d-1})$, see \cite[Theorem 15.1]{Lions}. 
\end{rem}

\section{The Boundary Values of Functions in the Domain of the Divergence}

Let $\Omega:=\R^{d-1}\times\Rg{0}$.
There is also a space of corresponding boundary traces for the divergence operator.
Similarly to the boundary values for the domain of the gradient operator, $H^{1}(\Omega)$, the construction of the boundary trace for $H(\dive)$-vector
fields rests on a density result. The proof can be done along the
lines of \prettyref{thm:CinfbdryDense} and will be addressed in \prettyref{exer:divdesnity}.
\begin{thm}
\label{thm:divdensity} $\mathcal{D}^{d}$ is dense in $H(\dive,\Omega)$. 
\end{thm}

Equipped with this result, we can describe all possible boundary values
of $H(\dive,\Omega)$. It will turn out that vector
fields in $H(\dive,\Omega)$ have a well-defined \emph{normal}
trace, which for $\Omega=\R^{d-1}\times \R_{>0}$ is just the negative of the last coordinate of the vector field.
\begin{thm}
\label{thm:divetrace} The operator 
\begin{align*}
\gamma_{\mathrm{n}}\colon\mathcal{\mathcal{D}}^{d}\subseteq H(\dive,\Omega) & \to\left(H^{1/2}(\R^{d-1})\right)'\eqqcolon H^{-1/2}(\R^{d-1})\\
q & \mapsto\bigl(\R^{d-1}\ni\check{x}\mapsto -q_{d}(\check{x},0)\bigr),
\end{align*}
is densely defined, continuous with norm bounded by $1$ and has dense
range. Thus $\gamma_{\mathrm{n}}$ admits a unique extension to $H(\dive,\Omega)$ again denoted by $\gamma_{\mathrm{n}}$. 
Here, $-q_{d}$ is the negative of the $d$-th component of $q$ pointing into
the outward normal direction of $\Omega$ and $-q_{d}$ is identified with the
linear functional 
\[
H^{1/2}(\R^{d})\ni\gamma f\mapsto\scp{-q_{d}(\cdot,0)}{\gamma f}_{\L(\R^{d-1})}.
\]
Moreover, for all $f\in\dom(\grad)$ and $q\in\dom(\dive)$ we have
\begin{equation}
\scp{\dive q}{f}+\scp{q}{\grad f}=(\gamma_{\mathrm{n}} q)(\gamma f),\label{eq:ibptrace}
\end{equation}
where we denoted the extension of $\gamma_{\mathrm{n}}$ again by $\gamma_{\mathrm{n}}$.
\end{thm}

\begin{proof}
Let $f\in\mathcal{D}$
and $q\in\mathcal{D}^{d}$. Then integration by parts yields
\begin{align*}
\scp{\dive q}{f}+\scp{q}{\grad f} & =\int_{\Omega}\dive(q^\ast f)
 =\int_{\R^{d-1}}\scp{q^\ast(\check{x},0)f(\check{x},0)}{-e_d}\d\check{x}\\
 & =-\int_{\R^{d-1}} \gamma q_{d}^\ast \gamma f
 =\scp{\gamma_{\mathrm{n}}q}{\gamma f}_{\L(\R^{d-1})}=(\gamma_{\mathrm{n}} q)(\gamma f).
\end{align*}
Hence, 
\[
\abs{\scp{\gamma_{\mathrm{n}}q}{\gamma f}_{\L(\R^{d-1})}}\leq\norm{q}_{H(\dive)}\norm{f}_{H^{1}}.
\]
Since $\mathcal{D}$ is dense in $H^{1}(\Omega)$, the inequality remains
true for all $f\in H^{1}(\Omega)$. Thus,
\[
\abs{\scp{\gamma_{\mathrm{n}}q}{\gamma f}_{\L(\R^{d-1})}}\leq\norm{q}_{H(\dive)}\norm{f}_{H^{1}}\quad(f\in H^{1}(\Omega)).
\]
Computing the infimum over all $g\in H^{1}(\Omega)$ with $\gamma g=\gamma f,$
we deduce
\[
\abs{\scp{\gamma_{\mathrm{n}}q}{\gamma f}_{\L(\R^{d-1})}}\leq\norm{q}_{H(\dive)}\norm{\gamma f}_{H^{1/2}(\R^{d-1})}\quad(f\in H^{1}(\Omega)).
\]
Therefore $\gamma_{\mathrm{n}}q\in H^{-1/2}(\R^{d-1})$ and $\norm{\gamma_{\mathrm{n}}q}_{H^{-1/2}}\leq\norm{q}_{H(\dive)}$,
which shows continuity of $\gamma_{\mathrm{n}}$. It is left to show
that ${\gamma}_{\mathrm{n}}$ has dense range. For this, take $\gamma f\in H^{1/2}(\R^{d-1})$
for some $f\in H^{1}(\Omega)$ such that
\[
\scp{\gamma_{\mathrm{n}}g}{\gamma f}_{\L(\R^{d-1})}=0
\]
for all $g\in\mathcal{D}^{d}$. Next, take $\tilde{g}\in\cci(\R^{d-1})$
and $\psi\in\cci(\R)$ with $\psi(0)=1$. Then we set $g\colon\Omega\ni(\check{x},\hat{x})\mapsto -e_{d}\tilde{g}(\check{x})\psi(\hat{x}) \in\mathcal{D}^{d}$
and note that $\gamma_{\mathrm{n}}g=\tilde{g}$. Hence 
\[
\scp{\gamma f}{\tilde{g}}_{\L(\R^{d-1})}=0\quad(\tilde{g}\in\cci(\R^{d-1})).
\]
Thus, $\gamma f=0$, which implies that the range of $\gamma_{\mathrm{n}}$
is dense, as $H^{1/2}(\R^{d-1})$ is a Hilbert space. The remaining
formula \prettyref{eq:ibptrace} follows by continuously extending both the left- and right-hand
side of the integration by parts formula from the beginning of the
proof. Note that for this, we have used both \prettyref{thm:CinfbdryDense}
and \prettyref{thm:divdensity}.
\end{proof}

\begin{cor}
\label{cor:CharDiv0}Let $f\in H^1(\Omega)$, $q\in H(\dive,\Omega)$. Then $f\in \dom(\grad_0)$ if and only if $\gamma f = 0$, and 
$q\in\dom(\dive_{0})$ if and only if $\gamma_{\mathrm{n}}q=0$.
\end{cor}

\begin{proof}
We only show the statement for $q$. The proof for $f$ is analogous.
If $q\in\dom(\dive_{0})$, then there exists a sequence $(\psi_{n})_{n}$
in $\cci(\Omega)^{d}$ such that $\psi_{n}\to q$ in
$H(\dive,\Omega)$ as $n\to\infty$. Thus, by continuity
of $\gamma_{\mathrm{n}}$, we infer $0=\gamma_{\mathrm{n}}\psi_{n}\to\gamma_{\mathrm{n}}q$.
Assume on the other hand that $q\in\dom(\dive)$ with $\gamma_{\mathrm{n}}q=0$.
Using \prettyref{eq:ibptrace}, we obtain for all $f\in\dom(\grad)$
\[
\scp{\dive q}{f}+\scp{q}{\grad f}=0.
\]
This equality implies that $q\in\dom(\grad^{*})=\dom(\dive_{0})$, which shows the remaining assertion.
\end{proof}
The remaining part of this section is devoted to showing that the continuous
extension of $\gamma_{\mathrm{n}}$ maps onto $H^{-1/2}(\R^{d-1})$.
For this we require the following observation, which will also be needed
later on.
\begin{prop}
\label{prop:BDdiv}Let $U\subseteq\R^{d}$ be open. Then
\[
H_{0}(\dive,U)^{\bot_{H(\dive,U)}}=\set{q\in H(\dive,U)}{\dive q\in H^{1}(U),q=\grad\dive q}.
\]
\end{prop}

\begin{proof}
Let $q\in H(\dive,U)$. Then $q\in H_{0}(\dive,U)^{\bot_{H(\dive,U)}}$
if and only if for all $r\in H_{0}(\dive,U)$ we have
\begin{align*}
0 & =\scp{r}{q}_{H(\dive,U)}
 =\scp{r}{q}_{\L(U)^{d}}+\scp{\dive r}{\dive q}_{\L(U)}
 =\scp{r}{q}_{\L(U)^{d}}+\scp{\dive_{0}r}{\dive q}_{\L(U)}.
\end{align*}
The latter, in turn, is equivalent to $\dive q\in\dom(\dive_{0}^{*})=\dom(\grad)=H^1(U)$
and $-\grad\dive q=\dive_{0}^{*}\dive q=-q$.
\end{proof}
\begin{thm}
\label{thm:NeuTrOnto}$\gamma_{\mathrm{n}}$ maps onto
$H^{-1/2}(\R^{d-1})$. In particular, we have
\[
\norm{q}_{H(\dive,\Omega)}\leq\norm{\gamma_{\mathrm{n}}q}_{H^{-1/2}(\R^{d-1})}
\]
for all $q\in H_{0}(\dive,\Omega)^{\bot_{H(\dive,\Omega)}}$.
\end{thm}

\begin{proof}
By \prettyref{thm:divetrace}
it suffices to show that $\gamma_{\mathrm{n}}$ has closed
range. For this, it suffices to show that there exists $c>0$ such that
\[
\norm{q}_{H(\dive,\Omega)}\leq c\norm{\gamma_{\mathrm{n}}q}_{H^{-1/2}(\R^{d-1})}
\]
for all $q\in\ker(\gamma_{\mathrm{n}})^{\bot_{H(\dive,\Omega)}}$. 
By \prettyref{cor:CharDiv0}, we obtain $\ker(\gamma_{\mathrm{n}})=H_{0}(\dive,\Omega)$.
Hence, by \prettyref{prop:BDdiv}, we deduce that $q\in\ker(\gamma_{\mathrm{n}})^{\bot_{H(\dive,\Omega)}}$
if and only if $q\in\dom(\grad\dive)$ and $q=\grad\dive q$. So,
assume that $q\in\dom(\grad\dive)$ with $q=\grad\dive q$. Then
\prettyref{eq:ibptrace} applied to $q\in\dom(\dive)$ and $f=\dive q\in\dom(\grad)$ yields
\begin{align*}
(\gamma_{\mathrm{n}}q)(\gamma \dive q) & =\scp{\dive q}{\dive q}+\scp{ q}{\grad\dive q}
 =\scp{\dive q}{\dive q}+\scp{q}{q}\\
& =\norm{q}_{H(\dive,\Omega)}^{2},
\end{align*}
where we used $\grad\dive q=q$.
Hence 
\begin{align*}
\norm{q}_{H(\dive,\Omega)}^{2} & \leq\norm{\gamma\dive q}_{H^{1/2}}\norm{\gamma_{\mathrm{n}}q}_{H^{-1/2}}
 \leq\norm{\dive q}_{H^{1}(\Omega)}\norm{\gamma_{\mathrm{n}}q}_{H^{-1/2}}\\
&  =\norm{q}_{H(\dive,\Omega)}\norm{\gamma_{\mathrm{n}}q}_{H^{-1/2}}
\end{align*}
where we again used that $\grad\dive q = q$.
This yields the assertion.
\end{proof}

\section{Inhomogeneous Boundary Value Problems}

Let $\Omega\coloneqq \R^{d-1}\times\Rg{0}$.
With the notion of traces we now have a tool at hand that allows us
to formulate inhomogeneous boundary value problems.  Here we focus on the
scalar wave type equation for given Neumann data $\tilde{g}\in H^{-1/2}(\R^{d-1})$. We shall address
other boundary value problems in the exercises.
Let $M\colon\dom(M)\subseteq\C\to \bo\bigl(\L(\Omega)\times\L(\Omega)^{d}\bigr)$ be a material law with $\sbb{M}<\nu_0$ for some $\nu_0\in\R$. We assume that $M$ satisfies the positive definiteness condition in \prettyref{thm:Solution_theory_EE}; that is, we assume there exists $c>0$ such that for all $z\in \C_{\Re \geq \nu_0}$ we have $\Re zM(z) \geq c$.
For $\nu\geq\nu_0$ we want to solve
\[
\begin{cases}
\left(\td{\nu}M(\td{\nu})+\begin{pmatrix}
 0  &  \dive\\
 \grad\   &  0 
\end{pmatrix}\right)\begin{pmatrix}
v\\
q
\end{pmatrix}=\begin{pmatrix}
0\\
0
\end{pmatrix} & \text{ on }\Omega,\\
\gamma_{\mathrm{n}}q(t,\cdot)=\tilde{g} & \text{ on }\partial\Omega\text{ for all }t\in\R.
\end{cases}
\]
Let us reformulate this problem. Let $\phi\in C^{\infty}(\R)$ such
that $0\leq\phi\leq1$ with $\phi=1$ on $\roi{0}{\infty}$ and $\phi=0$
on $\loi{-\infty}{-1}$. We define the function $g\coloneqq \bigl(t\mapsto\phi(t)\tilde{g}\in H^{-1/2}(\R^{d-1})\bigr)\in\bigcap_{\nu>0}\Lnu(\R;H^{-1/2}(\R^{d-1}))$
and consider
\begin{equation}
\begin{cases}
\left(\td{\nu}M(\td{\nu})+\begin{pmatrix}
 0  &  \dive\\
 \grad\   &  0 
\end{pmatrix}\right)\begin{pmatrix}
v\\
q
\end{pmatrix}=\begin{pmatrix}
0\\
0
\end{pmatrix} & \text{ on }\Omega,\\
\gamma_{\mathrm{n}}q(t)=g(t) & \text{ for all }t>0.
\end{cases}\label{eq:inhBVP}
\end{equation}
instead. 
\begin{thm}
\label{thm:wpInhBVP} Let $\nu\geq \max\{\nu_0,0\},\nu\neq0$. Then \prettyref{eq:inhBVP} admits a unique solution
$(v,q)\in H_{\nu}^{1}\Bigl(\R;\dom\Bigl(\begin{pmatrix}
0 & \dive\\
\grad & 0
\end{pmatrix}\Bigr)\Bigr)$. 
\end{thm}

\begin{proof}
We start with the existence part. By \prettyref{thm:NeuTrOnto}, we
find $\tilde{G}\in H(\dive,\Omega)$ such that $\gamma_{\mathrm{n}}\tilde{G}=\tilde{g}$;
set $G\coloneqq\phi(\cdot)\tilde{G}\in H_{\nu}^{3}(\R;H(\dive,\Omega))$.
Consider the following evolutionary equation
\begin{align*}
 & \left(\overline{\td{\nu}M(\td{\nu})+\begin{pmatrix}
0 & \dive_{0}\\
\grad & 0
\end{pmatrix}}\right)\begin{pmatrix}
u\\
r
\end{pmatrix}
 =\td{\nu}M(\td{\nu})\begin{pmatrix}
0\\
-G
\end{pmatrix}+\begin{pmatrix}
-\dive G\\
0
\end{pmatrix}.
\end{align*}
Note that the right-hand side is in $H_{\nu}^{2}(\R;\L(\Omega)\times \L(\Omega)^d)$.
By \prettyref{thm:Solution_theory_EE}, we obtain
\begin{align*}
 \begin{pmatrix}
u\\
r
\end{pmatrix}
 & =\left(\overline{\td{\nu}M(\td{\nu})+\begin{pmatrix}
0 & \dive_{0}\\
\grad & 0
\end{pmatrix}}\right)^{-1}\left(\td{\nu}M(\td{\nu})\begin{pmatrix}
0\\
-G
\end{pmatrix}+\begin{pmatrix}
-\dive G\\
0
\end{pmatrix}\right)\\
  & \in H_{\nu}^{1}(\R;\L(\Omega)\times \L(\Omega)^d)\cap \Lnu\Bigl(\R;\dom\Bigl(\begin{pmatrix}
0 & \dive\\
\grad & 0
\end{pmatrix}\Bigr)\Bigr) .
\end{align*}
Indeed, since the solution operator commutes with $\td{\nu}$ and the right-hand side lies in $H_\nu^2$, it even follows that $\begin{pmatrix}    
                                                                                                                          u\\
                                                                                                                          r
                                                                                                                         \end{pmatrix} \in H_\nu^2(\R;\L(\Omega)\times \L(\Omega)^d)$.
 From the equality 
\begin{align*}
 \left(\td{\nu}M(\td{\nu})+\begin{pmatrix}
0 & \dive_{0}\\
\grad & 0
\end{pmatrix}\right)\begin{pmatrix}
u\\
r
\end{pmatrix}
 & =\td{\nu}M(\td{\nu})\begin{pmatrix}
0\\
-G
\end{pmatrix}+\begin{pmatrix}
-\dive G\\
0
\end{pmatrix}
\end{align*}
it follows that 
\[
\begin{pmatrix}
u\\
r
\end{pmatrix}\in  H_{\nu}^{1}\Bigl(\R;\dom\Bigl(\begin{pmatrix}
0 & \dive_{0}\\
\grad & 0
\end{pmatrix}\Bigr)\Bigr).
\]
Also, we deduce that 
\[
\left(\td{\nu}M(\td{\nu})+\begin{pmatrix}
0 & \dive\\
\grad & 0
\end{pmatrix}\right)\begin{pmatrix}
u\\
r+G
\end{pmatrix}=\begin{pmatrix}
0\\
0
\end{pmatrix}.
\]
Since $r\in H_{\nu}^{1}(\R;\dom(\dive_{0}))$, by \prettyref{cor:CharDiv0}
and \prettyref{thm:Sobolev_emb} we obtain
\[
\gamma_{\mathrm{n}}\left((r+G)(t)\right)=\gamma_{\mathrm{n}}G(t)=g(t) \quad (t\in \R).
\]
Hence, $(u,r+G)$ solves \prettyref{eq:inhBVP}. 

Next we address the uniqueness result. For this we note that a straightforward
computation shows
\[
\begin{pmatrix}
v\\
q-G
\end{pmatrix}
  =\left(\overline{\td{\nu}M(\td{\nu})+\begin{pmatrix}
0 & \dive_{0}\\
\grad & 0
\end{pmatrix}}\right)^{-1}\left(\td{\nu}M(\td{\nu})\begin{pmatrix}
0\\
-G
\end{pmatrix}+\begin{pmatrix}
-\dive G\\
0
\end{pmatrix}\right),
 \]
which coincides with the formula for $(u,r+G)$. 
\end{proof}
 The upshot of the rationale exemplified in the proof is that inhomogeneous
boundary value problems can be reduced to an evolutionary equation
of the standard form with non-vanishing right-hand side. Of course
the treatment of inhomogeneous Dirichlet data works along similar
lines. 

\section{Abstract Boundary Data Spaces}

Of course inhomogeneous boundary value problems can be addressed for
other domains $\Omega$ than the half space $\R^{d-1}\times\Rg{0}$.
Classically, some more specific properties need to be imposed on the
description of the boundary $\partial\Omega$. In this section, however,
we deviate from the classical perspective in as much as we like to
consider \emph{arbitrary} open sets $\Omega\subseteq\R^{d}$. For
this we introduce
\begin{align*}
\BD(\dive) & =\set{q\in H(\dive,\Omega)}{\dive q\in\dom(\grad),\grad\dive q=q},\\
\BD(\grad) & =\set{u\in H^{1}(\Omega)}{\grad u\in\dom(\dive),\dive\grad u=u}.
\end{align*}
By \prettyref{prop:BDdiv} and \prettyref{exer:H=00003DW+Kasuga},
these spaces are closed subspaces of $H(\dive,\Omega)$ and $H^{1}(\Omega)$,
respectively, and therefore Hilbert spaces. Indeed, 
\[
\BD(\dive)=H_{0}(\dive,\Omega)^{\bot_{H(\dive,\Omega)}}
\]
and 
\[
\BD(\grad)=H^1_{0}(\Omega)^{\bot_{H^1(\Omega)}}.
\]
Now, we are in a position to solve inhomogeneous boundary
value problems, where the trace mappings $\gamma$ and $\gamma_{\mathrm{n}}$
are replaced by the canonical orthogonal projections $\pi_{\BD(\grad)}$
and $\pi_{\BD(\dive)}$ respectively; see \prettyref{exer:uniqunessBVP}.
We devote the rest of this section to describe the relationship between
the classical trace spaces introduced before and the $\BD$-spaces.
In the perspective outlined here, there is not much of a difference
between Neumann boundary values and Dirichlet boundary values. The
next result is an incarnation of this.
\begin{prop}
\label{prop:divgradunit} We have 
\[\grad[\BD(\grad)]\subseteq\BD(\dive)\quad\text{and} \quad \dive[\BD(\dive)]\subseteq\BD(\grad).\]
Moreover,
the mappings 
\begin{align*}
\grad_{\BD}\colon\mathrm{BD(\grad)} & \to \BD(\dive),\\
u & \mapsto\grad u
\end{align*}
 and 
\begin{align*}
\dive_{\BD}\colon\mathrm{BD(\dive)} & \to \BD(\grad),\\
q & \mapsto\dive q
\end{align*}
 are unitary, and $\grad_{\BD}^{*}=\dive_{\BD}$.
\end{prop}

\begin{proof}
Let $\phi\in\BD(\grad)$. Then $\grad\phi\in H(\dive,\Omega)$
and $\dive\grad\phi=\phi$. This implies $\dive\grad\phi\in\dom(\grad)$
and $\grad\dive\grad\phi=\grad\phi,$ which yields $\grad\phi\in\BD(\dive)$.
Thus, $\grad_{\BD}$ is defined everywhere; interchanging the roles
of $\grad$ and $\dive$, we obtain $\dive_{\BD}$ is also defined
everywhere. We infer $\dive_{\BD}\grad_{\BD}=\idop_{\BD(\grad)}$
and $\grad_{\BD}\dive_{\BD}=\idop_{\BD(\dive)}$ and thus $\grad_{\BD}$
is bijective with $\grad_{\BD}^{-1}=\dive_{\BD}$. It remains to show
that $\grad_{\BD}$ preserves the norm. For this we compute 
\begin{align*}
\scp{\grad_{\BD}\phi}{\grad_{\BD}\phi}_{\BD(\dive)} & =\scp{\grad\phi}{\grad\phi}_{H(\dive)}\\
 & =\scp{\grad\phi}{\grad\phi}_{\L(\Omega)^{d}}+\scp{\dive\grad\phi}{\dive\grad\phi}_{\L(\Omega)}\\
 & =\scp{\grad\phi}{\grad\phi}_{\L(\Omega)^{d}}+\scp{\phi}{\phi}_{\L(\Omega)}\\
 & =\scp{\phi}{\phi}_{\dom(\grad)}=\scp{\phi}{\phi}_{\BD(\grad)},
\end{align*}
which implies that $\grad_{\BD}$ is unitary. Hence, $\dive_{\BD}=\grad_{\BD}^{-1}=\grad_{\BD}^{*}$.
\end{proof}
It is also possible to show an `integration by parts' formula analogous
to \prettyref{eq:ibptrace} for the abstract situation:
\begin{prop}
\label{prop:ibpabstract}Let $u\in H^{1}(\Omega)$ and $q\in H(\dive,\Omega)$.
Then
\begin{align*}
\scp{\dive q}{u}_{\L(\Omega)}+\scp{q}{\grad u}_{\L(\Omega)^{d}} & =\scp{\dive_{\BD}\pi_{\BD(\dive)}q}{\pi_{\BD(\grad)}u}_{\BD(\grad)}\\
 & =\scp{\pi_{\BD(\dive)}q}{\grad_{\BD}\pi_{\BD(\grad)}u}_{\BD(\dive)}.
\end{align*}
\end{prop}

\begin{proof}
We decompose $u=u_{0}+u_{1}$ and $q=q_{0}+q_{1}$ with $u_{0}\in H_{0}^{1}(\Omega)$,
$q_{0}\in H_{0}(\dive,\Omega)$, $u_{1}=\pi_{\BD(\grad)}u$
and $q_{1}=\pi_{\BD(\dive)}q$. Then we obtain
\begin{align*}
 & \scp{\dive q}{u}_{\L(\Omega)}+\scp{q}{\grad u}_{\L(\Omega)^{d}}\\
 & =\scp{\dive_{0}q_{0}}{u}_{\L(\Omega)}+\scp{\dive q_{1}}{u}_{\L(\Omega)}+\scp{q_{0}}{\grad u}_{\L(\Omega)^{d}}+\scp{q_{1}}{\grad u}_{\L(\Omega)^{d}}\\
 & =\scp{q_{0}}{-\grad u}_{\L(\Omega)^d}+\scp{\dive q_{1}}{u}_{\L(\Omega)}+\scp{q_{0}}{\grad u}_{\L(\Omega)^{d}}+\scp{q_{1}}{\grad u}_{\L(\Omega)^{d}}\\
 & =\scp{\dive q_{1}}{u_{0}}_{\L(\Omega)}+\scp{\dive q_{1}}{u_{1}}_{\L(\Omega)}+\scp{q_{1}}{\grad u_{0}}_{\L(\Omega)^{d}}+\scp{q_{1}}{\grad u_{1}}_{\L(\Omega)^{d}}\\
 & =\scp{q_{1}}{-\grad_0 u_{0}}_{\L(\Omega)^d}+\scp{\dive q_{1}}{u_{1}}_{\L(\Omega)}+\scp{q_{1}}{\grad_0 u_{0}}_{\L(\Omega)^{d}}+\scp{q_{1}}{\grad u_{1}}_{\L(\Omega)^{d}}\\
 & =\scp{\dive q_{1}}{u_{1}}_{\L(\Omega)}+\scp{q_{1}}{\grad u_{1}}_{\L(\Omega)^{d}}\\
 & =\scp{\dive q_{1}}{u_{1}}_{\L(\Omega)}+\scp{\grad\dive q_{1}}{\grad u_{1}}_{\L(\Omega)^{d}}=\scp{\dive q_{1}}{u_{1}}_{\BD(\grad)}.
\end{align*}
The remaining equality follows from the unitarity of $\grad_{\BD}$ and
$\grad_{\BD}^{-1}=\dive_{\BD}$ by \prettyref{prop:divgradunit}.
\end{proof}
In view of \prettyref{prop:ibpabstract} the proper replacement of $\gamma_{\mathrm{n}}$
appears to be $\dive_{\BD}\pi_{\BD(\dive)}$ instead of just
$\pi_{\BD(\dive)}$. Next, we show the equivalence of the
trace spaces for the half space and the abstract ones introduced in
this section.
\begin{thm}
\label{thm:abstractconcretesame}Let $\Omega\coloneqq\R^{d-1}\times\Rg{0}$.
Then $\gamma|_{\BD(\grad)}:\BD(\grad)\to H^{1/2}(\R^{d-1})$ and $\gamma_{\mathrm{n}}|_{\BD(\dive)}:\BD(\dive)\to H^{-1/2}(\R^{d-1})$
are unitary mappings.
\end{thm}

\begin{proof}
We begin with $\gamma_{\mathrm{n}}$. We have shown in \prettyref{thm:divetrace}
that $\gamma_{\mathrm{n}}|_{\BD(\dive)}$ is continuous and
in \prettyref{thm:NeuTrOnto} it has been shown that $(\gamma_{\mathrm{n}}|_{\BD(\dive)})^{-1}$
is continuous. Also the two norm inequalities have been established.

The injectivity of $\gamma|_{\BD(\grad)}$ follows from $\ker\gamma=H_{0}^{1}(\Omega)$ by \prettyref{cor:CharDiv0}.
All that remains simply relies upon recalling that $H^{1/2}(\R^{d-1})$ is isomorphic to $\left(\ker\gamma\right)^{\bot}$
with the orthogonal complement computed in $H^{1}(\Omega)$. 
\end{proof}

\section{Robin Boundary Conditions}

The classical Robin boundary conditions involve both traces, the Dirichlet
trace $\gamma$ and the Neumann trace $\gamma_{\mathrm{n}}$. To motivate
things, let us again have a look at the case $\Omega=\R^{d-1}\times\Rg{0}$.
We consider the boundary condition for given $q\in H(\dive,\Omega)$
and $u\in H^{1}(\Omega)$
\[
\gamma_{\mathrm{n}}q+\i\gamma u=0,
\]
in the sense that
\[
 (\gamma_{\mathrm{n}}q)(v)= \scp{-\i \gamma u}{v}_{\L(\R^{d-1})} \quad (v\in H^{1/2}(\R^{d-1})).
\]
Note that this is an implicit regularity statement as  $\gamma_{\mathrm{n}}q\in H^{-1/2}(\R^{d-1})$ is representable as an $\L(\R^{d-1})$ function. The next result
asserts that an evolutionary equation with a spatial operator of the
type $\begin{pmatrix}
0 & \dive\\
\grad & 0
\end{pmatrix}$ with the above Robin boundary condition fits into the setting rendered
by \prettyref{thm:Solution_theory_EE}. In other words:
\begin{thm}
\label{thm:Robinssa} Let $\Omega=\R^{d-1}\times\Rg{0}$. Then the
operator $A\colon\dom(A)\subseteq\L(\Omega)^{d+1}\to\L(\Omega)^{d+1}$
with $A\subseteq\begin{pmatrix}
0 & \dive\\
\grad & 0
\end{pmatrix}$ with domain
\[
\dom(A)=\set{(u,q)\in H^{1}(\Omega)\times H(\dive,\Omega)}{\gamma_{\mathrm{n}}q+\i\gamma u=0}
\]
is skew-selfadjoint.
\end{thm}

\begin{proof}
Let $(u,q),(v,r)\in H^{1}(\Omega)\times H(\dive,\Omega)$. Then, by
\prettyref{eq:ibptrace} we obtain
\begin{align*}
 & \scp{\begin{pmatrix}
0 & \dive\\
\grad & 0
\end{pmatrix}\begin{pmatrix}
 u\\
 q 
\end{pmatrix}}{\begin{pmatrix}
v\\
r
\end{pmatrix}}+\scp{\begin{pmatrix}
 u\\
 q 
\end{pmatrix}}{\begin{pmatrix}
0 & \dive\\
\grad & 0
\end{pmatrix}\begin{pmatrix}
v\\
r
\end{pmatrix}}\\
 & =\scp{\dive q}{v}+\scp{\grad u}{r}+\scp{u}{\dive r}+\scp{q}{\grad v}
 = (\gamma_{\mathrm{n}}q)(\gamma v)+ ((\gamma_{\mathrm{n}}r)(\gamma u))^\ast
\end{align*}
If, in addition, $(u,q)\in\dom(A)$, we obtain
\begin{align*}
 & \scp{A\begin{pmatrix}
 u\\
 q 
\end{pmatrix}}{\begin{pmatrix}
v\\
r
\end{pmatrix}}+\scp{\begin{pmatrix}
 u\\
 q 
\end{pmatrix}}{\begin{pmatrix}
0 & \dive\\
\grad & 0
\end{pmatrix}\begin{pmatrix}
v\\
r
\end{pmatrix}}\\
 & = (\gamma_{\mathrm{n}}q)(\gamma v)+ ((\gamma_{\mathrm{n}}r)(\gamma u))^\ast
 =\scp{-\i\gamma u}{\gamma v}_{\L(\R^{d-1})}+((\gamma_{\mathrm{n}}r)(\gamma u))^\ast\\
 &= \scp{\gamma u}{\i\gamma v}_{\L(\R^{d-1})}+((\gamma_{\mathrm{n}}r)(\gamma u))^\ast
=((\i\gamma v+\gamma_{\mathrm{n}}r)(\gamma u))^\ast.
\end{align*}
Since for every $u\in\mathcal{D}$, we find $q\in\mathcal{D}^{d}$
such that $(u,q)\in\dom(A)$, 
\[
\gamma[\mathcal{D}]\subseteq\set{\gamma u}{\exists q\in H(\dive,\Omega)\colon(u,q)\in\dom(A)}.
\]
Thus, the set on the right-hand side is dense in $H^{1/2}(\R^{d-1})$. This in
turn implies that $(v,r)\in\dom(A^{*})$ if and only if $\i\gamma v+\gamma_{\mathrm{n}}r=0$,
and in this case we have $A^{*}(v,r)=-A(v,r)$. This implies that
$A$ is skew-selfadjoint.
\end{proof}

\begin{rem}
 The factor $\i$ in front of $\gamma u$ is chosen as a mere convenience in order to render the corresponding operator $A$ in \prettyref{thm:Robinssa} skew-selfadjoint. It is also possible to choose $\beta \in \bo(H^{1/2}(\partial\Omega))$ with $-\Re\beta\geq 0$ instead of $\i$. Then one obtains for all $U\in \dom(A)$ and $V\in \dom(A^*)$ the estimates $\Re\scp{U}{AU}\geq 0$ and $\Re\scp{V}{A^*V}\geq 0$. Appealing to \prettyref{rem:Aaccretive}, it can be shown that the corresponding evolutionary equation
 \[
     (\td{\nu}M(\td{\nu})+A)U=F
 \]
 for a suitable material law $M$ as in \prettyref{thm:Solution_theory_EE} is well-posed.
\end{rem}

Next, one could argue that in the case for arbitrary $\Omega$, the
condition 
\begin{equation}
\i\pi_{\BD(\grad)}u+\dive_{\BD}\pi_{\BD(\dive)}q=0\label{eq:abstractRobin1}
\end{equation}
amounts to a generalisation of the Robin boundary condition just
considered. However, this is not true as the following proposition shows.
\begin{prop}
 Let $u\in H^1(\Omega),q\in H(\dive,\Omega)$. Moreover, we set $\kappa\from\BD(\grad)\to \L(\R^{d-1})$ with $\kappa v=\gamma v$ for $v\in \BD(\grad)$. Then $\gamma_{\mathrm{n}}q + \i \gamma u=0$ if and only if
 \[
  \dive_\BD \pi_{\BD(\dive)} q + \i \kappa^\ast \kappa \pi_{\BD(\grad)} u=0.
 \]
\end{prop}

\begin{proof}
 We first observe that $\kappa \pi_{\BD(\grad)} w=\gamma w$ for each $w\in H^1(\Omega)$.\\
 Assume now that $\gamma_{\mathrm{n}}q + \i \gamma u=0$ and let $v\in \BD(\grad)$. Then we compute, using \prettyref{prop:ibpabstract} and \prettyref{eq:ibptrace}
 \begin{align*}
  \scp{\i \kappa^\ast \kappa \pi_{\BD(\grad)} u}{v}_{\BD(\grad)} &= \scp{\i \kappa \pi_{\BD(\grad)}u}{\kappa v}_{\L(\R^{d-1})}\\
  &= \scp{\i \gamma u}{\gamma v}_{\L(\R^{d-1})} \\
   &= -(\gamma_{\mathrm{n}}q)(\gamma v)\\
  &= \scp{-\dive q}{v}_{\L(\Omega)}+\scp{-q}{\grad v}_{\L(\Omega)^d} \\
  &= \scp{-\dive_\BD \pi_{\BD(\dive)}q}{v}_{\BD(\grad)},
 \end{align*}
which proves one of the asserted implications.

Assume now that $\dive_\BD \pi_{\BD(\dive)} q + \i \kappa^\ast \kappa \pi_{\BD(\grad)} u=0$ and let $v\in H^{1/2}(\R^{d-1})$. We take $w\in H^1(\Omega)$ with $\gamma w=v$ and compute
\begin{align*}
 (\gamma_{\mathrm{n}}q)(v) &= \scp{\dive q}{w}_{\L(\Omega)}+\scp{q}{\grad w}_{\L(\Omega)^d}\\
 &= \scp{\dive_{\BD} \pi_{\BD(\dive)} q}{\pi_{\BD(\grad)}w}_{\BD(\grad)}\\
 &= \scp{-\i \kappa^\ast \kappa \pi_{\BD(\grad)} u}{\pi_{\BD(\grad)}w}_{\BD(\grad)} \\
 &= \scp{-\i \kappa \pi_{\BD(\grad)} u}{\kappa \pi_{\BD(\grad) }w}_{\L(\R^{d-1})}\\
 &= \scp{-\i \gamma u}{v}_{\L(\R^{d-1})},
\end{align*}
which shows the remaining implication.
\end{proof}

\section{Comments}

The concept of abstract trace spaces has been introduced in \cite{PTW16_IM} in order to study a multi-dimensional analogue for port-Hamiltonian systems. Also concerning differential equations at the boundary (so-called impedance type boundary conditions), the concept of abstract boundary value spaces has been employed, see \cite{PSTW16_GD}.

A comparison between abstract and classical trace spaces has been provided in  \cite{EGW17_D2N,Trostorff2014} particularly concerning $H^{-1/2}(\R^{d-1})$.  A good introduction for trace mappings for
more complicated geometries can be found e.g. in \cite{Arendt2015}.  The trace operator can also be suitably established for $H(\curl,\Omega)$-regular vector fields given that $\Omega$ is a so-called Lipschitz domain, see \cite{Buffa2002}.

\section*{Exercises}
\addcontentsline{toc}{section}{Exercises}

\begin{xca}
\label{exer:convSmooth}Let $\phi\in\cci(\R^{d})$, $f\in\L(\R^{d})$.
Show that 
\[
\phi\ast f\colon x\mapsto\int_{\R^{d}}\phi(x-y)f(y)\d y
\]
belongs to $H^{1}(\R^{d})$ and that $\grad\left(\phi\ast f\right)=\left(\grad\phi\right)\ast f$.
If, in addition, $f\in H^{1}(\R^{d})=\dom(\grad)$, then $\grad(\phi\ast f)=\phi\ast\grad f$,
where the convolution is always taken component wise.
\end{xca}

\begin{xca}
\label{exer:H1translate}Let $\Omega\subseteq\R^{d}$ be open. Let $f\in \L(\Omega)$ and denote by $\tilde{f}\in \L(\R^{d})$
the extension of $f$ by zero. Let $v\in\R^{d}$, $\tau>0$ and
define $f_{\tau}\coloneqq\tilde{f}(\cdot+\tau v)|_{\Omega}$.
\begin{enumerate}
\item\label{exer:H1translate:item:1}
Show that $f_{\tau}\to f$ in $\L(\Omega)$ as $\tau\to0$.
\item\label{exer:H1translate:item:2}
Let now $f\in H^{1}(\Omega)$ and  $\Omega+\tau v\subseteq\Omega$ for all $\tau>0$. Show that
$f_{\tau}\to f$ in $H^{1}(\Omega)$ as $\tau\to0$. 
\end{enumerate}
\end{xca}

\begin{xca}
\label{exer:divdesnity} Prove \prettyref{thm:divdensity}.
\end{xca}

\begin{xca}
\label{exer:uniqunessBVP} Let $\Omega\subseteq\R^{d}$ be open, $M\from\dom(M)\subseteq \C\to \bo\bigl(\L(\Omega)\times \L(\Omega)^d\bigr)$ with $\sbb{M}<\nu_0$ for some $\nu_0\in\R$, $c>0$ such that for all $z\in \C_{\Re \geq \nu_0}$ we have $\Re zM(z) \geq c$, $\nu\geq \max\{\nu_0,0\}$ and $\nu\neq 0$.
Show that there exists a unique 
\[
\begin{pmatrix}
v\\
q
\end{pmatrix}\in H_{\nu}^{1}\Bigl(\R;\dom\Bigl(\begin{pmatrix}
0 & \dive\\
\grad & 0
\end{pmatrix}\Bigr)\Bigr)
\]
satisfying 
\[
\begin{cases}
\left(\td{\nu}M(\td{\nu})+\begin{pmatrix}
0 & \dive\\
\grad & 0
\end{pmatrix}\right)\begin{pmatrix}
v\\
q
\end{pmatrix}=\begin{pmatrix}
0\\
0
\end{pmatrix} & \text{ on }\Omega,\\
\pi_{\BD(\grad)}v(t)=\phi(t)f & \text{ for all }t\in\R,
\end{cases}
\]
for some bounded $\phi\in C^{\infty}(\R)$ with $\inf\spt\phi>-\infty$
and $f\in\BD(\grad)$.
\end{xca}

\begin{xca}
\label{exer:Extension} Let $\Omega=\R^{d-1}\times\Rg{0}$. Show that
there exists a continuous linear operator $E\colon H^{1}(\Omega)\to H^{1}(\R^{d})$
such that $E(\phi)|_{\Omega}=\phi$ for each $\phi \in H^1(\Omega)$. 
\end{xca}

\begin{xca}[Korn's second inequality]
\label{exer:Korn2} Let $\Omega=\R^{d-1}\times\Rg{0}$. Using \prettyref{exer:Extension}
show that there exists $c>0$ such that for all $\phi\in H^{1}(\Omega)^{d}$
we have
\[
\norm{\phi}_{H^{1}(\Omega)^{d}}\leq c\left(\norm{\phi}_{\L(\Omega)^{d}}+\norm{\Grad\phi}_{\L(\Omega)^{d\times d}}\right).
\]
Thus, describe the space of boundary values of $\dom(\Grad)$. 

\emph{Hint:} Prove a corresponding result for $\Omega=\R^{d}$ first
after having shown that $\cci(\R^{d})^{d}$ forms a dense subset of
both $H^{1}(\Omega)^{d}$ and $\dom(\Grad)$.
\end{xca}

\begin{xca}
\label{exer:MaxwellBdy} Let $\Omega\subseteq\R^{3}$ be open. Compute
$\BD(\curl)\coloneqq H_{0}(\curl,\Omega)^{\bot_{H(\curl,\Omega)}}$
and show that $\curl\colon\BD(\curl)\to\BD(\curl)$
is well-defined, unitary and skew-selfadjoint.
\end{xca}


\printbibliography[heading=subbibliography]
\chapter{Continuous Dependence on the Coefficients I}

The power of the functional analytic framework for evolutionary equations
lies in its variety. In fact, as we have outlined in earlier lectures,
it is possible to formulate many differential equations in the form
\[
\left(\partial_{t}M(\partial_{t})+A\right)U=F.
\]
In this chapter we want to use this versatility and address continuity
of the above expression (or more precisely of the solution operator)
in $M(\partial_{t})$. To see this more clearly, fix $F$ and take a sequence
of material laws $\left(M_{n}\right)_{n}$. We will address the following
question: what are the conditions or notions of convergence of $\left(M_{n}\right)_{n}$
to some $M$ in order that $(U_{n})_{n}$ with $U_n$ given as the solution of
\[
\left(\partial_{t}M_{n}(\partial_{t})+A\right)U_{n}=F
\]
converges to $U$, which satisfies 
\[
\left(\partial_{t}M(\partial_{t})+A\right)U=F?
\]

In the first of two chapters on this subject, we shall specialise to $A=0$;
that is, we will discuss ordinary differential equations with infinite-dimensional
state space. To begin with, we address the convergence of material laws
pointwise in the Fourier--Laplace transformed domain and its relation
to the convergence of material laws evaluated at the time derivative.

\section{Convergence of Material Laws}

Throughout, let $H$ be a Hilbert space. We briefly
recall that a sequence $(T_{n})_n$ in $\bo(H)$ converges in the
\emph{strong operator topology} to some $T\in\bo(H)$ if for all
$x\in H$ we have
\[
T_{n}x\to Tx\quad(n\to\infty).
\]
$(T_{n})_n$ is said to converge in the \emph{weak operator topology} to $T\in\bo(H)$ if for all $x,y\in H$ we have
\[
\scp{y}{T_{n}x}\to\scp{y}{Tx}\quad(n\to\infty).
\]
We denote the set of material laws on $H$ with abscissa of boundedness less than or equal to
$\nu_{0}\in\R$ by
\[
\mathcal{M}(H,\nu_{0})\coloneqq\set{M\colon\dom(M)\to\bo(H)}{M\text{ material law},\sbb{M}\leq\nu_{0}}.
\]

\begin{rem}
 Let $\nu_0\in\R$, $\nu>\nu_0$. Then $\mathcal{M}(H,\nu_0)$ is an algebra and $\mathcal{M}(H,\nu_0)\ni M\mapsto M(\td{\nu})\in \bo\bigl(\Lnu(\R;H)\bigr)$ is an algebra homomorphism which is one-to-one by \prettyref{thm:representaion}. 
\end{rem}

\begin{defn*}
Let $\nu_{0}\in\R$.
A sequence $(M_{n})_{n\in\N}$ in $\mathcal{M}(H,\nu_{0})$ is called \emph{bounded}
if 
\[
\sup_{n\in\N}\norm{M_{n}}_{\infty,\C_{\Re>\nu_{0}}}<\infty.
\]
\end{defn*}

\begin{thm}
\label{thm:converge_wotsot}Let $\nu_{0}\in\R$, $(M_{n})_n$ in $\mathcal{M}(H,\nu_{0})$ be
bounded. Assume that for all $z\in\C_{\Re>{\nu_0}}$ the sequence $(M_{n}(z))_n$
converges in the weak operator topology of $\bo(H)$ with limit $M(z)$ and let $\nu>\nu_0$. Then $M\in\mathcal{M}(H,\nu_0)$ and $M_{n}(\td{\nu})\to M(\td{\nu})$ as
$n\to\infty$ in the weak operator topology of $\bo\bigl(\Lnu(\R,H)\bigr)$.

If, in addition, $(M_{n}(z))_n$ converges in the strong
operator topology of $\bo(H)$ for all $z\in\C_{\Re>{\nu_0}}$, then, as $n\to\infty$, $M_{n}(\td{\nu})\to M(\td{\nu})$ in the strong operator topology of $\bo\bigl(\Lnu(\R,H)\bigr)$.
\end{thm}

\begin{proof}
Let $z_0\in \C_{\Re>{\nu_0}}$, $r\in \oi{0}{\Re z_0-\nu_0}$. For $x,y\in H$,
by Cauchy's integral formula, we deduce
\[
\scp{y}{M_{n}(z_{0})x}=\frac{1}{2\pi \i} \int_{\partial B(z_0,r)}\frac{\scp{y}{M_{n}(z)x}_H}{z-z_{0}}\d z \quad(n\in\N).
\]
Using boundedness of $(M_{n})_n$, Lebesgue's dominated
convergence theorem yields
\[
\scp{y}{M(z_{0})x}=\frac{1}{2\pi \i}\int_{\partial B(z_0,r)}\frac{\scp{y}{M(z)x}_H}{z-z_{0}}\d z.
\]
Since 
\begin{equation}
\norm{\scp{y}{M(z)x}}_H\leq\norm{x}_H\norm{y}_H\sup_{n\in\N}\norm{M_{n}}_{\infty,\C_{\Re>\nu_0}}\quad(z\in\C_{\Re>{\nu_0}}),\label{eq:elem_est_M}
\end{equation}
$\scp{y}{M(\cdot)x}_H$ is holomorphic in a neighbourhood of $z_{0}$.
By \prettyref{exer:holomorphic} we obtain that $M\colon\C_{\Re>{\nu_0}}\to\bo(H)$
is holomorphic. In fact, the estimate \prettyref{eq:elem_est_M}
even implies that $M\in\mathcal{M}(H,\nu_0)$. 

If $z\in\C_{\Re>{\nu_0}}$ and $(M_n(z))_n$ even converges in the strong operator topology, then the limit is clearly $M(z)$.

The convergence statements for $(M_{n}(\td{\nu}))_n$ (in the weak and strong operator topology)
are then implied by Fourier--Laplace transformation. 
\end{proof}

\begin{rem}
In \prettyref{thm:converge_wotsot}, it suffices
to assume that $(M_{n}(z))_n$ converges only for $z$ belonging
to a countable subset of $\C_{\Re>\nu_0}$ with an accumulation point in
$\C_{\Re>\nu_0}$.
\end{rem}

The next statement is essential for the convergence statement for
``ordinary'' evolutionary equations.

\begin{prop}
\label{prop:sotInv}Let $(T_{n})_n$ be a sequence in $\bo(H)$
converging in the strong operator topology to some $T\in\bo(H)$ with
$0\in\bigcap_{n\in\N}\rho(T_{n})$, $\sup_{n\in\N}\norm{T_{n}^{-1}}<\infty$
and $\ran(T)\subseteq H$ dense. Then $T$ is continuously invertible
and $(T_{n}^{-1})_n$ converges to $T^{-1}$ in the strong
operator topology.
\end{prop}

\begin{proof}
We set $K\coloneqq\sup_{n\in\N}\norm{T_{n}^{-1}}$. We show that $T$ is
continuously invertible first. For this, let $x\in H$. Then 
\[
\norm{x}=\norm{T_{n}^{-1}T_{n}x}\leq K\norm{T_{n}x}\to K\norm{Tx}\quad(n\to\infty).
\]
Hence, $T$ is one-to-one and it follows that $\ran(T)\subseteq H$
is closed. Hence, $0\in\rho(T)$. For $x\in H$ we conclude
\begin{align*}
\norm{T_{n}^{-1}x-T^{-1}x} & =\norm{T_{n}^{-1}(T-T_{n})T^{-1}x}
 \leq K\norm{(T-T_{n})T^{-1}x}\to0\quad(n\to\infty).\tag*{{\qedhere}}
\end{align*}
\end{proof}

We are now in the position to obtain the first result on continuous dependence.

\begin{thm}
\label{thm:ODEcondep}
Let $\nu_{0}\in\R$, $(M_{n})_n$ a bounded sequence
in $\mathcal{M}(H,\nu_{0})$, $c>0$ such that
for all $n\in\N$ and $z\in\C_{\Re>\nu_{0}}$ we have
\[
\Re zM_{n}(z)\geq c.
\]
If $(M_{n}(z))_n$ converges in the strong operator
topology for all $z\in\C_{\Re>\nu_{0}}$ then for the limit $M(z)$ we have $M\in\mathcal{M}(H,\nu_{0})$
with $\Re zM(z)\geq c$ for all $z\in\C_{\Re>\nu_{0}}$ and for $\nu>\nu_0$ we have
\[
\bigl(\td{\nu}M_{n}(\td{\nu})\bigr)^{-1}\to\bigl(\td{\nu}M(\td{\nu})\bigr)^{-1}
\]
in the strong operator topology.
\end{thm}

\begin{proof}
By \prettyref{thm:converge_wotsot}, we observe $M\in\mathcal{M}(H,\nu_{0})$. Let $z\in\C_{\Re>\nu_{0}}$.
Then we have $\Re zM(z) = \lim_{n\to\infty} \Re zM_n(z)\geq c$ and hence
$zM(z)$ is continuously invertible. Since $0\in\bigcap_{n\in\N}\rho(zM_{n}(z))$
and $\norm{(zM_{n}(z))^{-1}}\leq1/c$ by \prettyref{prop:block_op_realinv}\ref{prop:block_op_realinv:item:2}, we deduce by \prettyref{prop:sotInv}
applied to $T_{n}=zM_{n}(z)$ that $(zM_{n}(z))^{-1}\to(zM(z))^{-1}$
in the strong operator topology. By \prettyref{thm:converge_wotsot}, for $\nu>\nu_0$
we infer $\bigl(\td{\nu}M_{n}(\td{\nu})\bigr)^{-1}\to\bigl(\td{\nu}M(\td{\nu})\bigr)^{-1}$
in the strong operator topology.
\end{proof}

\section{A Leading Example}

We want to illustrate the findings of the previous section with the
help of an ordinary differential equation. Also, we shall provide an argument on the limitations
of the theory presented above. Let $(\Omega,\Sigma,\mu)$
be a finite measure space. 

Note that for $V\in L_\infty(\mu)$ with associated multiplication operator $V(\m)$ as in \prettyref{thm:mult1} we have that
\[
M\colon z\mapsto1+z^{-1}V(\m)\in\bo(\L(\mu))
\]
is a material law with $\sbb{M}=0$ unless $V=0$ (in case $V=0$ we have $\sbb{M} = -\infty$).
The corresponding evolutionary equation is given by
\[
\td{\nu}u+V(\m)u=f.
\]
We want to study sequences of material laws of this form; that is, material laws induced by sequences $(V_n)_n$ in $L_\infty(\mu)$.
First, we provide the following characterisation of the convergence
of multiplication operators. 
We recall that for a Banach space $X$ the weak$^*$ topology $\sigma(X',X)$ on $X'$ is the coarsest topology such that all the mappings $X'\ni x'\mapsto x'(x)$ ($x\in X$) are continuous.

\begin{prop}
\label{prop:conve_mult} Let $(V_{n})_n$ in $L_{\infty}(\mu)$ and
$V\in L_{\infty}(\mu)$. Then the following statements hold.
\begin{enumerate}
\item\label{prop:conve_mult:item:1} $V_{n}(\m)\to V(\m)$
in $\bo(\L(\mu))$ if and only if $V_{n}\to V$ in $L_{\infty}(\mu)$.

\item\label{prop:conve_mult:item:2} $V_{n}(\m)\to V(\m)$ in the strong
operator topology of $\bo(\L(\mu))$ if and only if $(V_{n})$ is bounded in $L_{\infty}(\mu)$ and $V_{n}\to V$
in $L_{1}(\mu)$.

\item\label{prop:conve_mult:item:3} $V_{n}(\m)\to V(\m)$ in the weak operator topology of $\bo(\L(\mu))$ if and
only if $V_{n}\to V$ in the weak$^*$ topology $\sigma\bigl(L_{\infty}(\mu),L_{1}(\mu)\bigr)$.
\end{enumerate}
\end{prop}

\begin{proof}
\ref{prop:conve_mult:item:1} This is a direct consequence of \prettyref{prop:mult1.75}.

\ref{prop:conve_mult:item:2} Assume $V_{n}\to V$ in $L_{1}(\mu)$ and that $(V_{n})_n$
is bounded in $L_{\infty}(\mu)$. Then $(V_n-V)_n$ is also bounded in $L_\infty(\mu)$.
For $f\in L_\infty(\mu)\subseteq \L(\mu)$ we obtain
\begin{align*}
  \norm{V_n(\m)f - V(\m)f}_{\L(\mu)}^2 & = \int_\Omega \abs{V_n - V}^2 \abs{f}^2\d \mu \\
  & \leq \sup_{n\in\N} \norm{V_n-V}_{L_\infty(\mu)} \norm{f}_{L_\infty(\mu)}^2 \int_\Omega \abs{V_n-V} \,d\mu \to 0.
\end{align*}
Since $L_\infty(\mu)$ is dense in $\L(\mu)$ and $(V_n(\m)-V(\m))_n$ is bounded by \prettyref{prop:mult1.75}, we obtain $V_n(\m)\to V(\m)$ in the strong operator topology of $\bo(\L(\mu))$.

Now, let $V_{n}(\m)\to V(\m)$ in the strong operator
topology of $\bo(\L(\mu))$.
Then $(V_{n}(\m))_n$ is bounded
in $\bo(\L(\mu))$ by the uniform boundedness
principle. Now \prettyref{prop:mult1.75} yields boundedness of $(V_n)_n$ in $L_\infty(\mu)$.
Moreover, since $\1_{\Omega}\in\L(\mu)$, we deduce $V_{n} = V_n(\m)\1_\Omega \to V(\m)\1_\Omega = V$
in $\L(\mu)$. Since $\L(\mu)$ embeds continuously into $L_{1}(\mu)$ we obtain $V_{n}\to V$
in $L_{1}(\mu)$.

\ref{prop:conve_mult:item:3} The assertion follows easily upon realising that $\phi\in L_{1}(\mu)$
if and only if there exists $\psi_{1},\psi_{2}\in\L(\mu)$ such that
$\phi=\psi_{1}\psi_{2}$.
\end{proof}

With the latter result at hand together with the results in the previous
section, we easily deduce the next theorem on continuous dependence
on the coefficients.

\begin{thm}
\label{thm:cd_sot_ex}Let $(V_{n})_n$ in $L_{\infty}(\mu)$ be bounded, $V\in L_\infty(\mu)$, and $V_{n}\to V$ in $L_{1}(\mu)$. Then there
exists $\nu>0$ such that 
\[
\bigl(\td{\nu}+V_{n}(\m)\bigr)^{-1}\to\bigl(\td{\nu}+V(\m)\bigr)^{-1}
\]
in the strong operator topology of $\bo\bigl(\Lnu(\R;\L(\mu))\bigr)$.
\end{thm}

Note that the convergence statement
can be improved, see \prettyref{exer:improved_cd_sot_ex}.

\begin{proof}
  By \prettyref{prop:conve_mult}\ref{prop:conve_mult:item:2} we obtain $V_n(\m)\to V(\m)$ in the strong
operator topology of $\bo(\L(\mu))$. Note that for $\nu\geq 1+\sup_{n\in\N} \norm{V_n}_{L_\infty(\mu)}$ we have
\[\Re (z+V_n(\m)) \geq 1 \quad (z\in\C_{\Re>\nu}, n\in\N).\]
Now \prettyref{thm:ODEcondep} applied to $M_n(z) = 1+z^{-1}V_n(\m)$ yields the assertion.
\end{proof}

\begin{rem}
  \prettyref{thm:cd_sot_ex} can be generalized in the following way. Let $(B_n)_n$ in $\bo(H)$, $B\in \bo(H)$, $B_n\to B$ in the strong operator topology.
  Then there txists $\nu>0$ such that 
\[
\bigl(\td{\nu}+B_n\bigr)^{-1}\to\bigl(\td{\nu}+B\bigr)^{-1}
\]
in the strong operator topology of $\bo\bigl(\Lnu(\R;\L(\mu))\bigr)$.
\end{rem}

In \prettyref{thm:cd_sot_ex} we did assume strong convergence of the sequence of multiplication operators $(V_n(\m))_n$.
A natural question to ask is whether the stated result can be improved to $(V_{n})_n$
converging in the weak$^*$ topology $\sigma\bigl(L_{\infty}(\mu),L_{1}(\mu)\bigr)$ only. The answer is neither
`yes' nor `no', but rather `not quite', as we will show in the following.
We start with a result on weak$^*$ limits of scaled periodic functions, which will serve as the prototypical example for a sequence converging in the weak$^*$ topology of $L_\infty$.

\begin{thm}
\label{thm:periodic_homo} Let $f\in L_{\infty}(\R^{d})$ be \emph{$\roi{0}{1}^d$-periodic}; that is,
\[
f(\cdot+k)=f\quad(k\in\Z^{d}).
\]
Then 
\[
f(n\cdot)\to\int_{\roi{0}{1}^{d}}f(x)\d x \1_{\R^d}
\]
in the weak$^*$ topology $\sigma\bigl(L_{\infty}(\R^{d}),L_{1}(\R^{d})\bigr)$ as $n\to\infty$.
\end{thm}

\begin{proof}
Without loss of generality, we may assume $\int_{\roi{0}{1}^{d}}f(x)\d x=0$.
By the density of simple functions in $L_{1}(\R^{d})$ and the boundedness
of $(f(n\cdot))_n$ in $L_{\infty}(\R^{d})$, it suffices to show
\[
\int_{Q}f(nx)\d x\to0\quad(n\to\infty)
\]
for $Q=\ci{a}{b}\coloneqq \ci{a_{1}}{b_{1}}\times\ldots\times\ci{a_{d}}{b_{d}}$ where
$a=(a_{1},\ldots,a_{d}),b=(b_{1},\ldots,b_{d})\in\R^d$. 
By translation and the periodicity of $f$ we may assume $a=0$. Thus, it suffices to show
\[
\int_{\ci{0}{b}}f(nx)\d x\to0\quad(n\to\infty)
\]
for all $b\in \oi{0}{\infty}^d$. So, let $b=(b_{1},\ldots,b_{d})\in\oi{0}{\infty}^d$.
Let $n\in\N$. Then we find $z\in\N^d$ and $\zeta\in\roi{0}{1}^d$
such that $nb=z+\zeta$.
We compute
\begin{align*}
 & \int_{\ci{0}{b}}f(nx)\d x\\
 & =\frac{1}{n^{d}}\int_{\ci{0}{nb}} f(x)\d x\\
 & =\frac{1}{n^{d}}\int_{\ci{0}{z_{1}}\times\ci{0}{nb_2}\times\ldots\times\ci{0}{nb_d}}f(x)\d x
 +\frac{1}{n^{d}}\int_{\loi{z_{1}}{z_{1}+\zeta_{1}}\times\ci{0}{nb_2}\times\ldots\times\ci{0}{nb_d}}f(x)\d x.
\end{align*}
We now estimate
\begin{align*}
\abs{\frac{1}{n^{d}}\int_{\loi{z_{1}}{z_{1}+\zeta_{1}}\times\ci{0}{nb_2}\times\ldots\times\ci{0}{nb_d}}f(x)\d x} 
 & \leq\frac{1}{n^{d}}\int_{\loi{z_{1}}{z_{1}+\zeta_{1}}\times\ci{0}{nb_2}\times\ldots\times\ci{0}{nb_d}}\abs{f(x)}\d x\\
 & \leq\frac{1}{n^{d}}\int_{\loi{0}{1}\times\ci{0}{nb_2}\times\ldots\times\ci{0}{nb_d}}\d x\norm{f}_{L_{\infty}(\mu)}\\
 & =\frac{1}{n}b_{2}\cdot\ldots\cdot b_{d}\norm{f}_{L_{\infty}(\mu)}.
\end{align*}
Continuing in this manner and using $z_{j}\leq nb_{j}$ for all $j\in\{1,\ldots,d\}$, we obtain
\begin{align*}
\abs{\int_{\ci{0}{b}}f(nx)\d x} & \leq\frac{1}{n^{d}}\abs{\int_{\ci{0}{z}}f(x)\d x}+\frac{1}{n}\sum_{j=1}^{d}\frac{b_{1}\cdot \ldots \cdot b_{d}}{b_{j}}\norm{f}_{L_{\infty}(\mu)}.
\end{align*}
Since $f$ is $\roi{0}{1}^{d}$-periodic and $z\in\N^d$ we observe
\[\int_{\ci{0}{z}}f(x)\d x = \prod_{j=1}^d z_j \int_{\roi{0}{1}^d} f(x)\d x = 0.\]
Thus,
\[
\abs{\int_{\ci{0}{b}}f(nx)\d x} \leq \frac{1}{n}\sum_{j=1}^{d}\frac{b_{1}\cdot\ldots\cdot b_{d}}{b_{j}}\norm{f}_{L_{\infty}(\mu)},
\]
which tends to $0$ as $n\to\infty$.
\end{proof}

\begin{rem}
Note that \prettyref{thm:periodic_homo} also yields 
\[
f(n\cdot)\to\int_{[0,1)^{d}}f(x)\d x \1_{\Omega}
\]
in the weak$^*$ topology $\sigma(L_{\infty}(\Omega),L_{1}(\Omega))$ for all
measurable subsets $\Omega\subseteq\R^{d}$ with non-zero Lebesgue measure.
\end{rem}

We now present an example which shows that weak$^*$ convergence of $(V_n)_n$ does not yield the result of \prettyref{thm:cd_sot_ex}.

\begin{example}
\label{exa:sin_Memory}
Let $(\Omega,\Sigma,\mu)=(\oi{0}{1},\mathcal{B}(\oi{0}{1}),\lambda|_{\oi{0}{1}})$.
For $n\in\N$ let $V_{n}$ be given by $V_{n}(x)\coloneqq\sin(2\pi nx)$
for $x\in\oi{0}{1}$. Then, by \prettyref{thm:periodic_homo}, we obtain $V_{n}\to0$ in $\sigma\bigl(L_{\infty}(\oi{0}{1}),L_{1}(\oi{0}{1})\bigr)$ as $n\to\infty$.
Let $\nu>1$. Then $\bigl(\td{\nu}+V_{n}(\m)\bigr)$ is continuously
invertible as an operator in $\Lnu\bigl(\R;\L(\oi{0}{1})\bigr)$. Let $\tilde{f}\in C(\ci{0}{1})$
and denote $f\colon(t,x)\mapsto\1_{\Rge{0}}(t)\tilde{f}(x)$. Then $f\in\Lnu\bigl(\R;\L(\oi{0}{1})\bigr)$.
The solution $u_{n}\in\Lnu\bigl(\R;\L(\oi{0}{1})\bigr)$ of 
\[
\bigl(\td{\nu}+V_{n}(\m)\bigr)u_{n}=f
\]
is given by the variations of constants formula; that is,
\[
u_{n}(t,x)=\1_{\roi{0}{\infty}}(t)\int_{0}^{t}\exp\bigl(-(t-s)\sin(2\pi nx)\bigr)\d s\tilde{f}(x)\quad(t\in\R,x\in\oi{0}{1}).
\]
Thus, if a variant of \prettyref{thm:cd_sot_ex} were true also in
this case, $(u_{n})_n$ needs to converge (in some sense) to the solution $u$
of
\[
\td{\nu}u=f,
\]
which is given by
\[
u(t,x)=\1_{\roi{0}{\infty}}(t)t\tilde{f}(x)\quad(t\in\R,x\in\oi{0}{1}).
\]
However, by \prettyref{thm:periodic_homo}, for $x\in\oi{0}{1}$ we deduce 
\[\int_{0}^{t}\exp\bigl(-(t-s)\sin(2\pi nx)\bigr)\d s \to \int_{0}^{t}J(-(t-s))\d s\quad(n\to\infty)
\]
in $\sigma\bigl(L_{\infty}(\oi{0}{1}),L_{1}(\oi{0}{1})\bigr)$ for each $t\geq0$, where
\[
J(s)\coloneqq\int_{0}^{1}\exp\bigl(s\sin(2\pi x)\bigr)\d x\quad(s\in\R)
\]
denotes the $0$-th order modified Bessel function of the first kind,
cf.~\cite[9.6.19]{Abramowitz1964}. Moreover, for $\varphi\in \cci(\R)$, $A\in\mathcal{B}(\oi{0}{1})$ and using domianted convergence we obtain
\begin{align*}
& \scp{u_{n}}{\varphi\1_{A}}_{\Lnu(\R;\L(\oi{0}{1}))} \\
& =\int_{0}^{\infty}\int_{0}^{1}\int_{0}^{t}\exp\bigl(-(t-s)\sin(2\pi nx)\bigr)\d s\tilde{f}(x)\1_{A}(x)\d x\varphi(t)\e^{-2\nu t}\d t\\
 & \to\int_{0}^{\infty}\int_{0}^{1}\int_{0}^{t}J(-(t-s))\d s\tilde{f}(x)\1_{A}(x)\d x\varphi(t)\e^{-2\nu t}\d t\\
 & =\scp{\tilde{u}}{\varphi\1_{A}}_{\Lnu(\R;\L(\oi{0}{1}))}
\end{align*}
with 
\[
\tilde{u}(t,x)\coloneqq\1_{\roi{0}{\infty}}(t)\int_{0}^{t}J(-(t-s))\d s\tilde{f}(x)\quad(t\in\R,x\in\oi{0}{1}).
\]
Since $(u_{n})_n$ is bounded in $\Lnu\bigl(\R;\L(\oi{0}{1})\bigr)$ and $\set{\varphi\1_{A}}{A\in\mathcal{B}(\oi{0}{1}),\,\varphi\in\cci(\R)}$
is total in $\Lnu\bigl(\R;\L(\oi{0}{1})\bigr)$ by \prettyref{lem:dense sets 2},
we infer $u_{n}\to\tilde{u}$
weakly in $\Lnu\bigl(\R;\L(\oi{0}{1})\bigr)$ as $n\to\infty$. In particular,
$\tilde{u}\ne u$.
Furthermore, $\tilde{u}$ is \emph{not} of the form
\[
\int_{0}^{t}\exp\bigl(-(t-s)\tilde{V}(x)\bigr)\d s\tilde{f}(x)
\]
for some $\tilde{V}\in L_{\infty}(\oi{0}{1})$ and hence, we \emph{cannot}
hope for $\tilde{u}$ to satisfy an equation of the type
\[
\bigl(\td{\nu}+\tilde{V}(\m)\bigr)\tilde{u}=f.
\]
\end{example}

As we shall see next, in the framework of evolutionary equations it is possible to
derive an equation involving suitable limits of $(V_{n})_n$
and $f$ as a right-hand side.

\section{Convergence in the Weak Operator Topology}

In this section, we consider a particular class of material laws and characterise convercenge of the solution operators of the corresponding evolutionary equations in the weak operator topology. The main theorem that will
serve to compute the limit equation satisfied by $\tilde{u}$ in \prettyref{exa:sin_Memory}
reads as follows.
\begin{thm}
\label{thm:char_thm_wot}Let $H$ be a Hilbert space, $(B_{n})_n$ a bounded sequence in $\bo(H)$ and $\nu>\sup_{n\in\N}\norm{B_{n}}$.
Then $\bigl((\td{\nu}+B_{n})^{-1}\bigr)_n$ converges
in the weak operator topology of $\bo(\Lnu(\R;H))$ if and only if
for all $k\in\N$ the sequence $(B_{n}^{k})_n$ converges
in the weak operator topology of $\bo(H)$. In either case, we have 
\[
(\td{\nu}+B_{n})^{-1}\to\sum_{k=0}^{\infty}\bigl(-\td{\nu}^{-1}\bigr)^{k}C_{k}\td{\nu}^{-1}
\]
in the weak operator topology of $\bo(\Lnu(\R;H))$, where $C_{k}\in\bo(H)$ denotes the weak limit of $(B_n^k)_n$ for $k\in\N$ and $C_0\coloneqq 1_H$.
\end{thm}

\begin{rem}
\label{rem:char_thm_wot}
In the situation of \prettyref{thm:char_thm_wot}, let $B_{n}^{k}\to C_k$ in the weak operator topology for all $k\in\N$.
Let $L\coloneqq\sup_{n\in\N}\norm{B_{n}}$, $\nu>2L$, and $f\in \Lnu(\R;H)$.
By \prettyref{thm:char_thm_wot}, if $(\td{\nu}+B_{n})u_{n}=f$ for all $n\in\N$, then
$(u_{n})_n$ converges weakly in $\Lnu(\R;H)$ to some element $\tilde{u}\in\Lnu(\R;H)$.
In order to determine the differential equation satisfied by $\tilde{u}$,
we make the following observations: by weak convergence,
\[
\norm{C_{k}}\leq\liminf_{n\to\infty}\norm{B_{n}^{k}}\leq L^{k}.
\]
Hence, since $\norm{\td{\nu}^{-1}}_{\Lnu}\leq \frac{1}{\nu}$ (see \prettyref{sec:time_derivative}) we infer that
\[
\sum_{k=1}^{\infty}\bigl(-\td{\nu}^{-1}\bigr)^{k}C_{k}
\]
converges in $\bo(\Lnu(\R;H))$ and 
\[
\norm{\sum_{k=1}^{\infty}\bigl(-\td{\nu}^{-1}\bigr)^{k}C_{k}} \leq \sum_{k=1}^\infty \norm{\td{\nu}^{-1}}^k \norm{C_k} < \sum_{k=1}^{\infty} \frac{1}{2^k} = 1.
\]
Hence, since $C_{0}=\idop_{H}$ we deduce that $\sum_{k=0}^{\infty}\bigl(-\partial_{t,\nu}^{-1}\bigr)^{k}C_{k}$
is boundedly invertible by the Neumann series. Thus, we obtain
\begin{align*}
f & = \td{\nu}\left(\sum_{k=0}^{\infty}\bigl(-\td{\nu}^{-1}\bigr)^{k}C_{k}\right)^{-1}\tilde{u}
 =\td{\nu}\left(\idop_{H}+\sum_{k=1}^{\infty}\bigl(-\td{\nu}^{-1}\bigr)^{k}C_{k}\right)^{-1}\tilde{u}\\
 & =\td{\nu}\sum_{\ell=0}^{\infty}\left(-\sum_{k=1}^{\infty}\bigl(-\td{\nu}^{-1}\bigr)^{k}C_{k}\right)^{\ell}\tilde{u}
 =\td{\nu}\tilde{u}+\td{\nu}\sum_{\ell=1}^{\infty}\left(-\sum_{k=1}^{\infty}\bigl(-\td{\nu}^{-1}\bigr)^{k}C_{k}\right)^{\ell}\tilde{u}.
\end{align*}
\end{rem}

Before we prove \prettyref{thm:char_thm_wot} we revisit \prettyref{exa:sin_Memory}.

\begin{example}[\prettyref{exa:sin_Memory} continued]
By \prettyref{thm:char_thm_wot}, we need to compute the limit
of $(\sin^{k}(2\pi n\cdot))_n$ in the weak$^*$ topology
of $L_{\infty}(\oi{0}{1})$ for all $k\in\N$. By \prettyref{thm:periodic_homo}, we
obtain for all $k\in\N$  
\begin{align*}
\lim_{n\to\infty}\sin^{k}(2\pi n\cdot) & =\int_{0}^{1}\sin^{k}(2\pi\xi)\d\xi \1_{\oi{0}{1}}
 =\begin{cases}
\frac{\left(2m\right)!}{\left(m!2^{m}\right)^{2}}\1_{\oi{0}{1}}, & k=2m\text{ for some }m\in\N,\\
0, & k\text{ odd},
\end{cases}
\end{align*}
in $\sigma\bigl(L_{\infty}(\oi{0}{1}),L_{1}(\oi{0}{1})\bigr)$. Hence, $u_{n}\to\tilde{u}$ weakly,
where $\tilde{u}$ satisfies
\[
\td{\nu}\tilde{u}+\td{\nu}\sum_{\ell=1}^{\infty}\left(-\sum_{m=1}^{\infty}\td{\nu}^{-2m}\frac{\left(2m\right)!}{\left(m!2^{m}\right)^{2}}\right)^{\ell}\tilde{u}=f
\]
for $\nu>2$ by \prettyref{rem:char_thm_wot}.
\end{example}

\begin{proof}[Proof of \prettyref{thm:char_thm_wot}]
 Before we prove the equivalence, we make some observations. Since
$\nu>\sup_{n\in\N}\norm{B_{n}}\eqqcolon L$, by a Neumann series argument we deduce that
\[
 \bigl(\td{\nu}+B_{n}\bigr)^{-1}=\sum_{k=0}^{\infty}\bigl(-\td{\nu}^{-1}B_{n}\bigr)^{k}\td{\nu}^{-1} = \sum_{k=0}^{\infty}\bigl(-\td{\nu}^{-1}\bigr)^k B_{n}^k \td{\nu}^{-1}.
\]
The series $\sum_{k=0}^{\infty}\bigl(-\td{\nu}^{-1}\bigr)^kB_{n}^{k}\td{\nu}^{-1}$
is absolutely convergent in $\bo(\Lnu(\R;H))$. Also note that for $M_{n}\from\C_{\Re>L}\ni z\mapsto\sum_{k=0}^{\infty}(-\frac{1}{z})^k B_{n}^{k}\frac{1}{z}$ we have $M_n\in\mathcal{M}(H,\nu)$. 

Assume now that $(B_{n}^{k})_{n}$ converges in the weak operator
topology to some $C_{k}$ for all $k\in\N$. A little computation reveals that 
as $n\to\infty$, 
\[
M_{n}(z)\to\sum_{k=0}^{\infty}\left(-\frac{1}{z}\right)^{k}C_{k}\frac{1}{z}\eqqcolon M(z)\quad(z\in\C_{\Re>L})
\]
in the weak operator topology, where the series on the right-hand
side converges in $\bo(H)$ since 
\[
\norm{C_{k}}\leq\liminf_{n\to\infty}\norm{B_{n}^{k}}\leq L^{k}\quad(k\in\N).
\]
Moreover, since $\nu>L$, the sequence $(M_{n})_n$ is bounded in
$\mathcal{M}(H,\nu)$ and thus, $M\in\mathcal{M}(H,\nu)$ and 
\[
M_{n}(\td{\nu})\to M(\td{\nu})
\]
in the weak operator topology by \prettyref{thm:converge_wotsot}.

Now, we assume that $\bigl((\td{\nu}+B_{n})^{-1}\bigr)_n$ converges in the weak operator topology.
Then $(M_{n}(\td{\nu}))_n$ converges in the weak operator topology. 
Let $k\in\N$. We need to show that for all $\phi,\psi\in H$ the sequence
$(\scp{\phi}{B_{n}^{k}\psi}_{H})_n$ is convergent to
some number $c_{k,\phi,\psi}$ as $n\to\infty$.
The Riesz representation theorem then yields the existence of $C_{k}\in\bo(H)$
with $\scp{\phi}{C_{k}\psi}=c_{k,\phi,\psi}$. So, let $\phi,\psi\in H$.
Moreover, we consider the functions $m_{n}$ and $h_{n}$ given by
\[
m_{n}(z)\coloneqq\sum_{k=0}^{\infty}(-z)^{k}z\scp{\phi}{B_{n}^{k}\psi}_{H}\quad(z\in B(0,1/L), n\in\N)
\]
and 
\[
h_{n}(z)\coloneqq\scp{\phi}{M_{n}(z)\psi}_{H}=\sum_{k=0}^{\infty}\frac{1}{z}\left(-\frac{1}{z}\right)^{k}\scp{\phi}{B_{n}^{k}\psi}_{H}\quad(z\in\C_{\Re>L}, n\in\N).
\]
Clearly, $m_{n}$ and $h_{n}$ are holomorphic on their respective domains
for each $n\in\N$ and the sequences $(m_{n})_n$ and $(h_{n})_n$ are
uniformly bounded on compact subsets (in other words they form normal
families). Moreover, 
\[
m_{n}(z)=h_{n}\Bigl(\frac{1}{z}\Bigr)\quad\bigl(z\in B\bigl(1/(2L),1/(2L)\bigr),n\in\N\bigr).
\]
We aim to show that the coefficients of the power series of $m_{n}$
converge as $n$ tends to infinity. The proof will be done in two
steps. In step 1, we will prove that the sequence $(h_{n})_{n}$
converges to a holomorphic function $h\from\C_{\Re>L}\to\C$ uniformly on
compact sets. Then, in the second step, we will use this to deduce
that $(m_{n})_n$ also converges uniformly on compact sets and prove
the assertion with the help of Cauchy's integral formula.

\emph{Step 1:} By \prettyref{prop:mat_law_function_of_td}, $(M_{n}(\i\m+\nu))_n$
converges in the weak operator topology of $\bo(\L(\R;H))$. For $f,g\in\L(\R)$
we thus obtain that
\begin{align*}
\bigl(\scp{f}{h_{n}(\i\m+\nu)g}_{\L(\R)}\bigr)_n & =\bigl(\scp{f\phi}{M_{n}(\i\m+\nu)g\psi}_{\L(\R;H)}\bigr)_n
\end{align*}
is convergent. Thus, using $\L(\R)\cdot\L(\R)=L_{1}(\R)$, we obtain
that
\[
\Psi\from L_{1}(\R)\ni u\mapsto\lim_{n\to\infty}\left(\int_{\R}h_{n}(\i t+\nu)u(t)\d t\right) \in \C
\]
defines a linear functional, which is continuous, since 
\[
\sup_{n\in\N}\sup_{t\in\R}\norm{M_{n}(\i t+\nu)}_{\bo(H)}=\sup_{n\in\N}\norm{M_{n}(\i\m+\nu)}_{\bo(\L(\R;H))}<\infty
\]
by boundedness of $(B_n)_n$. Hence, since $L_{1}(\R)'=L_{\infty}(\R)$,
we find a unique $\tilde{h}\in L_{\infty}(\R)$ with
\begin{align*}
\lim_{n\to\infty}\int_{\R}h_{n}(\i t+\nu)u(t)\d t & =\int_{\R}\tilde{h}(t)u(t)\d t\quad(u\in L_{1}(\R)).
\end{align*}
We now show that every subsequence $(h_{n_k})_k$ of $(h_n)_n$ has a subsequence $(h_{n_{k_l}})_l$ which converges locally uniformly to a holomorphic function $h\from \C_{\Re>L}\to \C$ such that $h(\i\cdot+\nu) = \tilde{h}$ a.e., and that this implies that the limit $h$ does not depend on the subsequences.
Then we conclude that $(h_n)_n$ itself converges locally uniformly to $h$.

So, let $(h_{n_k})_k$ be a subsequence of $(h_n)$. 
By Montel's theorem (see \cite[Theorem 6.2.2]{Simon2015}), we find a subsequence
$(h_{n_{k_l}})_{l}$ of $(h_{n_k})_k$ such that $h_{n_{k_l}}\to h$
as $l\to\infty$ uniformly on compact subsets of $\C_{\Re>L}$
for some holomorphic function $h\from\C_{\Re>L}\to\C$. In particular, we
obtain
\[
\lim_{k\to\infty}\int_{\R}h_{n_{k_l}}(\i t+\nu)\varphi(t)\d t=\int_{\R}h(\i t+\nu)\varphi(t)\d t\quad(\varphi\in \cc(\R))
\]
by dominated convergence and hence, $h(\i t+\nu)=\tilde{h}(t)$ for
almost every $t\in\R$. This shows that the limit $h$ is independent of choice of the subsequences $(h_{n_k})_k$ and $(h_{n_{k_l}})_l$.
Indeed, if $\hat{h}\from\C_{\Re>L}\to\C$ is the limit of another subsubsequence of $(h_n)_n$ as above, then $\hat{h}(\i\cdot+\nu) = \tilde{h}=h(\i\cdot+\nu)$ a.e.
Since $\hat{h}$ and $h$ are holomorphic, the identity theorem yields $\hat{h}=h$.

Now, assume for a contradiction that $(h_{n})_n$ does not converge locally uniformly to $h$. Then we
find a subsequence $(h_{n_{k}})_k$ of $(h_n)_n$, a compact set $K\subseteq\C_{\Re>L}$
and $\varepsilon>0$ such that
\begin{equation}
\norm{h_{n_{k}}-h}_{\infty,K}\geq\varepsilon\quad(k\in\N).\label{eq:subsubsequence}
\end{equation}
However, the subsequence $(h_{n_k})_k$ has a subsequence $(h_{n_{k_l}})_l$ which converges locally uniformly to $h$, contradicting \prettyref{eq:subsubsequence}.
Thus, $(h_n)_n$ itself converges locally uniformly to $h$, and, in particular, $h_n\to h$ pointwise on $\C_{\Re>L}$.

\emph{Step 2:} By what we have shown in Step 1, the sequence $(m_{n})_{n\in\N}$
converges pointwise on $B\bigl(1/(2L),1/(2L)\bigr)$. Since $(m_{n})_n$
is also uniformly bounded on compact subsets of $B(0,1/L)$, we derive
that $(m_{n})_n$ converges uniformly on compact subsets of $B(0,1/L)$
by Vitali's theorem (see \cite[Theorem 6.2.8]{Simon2015}). Choosing
$0<r<1/L$, we thus obtain by Cauchy's integral formula 
\[
\scp{\phi}{B_{n}^{k}\psi}_H=(-1)^{k}\frac{1}{2\pi\i}\int_{\partial B(0,r)}\frac{m_{n}(z)}{z^{k+2}}\d z.
\]
Thus $(B_n^k)_n$ converges in the weak operator topology as $n\to\infty$.
\end{proof}



\section{Comments}

The problems discussed here are contained in \cite{W16_H,W14_G} for
both the weak and the strong operator topology. The case of differential-algebraic
equations has been invoked as well. 

The appearance of memory effects; that is, the occurence of higher
order integral operators due to a weak convergence of the coefficients
has been first observed by Tartar and can, for instance, be found in
\cite{Tartar2009}. The limit equation, however, is described by a
convolution term rather than a power series of integral operators.
It is, however, possible to reformulate these resulting equations into
one another, see \cite{W11_P}. 

The last characterisation of weak convergence in \prettyref{thm:char_thm_wot} was formulated
for the first time in \cite{PTW15_WP_P}.

\section*{Exercises}
\addcontentsline{toc}{section}{Exercises}

\begin{xca}
Let $(V_{n})_n$ in $L_\infty(\R^{d})$ and $V\in L_\infty(\R^{d})$. Characterise
convergence of $V_{n}(\m)\to V(\m)$ in the strong operator topology
of $\bo(\L(\R^{d}))$ in terms of convergence of $(V_{n})_n$ similar
to as was done in \prettyref{prop:conve_mult}.
\end{xca}

\begin{xca}
Show that there exists an unbounded sequence $(V_{n})_n$ in $L_{\infty}(\oi{0}{1})$
and $V\in L_{\infty}(\oi{0}{1})$ with $V_{n}\to V$ in $L_{1}(\oi{0}{1})$. 
\end{xca}

\begin{xca}
\label{exer:improved_cd_sot_ex}
Let $(\Omega,\Sigma,\mu)$ be a finite measure space, $(V_{n})_n$ a
bounded sequence in $L_{\infty}(\mu)$ and assume that $V_{n}\to V$
in $L_{1}(\mu)$ for some $V\in L_{\infty}(\mu)$. Show that there
exists $\nu>0$ such that 
\[
\bigl(\td{\nu}+V_{n}(\m)\bigr)^{-1}\to\bigl(\td{\nu}+V(\m)\bigr)^{-1}
\]
in the strong operator topology of $\bo\bigl(\Lnu(\R;\L(\mu)),H_{\nu}^{1}(\R;\L(\mu))\bigr)$.
\end{xca}

\begin{xca}
Let $D=\bigcup_{n\in\Z} \ci{n+1/2}{n+1}$, $V_{n}\coloneqq\1_{D}(n\cdot)$.
For suitable $\nu>0$ compute the limit of 
\[
\bigl((\td{\nu}+V_{n}(\m))^{-1}\bigr)_n
\]
in the weak operator topology of $\Lnu\bigl(\R;\L(\oi{0}{1})\bigr)$.
\end{xca}

\begin{xca}
Let $H$ be a Hilbert space, $c>0$ and $c\leq B_{n}=B_{n}^{*}\in\bo(H)$
for all $n\in\N$. Characterise, in terms of convergence
of $(B_{n})_n$ in a suitable sense, that
\[
\bigl((\td{\nu}B_{n})^{-1}\bigr)_n
\]
converges in the weak operator topology. In the case of convergence, find
its limit and a sufficient condition for which there exists a $B\in\bo(H)$
such that 
\[
(\td{\nu}B_{n})^{-1}\to(\td{\nu}B)^{-1}
\]
in the weak operator topology.
\end{xca}

\begin{xca}
Let $H$ be a Hilbert space. Show that $B_{\bo(H)}\coloneqq\set{B\in\bo(H)}{\norm{B}\leq1}$
is a compact subset under the weak operator topology. If, in addition,
$H$ is separable, show that $B_{\bo(H)}$ is also metrisable under
the weak operator topology.
\end{xca}

\begin{xca}
Let $H$ be a separable Hilbert space, $(B_{n})_n$
in $\bo(H)$ bounded. Show that there exists a subsequence $(B_{n_{k}})_k$ of $(B_n)_n$, a material law $M\from\dom(M)\to\bo(H)$ and $\nu>0$ such that
given $f\in\Lnu(\R;H)$ and $(u_{k})_k$ in $\Lnu(\R;H)$ with 
\[
\td{\nu}u_{k}+B_{n_{k}}u_{k}=f\quad(k\in\N),
\]
we deduce that $(u_{k})_k$ converges weakly to some
$u\in\Lnu(\R;H)$ with the property that
\[
\td{\nu}M(\td{\nu})u=f.
\]
\end{xca}

\printbibliography[heading=subbibliography]

\chapter{Continuous Dependence on the Coefficients II}

This chapter is concerned with the study of problems of the form
\[
\left(\td{\nu}M_{n}(\td{\nu})+A\right)U_{n}=F
\]
for a suitable sequence of material laws $\left(M_{n}\right)_{n}$
when $A\neq0$. The aim of this chapter will be to provide the conditions required for
convergence of the material law sequence to imply the existence
of a limit material law $M$ such that the limit $U=\lim_{n\to\infty}U_{n}$
exists and satisfies 
\[
\left(\td{\nu}M(\td{\nu})+A\right)U=F.
\]
Additionally, for material laws of the form
$M_{n}(\td{\nu})=M_{0,n}+\td{\nu}^{-1}M_{1,n}$ it will be desirable to have the respective limit
material law satisfy $M(\td{\nu})=M_{0}+\td{\nu}^{-1}M_{1}$ for
some $M_{0},M_{1}\in\bo(H)$. This cannot be expected (as we have seen in the guiding example in the
previous chapter) if $A$ is
a bounded operator, the Hilbert space $H$ is infinite-dimensional, and
the material law sequence only converges pointwise in the weak operator
topology. It will turn out, however, that if $A$ is ``strictly unbounded''
then a suitable result can hold, even if we only assume weak convergence
of the material law operators. 

\section{A Convergence Theorem}

The main convergence theorem of this chapter will be presented next.
\begin{thm}
\label{thm:convPDE}Let $H$ be a Hilbert space, $\nu_{0}\in\R,$ $\left(M_{n}\right)_{n}$
in $\mathcal{M}(H,\nu_{0})$ and $M\in\mathcal{M}(H,\nu_{0})$. Assume there
exists $c>0$ such that for all $n\in\N$ we have
\[
\Re zM_{n}(z)\geq c\quad(z\in\C_{\Re>\nu_{0}}).
\]
Let $A\colon\dom(A)\subseteq H\to H$ be skew-selfadjoint and assume $\dom(A)\hookrightarrow H$ compactly. If $M_{n}(z)\to M(z)$
as $n\to\infty$ in the weak operator topology for all $z\in\C_{\Re>\nu_{0}}$,
then 
\[
\bigl(\overline{\td{\nu}M_{n}(\td{\nu})+A}\bigr)^{-1}\to \bigl(\overline{\td{\nu}M(\td{\nu})+A}\bigr)^{-1}
\]
in the strong operator topology of $\bo(\Lnu(\R;H))$ for each $\nu>\nu_0$. 
\end{thm}

For the proof of this theorem, we need a lemma first.
\begin{lem}
\label{lem:convPDE}Let $H$ be a Hilbert space, $A\colon\dom(A)\subseteq H\to H$
skew-selfadjoint, $c>0$, $(T_{n})_{n}$ in $\bo(H)$ with $\Re T_{n}\geq c$ for all $n\in\N$, and $T\in\bo(H)$. 
Assume $\dom(A)\hookrightarrow H$ compactly and $T_{n}\to T$ in the
weak operator topology. Then $0\in\bigcap_{n\in\N}\rho(T_{n}+A)\cap\rho(T+A)$
and 
\[
\left(T_{n}+A\right)^{-1}\to\left(T+A\right)^{-1}
\]
in the norm topology of $\bo(H)$. 
\end{lem}

\begin{proof}
From $\Re T_{n}\geq c$ it follows that $0\in\rho(T_{n}+A)$ ($n\in\N$) and $\bigl((T_{n}+A)^{-1}\bigr)_{n}$
is bounded in $\bo(H)$. Indeed, since $B\coloneqq T_n+A$ satisfies $\Re B=\Re T_n\geq c$ and $\dom(B)=\dom(A)=\dom(B^\ast)$ due to the skew-selfadjointness of $A$, \prettyref{prop:accr_inv} yields the assertion. Moreover, since
\[
 A(T_n+A)^{-1}=1-T_n(T_n+A)^{-1}
\]
for all $n\in \N$, it follows that $\left((T_n+A)^{-1}\right)_n$ is also bounded in $\bo(H,\dom(A))$ by the boundedness of $(T_n)_n$ in $\bo(H)$. Due to
the convergence of $(T_{n})_{n}$ to $T$, it follows that $\Re T\geq c$, and thus, $(T+A)^{-1}\in\bo(H,\dom(A))$.
Before we come to a proof of the desired result, we will prove an
auxiliary observation.

Claim: for all $(f_{n})_{n}$ in $H$ weakly converging to $f$, we
have $\left(T_{n}+A\right)^{-1}f_{n}\to\left(T+A\right)^{-1}f$
in the norm topology of $H$. 

For proving the claim, let $(f_{n})_{n}$ in $H$ be weakly convergent
to some $f$. Consider $u_{n}\coloneqq (T_{n}+A)^{-1}f_{n}$. Then $(u_n)_n$ is bounded in $\dom(A)$, since $\bigl((T_n+A)^{-1}\bigr)_n$ is bounded in $\bo(H,\dom(A))$ and $(f_n)_n$ is bounded in $H$. Hence, there exists a subsequence $(u_{n_{k}})_{k}$ which weakly converges to some $u$ in $\dom(A)$.
Since $\dom(A)\hookrightarrow H$ compactly, we infer $u_{n_k}\to u$ in
the norm topology of $H$. Hence, in the
equality
\[
T_{n_{k}}u_{n_{k}}+Au_{n_{k}}=f_{n_{k}},
\]
as $T_{n_k}\to T$ in the weak operator topology and $u_{n_k}\to u$ in $H$, we may let $k\to\infty$ and obtain for the weak limits
\[
Tu+Au=f;
\]
that is, $u=\left(T+A\right)^{-1}f$. Having identified the limit,
a contradiction argument (here a so-called `subsequence argument', see \prettyref{exer:convPDEstrong}) concludes that $(u_{n})_{n}$ itself converges
weakly in $\dom(A)$ and strongly in $H$ to $u$. Thus, the claim
is proved.

Next, assume by contradiction that $\bigl((T_{n}+A)^{-1}\bigr)_{n}$
does not converge in operator norm to $\left(T+A\right)^{-1}$. Then
we find an $\varepsilon>0$ and a strictly increasing sequence of
integers, $(n_{k})_{k}$, and a sequence of unit vectors $(f_{n_{k}})_{k}$
in $H$ such that 
\begin{equation}
\norm{(T_{n_{k}}+A)^{-1}f_{n_{k}}-\left(T+A\right)^{-1}f_{n_{k}}}\geq\varepsilon.\label{eq:contr}
\end{equation}
By possibly taking another subsequence, we may assume without loss
of generality that $\left(f_{n_{k}}\right)_{k}$ converges weakly
to some $f\in H$. By the claim proved above, we deduce  $\left(T_{n_{k}}+A\right)^{-1}f_{n_{k}}\to\left(T+A\right)^{-1}f$
and $\left(T+A\right)^{-1}f_{n_{k}}\to\left(T+A\right)^{-1}f$, both
in the norm topology of $H$ as $k\to\infty.$ Thus, we may let $k\to\infty$
in \prettyref{eq:contr}, and obtain the desired contradiction.
\end{proof}
\begin{proof}[Proof of \prettyref{thm:convPDE}]
 By \prettyref{thm:converge_wotsot} it suffices to show that for
all $z\in\C_{\Re>\nu_{0}}$ 
\[
\left(zM_{n}(z)+A\right)^{-1}\to\left(zM(z)+A\right)^{-1}\quad(n\to\infty)
\]
 in the strong operator topology. This, however, follows from \prettyref{lem:convPDE}
applied to $T_{n}=zM_{n}(z)$.
\end{proof}
\begin{rem}
Note that we only used convergence in the strong operator topology
in the proof of \prettyref{thm:convPDE}. However, the assertion in
\prettyref{lem:convPDE} is about convergence in the norm topology.
The reason that we cannot assert the convergence claimed in \prettyref{thm:convPDE}
in the norm topology is that the compact embedding of $\dom(A)\hookrightarrow H$
only works locally for fixed $z$, and not uniformly in $z$. This
situation can, however, be rectified. We refer to \prettyref{exer:normPDEconv}
for this.
\end{rem}

\section{The Theorem of Rellich and Kondrachov}

In order to apply \prettyref{thm:convPDE}, we need to provide a setting
where the condition on the compactness of the embedding is satisfied.
In fact, it is true that $H^{1}(\Omega)$ embeds compactly into $\L(\Omega)$
given $\Omega\subseteq\R^{d}$ is bounded and has `continuous boundary', see e.g.\ \cite[Theorem 7.11]{Arendt2015}. In this
chapter, we restrict ourselves to a proof of a less general statement. 

A preparatory result needed to prove the compact embedding theorem is given
next.
\begin{prop}
\label{prop:restriction} Let $I\subseteq\R$ be an open, bounded, non-empty
interval. Then the mapping $H^{1}(\R)\ni f\mapsto f|_{I}\in H^{1}(I)$
is well-defined, continuous and onto. Moreover, there exists a continuous right inverse $H^1(I)\to H^1(\R)$.
\end{prop}

For the proof of this proposition, we need an auxiliary result first.

\begin{lem}\label{lem:kernel_grad}
 Let $\Omega \subseteq \R^d$ be open and connected. Moreover, let $u\in H^1(\Omega)$ with $\grad u=0$. Then $u$ is constant.
\end{lem}

We leave the proof of this lemma as \prettyref{exer:kernel_grad}.

\begin{proof}[Proof of \prettyref{prop:restriction}]
The mapping $H^{1}(\R)\to H^{1}(I),f\mapsto f|_{I}$
is readily confirmed to be continuous. It remains to prove that it is onto. Let $I=\oi{a}{b}$, $u\in H^1(I)$ and define the function $v$ by
\[
 v(t)\coloneqq \int_a^t \partial u(s)\d s\quad (t\in \oi{a}{b}).
\]
Clearly, $v\in \L(\oi{a}{b})$ and we compute for each $\varphi\in \cci(\oi{a}{b})$
\begin{align*}
 \scp{v}{\varphi'}_{\L(\oi{a}{b})}
 & = \int_a^b \left(\int_a^t \partial u(s)\d s\right)^\ast \varphi'(t) \d t
 = \int_a^b \int_s^b \varphi'(t)\d t\, \partial u(s)^\ast \d s \\
 & = -\scp{\partial u}{\varphi}_{\L(\oi{a}{b})}.
\end{align*}
This shows $v\in H^1(\oi{a}{b})$ with $\partial v=\partial u$. Hence, by \prettyref{lem:kernel_grad} there exists a constant $c\in \C$ with $u=c+v$. We now define $f$ by
\[f(t) \coloneqq \begin{cases}
                 0 &\mbox{if } t<a-1 \mbox{ or } t>b+1,\\
                 ct+c(1-a)  &\mbox{if } a-1\leq t \leq a,\\
                 u(t) &\mbox{if } a<t<b,\\
                 -(c+v(b))t+(c+v(b))(1+b) &\mbox{if } b\leq t\leq b+1.
                \end{cases}
\]
We then easily see that $f\in H^1(\R)$ and clearly $f|_{\oi{a}{b}}=u$. In order to see that $u\mapsto f$ is continuous, we need to establish that the value $c$ depends continuously on $u$. This, however, follows from the estimate
\begin{align*}
 |c|&=\frac{1}{\sqrt{b-a}}\left(\int_a^b |c|^2\right)^{1/2} \leq \frac{1}{\sqrt{b-a}} (\norm{u}_{\L(a,b)}+\norm{v}_{\L(a,b)}) \\
   &\leq \frac{1}{\sqrt{b-a}}(\norm{u}_{\L(a,b)} + (b-a)\norm{\partial u}_{\L(a,b)}) 
   \leq \frac{\sqrt{2} \max\{1,(b-a)\}}{\sqrt{b-a}} \norm{u}_{H^1(a,b)}.\qedhere
\end{align*}
\end{proof}

\begin{thm}
\label{thm:compH1}Let $I\subseteq\R$ be an open bounded interval.
Then $H^{1}(I)\hookrightarrow\L(I)$ compactly.
\end{thm}

\begin{proof} By \prettyref{prop:restriction}, we find a continuous mapping $E\colon H^1(I)\to H^1(\R)$ such that for all $u\in H^1(I)$ we have $E(u)|_I=u$. Moreover, by \prettyref{exer:12Hoelder} the mapping $H^1(\R)\hookrightarrow C^{1/2}(\R)$ is continuous. Thus,
\[
   H^1(I) \stackrel{E}{\to} H^1(\R) \hookrightarrow C^{1/2}(\R) \to C^{1/2}(I),
\] 
is a composition of continuous mappings, where the last mapping is the restriction to $I$. Since $C^{1/2}(I)\hookrightarrow C(I)$ compactly by the Arzel\`{a}--Ascoli theorem, and $C(I)\hookrightarrow \L(I)$ continuously, we infer $H^1(I)\hookrightarrow \L(I)$ compactly.
\end{proof}
We now have the opportunity to study the limit behaviour of a periodic
mixed type problem.

\begin{example}[Highly oscillatory problems]
Let $s_{1},s_{2}\colon\R\to\ci{0}{1}$ be $1$-periodic, measurable
functions. Then for $\nu>0$, we set
\[
S^{(n)}\coloneqq\Biggl(\overline{\td{\nu}\begin{pmatrix}
s_{1}(n\m) & 0\\
0 & s_{2}(n\m)
\end{pmatrix}+\begin{pmatrix}
1-s_{1}(n\m) & 0\\
0 & 1-s_{2}(n\m)
\end{pmatrix}+\begin{pmatrix}
0 & \partial\\
\partial_{0} & 0
\end{pmatrix}}\Biggr)^{-1},
\]
where $\partial=\dive$ and $\partial_{0}=\grad_{0}$ are regarded as operators
in $\L(\oi{0}{1})$ with respective domains $H^{1}(\oi{0}{1})$ and $H_{0}^{1}(\oi{0}{1})$.
Then, by \prettyref{thm:compH1}, the operator $A\coloneqq\begin{pmatrix}
0 & \partial\\
\partial_{0} & 0
\end{pmatrix}$ satisfies the assumptions of \prettyref{thm:convPDE}. Moreover,
\prettyref{thm:periodic_homo} implies that the remaining assumptions
of \prettyref{thm:convPDE} are satisfied. Hence, we deduce that $\left(S^{(n)}\right)_{n}$
converges in the strong operator topology of $\Lnu\bigl(\R;\L(\oi{0}{1})\bigr)$
to the limit
\[
\Biggl(\overline{\td{\nu}\begin{pmatrix}
\int_{0}^{1}s_{1} & 0\\
0 & \int_{0}^{1}s_{2}
\end{pmatrix}+\begin{pmatrix}
1-\int_{0}^{1}s_{1} & 0\\
0 & 1-\int_{0}^{1}s_{2}
\end{pmatrix}+\begin{pmatrix}
0 & \partial\\
\partial_{0} & 0
\end{pmatrix}}\Biggr)^{-1}.
\]
\end{example}

Next, we aim to provide an application to more than one spatial dimension.
For this, we will also need a corresponding compactness statement. This is the subject of the rest of this section.
\begin{thm}[Rellich--Kondrachov]
\label{thm:Rellich} Let $\Omega\subseteq\R^{d}$ be open and bounded.
Then $H_{0}^{1}(\Omega)\hookrightarrow\L(\Omega)$ compactly.
\end{thm}

\begin{proof}
Without loss of generality (by shifting and
shrinking of $\Omega$ and extending by $0$), we may assume that $\Omega=\oi{0}{1}^{d}$.
We carry out the proof by induction on the spatial dimension $d$.
The case $d=1$ has been dealt with in \prettyref{thm:compH1}. Assume
the statement is true for some $d-1$. Using that $\cci(\oi{0}{1}^d)$ is dense in $H_{0}^{1}(\oi{0}{1}^d)$, we infer the continuity of the
injection 
\begin{align*}
R\colon H_{0}^{1}(\oi{0}{1}^d) & \to H^{1}\bigl(\R;\L(\oi{0}{1}^{d-1})\bigr)\cap\L\bigl(\R;H_{0}^{1}(\oi{0}{1}^{d-1})\bigr)\\
\phi & \mapsto\bigl(t\mapsto\bigl(\omega\mapsto\phi(t,\omega)\bigr)\bigr),
\end{align*}
where we identify $\phi$ with its extension to $\R^d$ by $0$. The range space is endowed with the usual sum scalar
product.

Let $(\phi_{n})_{n}$ be a weakly convergent null-sequence
in $H_{0}^{1}(\oi{0}{1}^d)$. In particular, $\left(R\phi_{n}\right)_{n}$ is bounded in $H^{1}\bigl(\R;\L(\oi{0}{1}^{d-1})\bigr)$ and hence, it 
is also bounded in $\cb\bigl(\R;\L\oi{0}{1}^{d-1}\bigr)$ by \prettyref{thm:Sobolev_emb} (and \prettyref{cor:vanish_at_inf});
that is,
\begin{equation}
\sup_{t\in\ci{0}{1},n\in\N}\norm{\phi_{n}(t,\cdot)}_{\L(\oi{0}{1}^{d-1})}<\infty.\label{eq:boundRellich}
\end{equation}
Let $f\in\L(\oi{0}{1}^{d-1})$. Then $(\phi_{n,f})_n$ given by
\[
\phi_{n,f}\colon t\mapsto\scp{\phi_{n}(t,\cdot)}{f}_{\L(\oi{0}{1}^{d-1})}
\]
is a weakly convergent null-sequence in $H^1(\oi{0}{1})$. We obtain
by \prettyref{thm:compH1} that $\phi_{n,f}\to 0$ in $\L(\oi{0}{1})$
as $n\to\infty$. By separability of $\L(\oi{0}{1}^{d-1})$ we find $D\subseteq\L(\oi{0}{1}^{d-1})$ countable
and dense, a subsequence (again labeled by $n$)  and 
a null-set $N\subseteq\R$ such that $\phi_{n,f}(t)\to 0$
for all $t\in\R\setminus N$ and $f\in D$ as $n\to \infty$. By \prettyref{eq:boundRellich}, we deduce $\phi_{n,f}(t)\to0$
for all $t\in\R\setminus N$ and $f\in\L(\oi{0}{1}^{d-1})$ as $n\to\infty$, or, in other words, $\phi_{n}(t,\cdot)\to 0$ weakly in $\L(\oi{0}{1}^{d-1})$ for each $t\in \R\setminus N$ as $n\to \infty$.

Next, we show that there exists a null set $N\subseteq N_{1}\subseteq\R$ such that $\phi_n(t,\cdot)\to 0$ in $\L(\oi{0}{1}^{d-1})$ for all $t\in\R\setminus N_{1}$.
For this, since $(R\phi_n)_n$ in $\L\bigl(\R;H_0^1(\oi{0}{1}^{d-1})\bigr)$ is bounded, we find a null set $N\subseteq N_{1}\subseteq\R$ such
that $\left(\phi_{n}(t,\cdot)\right)_{n}$ is bounded in $H_{0}^{1}(\oi{0}{1}^{d-1})$
for all $t\in\R\setminus N_{1}$. Let $t\in \R\setminus N_1$. Then there exists a further subsequence $(\phi_{n_k}(t,\cdot))_{k}$ which converges weakly in $H_{0}^{1}(\oi{0}{1}^{d-1})$. By the induction hypothesis, $(\phi_{n_k}(t,\cdot))_{n_k}$ converges strongly in $\L(\oi{0}{1}^{d-1})$, and since we have already seen that it is a weak null-sequence in $\L(\oi{0}{1}^{d-1})$, we derive
$\phi_{n_k}(t,\cdot)\to 0$ in $\L(\oi{0}{1}^{d-1})$. By a subsequence argument we derive that
\[\phi_n(t,\cdot)\to 0\]
in $\L(\oi{0}{1}^{d-1})$ for all $t\in \R\setminus N_1$.

Now, for
$n\in\N$ we deduce
\[
\norm{\phi_{n}}_{\L(\oi{0}{1}^{d})}^{2}=\int_{0}^{1}\norm{\phi_{n}(t,\cdot)}_{\L(\oi{0}{1}^{d-1})}^{2}\d t\to0,
\]
where we have used dominated convergence, which is possible due to \prettyref{eq:boundRellich}.
\end{proof}

\section{The Periodic Gradient}

In this section we investigate the gradient on periodic functions on $\R^d$. Throughout, we set
$Y\coloneqq\roi{0}{1}^{d}$.
\begin{defn*}[Periodic Gradient]
We define 
\[
C_{\sharp}^{\infty}(Y)\coloneqq\set{\phi|_{Y}}{\phi\in C^{\infty}(\R^{d}),\ \phi(\cdot+k)=\phi\ (k\in\Z^{d})}
\]
and
\begin{align*}
\grad_{\sharp,\infty}\colon C_{\sharp}^{\infty}(Y)\subseteq\L(Y) & \to\L(Y)^{d}\\
\phi & \mapsto\grad\phi.
\end{align*}
Moreover, we set $\dive_{\sharp}\coloneqq-\grad_{\sharp,\infty}^{*}$ and
$\grad_{\sharp}\coloneqq-\dive_{\sharp}^{*}=\overline{\grad_{\sharp,\infty}}$.
\end{defn*}
\begin{rem}
The operators just introduced can easily be shown to lie between the
operator realisations we have introduced in earlier chapters. Indeed, it is easy
to see that 
\[
\dive_{0}\subseteq\dive_{\sharp}\text{ and }\grad_{0}\subseteq\grad_{\sharp}
\]
and, consequently, we also have
\[
\grad_{\sharp}\subseteq\grad\text{ and }\dive_{\sharp}\subseteq\dive.
\]
The corresponding domains for the operators $\grad_{\sharp}$ and
$\dive_{\sharp}$ will be denoted by $H_{\sharp}^{1}(Y)$
and $H_{\sharp}(\dive,Y)$, respectively.

For the next results, we define the periodic extension operator. For
$\phi\in\L(Y)^{m}$ we put
\[
\phi_{\rmpe}(x+k)\coloneqq\phi(x)
\]
for almost every $x\in Y$ and all $k\in\Z^{d}$.
\end{rem}

We start with the following two observations.

\begin{lem}\label{lem:delta-seq-periodic}
 Let $f\in \L(Y)$ and $(\rho_k)_k$ be a $\delta$-sequence in $\cci(\R^d)$ (cf.~\prettyref{exer:delta-seq}). Define
 \[
  f_k \coloneqq (\rho_k\ast f_\rmpe)|_Y \quad (k\in \N).
 \]
 Then $f_k\in C_\sharp^\infty(Y)$ for each $k\in \N$ and $f_k\to f$ in $\L(Y)$ as $k\to \infty$.
\end{lem}

\begin{proof}
 It follows as in \prettyref{exer:C_cinfty dense} that $\rho_k\ast f_\rmpe$ is in $C^\infty$. Moreover, one easily sees that $\rho_k\ast f_\rmpe$ is $\roi{0}{1}^d$-periodic, and hence, $f_k\in C^\infty_\sharp(Y)$ for each $k\in \N$. For the convergence we observe 
 \[
  \bigl(\rho_k\ast (\1_{Y+B(0,1)} f_\rmpe)\bigr)(x)=f_k(x)\quad (x\in Y, k\in \N).
 \]
 Moreover, by \prettyref{exer:C_cinfty dense} we have $\rho_k\ast (\1_{Y+B(0,1)} f_\rmpe) \to \1_{Y+B(0,1)} f_\rmpe$ in $\L(\R^d)$ as $k\to \infty$, and thus,
 \[
  f_k =\bigl(\rho_k\ast (\1_{Y+B(0,1)} f_\rmpe)\bigr)|_Y\to (\1_{Y+B(0,1)} f_\rmpe)|_Y=f\quad (k\to \infty) \quad\text{in $\L(Y)$.}\qedhere
 \]

\end{proof}

\begin{lem}\label{lem:C_sharp_core}
 $C_\sharp^\infty(Y)^d$ is a core for $\dive_\sharp$ and $\dive_\sharp\Psi=\dive \Psi$ for each $\Psi\in C_\sharp^\infty(Y)^d$.
\end{lem}

\begin{proof}
 First we note that $C_\sharp^\infty(Y)^d\subseteq \dom(\dive_\sharp)$ and $\dive_\sharp \Psi=\dive \Psi$ for $\Psi\in C_\sharp^\infty(Y)^d$. To see this, for $\phi \in C_\sharp^\infty(Y),\Psi\in C_\sharp^\infty(Y)^d$ we compute
 \begin{align*}
  \scp{\grad \phi}{\Psi}_{\L(Y)^d} & =\int_Y \scp{\grad \phi(x)}{\Psi(x)}_{\K^d} \d x=-\int_Y \phi(x)^\ast \dive \Psi (x) \d x\\
  & =\scp{\phi}{-\dive \Psi}_{\L(Y)}
 \end{align*}
by integration by parts (note that the boundary values cancel out due to the periodicity of $\phi$ and $\Psi$). Now, let $q\in \dom(\dive_\sharp)$ and $(\rho_k)_k$ be a $\delta$-sequence in $\cci(\R^d)$. For $k\in\N$ we define 
\[
 q_k\coloneqq (\rho_k \ast q_\rmpe)|_Y,
\]
and obtain $q_k \in C^\infty_\sharp(Y)$ and $q_k\to q$ in $\L(Y)^d$ as $k\to \infty$ by \prettyref{lem:delta-seq-periodic}. It is left to show that $\dive q_k\to \dive_\sharp q$ in $\L(Y)$ as $k\to \infty$. For doing so, we show that $\dive q_k =\bigl(\rho_k \ast (\dive_\sharp q)_\rmpe\bigr)|_Y$, which would then yield the assertion again by \prettyref{lem:delta-seq-periodic}. So, let $k\in \N$ and $\phi\in C^\infty_\sharp(Y)$. We compute
\begin{align*}
 \scp{q_k}{\grad \phi}_{\L(Y)^d}&= \int_Y \scp{\int_{\R^d} \rho_k(y) q_\rmpe (x-y)\d y}{\grad \phi(x)}_{\K^d} \d x\\
 &= \int_{\R^d} \rho_k(y)  \int_Y \scp{q_\rmpe(x-y)}{\grad \phi(x)}_{\K^d} \d x\d y\\
 &= \int_{\R^d} \rho_k(y)  \int_{Y-y} \scp{q_\rmpe(x)}{(\grad \phi)_\rmpe (x+y)}_{\K^d} \d x\d y\\
 &= \int_{\R^d} \rho_k(y)  \int_{Y} \scp{q(x)}{(\grad \phi)_\rmpe (x+y)}_{\K^d} \d x\d y\\
 &=\int_{\R^d} \rho_k(y)  \int_{Y} \scp{q(x)}{(\grad \phi_\rmpe(\cdot +y)) (x)}_{\K^d} \d x\d y\\
 &= -\int_{\R^d} \rho_k(y)  \int_{Y} \scp{\dive_\sharp q(x)}{\phi_\rmpe(x+y)}_{\K^d} \d x\d y\\
 &= -\int_{\R^d} \rho_k(y) \int_{Y+y} \scp{(\dive_\sharp q)_\rmpe(x-y)}{\phi_\rmpe(x)}_{\K^d}\d x\d y \\
 &= -\scp{\bigl(\rho_k \ast (\dive_\sharp q)_\rmpe\bigr)|_Y}{\phi}_{\L(Y)},
\end{align*}
where we have used perodicity 
as well as $\phi_\rmpe(\cdot +y)\in C_\sharp^\infty(Y)$. 
\end{proof}

\begin{rem}\label{rem:approx_in_Kern_div_sharp}
  The proof of \prettyref{lem:C_sharp_core} reveals that every $q\in \ker(\dive_\sharp)$ can be approximated by elements in $C^\infty_\sharp(Y)^d\cap \ker(\dive_\sharp)$.
\end{rem}

\begin{prop}
\label{prop:periodicextension}
Let $\Omega\subseteq\R^{d}$ be open,
bounded, $u\in H_{\sharp}^{1}(Y)$ and $q\in H_\sharp(\dive,Y)$.
Then $u_{\rmpe}|_{\Omega}\in H^{1}(\Omega)$, $q_\rmpe|_\Omega \in H(\dive,\Omega)$ and 
\[
\grad\left(u_{\rmpe}|_{\Omega}\right)=\left(\grad_{\sharp}u\right)_{\rmpe}|_{\Omega}  \mbox{ and } \dive\left(q_{\rmpe}|_{\Omega}\right)=\left(\dive_{\sharp}q\right)_{\rmpe}|_{\Omega}.
\]
\end{prop}

\begin{proof}
Let first $\phi \in C^\infty_\sharp(Y)$. Then by definition $\phi_\rmpe \in C^\infty(\R^d)$ and we easily see
\[
 \grad \phi_{\rmpe} = (\grad \phi)_\rmpe = (\grad_\sharp \phi)_\rmpe. 
\]
Moreover, since $\Omega$ is bounded, we infer $\phi_\rmpe \in H^1(\Omega)$. By definition of $H^1_\sharp(Y)$ we find a sequence $(\phi_k)_{k\in \N}$ in $C^\infty_\sharp(Y)$ such that $\phi_k\to u$ in $\L(Y)$ and $\grad_\sharp \phi_k\to \grad_\sharp u$ in $\L(Y)^d$ as $k\to \infty$. Since 
\[
 \L(Y)\to \L(\Omega),\quad f\mapsto f_\rmpe
\]
is bounded due to the boundedness of $\Omega$, we also derive $\phi_{k,\rmpe}\to u_\rmpe$ in $\L(\Omega)$ and $(\grad_\sharp \phi_k)_\rmpe\to (\grad_\sharp u)_\rmpe$ in $\L(\Omega)^d$ as $k\to \infty$.  By what we have shown above, we infer
\[
 \grad \phi_{k,\rmpe}=(\grad_\sharp \phi_k)_\rmpe\to (\grad_\sharp u)_\rmpe\quad (k\to \infty)
\]
in $\L(\Omega)^d$, and thus, $u_\rmpe\in H^1(\Omega)$ with $\grad u_\rmpe=(\grad_\sharp u)_\rmpe$ by the closedness of $\grad$. The proof for $q$ follows by the same argument with \prettyref{lem:C_sharp_core} as an additional resource.
\end{proof}

The extension result just established yields the following compactness
statement.
\begin{thm}[Rellich--Kondrachov II]
\label{thm:RellichII} The embedding $H_{\sharp}^{1}(Y)\hookrightarrow\L(Y)$
is compact.
\end{thm}

\begin{proof}
Let $(\phi_{n})_{n}$ be a bounded sequence in $H_{\sharp}^{1}(Y)$.
Let $\Omega\subseteq\R^{d}$ be open and bounded such that $\overline{Y}\subseteq\Omega.$
By \prettyref{prop:periodicextension}, we deduce that $\left(\phi_{n,\rmpe}|_{\Omega}\right)_{n}$
is bounded in $H^{1}(\Omega)$. Let $\psi\in\cci(\Omega)$ with $\psi=1$
on $\overline{Y}$. Then $\left(\psi\phi_{n,\rmpe}\right)_{n}$
is bounded in $H_{0}^{1}(\Omega)$. By \prettyref{thm:Rellich}, we
find an $\L(\Omega)$-convergent subsequence. This sequence also converges
in $\L(Y)$. Since $\psi=1$ on $Y$, we obtain the assertion.
\end{proof}

Next, we provide a Poincar\'{e}-type inequality for the periodic gradient. 

\begin{prop}
\label{prop:Poincare_periodic} There exists $c\geq 0$ such that for
all $u\in H_{\sharp}^{1}(Y)$
\[
\norm{u-\int_{Y}u}_{\L(Y)}\leq c \norm{\grad_{\sharp}u}_{\L(Y)^{d}}.
\]
 In particular, $\ran(\grad_{\sharp})\subseteq\L(Y)^{d}$ is closed, $\ker(\grad_\sharp)=\lin\{\1_Y\}$ and the operator
 \[
  \grad_\sharp\from H^1_\sharp(Y)\cap \{\1_Y\}^\bot \to \ran(\grad_\sharp)
 \]
is an isomorphism.
\end{prop}

\begin{proof}
The proof is left as \prettyref{exer:poincare}.
\end{proof}
We are now in a position to formulate the particular example we have in mind.
Problems of this type with highly oscillatory coefficients are also
referred to as \emph{homogenisation problems}. We refer to the comments section
for more details on this.
\begin{example}[Homogenisation problem for the wave equation]
\label{exa:wave_homo} Let $c>0$, $a\colon\R^{d}\to\K^{d\times d}$ be bounded,
measurable, $a(x)=a(x)^{*}\geq c$ for all $x\in\R^{d}$. Furthermore,
assume that $a$ is $\roi{0}{1}^{d}$-periodic. Let $\nu>0$, $f\in\Lnu(\R;\L(Y))$
and for $n\in\N$ consider the problem of finding $u_{n}\in\Lnu(\R;\L(Y))$
such that 
\begin{equation}
\td{\nu}^{2}u_{n}-\dive_{\sharp}a(n\m)\grad_{\sharp}u_{n}=f. \label{eq:waven}
\end{equation}
We have already established that there exists a uniquely determined solution, $u_{n}$. 
Employing the same trick as in \prettyref{sec:Three_Exponentially_Stable_Models_for_Heat_Conduction},
we rewrite \prettyref{eq:waven} using $v_{n}\coloneqq\td{\nu}u_{n}$, the canonical embedding $\iota_{\sharp}\colon\ran(\grad_{\sharp})\hookrightarrow\L(Y)^{d}$
as well as $q_{n}\coloneqq-\iota_{\sharp}^{*}a(n\m)\iota_{\sharp}\iota_{\sharp}^{*}\grad_{\sharp}u_{n}$
to obtain
\[
\left(\td{\nu}\begin{pmatrix}
1 & 0\\
0 & \left(\iota_{\sharp}^{*}a(n\m)\iota_{\sharp}\right)^{-1}
\end{pmatrix}+\begin{pmatrix}
0 & \dive_{\sharp}\iota_{\sharp}\\
\iota_{\sharp}^{*}\grad_{\sharp} & 0
\end{pmatrix}\right)\begin{pmatrix}
v_{n}\\
q_{n}
\end{pmatrix}=\begin{pmatrix}
f\\
0
\end{pmatrix}.
\]
Note that we have used that $\bigl(\iota_{\sharp}^{*}a(n\m)\iota_{\sharp}\bigr)\colon\ran(\grad_{\sharp})\to\ran(\grad_{\sharp})$
is continuously invertible and strictly positive definite (uniformly
in $n$); see \prettyref{prop:exp_stab_heat}. Also note that $\iota_{\sharp}^{*}a(n\m)\iota_{\sharp}$
is selfadjoint. As in  \prettyref{exer:closed_range} we see that $\left(\iota_{\sharp}^{*}\grad_{\sharp}\right)^{*}=-\dive_{\sharp}\iota_{\sharp}$.
Thus, the operator
\[
S^{(n)}\coloneqq\Biggl(\overline{\td{\nu}\begin{pmatrix}
1 & 0\\
0 & \left(\iota_{\sharp}^{*}a(n\m)\iota_{\sharp}\right)^{-1}
\end{pmatrix}+\begin{pmatrix}
0 & \dive_{\sharp}\iota_{\sharp}\\
\iota_{\sharp}^{*}\grad_{\sharp} & 0
\end{pmatrix}}\Biggr)^{-1}
\]
 is well-defined and bounded in $\Lnu\bigl(\R;\L(Y)\times\ran(\grad_{\sharp})\bigr)$.
We aim to find the limit of $(S^{(n)})_{n}$ as $n\to\infty$. For this,
we want to apply \prettyref{thm:convPDE}. We readily see using \prettyref{thm:RellichII}
and \prettyref{exer:compactness of adjoint} that 
\[
A\coloneqq\begin{pmatrix}
0 & \dive_{\sharp}\iota_{\sharp}\\
\iota_{\sharp}^{*}\grad_{\sharp} & 0
\end{pmatrix}
\]
satisfies the assumptions in \prettyref{thm:convPDE}. Thus, it is
left to analyse $\bigl(\bigl(\iota_{\sharp}^{*}a(n\m)\iota_{\sharp}\bigr)^{-1}\bigr)_{n}$.
This is the subject of the next section. For this reason, we define
\[\mathfrak{a}_n:=\bigl(\iota_{\sharp}^{*}a(n\m)\iota_{\sharp}\bigr)^{-1} \quad(n\in\N).\]
\end{example}

\section{The Limit of the Scaled Coefficient Sequence}

In this section, we shall apply our earlier findings to higher-dimensional
problems. Again, we fix $Y\coloneqq\roi{0}{1}^{d}$ as well as $\iota_{\sharp}\colon\ran(\grad_{\sharp})\hookrightarrow\L(Y)^{d}$, the canonical embedding. Before we are able to present the central result of this section, we need a preliminary result.

Throughout, let $a\colon\R^{d}\to\K^{d\times d}$
be measurable, bounded and $\roi{0}{1}^{d}$-periodic such that $\Re a(x)\geq c$ for each $x\in \R^d$ for some $c>0$.

\begin{lem}
\label{lem:ExistenceLemma} 
Let $\xi\in\K^{d}$. Then there exists a unique $v_\xi\in\L(Y)^{d}$ with $v_\xi-\xi \in \ran(\grad_\sharp)$ and $a(\m)v\in \ker(\dive_\sharp)$.
\end{lem}

\begin{proof}
Take $w\in H^1_\sharp(Y)$ such that
\[
\grad_{\sharp}w=-\iota_{\sharp}\left(\iota_{\sharp}^{*}a(\m)\iota_{\sharp}\right)^{-1}\iota_{\sharp}^{*}a(\m)\xi = -\iota_{\sharp}\mathfrak{a}_n \iota_{\sharp}^{*}a(\m)\xi.
\]
This is possible, since the right-hand side belongs to $\ran(\grad_\sharp)$ by definition. We put $v_\xi\coloneqq\grad_{\sharp}w+\xi$. Then
$v_\xi-\xi \in \ran(\grad_\sharp)$ and we have
\begin{align*}
\iota_{\sharp}^{*}a(\m)v_\xi & =\iota_{\sharp}^{*}a(\m)\left(\grad_{\sharp}w+\xi\right)
 =\iota_{\sharp}^{*}a(\m)\left(-\iota_{\sharp}\mathfrak{a}_n\iota_{\sharp}^{*}a(\m)\xi+\xi\right)\\
 & =-\iota_{\sharp}^{*}a(\m)\iota_\sharp\mathfrak{a}_n\iota_{\sharp}^{*}a(\m)\xi+\iota_{\sharp}^{*}a(\m)\xi=0.
\end{align*}
The latter gives $a(\m)v_\xi\in \ran(\grad_\sharp)^\bot=\ker(\dive_\sharp)$. For the uniqueness, we assume $v\in\ran(\grad_{\sharp})$
with $a(\m)v\in \ker(\dive_\sharp)$. 
Then
\[
(\iota_\sharp^* a(\m)\iota_\sharp)\iota_\sharp^\ast v=\iota_\sharp^*a(\m)v=0,
\]
which implies $\iota_{\sharp}^{*}v=0$ since $\iota_{\sharp}^{*}a(\m)\iota_{\sharp}$
is invertible. Thus $v=0$.
\end{proof}
The previous result induces the linear mapping 
\[
a_{\hom}\colon\K^{d}\ni\xi\mapsto\int_{Y}av_\xi\in\K^{d},
\]
where $v_\xi \in\L(Y)^{d}$ is the unique vector field from \prettyref{lem:ExistenceLemma}.

\begin{rem}\label{rem:a_hom}
We gather some elementary facts on $a_{\hom}$. 
\begin{enumerate}
 \item\label{rem:a_hom:item:1}
  We have $(a^\ast)_{\rmhom}=a_{\rmhom}^\ast$. In particular, if $a$ is pointwise selfadjoint then so is $a_{\rmhom}$. Indeed, let $\xi,\zeta\in\K^{d}$
and $v_{\xi}$ and $v_{\zeta}\in\L(Y)^{d}$ be the corresponding functions for $a^\ast$ and $a$, respectively, according to \prettyref{lem:ExistenceLemma}. Then there exist $w_{\xi},w_{\zeta}\in\dom(\grad_{\sharp})$ with $v_\xi-\xi=\grad_\sharp w_\xi$ and $v_\zeta-\zeta=\grad_\sharp w_\zeta$. 
We compute
\begin{align*}
\scp{(a^\ast)_{\hom}\xi}{\zeta}_{\K^d} & =\int_{Y}\scp{\left(a^\ast v_{\xi}\right)(y)}{v_{\zeta}(y)-\grad_\sharp w_{\zeta}(y)}_{\K^d}\d y\\
 & =\int_{Y}\scp{\left(a^\ast v_{\xi}\right)(y)}{v_{\zeta}(y)}_{\K^d}\d y-\int_{Y}\scp{\left(a^\ast v_{\xi}\right)(y)}{\grad_\sharp w_{\zeta}(y)}_{\K^d}\d y\\
 & =\int_{Y}\scp{v_{\xi}(y)}{\left(av_{\zeta}\right)(y)}_{\K^d}\d y-\scp{a^\ast v_{\xi}}{\grad_\sharp w_{\zeta}}_{\L(Y)^{d}}\\
 & =\int_{Y}\scp{v_{\xi}(y)}{\left(av_{\zeta}\right)(y)}_{\K^d}\d y\\
 & =\int_{Y}\scp{\grad_\sharp w_{\xi}(y)+\xi}{\left(av_{\zeta}\right)(y)}_{\K^d}\d y\\
 & =\int_{Y}\scp{\xi}{\left(av_{\zeta}\right)(y)}_{\K^d}\d y=\scp{\xi}{a_{\hom}\zeta}_{\K^d}.
\end{align*}
\item $\Re a_{\hom}$ is strictly positive definite. As above, one shows
\[
\Re\scp{\xi}{a_{\hom} \xi}_{\K^d}=\Re\int_{Y}\scp{v_{\xi}(y)}{(av_{\xi})(y)}_{\K^d}\d y\geq c\norm{v_\xi}^2_{\L(Y)^d}\quad (\xi\in \K^d)
\]
and since the right-hand side is strictly positive if $\xi\ne 0$ by \prettyref{lem:ExistenceLemma}, we derive the assertion.
\end{enumerate}
\end{rem}

The construction of $a_{\hom}$ now allows us to formulate
the main result of this section.

\begin{thm}
\label{thm:homo} We have
\[
\mathfrak{a}_n = \left(\iota_{\sharp}^{*}a(n\m)\iota_{\sharp}\right)^{-1}\to\left(\iota_{\sharp}^{*}a_{\hom}\iota_{\sharp}\right)^{-1}\eqqcolon \mathfrak{a}_{\hom} \quad(n\to\infty)
\]
in the weak operator topology of $\bo(\ran(\grad_{\sharp}))$.
\end{thm}

The proof of \prettyref{thm:homo} requires some more preparation.
One of the results needed is a variant of \prettyref{thm:periodic_homo}
for $\L(Y)$. However, it will be beneficial to finish \prettyref{exa:wave_homo} first.
\begin{example}[\prettyref{exa:wave_homo} continued]
 The operator sequence $(S^{(n)})_{n}$ converges in the strong operator
topology of $\bo\bigl(\Lnu\bigl(\R;\L(Y)\times\ran(\grad_{\sharp})\bigr)\bigr)$ to the
following limit 
\[
\left(\overline{\td{\nu}\begin{pmatrix}
1 & 0\\
0 & \mathfrak{a}_{\hom}
\end{pmatrix}+\begin{pmatrix}
0 & \dive_{\sharp}\iota_{\sharp}\\
\iota_{\sharp}^{*}\grad_{\sharp} & 0
\end{pmatrix}}\right)^{-1}.
\]
\end{example}

\begin{lem}
\label{lem:homo} Let $f\colon\R^{d}\to\K$ be measurable and $\roi{0}{1}^{d}$-periodic.
Let $\Omega\subseteq\R^{d}$ be open, bounded and assume $f|_{Y}\in\L(Y)$.
Then 
\[
f(n\cdot)\to\Bigl(\int_{Y} f\Bigr) \1_\Omega
\]
weakly in $\L(\Omega)$ as $n\to\infty$.
\end{lem}

\begin{proof} Due to the boundedness of $\Omega$ we find a finite
set $F\subseteq\Z^{d}$ such that $\Omega\subseteq\bigcup_{k\in F}k+Y$.
Thus, by periodicity, it suffices to restrict our attention to the case when $\Omega=Y$. 
First of all we show that $(f(n\cdot))_{n}$ is a bounded sequence
in $\L(Y)$. For this, we compute  
\begin{align*}
\int_{Y}\abs{f(nx)}^{2}\d x & =\frac{1}{n^{d}}\int_{nY}\abs{f(y)}^{2}\d y
 =\frac{1}{n^{d}}n^{d}\int_{Y}\abs{f(y)}^{2}\d y=\norm{f}_{\L(Y)}^{2},
\end{align*}
where we used periodicity again. This chain of equalities also shows with the help of a density argument, that it suffices to assume $f|_Y$ to be a simple function. Note that in this case $f \in L_\infty(\R^d)$. 

Finally, we note that by \prettyref{thm:periodic_homo} 
\[
 \scp{f(n\cdot)}{g}_{\L(Y)}\to \scp{\Bigl(\int_Y f\Bigr)\1_{Y}}{g}_{\L(Y)}\quad (n\to \infty)
\]
for each $g\in  \L(Y)\subseteq L_1(Y)$, which implies the desired assertion.
\end{proof}

\begin{lem}
\label{lem:dcl_proto}Let
$(q_{n})_{n}$ and $\left(r_{n}\right)_{n}$ be weakly convergent
sequences in a Hilbert space $H$ with weak limits $q,r\in H$, respectively. Moreover,  let $X\subseteq H$ be a closed subspace and $\iota\from X\to H$ the canonical embedding. Assume that
\[
q_{n}\in X\text{ for each }n\in \N \text{ and }\left(\iota^{*}r_{n}\right)_{n}\text{ is strongly convergent in }X.
\]
Then
\[
\lim_{n\to\infty}\scp{r_{n}}{q_{n}}_{H}=\scp{r}{q}_{H}.
\]
\end{lem}

\begin{proof}
Since $\iota^\ast\from H\to X$ is continuous it is also weakly continuous, and thus, 
\[
 \iota^\ast r_n \to \iota^\ast r\quad (n\to \infty)
\]
strongly in $X$. For $n\in\N$ we compute
\begin{align*}
\scp{r_{n}}{q_{n}}_{H} & =\scp{r_{n}}{\iota\iota^{*}q_{n}}_{H}=\scp{\iota^{*}r_{n}}{\iota^{*}q_{n}}_X
 \to\scp{\iota^\ast r}{\iota^{*}q}_{X}.
\end{align*}
Since $X$ is a closed subspace, it is also weakly closed and thus $q\in X$ which yields
\[
 \scp{\iota^\ast r}{\iota^{*}q}_{X}=\scp{r}{q}_{H}.\qedhere
\]
\end{proof}
The next theorem is a version of the so-called `div-curl lemma'.
\begin{thm}
\label{thm:dcl_abst}Let $(q_{n})_{n}$ and $(r_{n})_{n}$ be weakly
convergent sequences in $\L(Y)^{d}$ to some $q,r\in \L(Y)^d$, respectively.
Assume that
\[
q_{n}\in\ran(\grad_{\sharp}) \text{ for each } n\in \N \text{ and }\left(\iota_\sharp^{*}r_{n}\right)_{n}\text{ is strongly convergent in } \ran(\grad_\sharp).
\]
Then
\[
\int_{Y}\scp{r_{n}(x)}{q_{n}(x)}_{\K^d}\phi(x)\d x\to\int_{Y}\scp{r(x)}{q(x)}_{\K^d}\phi(x) \d x
\]
for all $\phi\in\cci(Y)$ as $n\to\infty$. 
\end{thm}

\begin{proof}
Let $\phi\in\cci(Y)$, $n\in\N$. Since $q_n \in \ran(\grad_\sharp)$, we find a unique $w_{n}\in H^1_\sharp(Y)$ with $w_n \in \{\1_Y\}^\bot=\ker(\grad_\sharp)^\bot$ such that 
\[
\grad_{\sharp}w_{n}=q_{n}.
\]
Moreover, since $\grad_\sharp\from H^1_\sharp(Y)\cap \{\1_Y\}^\bot \to \ran(\grad_\sharp)$ is an isomorphism by \prettyref{prop:Poincare_periodic}, we infer that $(w_n)_{n}$ is a weakly convergent sequence in $H^1_\sharp(Y)$ and denote its weak limit by $w\in H^1_\sharp(Y)$.  By \prettyref{thm:RellichII}, we deduce $w_{n}\to w$
strongly in $\L(Y)^{d}$. Moreover, note that $(\phi w_{n})_{n}$
weakly converges to $\phi w$ in $H_{\sharp}^{1}(Y).$ In
particular, $\grad_{\sharp}\left(\phi w_{n}\right)\to\grad_{\sharp}\left(\phi w\right)$
weakly in $\L(Y)^d$.  For $n\in \N$,  we compute
\begin{align*}
\int_{Y}\scp{r_{n}(x)}{q_{n}(x)}_{\K^d}\phi(x)\d x &= \scp{r_n}{q_n \phi}_{L(Y)^d} 
=\scp{r_n}{(\grad_\sharp w_n) \phi}_{L(Y)^d}\\
&=\scp{r_n}{\grad_\sharp (\phi w_n)}_{L(Y)^d}- \scp{r_n}{w_n \grad_\sharp\phi}_{\L(Y)^d}.
\end{align*}
Now, the first term on the right-hand side of this equality tends to $\scp{r}{\grad_\sharp (\phi w)}_{\L(Y)^d}$ by 
\prettyref{lem:dcl_proto} applied to $X=\ran(\grad_{\sharp})$, which is
closed by \prettyref{prop:Poincare_periodic}. The second term tends to $\scp{r}{w \grad_\sharp \phi}_{\L(Y)^d}$ by strong convergence of $(w_{n})_n$ and weak
convergence of $(r_{n})_{n}$ in $\L(Y)^d$. Thus, we obtain
\begin{align*}
 \int_{Y}\scp{r_{n}(x)}{q_{n}(x)}_{\K^d}\phi(x)\d x&\to \scp{r}{\grad_\sharp(\phi w)}_{\L(Y)^d}-\scp{r}{w \grad_\sharp \phi}_{\L(Y)^d}\\
 &=\int_{Y}\scp{r(x)}{q(x)}_{\K^d}\phi(x) \d x \quad(n\to\infty). \qedhere
\end{align*}
\end{proof}

We will apply the latter theorem to the concrete case when $r_n=a(n\m) q_n$ in order to determine the weak limit of $(a(n\m)q_n)_n$.  
\begin{lem}
\label{lem:JikovLemma}
Let $\left(q_{n}\right)_{n}$ and $\left(a(n\m)q_{n}\right)_{n}$
be weakly convergent in $\L(Y)^{d}$ to some $q$ and $r$, respectively.
Assume that
\[
q_{n}\in\ran(\grad_{\sharp}) \text{ for each } n\in \N \text{ and } \left(\iota_\sharp^\ast a(n\m)q_{n}\right)_{n} \text{ is strongly convergent in } \ran(\grad_\sharp).
\]
Then $r=a_{\hom}q$. 
\end{lem}

\begin{proof}
Let $\xi\in\K^{d}$ and choose $v\coloneqq v_\xi\in\L(Y)^{d}$ according to \prettyref{lem:ExistenceLemma} for $a^\ast$ instead of $a$; that is, $v-\xi \in \ran(\grad_\sharp)$ and $a^\ast(\m) v \in \ker (\dive_\sharp)$.  For $n\in \N$, we define
$v_{n}\coloneqq v_\rmpe(n\cdot)\in \L(Y)^d$.  Next, let $g\in C_{\sharp}^{\infty}(Y)$. Then
we compute
\begin{align*}
\scp{a^\ast(n\m)v_{n}}{\grad_{\sharp}g}_{\L(Y)^d} & =\int_{Y}\scp{a^\ast(nx)v_\rmpe(nx)}{\grad_{\sharp}g(x)}_{\K^d}\d x\\
 & =\frac{1}{n^{d}}\int_{nY}\scp{a^\ast(y)v_\rmpe(y)}{\left(\grad_{\sharp}g\right)(y/n)}_{\K^d}\d y\\
 &= \frac{1}{n^{d-1}} \int_{nY} \scp{a^\ast(y)v_\rmpe (y)}{(\grad g(\cdot/n))(y)}_{\K^d} \d y.
\end{align*}
In order to compute the last integral, we employ \prettyref{lem:C_sharp_core} and \prettyref{rem:approx_in_Kern_div_sharp} to find a sequence $(\phi_k)_{k\in \N}$ in $C_\sharp^\infty(Y)^d\cap \ker(\dive_\sharp)$ such that $\phi_k \to a^\ast(\m)v$ as $k\to \infty$ in $\L(Y)^d$. The latter implies
$(\phi_k)_\rmpe \to a^\ast(\m) v_\rmpe$ as $k\to \infty$ in $\L(nY)^d$ for each $n\in \N$ and $\dive (\phi_k)_\rmpe = 0$ for all $k\in\N$ by \prettyref{prop:periodicextension}.  
Thus, we obtain with integration by parts (note that the boundary terms vanish due to the periodicity of $\phi_k$ and $g$)
\begin{align*}
 \scp{a^\ast(n\m)v_{n}}{\grad_{\sharp}g}_{\L(Y)^{d}}&=\frac{1}{n^{d-1}}  \scp{a^\ast(\m)v_\rmpe}{(\grad g(\cdot/n))}_{\L(nY)^d} \\
 &=\frac{1}{n^{d-1}}  \lim_{k\to \infty}\scp{(\phi_k)_\rmpe}{(\grad g(\cdot/n))}_{\L(nY)^d}=0.
\end{align*}
Since $C_\sharp^\infty(Y)$ is a core for $\grad_\sharp$, we infer that $a^\ast(n\m)v_n \in \ran(\grad_\sharp)^\bot$ and hence,
\[
 \iota_\sharp^\ast a^\ast(n\m)v_n=0\quad (n\in \N).
\]
Moreover, we have $a^\ast(n\m)v_n \to \int_Y a^\ast v=(a^\ast)_\rmhom \xi$ weakly in $\L(Y)$ as $n\to \infty$ by \prettyref{lem:homo}. Thus, by \prettyref{thm:dcl_abst} applied to $q_n$ and $r_n\coloneqq a^\ast(n\m)v_n$, we deduce that for all $\phi\in\cci(Y)$
\[
 \lim_{n\to \infty}\int_Y \scp{a^\ast(nx)v_n(x)}{q_n(x)}_{\K^d} \phi(x)\d x= \int_Y \scp{(a^\ast)_\rmhom \xi}{q(x)}_{\K^d} \phi(x) \d x.
\]
On the other hand, $v_n \to \bigl(\int_Y v\bigr)\1_Y =\xi \1_Y$ weakly in $\L(Y)$ as $n\to \infty$ by \prettyref{lem:homo}, where $\int_Yv=\xi$ follows from $v-\xi \in \ran(\grad_\sharp)$. Thus, we can apply \prettyref{thm:dcl_abst} to $q_n\coloneqq v_n$ and $r_n\coloneqq a(n\m)q_n$ and obtain for all $\phi \in \cci(Y)$
\begin{align*}
 \lim_{n\to \infty}\int_Y \scp{a^\ast(nx)v_n(x)}{q_n(x)}_{\K^d} \phi(x)\d x&=\lim_{n\to \infty}\int_Y \scp{v_n(x)}{a(nx) q_n(x)}_{\K^d} \phi(x)\d x \\
 &= \int_Y \scp{\xi}{r(x)}_{\K^d} \phi(x) \d x.
\end{align*}
Thus, we have
\[
\int_Y \scp{(a^\ast)_{\hom} \xi}{q(x)}_{\K^d} \phi(x) \d x=\int_Y \scp{\xi}{r(x)}_{\K^d} \phi(x) \d x
\]
for each $\phi \in \cci(Y)$. Hence, we infer
\[
\scp{\xi}{r(x)}_{\K^d}=\scp{(a^\ast)_{\hom} \xi}{q(x)}_{\K^d}=\scp{\xi}{a_{\hom}  q(x)}_{\K^d}
\]
for almost every $x\in Y$, where we have used \prettyref{rem:a_hom}\ref{rem:a_hom:item:1}. Since the latter holds for each $\xi \in \K^d$, we deduce $r=a_{\hom} q$.
\end{proof}

\begin{proof}[Proof of \prettyref{thm:homo}]
Let $n\in\N$ and for $u\in\ran(\grad_{\sharp})$ we put $q_{n}\coloneqq \mathfrak{a}_{n} u$.
We need to show that $(q_{n})_{n}$ weakly converges to $\mathfrak{a}_{\hom} u$.
For this, we choose subsequences (without relabeling) such that both
$\left(q_{n}\right)_{n}$ and $\left(a(n\m)q_{n}\right)_{n}$ weakly converge
 to some $q$ and $r$, respectively. By definition, we have
$q_{n}\in\ran(\grad_{\sharp})$ and $\iota_{\sharp}^{*}a(n\m)q_{n}=u$ for each $n\in \N$.
Hence, by \prettyref{lem:JikovLemma}, we deduce $a_{\hom}q=r$. As $\ran(\grad_\sharp)$ is closed, it is also weakly closed, and hence, $q\in \ran(\grad_\sharp)$. Thus, we have 
\[
 \iota_\sharp^\ast a_{\hom} \iota_\sharp q= \iota_\sharp^\ast r,
\]
or equivalently
\[
 q= \mathfrak{a}_{\hom} \iota_\sharp^\ast r.
\]
Now, since $u=\iota_{\sharp}^{*}a(n\m)q_{n}\to \iota_\sharp^\ast r$ weakly, we infer
\[
 q=\mathfrak{a}_{\hom} u.
\]
A subsequence argument now yields the claim.
\end{proof}

\section{Comments}

The theory of finding partial differential equations as appropriate
limit problems of partial differential equations with highly oscillatory
coefficients is commonly referred to as `homogenisation'. The mathematical
theory of homogenisation goes back to the late 1960s and early 70s.
We refer to \cite{Bensoussan1978} as an early monograph wrapping
up the available theory to that date.

The usual way of addressing homogenisation problems is to look at
static (i.e., time-independent) problems first. The corresponding
elliptic equation is then intensively studied. Even though it might
be hidden in the derivations above, the `study of the elliptic problem'
essentially boils down to addressing the limit behaviour of $\mathfrak{a}_{n}$
as $n\to\infty$; see \cite{EGW17_D2N,W16_Gcon}. Consequently, generalisations
of the periodic case have been introduced. The periodic case (and
beyond) is covered in \cite{Bensoussan1978,Cioranescu1999}; non-periodic
cases and corresponding notions have been introduced in \cite{Spagnolo1967,Spagnolo1968}
and, independently, in \cite{Murat1997,Murat1978}.

An important technical tool to obtain results in this direction is
the $\dive$-$\curl$ lemma or the notion of `compensated
compactness'. In the above presented material, this is \prettyref{thm:dcl_abst};
the main difficulty to overcome is that of finding a limit of a product $\left(\scp{q_{n}}{r_{n}}\right)_{n}$
of weakly convergent sequences $\left(q_{n}\right)_{n},\left(r_{n}\right)_{n}$
in $\L(\Omega)^{3}$ for some open $\Omega\subseteq\R^{3}$. It turns
out that if $\left(\curl q_{n}\right)_{n}$ and $\left(\dive r_{n}\right)_{n}$
converge strongly in an appropriate sense, then $\int_{\Omega}\scp{q_{n}}{r_{n}}\phi$
converges to the desired limit for all $\phi\in\cci(\Omega)$. In
\prettyref{thm:dcl_abst} the $\curl$-condition is strengthened in
as much as we ask $q_{n}$ to be a gradient, which results in $\curl q_{n}=0$.
The $\dive$-condition is replaced by the condition involving $\iota_{\sharp}^{*}$,
which can in fact be shown to be equivalent, see \cite{W17_DCL}. The restriction to periodic boundary
value problems is a mere convenience. It can be shown that the arguments
work similarly for non-periodic boundary conditions, and even with the same
limit, see \cite[Lemma 10.3]{Tartar2009}.

There are many generalisations of the $\dive$-$\curl$ lemma. For
this, we refer to \cite{Briane2009} (and the references given
there) and to the rather recently found operator-theoretic perspective,
with plenty of applications not solely restricted to the operators $\dive$ and $\curl$, see \cite{W17_DCL,Pauly2019}.

The way of deriving the homogenised equation (i.e., the limit of $\mathfrak{a}_{n}$)
is akin to some derivations in \cite{JKO94,Cioranescu1999}. Further
reading on homogenisation problems can also be found in these references.
The first step of combining homogenisation processes and evolutionary
equations has been made in \cite{W11_P} and has had some profound
developments for both quantitative and qualitative results; see \cite{W16_H,FW17_1D,CW17_FH,W18_NHC}.

\section*{Exercises}
\addcontentsline{toc}{section}{Exercises}

\begin{xca}
\label{exer:normPDEconv}Under the same assumptions of \prettyref{thm:convPDE}
show 
\[
\norm{\left(\bigl(\overline{\td{\nu}M_{n}(\td{\nu})+A}\bigr)^{-1} - \bigl(\overline{\td{\nu}M(\td{\nu})+A}\bigr)^{-1}\right)\td{\nu}^{-1}}_{\bo(\Lnu(\R;H))}\to0.
\]
\end{xca}

\begin{xca}
\label{exer:kernel_grad} Let $\Omega \subseteq \R^d$ be open and $\varepsilon >0$. We define the set 
\[
 \Omega_\varepsilon\coloneqq \set{x\in \Omega}{\dist(x,\partial\Omega)>\varepsilon}.
\]
\begin{enumerate}
 \item\label{exer:kernel_grad:item:1} Let $(\phi_k)_{k\in \N}$ in $\cci(\R^d)$ be a $\delta$-sequence (cf.~\prettyref{exer:delta-seq}) and $u\in H^1(\Omega)$. We identify each function on $\Omega$ by its extension to $\R^d$ by $0$. Prove that for $k\in \N$ large enough, $\phi_k\ast u\in H^1(\Omega_\varepsilon)$ with 
 \[
  \grad(\phi_k\ast u)=\phi_k\ast \grad u \mbox{ on } \Omega_\varepsilon.
 \]
\item\label{exer:kernel_grad:item:2} Use \ref{exer:kernel_grad:item:1} to prove \prettyref{lem:kernel_grad}.
\end{enumerate}
\end{xca}

\begin{xca}
\label{exer:convPDEstrong} Prove the `subsequence argument': Let $X$ be a topological space and $(x_n)_n$ a sequence in $X$. Assume that there exists $x\in X$ such that each subsequence of $(x_n)_n$ has a subsequence converging to $x$. Show that $x_n\to x$ as $n\to \infty$.
\end{xca}

\begin{xca}
\label{exer:poincare}
Let $H_0, H_1$ be Hilbert spaces and $C\colon\dom(C)\subseteq H_0\to H_1$ be a closed linear operator such that $\dom(C)\hookrightarrow H_0$ compactly. Let $P_{\ker(C)^\bot}\from H_0\to H_0$ denote the orthogonal projection onto the closed subspace $\ker(C)^\bot$. Prove that there exists $c>0$ such that 
\[
 \forall u\in \dom(C):\, \norm{P_{\ker(C)^\bot} u}_{H_0} \leq c \norm{Cu}_{H_1}. 
\]
Apply this result to prove \prettyref{prop:Poincare_periodic}.
\end{xca}

\begin{xca}
\label{exer:compactness of adjoint}Let $H_{0},H_{1}$ be Hilbert
spaces. Let $C\colon\dom(C)\subseteq H_{0}\to H_{1}$ be closed and
densely defined. Assume that $\dom(C)\cap\ker(C)^{\bot}\hookrightarrow H_{0}$
compactly. Show that, then, $\dom(C^{*})\cap\ker(C^{*})^{\bot}\hookrightarrow H_{1}$
compactly.
\end{xca}

\begin{xca}
Let $\nu>0$, $\Omega=\roi{0}{1}^{d}$, $s\in L_{\infty}(\R)$ be $1$-periodic, $0\leq s\leq 1$,
and $a$ as in \prettyref{exa:wave_homo}. Show that
$(u_{n})_{n}$ in $\Lnu(\R;\L(Y))$ satisfying 
\[
\td{\nu}^{2}s(n\m)u_{n}+\td{\nu}(1-s(n\m))u_{n}-\dive_{\sharp}a(n\m)\grad_{\sharp}u_{n}=f
\]
for some $f\in\Lnu(\R;\L(Y))$ is convergent to some $u\in\Lnu(\R;\L(Y))$.
Which limit equation is satisfied by $u$?
\end{xca}

\begin{xca}
Let $(\alpha_{n})_{n}$ be a null-sequence in $\ci{0}{1}$ and let
$a$ be as in \prettyref{exa:wave_homo}. Show 
\begin{align*}
 & \Biggl(\overline{\begin{pmatrix}
\td{\nu} & 0\\
0 & \td{\nu}^{\alpha_{n}}\mathfrak{a}_{n}
\end{pmatrix}+\begin{pmatrix}
0 & \dive_{\sharp}\iota_{\sharp}\\
\iota_{\sharp}^{*}\grad_{\sharp} & 0
\end{pmatrix}}\Biggr)^{-1}\\
 & \to\Biggl(\overline{\begin{pmatrix}
\td{\nu} & 0\\
0 & \mathfrak{a}_{\hom}
\end{pmatrix}+\begin{pmatrix}
0 & \dive_{\sharp}\iota_{\sharp}\\
\iota_{\sharp}^{*}\grad_{\sharp} & 0
\end{pmatrix}}\Biggr)^{-1}
\end{align*}
in the strong operator topology. Show that if $f\in\Lm{-\mu}(\R;\L(Y)_{\bot})$,
where $\L(Y)_{\bot}\coloneqq\set{\phi\in\L(Y)}{\int_{Y}\phi=0}$ for
some small enough $\mu>0$, we have
\[
\Biggl(\overline{\begin{pmatrix}
\td{\nu} & 0\\
0 & \mathfrak{a}_{\hom}
\end{pmatrix}+\begin{pmatrix}
0 & \dive_{\sharp}\iota_{\sharp}\\
\iota_{\sharp}^{*}\grad_{\sharp} & 0
\end{pmatrix}}\Biggr)^{-1}\begin{pmatrix}
f\\
0
\end{pmatrix}\in\Lm{-\mu}\bigl(\R;\L(Y)\times\ran(\grad_{\sharp})\bigr).
\]
\end{xca}

\printbibliography[heading=subbibliography]

%



\backmatter

\cleardoublepage
\addcontentsline{toc}{chapter}{Index}
\printindex

%
%

\end{document}